\documentclass{amsart}[11pt]

\usepackage{etex}
\usepackage[lmargin=1in,rmargin=1in,tmargin=1in,bmargin=1in]{geometry}
\usepackage{amsmath,amsthm,amsfonts,amssymb,verbatim}
\usepackage{graphicx}
\usepackage{tabularx}
\usepackage{overpic}
\usepackage{tikz}\usetikzlibrary{decorations.markings}
\usetikzlibrary{arrows}
\usepackage{rotating}
\usepackage{hyperref}
\usepackage{framed}
\usepackage{xcolor}
\usepackage{placeins}
\usepackage{booktabs}
\usepackage{pinlabel}
 \usepackage{picins}
\usepackage{color}
\usepackage{multirow}
\usepackage{overpic}
\usepackage{colortbl}
\usepackage{geometry}
\usepackage{slashed}
\usepackage[font=footnotesize,labelfont=bf]{caption}
\usepackage{arydshln}
\usepackage{blkarray}

\usepackage{wrapfig}

\DeclareMathAlphabet{\mathpzc}{OT1}{pzc}{m}{it}

\def\co{\colon\thinspace}

\newcommand\scalemath[2]{\scalebox{#1}{\mbox{\ensuremath{\displaystyle #2}}}}

\newcommand{\Z}{\mathbb{Z}}
\newcommand{\Q}{\mathbb{Q}}
\newcommand{\R}{\mathbb{R}}
\newcommand{\F}{\mathbb{F}}
\newcommand{\C}{\mathbb{C}}

\newcommand{\RP}{\R P^1}
\newcommand{\sA}{\mathcal{A}}
\newcommand{\sB}{\mathcal{B}}
\newcommand{\sC}{\mathcal{C}}
\newcommand{\sI}{\mathcal{I}}
\newcommand{\sL}{\mathcal{L}}
\newcommand{\sE}{\mathcal{E}}
\newcommand{\sV}{\mathcal{V}}

\newcommand{\sF}{\mathcal{F}}

\newcommand{\sM}{\mathcal{M}}
\newcommand{\sP}{\mathcal{P}}

\newcommand{\bv}{\mathbf{v}}
\newcommand{\bx}{\mathbf{x}}
\newcommand{\by}{\mathbf{y}}
\newcommand{\bz}{\mathbf{z}}
\newcommand{\bw}{\mathbf{w}}
\newcommand{\half}{\frac{1}{2}}

\newcommand{\sAt}{\widetilde{\sA}}
\newcommand{\td}{\widetilde{\delta}}

\newcommand{\tN}{\widetilde{N}}

\DeclareMathOperator{\Hom}{Hom}
\DeclareMathOperator{\End}{End}

\DeclareMathOperator{\coker}{coker}

\newcommand{\bu}{\bullet}
\newcommand{\ci}{\circ}

\newcommand{\Tp}{T} 
\newcommand{\tildeT}{\widetilde{T}}
\newcommand{\barT}{\overline{T}}
\newcommand{\Tbar}{\barT}
\newcommand{\spinc}{\operatorname{Spin}^c}

\newcommand{\Ztwo}{\Z/2\Z}

\newcommand{\gr}{\operatorname{gr}}

\newcommand{\balpha}{\mathbf{\alpha}}
\newcommand{\bbeta}{\mathbf{\beta}}
\newcommand{\rhoarrow}{\vec{\rho}}

\newcommand{\bargamma}{\overline{\gamma}}
\newcommand{\bgamma}{{\boldsymbol \gamma}}

\newcommand{\tracks}{{\boldsymbol\vartheta}}
\newcommand{\bartracks}{{\overline{\boldsymbol\vartheta}}}
\newcommand{\barp}{\overline{p}}

\newcommand{\rk}{\operatorname{rk}}
\newcommand{\im}{\operatorname{im}}

\newcommand{\HF}{\mathit{HF}}
\newcommand{\CF}{\mathit{CF}}

\newcommand{\HFhat}{\widehat{\mathit{HF}}}

\newcommand{\CFhat}{\widehat{\mathit{CF}}}

\newcommand{\HFKhat}{\widehat{\mathit{HFK}}}

\newcommand{\CFD}{\widehat{\mathit{CFD}}}
\newcommand{\CFA}{\widehat{\mathit{CFA}}}

\newcommand{\HFK}{\widehat{\mathit{HFK}}}

\newcommand{\HFa}{\mathit{HF}}

\newcommand{\Alg}{\mathcal{A}}
\newcommand{\Ainfty}{\mathcal{A}_\infty}

\newcommand{\AlgExt}{\widetilde{\Alg}}

\newcommand{\spin}{\mathfrak{s}}

\newcommand{\spinbar}{\overline{\spin}}

\newcommand{\curves}[1]{\HFhat({#1})}

\newcommand{\barcurves}[1]{\HFhat({#1})}
\DeclareMathOperator{\Sym}{Sym}

\def\Extend{\mathcal{D}}	
\def\Curves{\mathcal{C}}	

\newcommand{\wtp}{\check{w}}
\newcommand{\wtm}{\hat{w}}

\newtheorem{theorem}{Theorem}[section]

\newtheorem{corollary}[theorem]{Corollary}
\newtheorem{proposition}[theorem]{Proposition}
\newtheorem{lemma}[theorem]{Lemma}
\newtheorem{conjecture}[theorem]{Conjecture}

\newtheorem*{namedtheorem}{\theoremname}
\newcommand{\theoremname}{testing}

\theoremstyle{definition}
\newtheorem{definition}[theorem]{Definition}
\newtheorem{question}[theorem]{Question}

\newtheorem{remark}[theorem]{Remark}
\newtheorem{example}[theorem]{Example} 

\title[Bordered Floer homology via immersed curves]{Bordered Floer homology for manifolds with torus boundary via immersed curves}
\date{March 8, 2023}

\author[Jonathan Hanselman]{Jonathan Hanselman}
\address {Department of Mathematics, Princeton University.\newline \it{E-mail address:} \tt{jh66@princeton.edu}}

\author[Jacob Rasmussen]{Jacob Rasmussen}
\address {Department of Pure Mathematics and Mathematical Statistics, University of Cambridge.\newline \it{E-mail address:} \tt{J.Rasmussen@dpmms.cam.ac.uk}}

\author[Liam Watson]{Liam Watson}
\thanks{JH was partially supported by NSF RTG grant DMS-1148490; JR was partially supported by EPSRC grant EP/M000648/1; LW was partially supported by a Marie Curie career integration grant, by a CIRGET research fellowship, and by a Canada Research Chair; JR and LW were Isaac Newton Institute program participants while part of this work was completed and acknowledge partial support from EPSRC grant EP/K032208/1; additionally, LW was partially supported by a grant from the Simons Foundation while at the Isaac Newton Institute}
\address {Department of Mathematics, University of British Columbia.\newline \it{E-mail address:} \tt{liam@math.ubc.ca}}

\begin{document}
\maketitle

\begin{abstract}
This paper gives a geometric interpretation of bordered Heegaard Floer homology for manifolds with torus boundary. If \(M\) is such a manifold, we show that the type D structure \(\CFD(M)\)   may be viewed as a set of immersed curves  decorated with  local systems in \(\partial M\). These curves-with-decoration are invariants of the underlying three-manifold up to regular homotopy of the curves and isomorphism of the local systems. Given two such manifolds and a homeomorphism \(h\) between the boundary tori, the Heegaard Floer homology of the closed manifold obtained by gluing with \(h\) is obtained from the Lagrangian intersection Floer homology of the curve-sets. This machinery has several applications: We establish that the dimension of \(\HFhat\) decreases under a certain class of degree one maps (pinches) and we establish that the existence of an essential separating torus gives rise to a lower bound on the dimension of \(\HFhat\). In particular, it follows that a prime rational homology sphere $Y$ with \(\HFhat(Y)<5\) must be geometric. Other results include a new proof  of Eftekhary's theorem that L-space homology spheres are atoroidal; a complete characterisation of toroidal L-spaces in terms of gluing data; and a proof of a conjecture of Hom, Lidman, and Vafaee on satellite L-space knots.
\end{abstract}

\section{Introduction}

Bordered Floer homology is a suite of invariants, introduced by Lipshitz, Ozsv\'ath and Thurston \cite {LOT}, assigned to three-manifolds with boundary. These invariants are particularly well-adapted to cut-and-paste techniques. In the case of manifold with torus boundary, this theory has been developed in various ways yielding effective combinatorial tools for studying certain classes of toroidal three-manifolds \cite{BGW2013,Hanselman-splice,Hanselman2016,HRRW,HW,HL2014,Levine2012}. The goal of this paper is to provide a geometric interpretation of bordered Floer homology for manifolds with torus boundary in terms of one-dimensional objects in the boundary of the manifold. 

\subsection{Bordered invariants as immersed curves} 
If $M$ is a closed orientable 3-manifold with torus boundary, we define
\(T_M\) to be the complement of \(0\) in \(H_1(\partial M;\R)/H_1(\partial M;\Z)\). The punctured torus \(T_M\) can be identified with 
the complement of a point \(z\) in \(\partial M\), and this identification is well-defined up to isotopy. 

In order to define the bordered invariant \(\CFD(M,\alpha, \beta)\), we must choose a parametrization
\((\alpha, \beta)\) of \(\partial M\). That is, \(\alpha\) and \(\beta\) are the cores of the \(1\)-handles 
in a handle decomposition of \(\partial M\). 
Starting from  $\CFD(M,\alpha,\beta)$ 
we will define a collection $\boldsymbol\gamma = \{\gamma^1,\ldots,\gamma^n\}$ of immersed closed curves
$\gamma^i\co S^1\looparrowright T_M$, each decorated with a local system $(V_i, \Phi_i)$ consisting of a finite dimensional
vector space over \(\F=\Z/2\Z\)  and an automorphism $\Phi_i\co V_i\to V_i$; we denote this data by $\curves M$.  

\begin{theorem}\label{thm:invariance}
The collection of decorated curves $\curves M$ is a well defined invariant of $M$, up to regular homotopy of curves and isomorphism of local systems.  In particular, $\curves M$ does not depend on the choice of parametrization $(\alpha,\beta)$.
\end{theorem}

This data is equivalent to bordered Floer homology: $\curves M$ is determined by $\CFD(M,\alpha,\beta)$ and, for any choice of parametrization $(\alpha,\beta)$, $\CFD(M,\alpha,\beta)$ can be recovered from $\curves M$.
 
Now suppose that \(M_0\) and \(M_1\) are manifolds as above, and that 
$Y=M_0\cup_h M_1$ is the closed manifold obtained by gluing them together by the 
orientation reversing homeomorphism $h\co\partial M_1\to\partial M_0$. 
The Heegaard Floer homology $\HFhat(Y)$ can be recovered from 
$\boldsymbol\gamma_0=\curves {M_0}$ and $\boldsymbol\gamma_1=\bar h(\curves{M_1})$, where
the homeomorphism $\bar h$ is the composition of $h$ with the elliptic involution on $\partial M_0$.
(In fact,  the invariant \( \curves{M}\)  is fixed by the elliptic
involution on \(\partial M\); this symmetry is established in the
 companion to this paper \cite{HRW-companion}. Thus \(\bar h\) can be replaced by \(h\) in the definition of \(\boldsymbol \gamma_1\).)

\begin{theorem}\label{thm:pairing}
For $Y = M_0\cup_h M_1$ as above, \[\HFhat(Y) \cong \HFa(\boldsymbol\gamma_0, \boldsymbol\gamma_1)\]
where $\HFa(\cdot,\cdot)$ 
is an appropriately defined version of the Lagrangian intersection Floer homology in \(T_{M_0}\).
\end{theorem}

\subsection{Trivial local systems} In order to illustrate the content of Theorem \ref{thm:invariance} and Theorem \ref{thm:pairing}, it is instructive to consider the case where the local systems are trivial. In this case the associated vector spaces are 1-dimensional and can be dropped from the notation; the resulting invariants are simply (systems of) immersed curves. As a very simple example, the invariant \(\curves{D^2\times S^1}\) consists of a  single closed circle parallel to the longitude \(\lambda = \partial D^2\times \{\text{pt}\}\).

\labellist
\tiny
\pinlabel {$z$} at 353 89
\endlabellist
\piccaption[]{The marked exterior of the figure eight knot, together with its bordered invariant as a pair of immersed curves. \label{fig:figure-eight-with-curves}}
\parpic[r]{
 \begin{minipage}{60mm}
 \centering
 \includegraphics[scale=0.3]{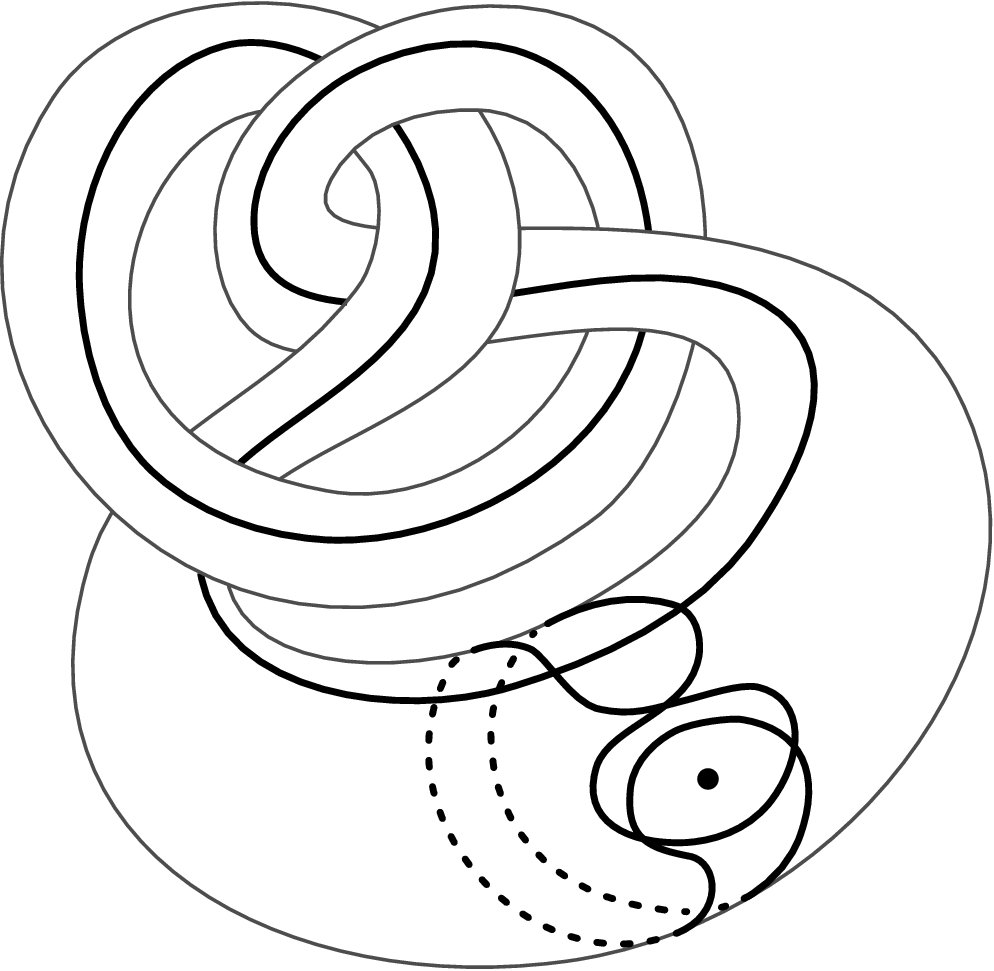}
  \end{minipage}%
}
A more interesting example is illustrated in Figure~\ref{fig:figure-eight-with-curves}: The invariant associated with the complement of the figure eight knot has two components. When \(\HFhat(M)\)  is more complicated, it is generally easier to represent it by drawing some curves in the cover \(\tildeT _M=  H_1(M;\R) \setminus H_1(M;\Z) \cong \R^2 \setminus \Z^2\) whose image in \(T_M\) is  \(\HFhat(M)\). From this point of view, the effect of orientation reversal on bordered invariants corresponds to reflection in the line determined by the longitude.   Figure \ref{fig:lifted-trefoil} shows the invariant associated with the right-hand trefoil. We note that, in this context, calculation is extremely efficient:

\begin{corollary}In the case of trivial local systems, following the notation of Theorem \ref{thm:pairing}, if no two components of $\boldsymbol\gamma_0$ and $\boldsymbol\gamma_1$ are parallel then $\dim\HFhat(Y)$ is the minimal geometric intersection number between $\boldsymbol\gamma_0$ and $\boldsymbol\gamma_1$.  \end{corollary}

\piccaption[]{A curve for the right-hand trefoil: the horizontal direction corresponds to the preferred longitude \(\lambda\), and the vertical direction to the standard meridian of the knot.\label{fig:lifted-trefoil}}
\parpic[r]{
 \begin{minipage}{60mm}
 \centering
 \includegraphics[scale=0.5]{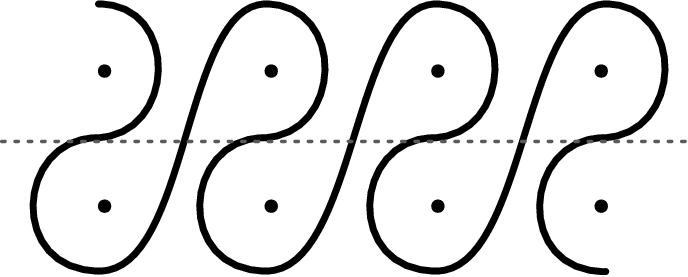}
  \end{minipage}%
}
For example, let \(Y_1\) be the manifold obtained by splicing the complements of the left-hand trefoil \(K_L\) and the right-hand trefoil \(K_R\); that is by taking \(M_0 = S^3\setminus \nu(K_L)\), \( M_1 = S^3 \setminus\nu(K_R)\), and \(h\co \partial M_0 \to \partial M_1\) such that 
\(h(\mu) = \lambda\) and \(h(\lambda) = \mu\), where \(\mu\) and \(\lambda\) are the preferred meridian and longitude of each trefoil in \(S^3\). Similarly, let \(Y_2\) be the manifold obtained by splicing two copies of the complement of the right-hand trefoil. Consulting Figure~\ref{fig:two-splices}, we see that \(\HFhat(Y_1)\) has dimension \(9\), while \(\HFhat(Y_2)\) has dimension \(7\), as calculated by Hedden and Levine \cite{HL2014}. 

\begin{figure}[ht]
\includegraphics[scale=0.5]{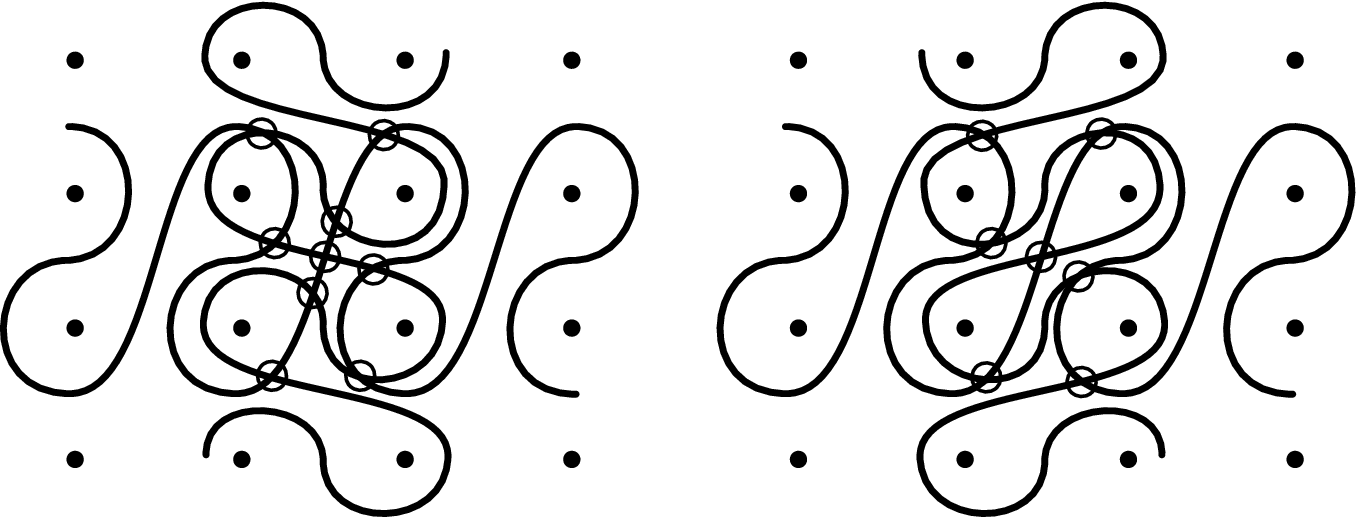}
\caption{Splicing trefoils: The diagram on the left illustrates  the intersection (carried out in $\tildeT_M$) calculating $\dim\HFhat(Y_1)=9$ where $Y_1$ is the splice of a right-hand and left-hand trefoil, while the diagram on the right illustrates the intersection calculating $\dim\HFhat(Y_2)=7$ where $Y_2$ is the splice of two right-hand trefoils.}\label{fig:two-splices}
\end{figure}

Manifolds for which the local systems are trivial are precisely the loop-type manifolds introduced by the first and third author \cite{HW}; the graphical formalism can thus be viewed as a geometric representation of the loop calculus. We remark that no explicit examples of three-manifolds are known for which the associated bordered invariant gives rise to a non-trivial local system. In practice---for instance for the gluing theorem needed in \cite{HRRW}---it is often enough to restrict attention to loop-type manifolds. As such, this class of manifolds seems interesting in its own right, and is discussed in further detail in a companion article \cite{HRW-companion}.

\subsection{Train tracks and a structure theorem} 
We now discuss the ideas behind the proof of Theorem~\ref{thm:invariance}. If \(M\) is as in the statement of the theorem,
the bordered Floer homology \(\CFD(M,\alpha, \beta)\) is a type D structure (in the sense of Lipshitz, Ozsv\'ath, and Thurston \cite{LOT}) over 
an algebra \(\sA\), known as the {\em torus algebra}. In this case, a  type D structure over \(\sA\) is simply a chain complex over 
\(\sA\) satisfying certain conditions. These are equivalent to projective differential modules  over $\sA$; see \cite[Remark 2.25]{LOT}.

In this paper, we will restrict attention to a certain class of type D structures over \(\sA\), 
which we call {\em extendable}. The precise definition is given in section~\ref{sec:extended-type-D-strs} but, briefly put, an {\it extended type D structure} is a curved differential module over a certain algebra \(\sAt\) that
has \(\sA\) as a quotient; these have certain properties in common with matrix factorizations. A type D structure \(N\) is extendable if there is an extended type D structure \(\widetilde{N}\) such that \(N \cong  \sA \otimes_{\sAt}  \widetilde{N} \). 
The following theorem, which is essentially due to Lipshitz, Ozsv\'ath and Thurston \cite[Chapter 11]{LOT}, shows that 
this is not much of a restriction. 

 \begin{theorem}\label{thm:extension} If \(M\) is a compact oriented manifold with torus boundary, 
 \(\CFD(M,\alpha, \beta)\) is extendable. \end{theorem}

We introduce a new graphical calculus, based on immersed train tracks in the punctured torus, 
which describes extended type D structures. The graphical calculus provides an effective practical method for reducing a given 
extended type D structure to a collection of curves decorated with local systems. Using it, we prove the following structure
theorem:

\begin{theorem}\label{thm:structure} Every extendable type D structure over $\Alg$ can be represented by a 
collection of immersed curves in the punctured torus, decorated with local systems.  \end{theorem}

The class of extendable type D structures known  to arise as $\CFD(M,\alpha, \beta)$ 
for some three-manifold \(M\) is considerably smaller than the set of all extendable type D structures.
Indeed, as mentioned above, we do not have any explicit examples of a manifold $M$ for 
which $\curves M$ has a non-trivial local system (though we expect that such $M$ should exist). Even in the case of trivial local systems there are curve sets that do not correspond to three-manifold invariants. For instance, Proposition \ref{prop:loose} shows that certain configurations of \emph{loose} curves (introduced in Section \ref{sub:STL}) do not arise as invariants of three-manifolds. This improves on a result due to Gillespie \cite{Gillespie}, and should be compared with work of Alishahi and Lipshitz \cite{AL2017}.

\subsection{Relation with the Fukaya category}
Theorem~\ref{thm:invariance} can be interpreted as saying that if $M$ is a manifold with torus boundary then \(\curves{M}\) is a compactly supported element of the Fukaya category \(\sF(T_M)\), while Theorem \ref{thm:pairing}
says that the Floer homology of a closed manifold  \(Y = M_0 \cup_h M_1\) is given by the \(\Hom\) pairing in the Fukaya category:
$$\HFhat(Y) = \Hom\big(\curves{M_0}, \bar{h}(\curves{M_1})\big)$$ The connection between the Heegaard Floer theory and the Fukaya category dates back to the introduction of Heegaard Floer homology, whose definition was motivated by a Seiberg-Witten analog of the Atiyah-Floer conjecture \cite{OSz2004-survey}. Following the introduction of bordered Floer homology \cite{Lipshitz-thesis,LOT},    Auroux \cite{Auroux2010-ICM, Auroux2010-Gokova} and Lekili and Perutz \cite{LP2011}  suggested that the Floer homology of a compact oriented \(M\) with \(\partial M = \Sigma_g\) should be an  object in \(\sF(\Sym^g(\Sigma_g\setminus z))\). In particular, Auroux showed that if \(\sP\) is a parametrization 
of \(\partial M\) (in other words, an identification of \(\partial M\) with a specific genus \(g\) surface \(F\) equipped with a handle
body decomposition), the bordered Floer homology \(\CFA(M,\sP)\) defines an object \(L(M,\sP)\) in the partially wrapped Fukaya category 
\(\sF(\Sym^g(F\setminus z))\), and that 
the pairing in bordered Floer homology is given by the Hom pairing. However it is unclear from this construction whether 
\(L(M,\sP)\) should be compactly supported. 

Auroux's construction is conceptually very useful, but for general \(g\) we don't have any way of making calculations in 
 \(\sF(\Sym^g(F\setminus z))\) other than the one provided by bordered Floer homology.  The one obvious  exception is  
the case \(g=1\) where, naively, one might expect that objects in  \(\sF(T^2\setminus z)\) are given by 
curves in the punctured torus. In fact, the situation is somewhat more complicated. The Fukaya category is triangulated, so a typical object actually has the form of an iterated mapping cone built out of geometric curves. In \cite{HKK}, Haiden, Katzarkov and Kontsevich give an especially nice algebraic model for the partially wrapped Fukaya category of a punctured surface \(\Sigma\); in the 
case of \(T^2\setminus z\), objects of \(\sF(T^2\setminus z)\) correspond to chain complexes over the algebra \(\sA\) or, equivalently, with type D structures. 

One of the main results of \cite{HKK} is a structure theorem for objects of \(\sF(\Sigma)\), which says that any object can be expressed as a direct sum of immersed, possibly noncompact curves.    Theorem~\ref{thm:structure} was motivated by this result but our proof is quite different. In particular,  the condition that the type D structure is extendable implies that the curves are all compactly supported. More generally, if \(\Sigma\) is a punctured surface with stops on each boundary component, objects of \(\sF(\Sigma)\) correspond to type D structures over an appropriate algebra. Our proof of Theorem~\ref{thm:structure} generalizes to show that if the type D structure is extendable, it is isomorphic to a disjoint union of compactly supported curves equipped with local systems. This gives a new proof of the structure theorem of Haiden, Katzarkov, and Kontsevich in the compactly supported case (see Theorem \ref{thm:higher-genus}), which is constructive and stays internal to the language of bordered Floer homology (namely, type D structures). This constructive approach has a key advantage: we obtain new information about Heegaard Floer theory and applications to 3-manifolds, as described below.
 
\subsection{Gradings}
There is a refined version of the invariant that takes spin$^c$ structures and the absolute \(\Z/2\Z\) grading on Heegaard Floer homology  into account. 

\begin{definition}
If \(M\) is as in Theorem~\ref{thm:invariance}, let \(\barT_M\) be the covering space of \(T_M\) whose 
fundamental group is the kernel of the composition
$$ \pi_1(T_M) \to \pi_1(\partial M) \to H_1(\partial M) \to H_1(M)$$
\end{definition}
Equivalently, if \(\lambda \in H_1(\partial M;\Z)\) generates the kernel of the inclusion \(j_*\co H_1(\partial M;\Z)\to H_1(M;\Z )\)
\(\barT_M\) is homeomorphic to the quotient \(((H_1(\partial M;\R) \setminus H_1(\partial M;\Z)))/\langle\lambda\rangle\). We let \(p\co\barT_M \to T_M\) be the covering map. 

The set of spin$^c$ structures on \(M\) can be identified with \(H^2(M) \simeq H_1(M, \partial M) \simeq \coker j_*\). We have:

\begin{theorem}
	\label{thm:gradings}
For each \(\spin \in \spinc(M)\) there is an invariant \(\barcurves{M,\spin}\), which is a collection of oriented immersed closed curves equipped with local systems
in \(\barT_M\).  Moreover,  \(\barcurves{M,\spin}\) is well-defined up to the action of translation by the deck group of \(p\), and  
$$\curves{M} = \bigcup_{\spin \in \spinc(M)} p(\barcurves{M,\spin}).$$
\end{theorem}

There is an analog of Theorem~\ref{thm:pairing} recovering the spin$^c$ decomposition of \(\HFhat(M_0 \cup_h M_1)\) from 
\( \barcurves{M_0,\spin_0}\) and \(\barcurves{M_1,\spin_1}\), where \(\spin_i\) runs over spin$^c$ structures on \(M_i\). 
(See Proposition~\ref{prop:refined-pairing} for a precise statement.) The \(\Z/2\Z\) grading on \(\HFhat(M_0 \cup_h M_1)\) is determined by the sign of the intersections 
between the oriented curves \(  \barcurves{M_0}\) and \(  \barcurves{M_1}\).

\subsection{Immediate consequences} The geometric interpretation of bordered Floer invariants described above has several  applications. For instance:

\begin{theorem}\label{thm:total-dim-toroidal} Let $Y$ be a  closed, orientable three-manifold. If $Y$ contains a separating essential  torus then $\dim\HFhat(Y) \ge 5$.\end{theorem}

This follows quickly from Theorem \ref{thm:pairing}. Indeed, a simple geometric argument shows that any two sufficiently non-trivial immersed curves in the punctured torus intersect in at least 5 points (a specific example realizing $\dim\HFhat(Y)= 5$ is illustrated in Figure \ref{fig:ZHS3-small}; see also Theorem \ref{thm:toroidal-detailed}).  Theorem \ref{thm:total-dim-toroidal} gives rise to an interesting geometric statement:

\begin{corollary} If $Y$ is a prime rational homology sphere with $\dim\HFhat(Y)<5$ then $Y$ is geometric.\end{corollary}

\begin{proof}Note that any essential torus in a rational homology sphere must be separating. Thus, if $\dim\HFhat(Y)<5$ then $Y$ is atoroidal and we may appeal to Perelman's resolution of Thurston's geometrization conjecture to conclude that $Y$ admits a geometric structure.
\end{proof}
Note that, more precisely, the geometric structure in question is either hyperbolic or it is one of 6 Seifert fibered geometries; see Scott \cite{Scott1983}, for example. As another immediate corollary of Theorem \ref{thm:total-dim-toroidal}, we obtain a new proof of a theorem of Eftekhary \cite{Eftekhary}:
\begin{corollary}\label{crl:ZHS3}
L-space integer homology spheres are atoroidal. \qed
	\end{corollary}
Recall that an L-space is a rational homology sphere for which $\dim\HFhat(Y)=|H_1(Y)|$. Corollary \ref{crl:ZHS3} is a part of a conjecture of Ozsv{\'a}th and Szab{\'o}; it implies that prime integer homology sphere L-spaces with infinite fundamental group---should examples exist---are hyperbolic.

We also obtain a statement about Khovanov homology. 
 
 \begin{corollary} If $L$ is a link in $S^3$ containing an essential Conway sphere then $\dim\widetilde{\mathit{Kh}}(L)\ge 5$. \end{corollary}
\begin{proof}
Denoting by $\Sigma_L$ the two-fold branched cover of $S^3$, branched along the link $L$, the presence of an essential Conway sphere in $L$ is equivalent to the existence of an essential separating torus in $\Sigma_L$. Thus $\dim\widetilde{\mathit{Kh}}(L)\ge\HFhat(\Sigma_L)\ge 5$, where the first inequality arises from the spectral sequence from the reduced Khovanov homology of (the mirror of) $L$ to Heegaard Floer homology \cite{OSz2005} and the second applies Theorem \ref{thm:total-dim-toroidal}. 
\end{proof}

It is worth comparing this result with the fact that the unknot \cite{KM2011} and the two trefoils \cite{BS} are characterized by their Khovanov homology. For statements about links, see \cite{HN2010, HN2013}. 

In another direction, we can establish strong behaviour for Heegaard Floer homology in the presence of certain degree one maps. Consider an integer homology sphere $Y$ containing an essential torus, and write $Y=M_0\cup_h M_1$ where the $M_i$ are (necessarily) integer homology solid tori.  Following our conventions above, $h(\lambda)$ must be a meridian in $\partial M_0$ when $\lambda$ is the longitude of $M_1$. The Dehn filling $Y_0=M_0(h(\lambda))$ is the result of replacing $M_1$ with a solid torus (called a {\em pinch}); there is a degree one map $Y\to Y_0$. 

\begin{theorem}\label{thm:degree-one-pinch}If $Y$ is a toroidal integer homology sphere and $Y_0$ is the result of a pinch on $Y$ then $\dim\HFhat(Y)\ge\dim\HFhat(Y_0)$. \end{theorem}

This follows from a more general statement about pinching along tori in rational homology spheres; see Theorem \ref{thm:pinch}. In particular, to the best of the authors' knowledge, this is the only result concretely treating the following:

\begin{question}If $Y\to Y_0$ is a degree one map between closed, connected, orientable three-manifolds, is it the case that $\dim\HFhat(Y)\ge\dim\HFhat(Y_0)$? \end{question}

\subsection{The L-space gluing theorem}

The geometric invariants defined in this paper can also be applied to the classification of L-spaces; in particular, we give a complete characterization of when gluing along a torus produces an L-space.  Again, an L-space is a rational homology sphere $Y$ for which $\dim\HFhat(Y)=|H_1(Y)|$. For a manifold with torus boundary  $M$, the set of L-space slopes of \(M\) is given by \[\sL_M=\{\alpha \,|\,  M(\alpha) \ \text{is an L-space}\}\subset \Q P^1\] where $M(\alpha)$ denotes Dehn filling along the slope $\alpha$.
Denote the interior of $\sL_M$ by $\sL_M^\circ$. 

The set $\sL_M$ is encoded by and easily extracted from the invariant $\curves M$. In particular, a necessary condition for $|\sL_M|>1$ is that the associated local system be trivial. The complement of $\sL_M$ is obtained by considering the minimal set of tangent lines to $\curves M$; see Theorem \ref{thm:slope_detection} for a precise statement. This provides a satisfying solution to a problem posed by Boyer and Clay \cite[Problem 1.9]{BC}. Toroidal L-spaces are then characterized as follows. 

\begin{theorem}\label{L-space-gluing-theorem}Let $Y=M_0\cup_h M_1$ be a 3-manifold where $M_i$ are boundary incompressible manifolds and $h\co \partial M_1\to\partial M_0$ is  an orientation reversing homeomorphism between torus boundaries. Then $Y$ is an L-space if and only if $\sL_{M_0}^\circ\cup h(\sL_{M_1}^\circ)=\Q P^1$.   \end{theorem}

This confirms a conjecture due to the first author \cite{Hanselman-splice}, and strengthens results in \cite{Hanselman2016,HRRW,HW,HL2014,RR}. In addition to the applications discussed below, the L-space gluing theorem plays a key role in the program to understand L-spaces arising as cyclic branched covers of knots in the three-sphere; see Gordon and Lidman \cite{GL2014-cor,GL2014} and Boileau, Boyer, and Gordon \cite{BBG}. A weaker version of the L-space gluing theorem was the key ingredient, along with work of Boyer and Clay \cite{BC}, in proving the L-space conjecture for graph manifolds \cite{HRRW} (an alternate, constructive, proof is given in \cite{Rasmussen2017}). We note that, in light of the L-space conjecture \cite{BGW2013}, the L-space gluing theorem makes some striking predictions about the behaviour of foliations on manifolds with torus boundary, as well as the behaviour of left-orders on the fundamental groups of these manifolds. Another notable application is given in the work of N{\'e}methi on links of rational singularities \cite{Nemethi2017}. 

Theorem \ref{L-space-gluing-theorem} allows us to refine Theorem \ref{thm:total-dim-toroidal} in the case that $Y$ is a toroidal L-space. In addition to ruling out $|H_1(Y)| \le 4$, we can easily enumerate all examples with $|H_1(Y)| < 7$; see Theorem \ref{thm:toroidal-detailed}. Note that this leads to another proof of Corollary \ref{crl:ZHS3}. Another quick consequence of Theorem \ref{L-space-gluing-theorem} is the following:

\begin{theorem}\label{thm:sat-intro} Suppose that $K$ is a satellite knot in $S^3$. If $K$ is an L-space knot (that is, $K$ admits non-trivial 
L-space surgeries) then both the pattern knot and the companion knot are L-space knots as well.\end{theorem}

This confirms a conjecture of Hom, Lidman, and Vafaee \cite[Conjecture 1.7]{HLV2014}. It has subsequently been used by Baker and Motegi \cite{BakerMotegi} to show that the pattern knot must be a braid in the solid torus (see also Hom \cite{Hom}).

\subsection*{Organization} This paper is a substantially revised version of our earlier preprint of the same name, which dealt only with loop-type manifolds. A subsequent paper \cite{HRW-companion} will discuss further properties of the invariant and give some examples. The current paper is laid out as follows. 

Section \ref{sec:tracks} summarizes the relevant background from bordered Floer homology (in the case of manifolds with torus boundary) and describes a geometric representation for these invariants in terms of immersed train tracks. The main new content in this section is a  graphical interpretation of the box tensor product---the chain complex $\CFA(M_0,\alpha_0,\beta_0)\boxtimes\CFD(M_1,\alpha_1,\beta_1)$---in terms of intersection between train tracks; see Theorem \ref{thm:pairtracks}. For train tracks that are immersed curves, this has a clear connection to Lagrangian intersection Floer theory. In general, however, the naive intersection of train tracks is not an invariant in any sense. Extracting invariants is the principal aim of the remaining sections.  

Section \ref{sec:extend} is devoted to algebraic issues concerning a well-behaved class of type D structures: we prove a structure theorem (Theorem \ref{thm:structure}) for extendable type D structures in two steps. We first reduce train tracks in this class to train tracks satisfying certain nice properties (Section \ref{subsec:simplify tracks}, particularly Proposition \ref{prop:curves-plus-crossovers}), and then show that such train tracks can be manipulated and ultimately interpreted in terms of immersed curves and local systems (Section \ref{sub:removing}, particularly Proposition \ref{prop:IncreaseDepth}). Indeed, the class of train tracks considered ultimately gives a geometric interpretation of the relevant local systems. The section concludes with Theorem \ref{thm:higher-genus}, which extends our techniques to surfaces of higher genus. 

Section \ref{sec:cat} establishes the equivalence between, and independence of choices made in the construction of, extendable type D structures and immersed curved with local systems. We note that, while only the existence of an extension is required for our purposes, it follows from the work in this section that  extendable type D structures have essentially unique extensions; see Proposition \ref{prop:unique-extension}.

Section \ref{sec:invariance} returns the focus to three-manifolds, completing the proof of Theorem \ref{thm:invariance} and Theorem \ref{thm:pairing}, while Section \ref{sec:gradings} explains the modifications to the invariant needed to recover the spin\(^c\) and \(\Z/2\Z\) gradings on \(\HFhat\), as described in Theorem~\ref{thm:gradings}. 

Section \ref{sec:applications} is devoted to the applications described above, and the paper concludes with a short Appendix containing the proof of Theorem \ref{thm:extension}.

\subsection*{Conventions and coefficients} All three-manifolds arising in this work are compact, connected, smooth, and orientable. Consistent with the set-up in bordered Floer homology \cite{LOT}, all Floer invariants in this work take coefficients in the two element field $\F$. We expect that the results we describe here should work over other fields (and perhaps over \(\Z\) as well) but setting it  up would require defining the bordered Floer homology of a manifold with torus boundary over \(\Z\). Unless explicitly stated otherwise, (singular) homology groups of manifolds should be assumed to take coefficients in $\Z$. 

\subsection*{Acknowledgements} The authors would like to thank Peter Kronheimer, Yank\i\ Lekili, Tye Lidman, Robert Lipshitz, Gabriel Paternain, Sarah Rasmussen, Ivan Smith, and Claudius Zibrowius for helpful discussions (some of them dating back a very long time). In addition, we would like to thank the referee for many thoughtful comments and suggestions, which led to improvements at a number of key points. Part of this work was carried out while the third author was visiting Montr\'eal as CIRGET research fellow, and part was carried out while the second and third authors were participants in the program {\em Homology Theories in Low Dimensions} at the Isaac Newton Institute.

\section{Bordered invariants as train tracks}\label{sec:tracks}

This section gives a geometric interpretation of bordered invariants in terms of immersed train tracks for manifolds with torus boundary. We begin with a brief review of the relevant notions from bordered Floer homology \cite{LOT}. A less terse introduction, with essentially the same notation used here, is given in \cite{HW}.

\subsection{Modules over the torus algebra}\label{sub:modules} Let $M$ be an orientable three-manifold with torus boundary, and choose oriented essential simple closed curves $\alpha,\beta$ in $\partial M$ with $\beta\cdot\alpha=1$. The ordered triple $(M,\alpha,\beta)$ is called a bordered three-manifold; the pair $(\alpha,\beta)$ may be regarded as a parametrization of the torus boundary, in the sense that $\langle\alpha,\beta\rangle$ specifies the peripheral subgroup $\pi_1(\partial M)$.  

\piccaption[]{A quiver for the torus algebra, with relations $\rho_2\rho_1=\rho_3\rho_2=0$. \label{fig:torus-algebra}}
\parpic[r]{
 \begin{minipage}{45mm}
 \centering
\begin{tikzpicture}[->,>=stealth',auto,node distance=3cm,thick,main node/.style={circle,draw}]
\node[main node] (1) {$\iota_0$};
  \node[main node] (2) [right of=1] {$\iota_1$};
  \path
  (1) edge[bend left=45] node [above] {\normalsize$\rho_1$} (2)
  (2) edge node [above] {\normalsize$\rho_2$} (1)
   (1) edge[bend right=45] node [above] {\normalsize$\rho_3$} (2);
\end{tikzpicture}
  \end{minipage}%
}
We focus on two (equivalent) bordered invariants assigned to a bordered manifold $(M,\alpha,\beta)$. These will take the form of certain modules over the torus algebra $\sA$. Various descriptions of this algebra are given by Lipshitz, Ozsv\'ath and Thurston \cite{LOT}, but for our purposes, recall that $\sA$ is generated, as an algebra over the two-element field $\F$, by elements $\rho_1$, $\rho_2$ and $\rho_3$ and idempotents $\iota_0$ and $\iota_1$. Multiplication in $\sA$ is described by the quiver depicted in Figure \ref{fig:torus-algebra} together with the relations $\rho_2\rho_1=\rho_3\rho_2=0$. The shorthand $\rho_{12}=\rho_1\rho_2$, $\rho_{23}=\rho_2\rho_3$, $\rho_{123} =\rho_1\rho_2\rho_3$ is standard. Let $\sI$ denote the subring of idempotents, and write $\boldsymbol 1=\iota_0+\iota_1$ for the unit in $\sA$. 

The relevant bordered invariants are  $\CFA(M,\alpha,\beta)$ and $\CFD(M,\alpha,\beta)$; both are invariants of the underlying bordered manifold up to homotopy \cite{LOT}. The former is a type A structure, that is, a right $A_\infty$-module over $\sA$. The latter is a type D structure, which consists of (1) a vector space $V$ (over $\F$) together with a splitting as a direct sum over a left action of the idempotents $V\cong \iota_0V\oplus\iota_1V$; and (2) a map $\delta^1\co V\to\sA\otimes_\sI V$ satisfying a compatibility condition ensuring that $\partial(a\otimes x) = a\cdot\delta(x)$ is a differential on $\sA\otimes_\sI V$ (with left $\sA$-action specified by $a\cdot(b\otimes x)=ab\otimes x$). In particular, $\sA\otimes_\sI V$ has the structure of a left differential module over $\sA$. We will sometimes confuse the type D structure $\CFD(M,\alpha,\beta)$ (consisting of $V$ and $\delta^1$) and this associated differential module. Given any type D structure there is a collection of recursively defined maps $\delta^k\co V\to \sA^{\otimes k}\otimes V$ where $\delta^0\co V\to V$ is the identity and $\delta^k=(\operatorname{id}_{\sA^{\otimes k-1}}\otimes\delta^1) \circ\delta^{k-1}$. The type D structure $\CFD(M,\alpha,\beta)$ is bounded if $\delta^k$ vanishes for all sufficiently large integers $k$. 

Given a bordered manifold $(M,\alpha,\beta)$ the associated type A and type D structures are related. In particular, according to \cite[Corollary 1.1]{LOT-bimodules}, $\CFA(M,\alpha,\beta)$ is dual (in an appropriate sense) to $\CFD(M,\alpha,\beta)$. However, the utility of the two structures, taken together, is a computable model for the  $A_\infty$ tensor product: Given bordered manifolds $(M_0,\alpha_0,\beta_0)$ and $(M_1,\alpha_1,\beta_1)$ consider the chain complex $\CFA(M_0,\alpha_0,\beta_0)\boxtimes\CFD(M_1,\alpha_1,\beta_1)$ obtained from $\CFA(M_0,\alpha_0,\beta_0)\otimes_\sI \CFD(M_1,\alpha_1,\beta_1)$ with differential defined by \[\partial^\boxtimes(x\otimes y) = \sum_{k=0}^\infty (m_{k+1}\otimes\operatorname{id})(x\otimes\delta^k(y))\]  and the requirement that $\CFD(M_1,\alpha_1,\beta_1)$ is bounded (this ensures that the above sum is finite). Then the pairing theorem of Lipshitz, Ozsv\'ath and Thurston asserts that the homology of $\CFA(M_0,\alpha_0,\beta_0)\boxtimes\CFD(M_1,\alpha_1,\beta_1)$ is isomorphic to the Heegaard Floer homology $\HFhat(M_0\cup_h M_1)$ of the closed manifold obtained from the homeomorphism $h\co \partial M_1\to \partial M_0$ specified by $h(\alpha_1)=\beta_0$ and $h(\beta_1)=\alpha_0$ \cite[Theorem 1.3]{LOT}. 

\subsection{Conventions for bordered manifolds}\label{sec:conventions}
In the framework given by Lipshitz, Ozsv\'ath and Thurston \cite{LOT}, the boundary parametrization of $M$ is recorded by specifying a diffeomorphism $\phi\co F \to \partial M$ where $F$ is a torus with a fixed handle decomposition. The image of the cores of the one-handles in $F$ correspond to the pair of $\alpha$-arcs $(\alpha_1^a,\alpha_2^a)$ in a bordered Heegaard diagram; completing this pair to curves in $\partial M$ gives a pair corresponding to our parametrizing curves $\alpha$ and $\beta$. To relate the two notations, we must specify which curve corresponds to each $\alpha$-arc in a bordered Heegaard diagram. We orient the arcs $\alpha_1^a$ and $\alpha_2^a$ so that, starting at the basepoint and following the boundary of the Heegaard surface, we pass the initial point of $\alpha_1^a$, the initial point of $\alpha_2^a$, the final point of $\alpha_1^a$, and the final point of $\alpha_2^a$. With this labelling, $\alpha$ in our notation corresponds to $-\alpha_1^a$ and $\beta$ corresponds to $\alpha_2^a$. Note that when gluing two bordered manifolds together by an orientation reversing diffeomorphism taking basepoint to basepoint, $\alpha_1^a$ must glue to $-\alpha_2^a$ and $\alpha_2^a$ must glue to $-\alpha_1^a$; thus in our notation, $\alpha$ glues to $\beta$ and $\beta$ glues to $\alpha$. Finally, the arcs defined above in a Heegaard diagram seem to satisfy $(-\alpha_1) \cdot \alpha_2 = 1$, which would imply that $\alpha\cdot\beta = 1$. However, the orientation of the Heegaard surface agrees with the opposite of the orientation on $\partial M$. Said another way, the surface in a bordered Heegaard diagram for $M$ should be interpreted as being viewed from \emph{inside} $M$, but we will generally look at $\partial M$ from \emph{outside} $M$ so that $\beta\cdot\alpha = 1$.

\subsection{Decorated graphs} 
\label{subsec:graphs and loops}
There is a convenient graph-theoretic shorthand for describing these bordered invariants  (see, for instance, \cite{HW, HL2014}). Let $\Gamma$ be a directed graph with vertex set $\sV_\Gamma$ and  edge set $\sE_\Gamma$. Suppose that every $v\in \sV_\Gamma$ is labelled with exactly one element from $\{\bu, \ci\}$ and every $e\in\sE_\Gamma$ is labelled with exactly one element from $\{\varnothing,1,2,3,12,23,123\}$. Moreover, suppose the edge labels are compatible with the vertex labels so that the only arrangements that occur are 
\raisebox{-4pt}{$\begin{tikzpicture}[thick, >=stealth',shorten <=0.1cm,shorten >=0.1cm] 
\draw[->] (0,0) -- (1,0);\node [above] at (0.45,-0.07) {$\scriptstyle{\varnothing}$};
\node at (0,0) {$\bu$};\node at (1,0) {$\bu$};\end{tikzpicture}$},
\raisebox{-4pt}{$\begin{tikzpicture}[thick, >=stealth',shorten <=0.1cm,shorten >=0.1cm] 
\draw[->] (0,0) -- (1,0);\node [above] at (0.45,-0.07) {$\scriptstyle{\varnothing}$};
\node at (0,0) {$\circ$};\node at (1,0) {$\circ$};\end{tikzpicture}$},
\raisebox{-4pt}{$\begin{tikzpicture}[thick, >=stealth',shorten <=0.1cm,shorten >=0.1cm] 
\draw[->] (0,0) -- (1,0);\node [above] at (0.45,-0.07) {$\scriptstyle{1}$};
\node at (0,0) {$\bu$};\node at (1,0) {$\circ$};\end{tikzpicture}$},
\raisebox{-4pt}{$\begin{tikzpicture}[thick, >=stealth',shorten <=0.1cm,shorten >=0.1cm] 
\draw[->] (0,0) -- (1,0);\node [above] at (0.45,-0.07) {$\scriptstyle{2}$};
\node at (0,0) {$\circ$};\node at (1,0) {$\bu$};\end{tikzpicture}$}, 
\raisebox{-4pt}{$\begin{tikzpicture}[thick, >=stealth',shorten <=0.1cm,shorten >=0.1cm] 
\draw[->] (0,0) -- (1,0);\node [above] at (0.45,-0.07) {$\scriptstyle{3}$};
\node at (0,0) {$\bu$};\node at (1,0) {$\circ$};\end{tikzpicture}$}, 
\raisebox{-4pt}{$\begin{tikzpicture}[thick, >=stealth',shorten <=0.1cm,shorten >=0.1cm] 
\draw[->] (0,0) -- (1,0);\node [above] at (0.45,-0.07) {$\scriptstyle{12}$};
\node at (0,0) {$\bu$};\node at (1,0) {$\bu$};\end{tikzpicture}$},
\raisebox{-4pt}{$\begin{tikzpicture}[thick, >=stealth',shorten <=0.1cm,shorten >=0.1cm] 
\draw[->] (0,0) -- (1,0);\node [above] at (0.45,-0.07) {$\scriptstyle{23}$};
\node at (0,0) {$\circ$};\node at (1,0) {$\circ$};\end{tikzpicture}$}, and
\raisebox{-4pt}{$\begin{tikzpicture}[thick, >=stealth',shorten <=0.1cm,shorten >=0.1cm] 
\draw[->] (0,0) -- (1,0);\node [above] at (0.45,-0.07) {$\scriptstyle{123}$};
\node at (0,0) {$\bu$};\node at (1,0) {$\circ$};\end{tikzpicture}$}.
Finally, we will require that for any pair of vertices $v_1$ and $v_2$ in a decorated graph and any element $I \in \{\varnothing,1,2,3,12,23,123\}$, there is an even number of length two directed paths from $v_1$ to $v_2$ for which the concatenation of the two edge labels is $I$. We will call a graph of this form an \emph{$\Alg$-decorated graph}.

\piccaption[]{A decorated graph. \label{fig:decorated-type-D}}
\parpic[r]{
 \begin{minipage}{45mm}
 \centering
 \begin{tikzpicture}[scale=0.75,>=stealth', thick] 
\def \radius {2cm} \def \outer {2.25cm}
 \node at ({360/(7) * (1- 1)}:\radius) {$\bu$};
  \node at ({360/(7) * (2- 1)}:\radius) {$\ci$};
  \node at ({360/(7) * (3 -1)}:\radius) {$\bu$};
   \node at ({360/(7) * (4 -1)}:\radius) {$\ci$};
    \node at ({360/(7) * (5 -1)}:\radius) {$\bu$};
     \node at ({360/(7) * (6 -1)}:\radius) {$\ci$};
      \node at ({360/(7) * (7 -1)}:\radius) {$\ci$};
  \draw[->,shorten <= 0.125cm, shorten >= 0.125cm]({360/7 * (1 - 1)}:\radius) arc ({360/7 * (1- 1)}:{360/7 * (2-1)}:\radius);
 \draw[<-,shorten <= 0.125cm, shorten >= 0.125cm]({360/7 * (2 - 1)}:\radius) arc ({360/7 * (2- 1)}:{360/7 * (3-1)}:\radius);
 \draw[->,shorten <= 0.125cm, shorten >= 0.125cm]({360/7 * (3 - 1)}:\radius) arc ({360/7 * (3- 1)}:{360/7 * (4-1)}:\radius);
 \draw[->,shorten <= 0.125cm, shorten >= 0.125cm]({360/7 * (4 - 1)}:\radius) arc ({360/7 * (4- 1)}:{360/7 * (5-1)}:\radius);
 \draw[->,shorten <= 0.125cm, shorten >= 0.125cm]({360/7 * (5 - 1)}:\radius) arc ({360/7 * (5- 1)}:{360/7 * (6-1)}:\radius);
 \draw[<-,shorten <= 0.125cm, shorten >= 0.125cm]({360/7 * (6 - 1)}:\radius) arc ({360/7 * (6- 1)}:{360/7 * (7-1)}:\radius);
 \draw[<-,shorten <= 0.125cm, shorten >= 0.125cm]({360/7 * (7 - 1)}:\radius) arc ({360/7 * (7- 1)}:{360/7 * (8-1)}:\radius);
  \node at ({360/(7) * (1- (1/2))}:2.45cm) {$\scriptstyle{123}$};
  \node at ({360/(7) * (2-(1/2))}:\outer) {$\scriptstyle{1}$};
  \node at ({360/(7) * (3-(1/2))}:\outer) {$\scriptstyle{3}$};
  \node at ({360/(7) * (4-(1/2))}:\outer) {$\scriptstyle{2}$};
  \node at ({360/(7) * (5-(1/2))}:\outer) {$\scriptstyle{1}$};
  \node at ({360/(7) * (6-(1/2))}:\outer) {$\scriptstyle{23}$};
  \node at ({360/(7) * (7-(1/2))}:\outer) {$\scriptstyle{3}$};
\end{tikzpicture}
  \end{minipage}%
  }
  An $\Alg$-decorated graph $\Gamma$ describes a type D structure over $\Alg$ as follows. The underlying vector space is generated by $\sV_\Gamma$, with the idempotent splitting specified by $\bu$ labels identifying the $\iota_0$ generators and $\ci$ labels identifying the $\iota_1$ generators. The map $\delta^1$ is determined by the edge labels: Given $e\in\sE_\Gamma$ with label $I\in \{\varnothing,1,2,3,12,23,123\}$ consider the source $x$ and target $y$ in $\sV_\Gamma$. Then $\rho_I\otimes y$ is a summand of $\delta^1(x)$, with the interpretation that $\rho_\varnothing = \boldsymbol 1$. A decorated graph (and its associated type D structure) is \emph{reduced} if none of the edges is labelled by $\varnothing$. An example is shown in Figure \ref{fig:decorated-type-D}, describing the bordered invariant $\CFD(M,\mu,\lambda)$ when $M$ is the complement of the right-hand trefoil, $\mu$ is the knot meridian and $\lambda$ is the Seifert longitude (see \cite{LOT}). Notice that the higher $\delta^k$ are determined by directed paths in $\Gamma$; for the example shown, there exist generators $x$ and $y$ such that $\delta^3(x)=\rho_3\otimes\rho_2\otimes\rho_1\otimes y$ corresponding to the unique directed path of length 3. Clearly, the associated type D structure is bounded if and only if the decorated graph contains no directed cycles. Note that the assumption that length two paths cancel mod 2 implies that $\partial^2 = 0$ in the corresponding differential module. 
  
\piccaption[]{  Relabelling by $1\leftrightarrow3$, $2\leftrightarrow2$, $3\leftrightarrow1$ to extract a type A structure from a decorated graph.\label{fig:decorated-type-A}}
\parpic[r]{
 \begin{minipage}{45mm}
 \centering
 {\begin{tikzpicture}[scale=0.75,>=stealth', thick] 
\def \radius {2cm} \def \outer {2.25cm}
 \node at ({360/(7) * (1- 1)}:\radius) {$\bu$};
  \node at ({360/(7) * (2- 1)}:\radius) {$\ci$};
  \node at ({360/(7) * (3 -1)}:\radius) {$\bu$};
   \node at ({360/(7) * (4 -1)}:\radius) {$\ci$};
    \node at ({360/(7) * (5 -1)}:\radius) {$\bu$};
     \node at ({360/(7) * (6 -1)}:\radius) {$\ci$};
      \node at ({360/(7) * (7 -1)}:\radius) {$\ci$};
  \draw[->,shorten <= 0.125cm, shorten >= 0.125cm]({360/7 * (1 - 1)}:\radius) arc ({360/7 * (1- 1)}:{360/7 * (2-1)}:\radius);
 \draw[<-,shorten <= 0.125cm, shorten >= 0.125cm]({360/7 * (2 - 1)}:\radius) arc ({360/7 * (2- 1)}:{360/7 * (3-1)}:\radius);
 \draw[->,shorten <= 0.125cm, shorten >= 0.125cm]({360/7 * (3 - 1)}:\radius) arc ({360/7 * (3- 1)}:{360/7 * (4-1)}:\radius);
 \draw[->,shorten <= 0.125cm, shorten >= 0.125cm]({360/7 * (4 - 1)}:\radius) arc ({360/7 * (4- 1)}:{360/7 * (5-1)}:\radius);
 \draw[->,shorten <= 0.125cm, shorten >= 0.125cm]({360/7 * (5 - 1)}:\radius) arc ({360/7 * (5- 1)}:{360/7 * (6-1)}:\radius);
 \draw[<-,shorten <= 0.125cm, shorten >= 0.125cm]({360/7 * (6 - 1)}:\radius) arc ({360/7 * (6- 1)}:{360/7 * (7-1)}:\radius);
 \draw[<-,shorten <= 0.125cm, shorten >= 0.125cm]({360/7 * (7 - 1)}:\radius) arc ({360/7 * (7- 1)}:{360/7 * (8-1)}:\radius);
  \node at ({360/(7) * (1- (1/2))}:2.45cm) {$\scriptstyle{321}$};
  \node at ({360/(7) * (2-(1/2))}:\outer) {$\scriptstyle{3}$};
  \node at ({360/(7) * (3-(1/2))}:\outer) {$\scriptstyle{1}$};
  \node at ({360/(7) * (4-(1/2))}:\outer) {$\scriptstyle{2}$};
  \node at ({360/(7) * (5-(1/2))}:\outer) {$\scriptstyle{3}$};
  \node at ({360/(7) * (6-(1/2))}:\outer) {$\scriptstyle{21}$};
  \node at ({360/(7) * (7-(1/2))}:\outer) {$\scriptstyle{1}$};
\end{tikzpicture} }
  \end{minipage}%
  }
In the reduced case, the duality between type D and type A structures is encoded in the $\Alg$-decorated graphs following an algorithm described by Hedden and Levine \cite{HL2014}. This takes the same interpretation for the underlying vector space (that is, the idempotent splitting according to vertex labels) but requires a different interpretation of the edge labels. First, one re-writes/re-interprets the edge labels according to the bijection $1\leftrightarrow3$ , $2\leftrightarrow2$, $3\leftrightarrow1$. Next, given any length $n$ directed path in $\Gamma$ with source vertex $x$ and target vertex $y$ we construct a sequence $I={I_1,\ldots,I_k}$ and assign a multiplication $m_{k+1}(x\otimes\rho_{I_1}\otimes\cdots\otimes\rho_{I_k})=y$. The sequence $I$ is constructed by forming a word in $\{1,2,3\}$ determined according to the labels of the directed edge read in order, and then regrouping to find the minimum $k$ such that each $I_j$ (for $1\le j \le k$) is an element of $\{1,2,3,12,23,123\}$. Thus, in our example, to the length 3 directed path we assign the label sequence  $I=\{123\}$ so that $m_2(x,\rho_{123})=y$ while the edge label $321$ (formerly $123$ in the original decorated graph) gives rise to a sequence $I=\{3,2,1\}$ and the product $m_4(x',\rho_3,\rho_2,\rho_1)=y'$.

\subsection{Train tracks}

Given a reduced $\Alg$-decorated graph \(\Gamma\) as in the previous section, we can immerse \(\Gamma\) in the torus, as described below. 
First, we fix a specific model for the punctured torus. 

\begin{definition}
The {\em marked torus} \(T\) is the oriented surface \(T=\R^2/\Z^2\) punctured at \(z=(1-\epsilon,1 - \epsilon)\). 
The images of the \(y\) and \(x\)-axes in \(T\) will be referred to as \(\alpha\) and \(\beta\) respectively. 
\end{definition}

We embed the vertices of \(\Gamma\) into \(T\) so that the vertices corresponding to $V_0=\iota_0V$ generators (the $\bu$ vertices) are distinct points on \(\alpha\) and the vertices corresponding to $V_1=\iota_1V$ generators (the $\circ$ vertices) are distinct points on \(\beta\). We then embed each edge in \(T\) according to its label, as shown in Figure \ref{fig:algebra-edges}. While each edge is embedded, different edges may intersect so the result is an immersion of $\Gamma$. We will denote the image of this immersion by $\tracks_\Gamma$. Note that the order in which vertices are embedded along $\alpha$ and $\beta$ is arbitrary, and the arcs in Figure \ref{fig:algebra-edges} may be replaced by any homotopic path, so $\tracks_\Gamma$ is defined up to regular homotopy. A concrete example is shown in Figure \ref{fig:naive-track-example}.

\begin{figure}[ht]
\labellist
\tiny
\pinlabel {$z$} at 63 63 
\pinlabel {$z$} at 176 63 
\pinlabel {$z$} at 291 63 
\pinlabel {$z$} at 403 63 
\pinlabel {$z$} at 516 63 
\pinlabel {$z$} at 628 63 
\small
\pinlabel {$1$} at 33 35
\pinlabel {$2$} at 161 35
\pinlabel {$3$} at 266 35
\pinlabel {$12$} at 369 29
\pinlabel {$23$} at 474 35
\pinlabel {$123$} at 593 35
\endlabellist
\includegraphics[scale=0.5]{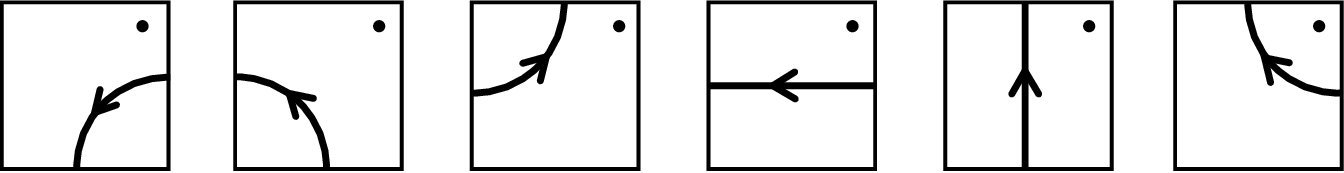}
\caption{Assigning edges in a reduced decorated graph $\Gamma$ to directed edges in the marked torus $T$. Notice that, by labelling corners of the (bordered) marked torus clockwise from the base point, the labels and orientations on edges may be suppressed without ambiguity. }\label{fig:algebra-edges}.
\end{figure}

We require that (1) all intersections between edges are transverse and away from the parametrizing curves $\alpha$ and $\beta$ and 
(2) near a vertex, all edges are orthogonal to the curve (\(\alpha\) or \(\beta\)) that the vertex lies on. This gives  $\tracks_\Gamma$ the structure of an immersed train track (following Thurston's notion of a train track in a surface, see for example \cite{Mosher2003}). An immersed train track is an immersed graph for which all edges incident at a vertex share a common tangent line. Following the usual terminology (inspired by railroad junctions), the vertices of an immersed train track are referred to as {\em switches}. An important point is that, just as $\Gamma$ need not be connected as a graph, the immersed train track $\tracks_\Gamma$, constituting equivalent data, may consist of a union of a collection of immersed train tracks.

A train track $\tracks_\Gamma$ constructed as above has a very particular form; for instance, it has no switches that are not on $\alpha$ or $\beta$. We will later have need to work with a more general class of immersed train tracks in $T$ (see Definition \ref{def:toroidal-train-track}), and we will call $\tracks_\Gamma$ as constructed here a \emph{naive train track} representing $\Gamma$. We remark that the edge orientations may be dropped from $\tracks_\Gamma$ without losing information, since they are determined by the rule that the basepoint lies on the right side of each arc, and thus $\tracks_\Gamma$ may be viewed as an unoriented train track. However, orientations will be meaningful for more general train tracks and we find it convenient to record them even in this case.

%
%

\begin{figure}[ht]
 \begin{tikzpicture}[scale=0.75,>=stealth', thick] 
\def \radius {3.5cm} \def \outer {3.9cm}
\node (d) at ({360/(11) * (1- 1)}:\radius) {$\bu$};
\node (j) at ({360/(11) * (2- 1)}:\radius) {$\ci$};
\node (a) at ({360/(11) * (3 -1)}:\radius) {$\bu$};
\node (k) at ({360/(11) * (4 -1)}:\radius) {$\ci$};
\node (f) at ({360/(11) * (5 -1)}:\radius) {$\bu$};
\node (g) at ({360/(11) * (6 -1)}:\radius) {$\ci$};
\node (h) at ({360/(11) * (7 -1)}:\radius) {$\ci$};
\node (c) at ({360/(11) * (8 -1)}:\radius) {$\bu$};
\node (e) at ({360/(11) * (9 -1)}:\radius) {$\bu$};
\node (i) at ({360/(11) * (10 -1)}:\radius) {$\ci$};
\node (b) at ({360/(11) * (11 -1)}:\radius) {$\bu$};

\draw[->, bend right = 15] (d) to (j);
\draw[->, bend left = 15] (a) to (j);
\draw[->, bend right = 15] (a) to (k);
\draw[->, bend right = 15] (k) to (f);
\draw[->, bend right = 15] (f) to (g);
\draw[->, bend left = 15] (h) to (g);
\draw[->, bend left = 15] (c) to (h);
\draw[->, bend right = 15] (c) to (e);
\draw[->, bend right = 15] (e) to (i);
\draw[->, bend left = 15] (b) to (i);
\draw[->, bend right = 15] (b) to (d);

\draw[->, bend right = 15] (d) to node[right]{$\scriptstyle{1}$} (i);
\draw[->, bend right = 30] (d) to node[right]{$\scriptstyle{12}$} (e);
\draw[->, bend right = 10] (d) to node[above]{$\scriptstyle{1}$} (h);
\draw[->, bend right = 10] (j) to node[above]{$\scriptstyle{23}$} (g);

  \node at ({360/(11) * (1- (1/2))}:\outer) {$\scriptstyle{1}$};
  \node at ({360/(11) * (2-(1/2))}:\outer) {$\scriptstyle{3}$};
  \node at ({360/(11) * (3-(1/2))}:\outer) {$\scriptstyle{123}$};
  \node at ({360/(11) * (4-(1/2))}:\outer) {$\scriptstyle{2}$};
  \node at ({360/(11) * (5-(1/2))}:\outer) {$\scriptstyle{1}$};
  \node at ({360/(11) * (6-(1/2))}:\outer) {$\scriptstyle{23}$};
  \node at ({360/(11) * (7-(1/2))}:\outer) {$\scriptstyle{3}$};
  \node at ({360/(11) * (8-(1/2))}:\outer) {$\scriptstyle{12}$};
  \node at ({360/(11) * (9-(1/2))}:\outer) {$\scriptstyle{1}$};
  \node at ({360/(11) * (10-(1/2))}:\outer) {$\scriptstyle{3}$};
  \node at ({360/(11) * (11-(1/2))}:\outer) {$\scriptstyle{12}$};

  \node at ({360/(11) * (0)}:\outer) {$\footnotesize{d}$};
  \node at ({360/(11) * (1)}:\outer) {$\footnotesize{j}$};
  \node at ({360/(11) * (2)}:\outer) {$\footnotesize{a}$};
  \node at ({360/(11) * (3)}:\outer) {$\footnotesize{k}$};
  \node at ({360/(11) * (4)}:\outer) {$\footnotesize{f}$};
  \node at ({360/(11) * (5)}:\outer) {$\footnotesize{g}$};
  \node at ({360/(11) * (6)}:\outer) {$\footnotesize{h}$};
  \node at ({360/(11) * (7)}:\outer) {$\footnotesize{c}$};
  \node at ({360/(11) * (8)}:\outer) {$\footnotesize{e}$};
  \node at ({360/(11) * (9)}:\outer) {$\footnotesize{i}$};
  \node at ({360/(11) * (10)}:\outer) {$\footnotesize{b}$};

\end{tikzpicture} \hspace{2 cm}
\labellist
\pinlabel {$a$} at 50 228
\pinlabel {$b$} at 50 192
\pinlabel {$c$} at 50 153
\pinlabel {$d$} at 50 116
\pinlabel {$e$} at 50 76
\pinlabel {$f$} at 50 37

\pinlabel {$a$} at 333 228
\pinlabel {$b$} at 333 192
\pinlabel {$c$} at 333 153
\pinlabel {$d$} at 333 116
\pinlabel {$e$} at 333 76
\pinlabel {$f$} at 333 37

\pinlabel {$g$} at 286 -10
\pinlabel {$h$} at 237 -10
\pinlabel {$i$} at 190 -10
\pinlabel {$j$} at 142 -10
\pinlabel {$k$} at 94 -10

\pinlabel {$g$} at 286 277
\pinlabel {$h$} at 237 277
\pinlabel {$i$} at 190 277
\pinlabel {$j$} at 142 277
\pinlabel {$k$} at 94 277

\tiny \pinlabel {$z$} at 300 245
\endlabellist
\raisebox{1cm}{\includegraphics[scale=0.5]{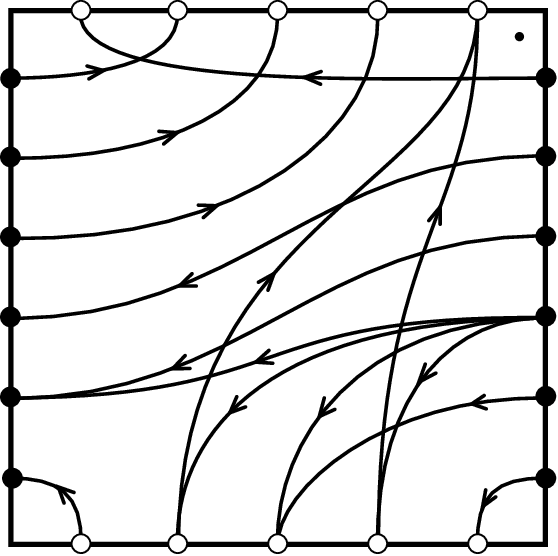}}
\caption{The train track $\tracks_\Gamma$ (right) associated with the $\Alg$-decorated graph $\Gamma$ shown on the left. All train track switches occur at the vertices on the $\alpha$ and $\beta$ curves; intersection points between edges of the graph are not (new) vertices of the graph. Note that the type D structure described is not known to come from a three-manifold.}\label{fig:naive-track-example}
\end{figure}


There is an interesting special case in which $\tracks_\Gamma$ (ignoring edge orientations) is simply a collection of immersed curves in $T$. This corresponds to the case when $\Gamma$ is a valence 2 graph and satisfies an extendability condition ensuring that the two arcs out of the switch of $\tracks_\Gamma$ associated with a given vertex leave the switch in opposite directions. The extendibility condition always holds for $\CFD(M,\alpha,\beta)$ (see Theorem \ref{thm:extension}); in the case where $\CFD(M,\alpha,\beta)$ admits a representative that may be described by a valence 2 decorated graph, $M$ is called loop-type; compare \cite{HW}. As an example, the trefoil  exterior is loop-type; see Figure \ref{fig:decorated-type-D}. 
It is a surprising fact that a great many classes of manifolds with torus boundary are loop-type, including any manifold admitting multiple L-space fillings \cite{HRRW,RR}; this class of manifolds is considered in detail in a companion paper \cite{HRW-companion}.  Indeed, we have no explicit example of a manifold that is not loop-type.

\subsection{Pairing train tracks}\label{sec:pairing-train-tracks} Just as an $\Alg$-decorated graph encodes both a type D structure and a type A structure over $\Alg$, we now establish our conventions for train tracks with respect to this duality. Let $\tracks$ be a train track that is a naive train track representing some reduced $\Alg$-decorated graph. Fix a standard marked torus $T$ and divide it into four quadrants. The \emph{type A realization} $A(\tracks)$ is obtained by including $\tracks$ (cut along $\alpha$ and $\beta$) into the first quadrant and adding unoriented horizontal edges in the second quadrant and unoriented vertical edges in the fourth quadrant connecting the two resulting copies of each switch; see Figure \ref{fig:typeDtypeA}. Similarly, the \emph{type D realization} $D(\tracks)$ is obtained by reflecting $\tracks$ across the anti-diagonal $y=-x$, including into the third quadrant, and adding unoriented vertical (resp. horizontal) edges in the second (resp. fourth) quadrant. In other words, $D(\tracks)$ is the reflection of $A(\tracks)$ across $y = -x$. We will refer to the edges of $A(\tracks)$ or $D(\tracks)$ in the first or third quadrants as \emph{corners}, since these are where paths carried by the train tracks can change direction, and the edges in the second and fourth quadrant are \emph{horizontal edges} or \emph{vertical edges}.

 \begin{figure}
\labellist
\tiny \pinlabel {$z$} at 111 241
\pinlabel {$z$} at 412 241
\pinlabel {$z$} at 717 241
\pinlabel {$z$} at 1025 241
\endlabellist
\includegraphics[scale=0.4]{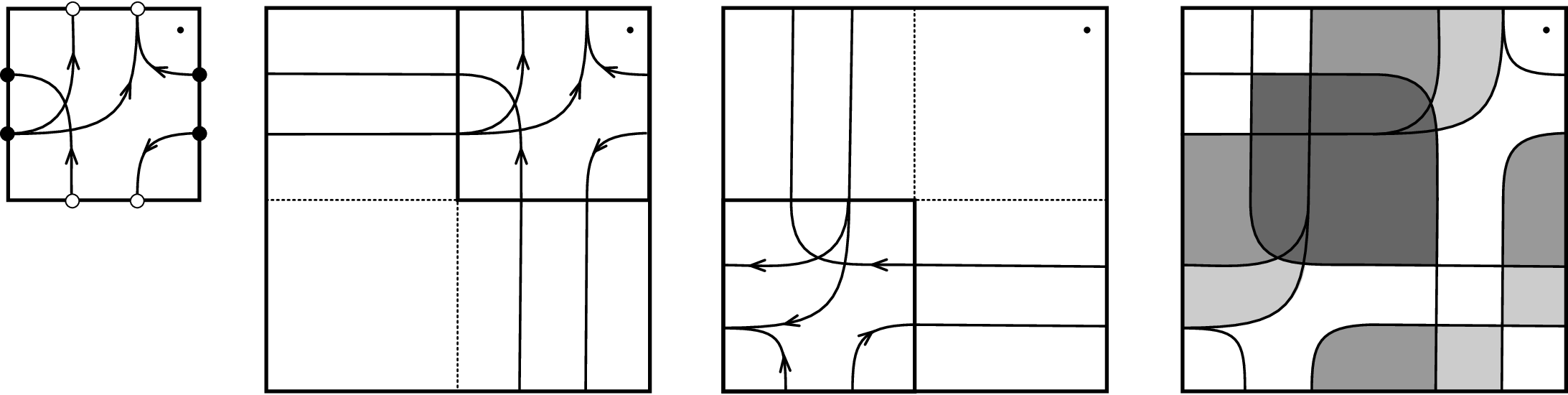}
\caption{From left to right: a sample train track $\tracks$, its type A realization $A(\tracks)$, its type D realization $D(\tracks)$, and the pairing complex $\sC(\tracks,\tracks)$ between the two. The 5 bigons contributing to the map $d^\tracks$ in this case have been shaded.}\label{fig:typeDtypeA}
\end{figure}

Now consider a pair of naive train tracks $\tracks_0$ and $\tracks_1$. 
We define $\sC(\tracks_0,\tracks_1)$ to be the vector space over \(\F\) generated  by the intersection points between \(A(\tracks_0)\) and \(D(\tracks_1)\). Note that these intersection points are confined to the second and fourth quadrants in $T$, by construction. We will define a linear map $d^\tracks\co\sC(\tracks_0,\tracks_1)\to\sC(\tracks_0,\tracks_1)$ by counting bigons.

If $x$ and $y$ are two intersection points between $A(\tracks_0)$ and $D(\tracks_1)$, we consider the set of continuous maps $u\co D^2 \to T$  satisfying the following conditions (compare with \cite[Definition 2.8]{Abouzaid2008}):

\begin{itemize}
\item $u$ is an orientation-preserving immersion away from $-i$ and $i$;
\item $u(-i) = x$ and $u(i) = y$ (viewing $D^2$ as the unit disk in $\C$);
\item $u$ maps the positive real part of $\partial D^2$ (oriented from $-i$ to $i$) to an oriented path in $A(\tracks_0)$;
\item $u$ maps the negative real part of $\partial D^2$ (oriented from $-i$ to $i$) to an oriented path in $D(\tracks_1)$;
\item both \(x\) and \(y\) are convex corners of the disk \(u\); and 
\item the basepoint $z \in T$ is not in the image of $u$.
\end{itemize}
We say two such maps \(u, u'\) are equivalent if there is an orientation preserving diffeomorphism $\phi:D^2 \to D^2$ with $u' = u\circ\phi$ and $\phi(\pm i) = \pm i$. Assuming the set of equivalence classes of such maps is finite, let   $n(x,y)$ denote the number of equivalence classes modulo 2. (We will see below that by making a small modification of \(\tracks_1\), we can arrange for the set of equivalence classes to be finite for all intersection points \(x,y\).) We then define

$$d^\tracks(x) = \sum_{y \in A(\tracks_0)\cap D(\tracks_1)} n(x,y) \, y.$$

We remark that an oriented path in $A(\tracks_0)$ or $D(\tracks_1$) may traverse the horizontal and vertical edges in either direction, since we view these as unoriented edges. In fact, just as we may drop the edge orientations on naive train tracks, we can also drop the requirement that paths from $-i$ to $i$ in $\partial D^2$ map to oriented paths in $A(\tracks_0)$ and $D(\tracks_1)$ in the present setting that $\tracks_0$ and $\tracks_1$ are naive train tracks. Indeed, the fact that every edge in $\tracks_0$ or $\tracks_1$ has the basepoint on its right implies that for any immersed disk not covering $z$ the paths determined by the boundary have the correct orientations. However, we state this requirement as it is relevant for generalizing this pairing to more general train tracks.

Under mild hypotheses, the vector space $\sC(\tracks_0,\tracks_1)$ with linear map $d^{\tracks}$ is, in fact, a chain complex: it may be identified with the box tensor product 
of the associated bordered invariants. To explain this, we need to introduce some variants of our train tracks (this will coincide with a notion of admissibility when we return 
our focus to three-manifold invariants later).  Recall that, to this point, we have assumed our train tracks come from reduced $\Alg$-decorated graphs (i.e graphs where no edges are labelled with $\varnothing$). Since we require a bounded type D structure for the box tensor product to make sense, we will need to relax the reduced assumption slightly.

\begin{figure}
\labellist
\pinlabel {\begin{tikzpicture}[thick, >=stealth',shorten <=0.1cm,shorten >=0.1cm] 
\draw[->] (0,1) -- (1.25,0.5);\node [above] at (0.65,0.70) {$\scriptstyle{1}$};
\draw[<-] (1.25,0.5) -- (0,0);\node [above] at (0.55,0.15) {$\varnothing$};
\draw[->] (0,0) -- (1.25,-0.5);\node [above] at (0.65,-0.70) {$\scriptstyle{2}$};
\node at (0,1) {$\bu$};\node at (1.25,0.5) {$\circ$};\node at  (0,0) {$\circ$};\node at (1.25,-0.5) {$\bu$};\end{tikzpicture}} at 65 -55
\pinlabel {\begin{tikzpicture}[thick, >=stealth',shorten <=0.1cm,shorten >=0.1cm] 
\draw[->] (0,1) -- (1.25,0.5);\node [above] at (0.65,0.70) {$\scriptstyle{2}$};
\draw[<-] (1.25,0.5) -- (0,0);\node [above] at (0.55,0.15) {$\varnothing$};
\draw[->] (0,0) -- (1.25,-0.5);\node [above] at (0.65,-0.70) {$\scriptstyle{3}$};
\node at (0,1) {$\circ$};\node at (1.25,0.5) {$\bu$};\node at  (0,0) {$\bu$};\node at (1.25,-0.5) {$\circ$};\end{tikzpicture}} at 250 -55
\pinlabel {\begin{tikzpicture}[thick, >=stealth',shorten <=0.1cm,shorten >=0.1cm] 
\draw[->] (0,1) -- (1.25,0.5);\node [above] at (0.65,0.70) {$\scriptstyle{1}$};
\draw[<-] (1.25,0.5) -- (0,0);\node [above] at (0.55,0.15) {$\varnothing$};
\draw[->] (0,0) -- (1.25,-0.5);\node [above] at (0.65,-0.70) {$\scriptstyle{23}$};
\node at (0,1) {$\bu$};\node at (1.25,0.5) {$\circ$};\node at  (0,0) {$\circ$};\node at (1.25,-0.5) {$\circ$};\end{tikzpicture}} at 430 -55
\pinlabel {\begin{tikzpicture}[thick, >=stealth',shorten <=0.1cm,shorten >=0.1cm] 
\draw[->] (0,1) -- (1.25,0.5);\node [above] at (0.65,0.70) {$\scriptstyle{12}$};
\draw[<-] (1.25,0.5) -- (0,0);\node [above] at (0.55,0.15) {$\varnothing$};
\draw[->] (0,0) -- (1.25,-0.5);\node [above] at (0.65,-0.70) {$\scriptstyle{3}$};
\node at (0,1) {$\bu$};\node at (1.25,0.5) {$\bu$};\node at  (0,0) {$\bu$};\node at (1.25,-0.5) {$\circ$};\end{tikzpicture}} at 605 -55
\tiny \pinlabel {$z$} at 114 115
\pinlabel {$z$} at 293 115
\pinlabel {$z$} at 468 115
\pinlabel {$z$} at 646 115
\endlabellist
\includegraphics[scale=0.5]{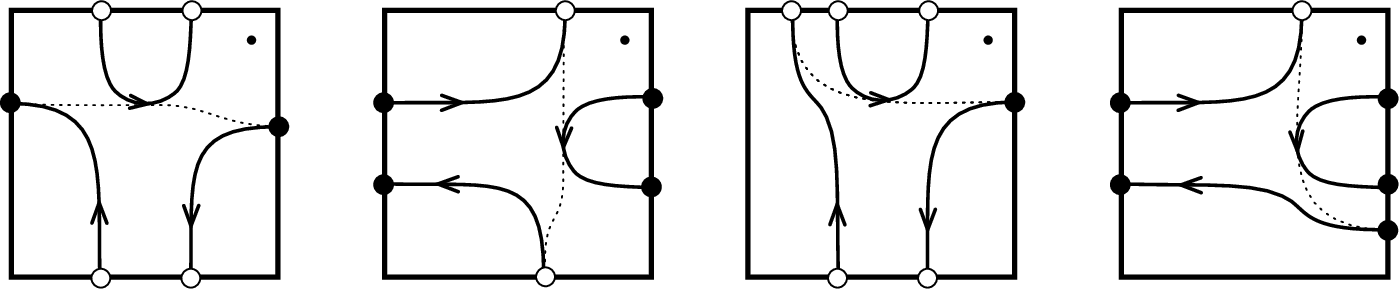}
\vspace*{50pt}
\caption{Modifications permitted in almost reduced train tracks. These local moves are useful for destroying oriented cycles and achieving admissibility in the proof of Theorem \ref{thm:pairtracks}.}\label{fig:almost}
\end{figure}

We will say an $\Alg$-deocrated graph is \emph{almost reduced} if the only edges labelled $\varnothing$ begin and end on valence 2 vertices and occur in one of precisely four configurations as shown in Figure \ref{fig:almost}. Note that each of these configurations can be replaced by a single arrow labelled $12$, $23$, $123$, or $123$, respectively, without changing the corresponding type D structure (up to homotopy). We can extend our naive train track construction to almost reduced $\Alg$-decorated graphs by embedding the edges labeled by $\varnothing$ as an edge starting and ending on the same side of the square $T \setminus (\alpha \cup \beta)$, as shown in the figure. Recall that an $\Alg$-decorated graph is bounded if it contains no oriented cycle. It is clear that we can replace any reduced graph by a bounded almost reduced graph representing a  type D structure homotopy equivalent to the original by replacing some number of $12$, $23$, or $123$ arrows by the corresponding zig-zag configuration. A train track will be called \emph{special bounded} if the underlying $\Alg$-decorated graph is both bounded and almost reduced.


\begin{theorem}\label{thm:pairtracks} Fix train tracks $\tracks_0$ and $\tracks_1$, where $\tracks_0$ is reduced and $\tracks_1$ is special bounded. Let $N^0_\Alg$ denote the type A structure associated with $\tracks_0$ and let ${}^\Alg\! N^1$ denote the type D structure associated with $\tracks_1$. Then $d^\tracks$ may be identified with $\partial^\boxtimes$ and there is an isomorphism of chain complexes \[\sC(\tracks_0,\tracks_1)\cong N^0_\Alg\boxtimes {}^\Alg\! N^1\]\end{theorem}

\begin{remark}\label{rmk:pairtracks}With the exception of the proof that follows, we will typically opt for subscripts on the given type D structures and make use of the shorthand notation $N^0_\Alg\boxtimes {}^\Alg\! N^1=N_0^A\boxtimes N_1$.\end{remark}
 
 \begin{proof}[Proof of Theorem \ref{thm:pairtracks}] First observe that there is an identification of the generating sets for $\sC(\tracks_0,\tracks_1)$ and $N^0_\Alg\boxtimes {}^\Alg\! N^1$, since each vertical (respectively	horizontal) segment in $A(\tracks_0)$ intersects each horizontal (respectively vertical) segment of $D(\tracks_1)$ and there are no other intersection points. Note that the horizontal segments of $A(\tracks_0)$ correspond directly to the $\iota_0$ generators of $N^0_\Alg$, while the vertical segments correspond to the $\iota_1$ generators. Similarly, the vertical segments of $D(\tracks_1)$ correspond directly to the $\iota_0$ generators of ${}^\Alg\! N^1$, while the horizontal segments correspond to the $\iota_1$ generators.
 
Since $\tracks_1$ is bounded, the associated type D structure ${}^\Alg\! N^1$ is bounded and hence $\partial^\boxtimes$ gives a well-defined differential on $N^0_\Alg\boxtimes {}^\Alg\! N^1$. We will identify the linear map $d^\tracks$ with the differential $\partial^\boxtimes$ by showing that that each immersed bigon contributing to $d^\tracks$ is uniquely associated with a pairing between type D and type A operations contributing to $\partial^\boxtimes$. 
 

We first aim to classify the immersed bigons that contribute to $d^\tracks$ by decomposing them into pieces. Observe that by cutting the torus $T$ along the curves $\alpha$ and $\beta$, any immersed bigon connecting intersection points $p$ to $q$ decomposes into embedded pieces. We will show that the only possible pieces are those shown in Figure \ref{fig:bigon-pieces}. We find it helpful to keep track of the type of corners on each boundary of the bigon (our convention for doing so will be explained shortly), and the dotted boxes in the top right and lower left quadrants in the figures are meant to highlight these corners. The top right quadrant will be called the corner region for $A(\tracks_0)$ and the lower left quadrant is the corner region for $D(\tracks_1)$.

\begin{figure}[t]
\labellist
\pinlabel {\bf Pieces of small bigons:} at 94 770
\pinlabel {\bf Starting pieces:} at 66 605
\pinlabel {\bf Ending pieces:} at 60 440
\pinlabel {\bf Connecting pieces:} at 79 277
\pinlabel {\bf Single piece bigon:} at 130 100
\scriptsize
\pinlabel {$\varnothing$} at 171 676
\pinlabel {$\varnothing$} at 464 670
\pinlabel {$3$} at 25 520
\pinlabel {$3$} at 165 520 \pinlabel {$2$} at 184 520 \pinlabel {$2$} at 214 551
\pinlabel {$1$} at 358 567
\pinlabel {$1$} at 501 567 \pinlabel {$2$} at 501 551
\pinlabel {$2$} at 472 520
\pinlabel {$3$} at 91 386
\pinlabel {$2$} at 185 355  \pinlabel {$2$} at 214 386 \pinlabel {$3$} at 233 386
\pinlabel {$1$} at 328 337
\pinlabel {$1$} at 472 337 \pinlabel {$2$} at 472 354  \pinlabel {$2$} at 503 386
\pinlabel {$1$} at 44 174 \pinlabel {$2$} at 44 190 \pinlabel {$1$} at 74 239 \pinlabel {$2$} at 74 223
\pinlabel {$3$} at 169 191 \pinlabel {$2$} at 187 191 \pinlabel {$2$} at 216 222 \pinlabel {$3$} at 235 222
\pinlabel {$2$} at 360 222 \pinlabel {$1$} at 330 174 \pinlabel {$2$} at 330 191 \pinlabel {$3$} at 312 191
\pinlabel {$2$} at 475 192 \pinlabel {$3$} at 522 222 \pinlabel {$2$} at 504 222 \pinlabel {$1$} at 504 239
\pinlabel {$2$} at 259 44 \pinlabel {$2$} at 289 74
\endlabellist
\vspace*{30pt}
\includegraphics[scale=0.6]{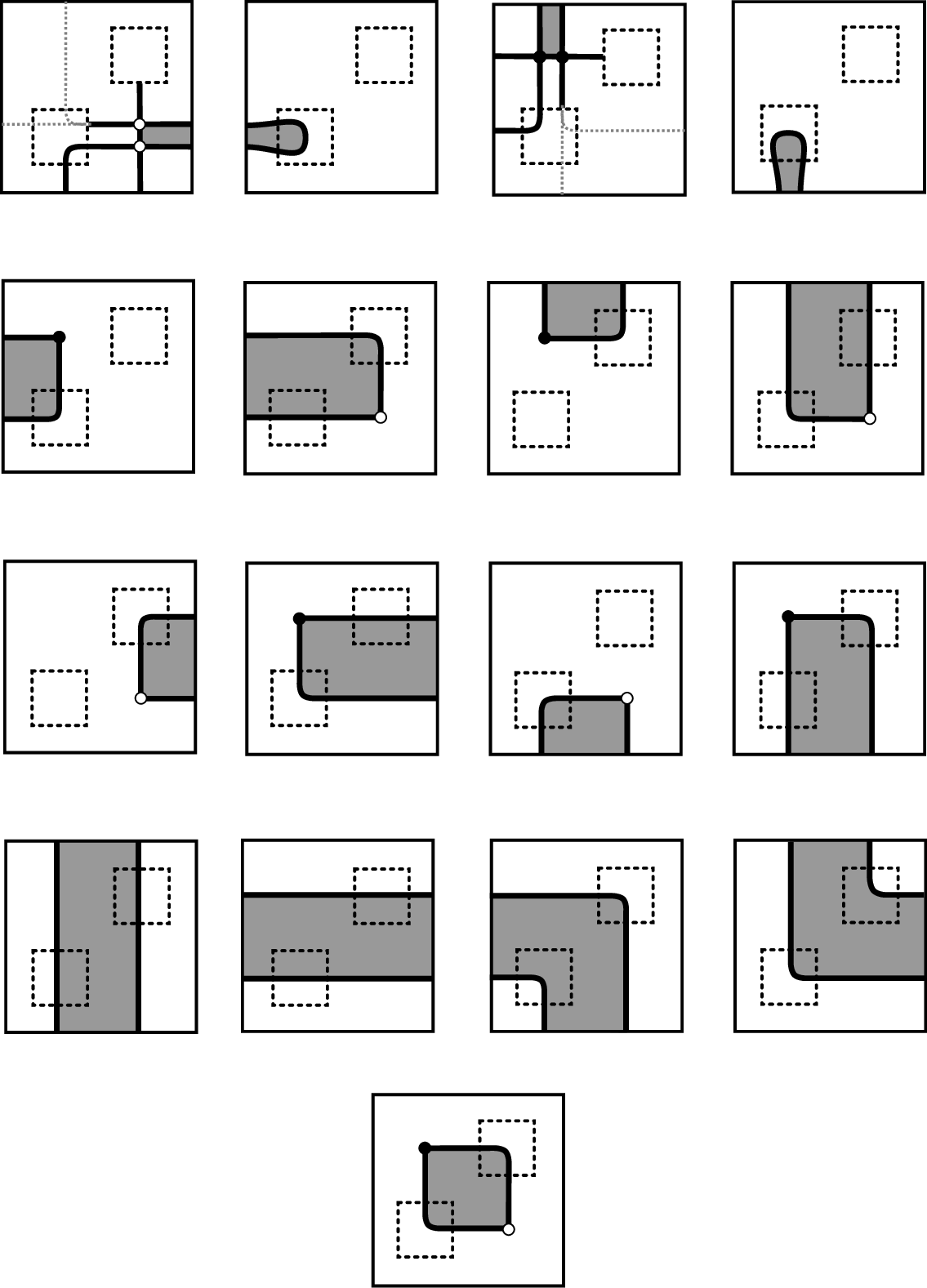}
\caption{Decomposing a bigon into tiles: Any small bigon carried by $(\tracks_0,\tracks_1)$ consists of exactly two pieces of either the first two types or the second two types in the first row. Any other bigon either is disjoint from $\alpha \cup \beta$ and has the form shown in the bottom row, or it decomposes into one starting piece, some number of connecting pieces, and one ending piece.}
\label{fig:bigon-pieces}
\end{figure}

The first row of the figure shows pieces arising from bigons which involve a $\varnothing$-labelled corner in $D(\tracks_1)$ (these are corners that start and end on the same side of the corner region, corresponding to $\varnothing$-labelled edges in the $\Alg$-decorated graph associated to $\tracks_1$); we will call such bigons \emph{small bigons}. Small bigons are always cut into exactly two pieces; we explain this for $\varnothing$-labelled corners connecting two horizontal edges in $D(\tracks_1)$, and the case of a $\varnothing$ corner connecting two vertical edges is similar. The $\varnothing$ corner in question connects the right ends of two horizontal segments corresponding to two $\circ$-generators in the type D structure corresponding to $\tracks_1$, and because $\tracks_1$ is almost reduced these generators appear in either the first or third zig-zag arrangement in Figure \ref{fig:almost}. At the left end of these horizontal segments, the lower one has a corner turning downward (corresponding to the $1$ labelled arrow in the decorated graph) and the upper one either continues leftward or turns upward (corresponding to a $23$ or $2$ labelled arrow). It is clear that the bigon can not extend past the left end of these horizontal edges, since it would then necessarily cover the basepoint, and so it must stop in the lower right quadrant of $T$ at some vertical edge of $A(\tracks_0)$. 

For the remaining bigons, we will assume they do not contain any corners of type $\varnothing$. Consider a bigon from an intersection point $p$ to an intersection point $q$. The only bigons that do not intersect $\alpha$ and $\beta$ and thus do not decompose into multiple pieces have the form shown in the last row of Figure \ref{fig:bigon-pieces}; they begin at $p$ in the bottom right quadrant of $T$ and end at $q$ in the top left quadrant, and each side of the boundary has a single corner labelled by 2. Apart from bigons of this type, the piece of the bigon containing the initial generator $p$ must look like one of the \emph{starting pieces} in the second row of Figure \ref{fig:bigon-pieces}. For example, if $p$ is in the lower quadrant then the bigon must open up and left from $p$, since a bigon opening down and right would cover the basepoint. The $A(\tracks_0)$ boundary must either pass straight through its corner region (the top right quadrant) or turn left, since turning right would again lead to the basepoint being covered. If the $A(\tracks_0)$ continues upward, then the $D(\tracks_1)$ boundary must turn upward in its corner region to avoid covering the basepoint, leading to the fourth starting piece in Figure \ref{fig:bigon-pieces}, and if the $A(\tracks_0)$ boundary turns left then the $D(\tracks_1)$ boundary must continue leftward through its corner region, leading to the second starting piece shown. If $p$ lies in the top left quadrant then the bigon can either open down and left or up and right from $p$; in the former case the $D(\tracks_1)$ boundary must turn leftward and in the latter case the $A(\tracks_0)$ boundary must turn upwards, producing starting pieces of the first and third types. Note that the type D and type A boundaries can only change direction in their respective corner regions, and backtracking is not possible since $\tracks_0$ is reduced and we assume the $\varnothing$-labelled edges of $D(\tracks_1)$  are not involved. 
Similar consideration shows that the piece of a bigon containing the terminal generator $q$ takes the form of one of the \emph{ending pieces} in Figure \ref{fig:bigon-pieces}, and all other pieces must be one of the \emph{connecting pieces} shown in the figure.




In order to identify $d^\tracks$ with $\partial^\boxtimes$, we will relate sequences of bigon pieces to type A operations and sequences of type D operations which pair in the box tensor product. Toward that end, we now describe our convention for labeling corners of bigon pieces. Recall that each corner in $A(\tracks_0)$ or $D(\tracks_0)$ represents an arrow in the corresponding $\Alg$-decorated graph. Referring to Figure \ref{fig:algebra-edges} (and remembering that $\tracks_1$ is reflected across the antidiagonal when constructing $D(\tracks_1)$), we label the corners of $D(\tracks_1)$ the same as the corresponding edge. The label on a corner of $A(\tracks_0)$ is the subword of $321$ obtained from the arrow labels in the corresponding $\Alg$-decorated graph by interchanging $3$'s and $1$'s. An equivalent graphical way of describing these labels is to label the corners of each corner region 0, 1, 2, and 3, ordered counterclockwise starting from the top right corner for the corner region of $A(\tracks_0)$ and from the bottom left corner for the corner region of $D(\tracks_1)$, and then label each corner edge in the boundary of a bigon by the sequence of these labels which are covered (in the order obtained by following the boundary from $p$ to $q$). Note that the path from $p$ to $q$ in $A(\tracks_0)$ follows the boundary orientation of the bigon and thus has the basepoint to its right at each corner, while the path from $p$ to $q$ in $D(\tracks_1)$ opposes the boundary orientation and has the basepoint to its left at each corner. It follows that traversing either path from $p$ to $q$ agrees with the orientation of the edges in the corresponding $\Alg$-decorated graphs, and thus each boundary path from $p$ to $q$ determines a directed path in the $\Alg$-decorated graphs. Finally, recall that a directed path in an $\Alg$-decorated graph defines either a $\delta^k$ map in the associated type D structure or an $m_k$ map in associated $\Ainfty$ module.

For a given bigon from $p$ to $q$, let $x$ denote the horizontal or vertical edge of $A(\tracks_0)$ containing $p$ and let $x'$ denote the horizontal or vertical edge containing $q$; similarly, let $y$ and $y'$ denote the edges of $D(\tracks_1)$ containing $p$ and $q$. Let $\tilde x$, $\tilde x'$, $\tilde y$, and $\tilde y'$ denote the corresponding generators of $N^0_\Alg$ or ${}^\Alg\! N^1$, so that $p$ corresponds to $\tilde x \otimes \tilde y$ and $q$ corresponds to $\tilde x' \otimes \tilde y'$. Let $I_1$, \ldots, $I_k$ is the sequence of corner labels along the $D(\tracks_1)$ part of the boundary of the bigon, viewed as a path from $y$ to $y'$. This path encodes an operation
$$\delta^k(\tilde y) = \rho_{I_1} \otimes \cdots \otimes \rho_{I_k} \otimes \tilde y'$$
in ${}^\Alg\! N^1$. Let $J_1, \ldots, J_\ell$ be the sequence of corner labels along the $A(\tracks_0)$ part of the boundary, viewed as a path from $x$ to $x'$. To determine the corresponding $\Ainfty$ operation, note that $1$'s and $3$'s are already swapped relative to labels in the $\Alg$-decorated graph, so by the duality described in Section \ref{subsec:graphs and loops} it just remains to concatenate the corner labels and divide into maximal increasing subwords. That is, let $I'_1, \ldots, I'_{k'}$ be the sequence such that $I'_1 \cdots I'_{k'} = J_1 \cdots J_\ell$, each $I'_i$ is a subword of $123$, and $k'$ is minimal with respect to these conditions. Then the path from $x$ to $x'$ in $A(\tracks_0)$ encodes an operation
$$\mu^{k'+1}(\tilde x, \rho_{I'_1}, \cdots, \rho_{I'_{k'}}) = \tilde x'$$
in $N^0_\Alg$.

A key observation is that for any bigon, the sequences of corner labels from each half of the boundary label have the same concatenation---that is, $I_1 \cdots I_k = J_1 \cdots J_\ell$. To see this, observe that for each piece in Figure \ref{fig:bigon-pieces}, the label on the $D(\tracks_1)$ corner can be obtained from the label on the $A(\tracks_0)$ corner by removing a 3 at the beginning if the piece has boundary on the right side of the square, adding a 3 at the end if the piece has boundary on the left side of the square, removing a 1 at the end if the piece has boundary on the top side of the square, and adding a 1 at the beginning if the piece has boundary on the bottom of the square. Note that for any bigon formed from these pieces, excluding small bigons, both boundary arcs move only upward and leftward when traveling from $x$ to $y$, so a piece with a boundary on the left side of the square is always followed by a piece with boundary on the right side of the square, and a piece with boundary on the top is always followed by a piece with boundary on the bottom. It follows that the changes made to obtain the $D(\tracks_1)$ labels from the $A(\tracks_0)$ labels cancel out when concatenating all labels along the boundary paths from $x$ to $y$. For example, for the bigon in Figure \ref{fig:minus-one} the concatenation of the corner labels along either boundary path is the word $123$. Since $I_1 \cdots I_k = J_1 \cdots J_\ell$, and $\{I_1, \cdots, I_k\}$ is a sequence of subwords of $123$ with the first letter of $I_i$ less than the last letter of $I_{i-1}$, we have immediately that the sequences $\{I_1, \ldots, I_k\}$ and $\{I'_1, \ldots, I'_{k'}\}$ agree. From this it follows that the operations discussed in the preceding paragraph pair in the box tensor product to produce a $\tilde x' \otimes \tilde y'$ term in $\partial^\boxtimes(\tilde x \otimes \tilde y)$.

We have shown that every non-small bigon encodes a term in $\partial^\boxtimes$. It is easy to check the same is true for small bigons. In this case, $p$ and $q$ lie on the same horizontal or vertical edge of $A(\tracks_0)$, so $x = x'$, and the $\varnothing$ corner on $D(\tracks_1)$ encodes an operation $\delta^1(\tilde y)=\boldsymbol 1\otimes \tilde y'$ in ${}^\Alg\! N^1$. Such a term in $\delta^1$ contributes $\tilde x\otimes \tilde y'$ to $\partial^\boxtimes(\tilde x\otimes \tilde y)$, as desired.

 It remains to show that every term in $\partial^\boxtimes$ corresponds to a bigon in the way described above. First consider terms in $\partial^\boxtimes(x \otimes y)$ arising from differentials of the form $\delta^1(y)=\boldsymbol 1\otimes y'$ in ${}^\Alg\! N^1$. For each such term, $x \otimes y'$ appears in $\partial^\boxtimes(x \otimes y)$ for each generator $x$ of $N^0_\Alg$ with the appropriate idempotent. It is clear that each of these terms correspond to a small bigon, since the horizontal (or vertical) edges of $D(\tracks_1)$ corresponding to $y$ and $y'$ intersect each vertical (or horizontal) edge of $A(\tracks_0)$ and are connected by a $\varnothing$-labelled corner on one end. All other terms in $\partial^\boxtimes(x \otimes y)$ arise from pairing two operations of the form 
$$\mu^{k+1}(\tilde x, \rho_{I_1}, \cdots, \rho_{I_k}) = \tilde x' \qquad \text{and} \qquad \delta^k(\tilde y) = \rho_{I_1} \otimes \cdots \otimes \rho_{I_k} \otimes \tilde y'$$
for some sequence of algebra elements $\rho_{I_1}, \cdots, \rho_{I_k}$. The presence of this $\delta^k$ map in ${}^\Alg\! N_1$ implies the existence of a path in $D(\tracks_1)$ from $y$ to $y'$ with $k$ corners of types $I_1, \ldots, I_k$. Similarly, the presence of this $m^{k+1}$ operation in $N^0_\Alg$ implies the existence of a path in $A(\tracks_0)$ from $x$ to $x'$ with sequence of corner labels $\{J_1, \ldots, J_\ell\}$ obtained by dividing $I_1 \cdots I_k$ into a minimal number of subwords of $321$. We need to show that these two paths must bound a bigon. It is clear that there is some path consisting of horizontal edges in the top left quadrant, vertical edges in the bottom right quadrant, and corners in the top right quadrant that would form a bigon with the given path in $D(\tracks_1)$---such a path could be obtained by applying a homotopy to the path in $D(\tracks_1)$ avoiding the basepoint to achieve the desired form. By the discussion above, this path must have sequence of corner labels $\{J_1, \ldots, J_\ell\}$, and any path in $A(\tracks_0)$ with the same sequence of corners would also bound a bigon with the path in $D(\tracks_1)$.
\end{proof}

\begin{figure}[t]
\includegraphics[scale=0.5]{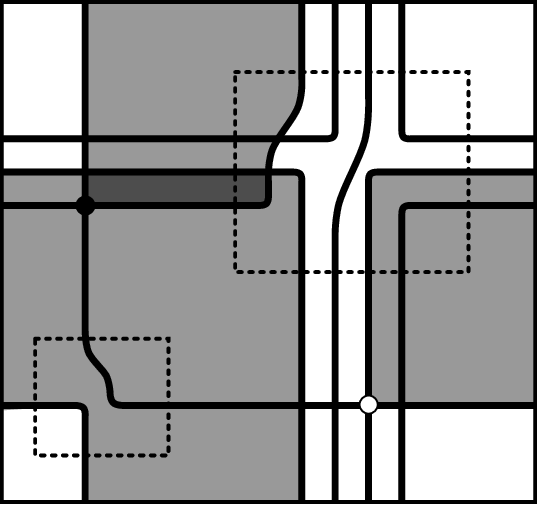}
\qquad\qquad
\labellist\scriptsize
\pinlabel {$1$} at 152 137 \pinlabel {$2$} at 152 154 \pinlabel {$3$} at 135 154
\pinlabel {$1$} at 182 89 \pinlabel {$2$} at 182 185 \pinlabel {$3$} at 90 185
\endlabellist
\includegraphics[scale=0.55]{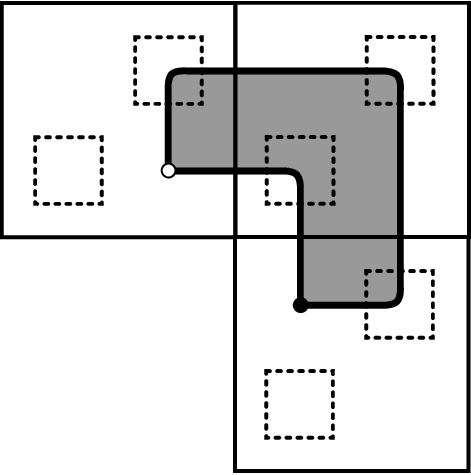}
\caption{The chain complex arising from $-1$ surgery on the right-hand trefoil. The shaded region highlights a bigon with edges corresponding to $m_2(x_0,\rho_{123}) = y_0$ and $\delta^1(x_1)=\rho_{123}\otimes y_1$ so that $y_0\otimes y_1$ is a summand of $\partial^\boxtimes(x_0\otimes x_1)$. The reader can also find an example of a single-piece bigon; compare Figure \ref{fig:bigon-pieces}.}\label{fig:minus-one}
\end{figure}

\section{Classifying extendable type D structures}\label{sec:extend} 

This section proves Theorem \ref{thm:structure}, our structure theorem for type D structures under the hypothesis that the module in question is extendable. While this is a purely algebraic result, it is of interest in the present context owing to the fact that the three-manifold invariants that we wish to study all satisfy this additional hypothesis; compare Theorem \ref{thm:extension}. As such, the ultimate outcome of this section will be the existence part of Theorem \ref{thm:invariance}. 

Broadly, after defining extendable type D structures in Section \ref{sec:extended-type-D-strs}, we establish the structure theorem in two steps. First, we give a partial characterization in terms of matrices over the ring $\F[U]/U^2$, and second, we give a combinatorial algorithm that simplifies the train tracks resulting from this partial characterization.

\subsection{Extended type D structures} \label{sec:extended-type-D-strs}
The torus algebra $\Alg$ admits an extension, denoted by $\AlgExt$, which plays a key role. 

\piccaption[]{A quiver for the extended torus algebra. \label{fig:extended-algebra-quiver}}
\parpic[r]{
 \begin{minipage}{45mm}
 \centering
\begin{tikzpicture}[->,>=stealth',auto,node distance=3cm,thick,main node/.style={circle,draw}]
\node[main node] (1) {$\iota_0$};
  \node[main node] (2) [right of=1] {$\iota_1$};
  \path
  (1) edge[bend left=22] node [above] {\normalsize$\rho_1$} (2)
  (2) edge[bend left=22] node [above] {\normalsize$\rho_2$} (1)
  (1) edge[bend right=60] node [above] {\normalsize$\rho_3$} (2)
  (2) edge[bend right=60] node [above] {\normalsize$\rho_0$} (1);
\end{tikzpicture}
  \end{minipage}%
}
The algebra $\AlgExt$ is a path algebra over
the quiver shown in Figure \ref{fig:extended-algebra-quiver} with relations $\rho_i \rho_{i-1} = 0$, where the indices are to be interpreted mod 4, along with the relation
$\rho_0 \rho_1 \rho_2 \rho_3 \rho_0 = 0$. As a vector space over $\F$ it is spanned by the idempotents $\iota_0$ and $\iota_1$, along with elements 
$\rho_I = \rho_{i_1} \rho_{i_2} \cdots \rho_{i_k}$ for each sequence $I = i_1\ldots i_k$ of indices in $\Z/4\Z$ containing at most one 0 such that $i_{j+1} \equiv i_j + 1$ for each $j$. 
The torus algebra may be recovered from $\AlgExt$ by setting $\rho_0 = 0$. More precisely, if $\mathcal{J}$ is the ideal generated by $\rho_0$ then $\Alg = \AlgExt / \mathcal{J}$. As with $\Alg$, write $\mathbf{1} = \iota_0 + \iota_1$ and use the convention $\rho_\varnothing = \mathbf{1}$. We distinguish the element $U = \rho_{1230} + \rho_{2301} + \rho_{3012} + \rho_{0123}$ and note that $U$ is central in $\AlgExt$.

 \begin{definition}
 An \emph{extended type D structure over $\AlgExt$} is an $\sI$ module $V = \iota_0 V \oplus \iota_1 V$  equipped with a map
 $\widetilde{\delta}^1: V \to \AlgExt \otimes_\sI V$ such that the map 
$\widetilde\partial: \AlgExt \otimes_\sI V \to \AlgExt \otimes_\sI V $ defined by $\widetilde\partial (a \otimes x) = a \cdot \td^1 x$  satisfies $\widetilde\partial^2(x) = U x$ for all $x\in V$.  The module $\widetilde N = \AlgExt \otimes _ \sI V$ is the \emph{extended type D module} associated with $V$ and $\widetilde\delta^1$. 
 \end{definition}

An extended type D structure over $\AlgExt$ can be specified by the pair of vector spaces $\iota_0 V$ and $\iota_1 V$ along with a collection of \emph{coefficient maps} $D_I: \iota_{i_1-1} V \to \iota_{i_k} V$ where $I = i_1\ldots i_k$ is a (possibly empty) sequence of indices in $\Z/4\Z$ containing at most one 0 such that $i_{j+1} \equiv i_j + 1$ for each $j$. In the case of the empty sequence we have two maps $D_\varnothing: \iota_i V\to \iota_i V$ for $i \in \{0,1\}$. The coefficient maps encode $\widetilde\delta^1$ by
$$\widetilde\delta^1 = \sum_I \rho_I \otimes D_I.$$ 
The condition that $\widetilde\partial^2(x) = Ux$ translates to the condition that for each $I$, the sum $\sum_{J\cup K = I} D_K \circ D_J$ is zero unless $I$ is 0123, 1230, 2301, or 3012, in which case the sum is the identity on the appropriate summand of $V$. A type D structure over $\Alg$ admits a similar description in terms of coefficient maps $D_I$ where $I \in \{\varnothing, 1, 2, 3, 12, 23, 123\}$.

Any extended type D structure over $\AlgExt$ determines a type D structure over $\Alg$ by ignoring all coefficient maps for which $I$ contains 0 or, equivalently, ignoring all terms in $\widetilde\delta_1$ involving a multiple of $\rho_0$; this corresponds to taking the quotient by the ideal $\mathcal{J}$. In this case we say that the initial extended type D structure is an \emph{extension} of the quotient type D structure. Equivalently, an extended type D module \(\tN\) is an extension of \(N\) if \(\sA \otimes_{\AlgExt} \tN\) is isomorphic to \(N\) as a differential \(\sA\) module. We say that a type D module $N$ is \emph{extendable} if it has an extension $\widetilde N$, and we say that $N = \sA \otimes_{\AlgExt} \widetilde N$ is the underlying type D module of $N$.

We remark that, though it is not clear \emph{a priori}, Proposition \ref{prop:unique-extension} shows that any two extensions of a type D structure $N$ must be homotopy equivalent. In  other words, an extended type D structure contains no more information than its underlying type D structure. Thus the chief advantage to working with extended type D structures is not the extension that we choose, but simply the fact that an extension exists. Being extendable is a rather strong algebraic constraint and, by restricting to extendable modules (as we will do for the rest of the paper), the geometric description of type D modules over $\Alg$ can be simplified considerably. 

Just as a type D structure (along with a choice of basis) can be represented by an $\Alg$-decorated graph, an extended type D structure with a specified basis can be represented by an $\AlgExt$-decorated graph, where the set of possible edge labels is expanded to include all sequences $I = i_1 \ldots i_k$ of elements of $\Z/4\Z$ with $i_\ell \equiv i_{\ell-1} + 1$ and with at most one $i_\ell=0$. The method for constructing an extended type D structure from a graph is the same: vertices labeled by $\bullet$ and $\circ$ correspond to $\iota_0$-generators and $\iota_1$-generators, respectively, and an arrow labeled by $I$ from $x$ to $y$ contributes a $\rho_I\otimes y$ term to $\td^1(x)$. Analogous to $\Alg$-decorated graphs, $\AlgExt$-decorated graphs satisfy a condition on the counts of length two paths of arrows, which ensures the result is in fact an extended type D structure.

In Section \ref{sec:tracks} we realized the $\Alg$-decorated graph $\Gamma$  representing a reduced or almost reduced type D structure geometrically by immersing it in the parametrized torus $T$. This construction extends immediately to decorated graphs associated with extensions, where the arcs corresponding to the additional edge labels are determined by Figure \ref{fig:extended-algebra}; the result is a naive train track representing the extended type D structure. Note that all edges of this train track are oriented, but a pair of vertices may be connected by a pair of parallel but oppositely oriented edges; in this case, for convenience, we adopt the convention that the two opposing oriented edges are replaced with a single unoriented edge. For example, consider the type D module in Figure \ref{fig:naive-track-example}; a choice of extension is illustrated in Figure \ref{fig:naive-track-example-extended} together with the corresponding train track.

\begin{figure}[ht]
\labellist 
\small
\pinlabel {$\rho_1$} at 181 393
\pinlabel {$\rho_{12}$} at 295 393
\pinlabel {$\rho_{123}$} at 410 393
\pinlabel {$\rho_{1230}$} at 520 393
\pinlabel {$\rho_{12301}$} at 637 393
\pinlabel {$\rho_{123012}$} at 747 393
\pinlabel {$\rho_{1230123}$} at 860 393

\pinlabel {$\iota_0$} at 71 257
\pinlabel {$\rho_2$} at 181 258
\pinlabel {$\rho_{23}$} at 295 258
\pinlabel {$\rho_{230}$} at 410 258
\pinlabel {$\rho_{2301}$} at 520 258
\pinlabel {$\rho_{23012}$} at 637 258
\pinlabel {$\rho_{230123}$} at 747 258

\pinlabel {$\iota_1$} at 71 120
\pinlabel {$\rho_3$} at 181 121
\pinlabel {$\rho_{30}$} at 295 121
\pinlabel {$\rho_{301}$} at 410 121
\pinlabel {$\rho_{3012}$} at 520 121
\pinlabel {$\rho_{30123}$} at 637 121

\pinlabel {$\rho_0$} at 181 -15
\pinlabel {$\rho_{01}$} at 295 -15
\pinlabel {$\rho_{012}$} at 410 -15
\pinlabel {$\rho_{0123}$} at 520 -15
\endlabellist
\includegraphics[scale=0.4]{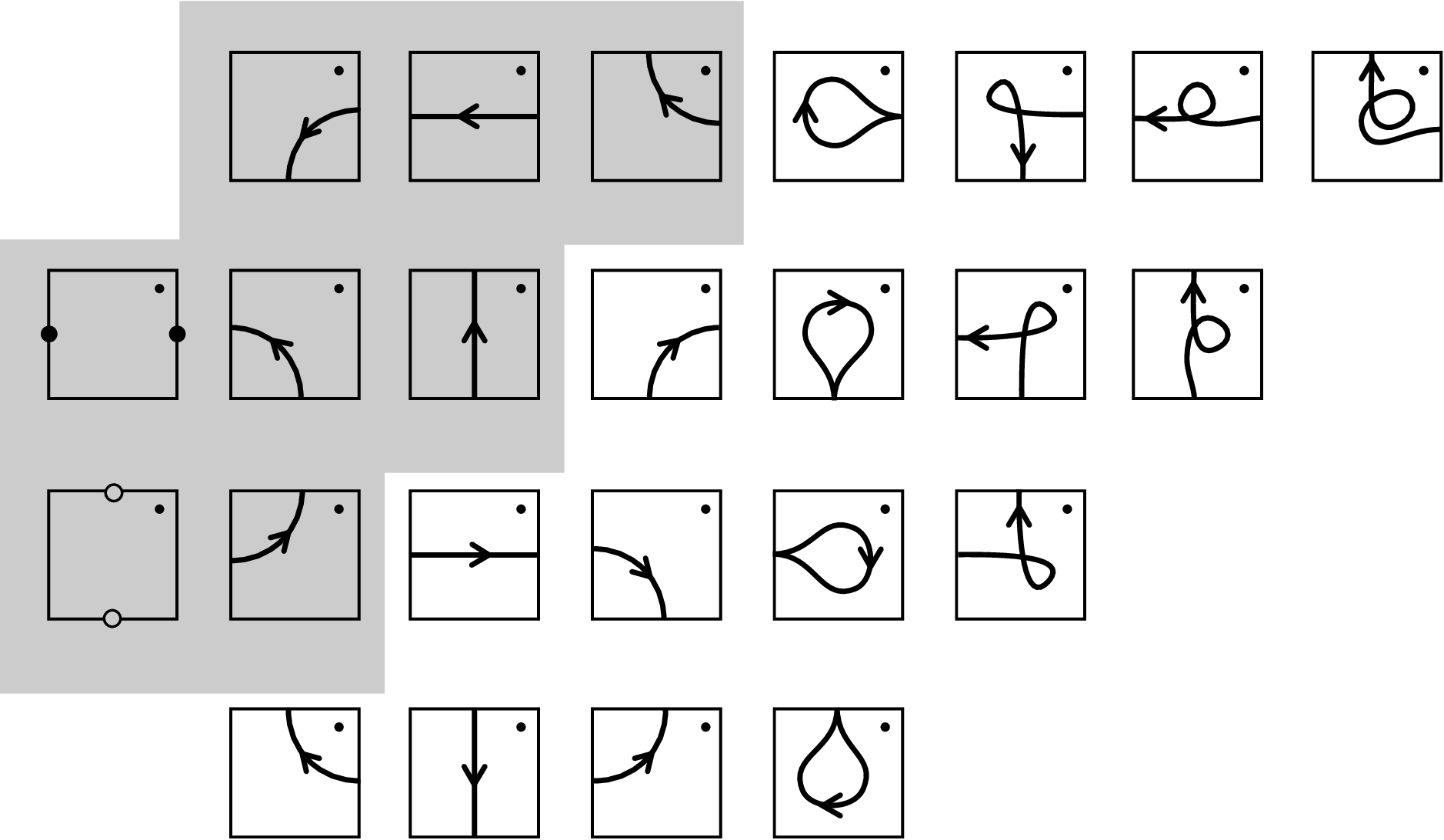}
\caption{Train track edges corresponding to elements of the extended algebra $\AlgExt$. These can also be interpeted as homotopy classes of clockwise moving paths (or constant paths, in the case of $\iota_0$ and $\iota_1$) connecting sides of a square and passing the top right corner at most once, and $\AlgExt$ is generated as an $\F$-vector space by these homotopy classes; algebra multiplication is concatenation of paths, where the product is zero if the concatenation does not exist or passes the top right corner twice. The shaded elements generate $\Alg$.}\label{fig:extended-algebra}
\end{figure}

\begin{figure}[ht]
 \begin{tikzpicture}[scale=0.75,>=stealth', thick] 
\def \radius {3.5cm} \def \outer {3.9cm} \def \inner {2.85cm}
\node (d) at ({360/(11) * (1- 1)}:\radius) {$\bu$};
\node (j) at ({360/(11) * (2- 1)}:\radius) {$\ci$};
\node (a) at ({360/(11) * (3 -1)}:\radius) {$\bu$};
\node (k) at ({360/(11) * (4 -1)}:\radius) {$\ci$};
\node (f) at ({360/(11) * (5 -1)}:\radius) {$\bu$};
\node (g) at ({360/(11) * (6 -1)}:\radius) {$\ci$};
\node (h) at ({360/(11) * (7 -1)}:\radius) {$\ci$};
\node (c) at ({360/(11) * (8 -1)}:\radius) {$\bu$};
\node (e) at ({360/(11) * (9 -1)}:\radius) {$\bu$};
\node (i) at ({360/(11) * (10 -1)}:\radius) {$\ci$};
\node (b) at ({360/(11) * (11 -1)}:\radius) {$\bu$};

\draw[->, bend right = 15] (d) to (j);
\draw[->, bend left = 15] (a) to (j);
\draw[->, bend right = 15] (a) to (k);
\draw[->, bend right = 15] (k) to (f);
\draw[->, bend right = 15] (f) to (g);
\draw[->, bend left = 15] (h) to (g);
\draw[->, bend left = 15] (c) to (h);
\draw[->, bend right = 15] (c) to (e);
\draw[->, bend right = 15] (e) to (i);
\draw[->, bend left = 15] (b) to (i);
\draw[->, bend right = 15] (b) to (d);

\draw[->, bend right = 45] (d) to node[left, pos = .7]{$\scriptstyle{1}$} (i);
\draw[->, bend right = 30] (d) to node[above left = -3pt, pos = .3]{$\scriptstyle{12}$} (e);
\draw[->, bend right = 20] (d) to node[above]{$\scriptstyle{1}$} (h);
\draw[->, bend right = 10] (j) to node[above]{$\scriptstyle{23}$} (g);

  \node at ({360/(11) * (1- (1/2))}:\outer) {$\scriptstyle{1}$};
  \node at ({360/(11) * (2-(1/2))}:\outer) {$\scriptstyle{3}$};
  \node at ({360/(11) * (3-(1/2))}:\outer) {$\scriptstyle{123}$};
  \node at ({360/(11) * (4-(1/2))}:\outer) {$\scriptstyle{2}$};
  \node at ({360/(11) * (5-(1/2))}:\outer) {$\scriptstyle{1}$};
  \node at ({360/(11) * (6-(1/2))}:\outer) {$\scriptstyle{23}$};
  \node at ({360/(11) * (7-(1/2))}:\outer) {$\scriptstyle{3}$};
  \node at ({360/(11) * (8-(1/2))}:\outer) {$\scriptstyle{12}$};
  \node at ({360/(11) * (9-(1/2))}:\outer) {$\scriptstyle{1}$};
  \node at ({360/(11) * (10-(1/2))}:\outer) {$\scriptstyle{3}$};
  \node at ({360/(11) * (11-(1/2))}:\outer) {$\scriptstyle{12}$};

  \node at ({360/(11) * (0)}:\outer) {$\footnotesize{d}$};
  \node at ({360/(11) * (1)}:\outer) {$\footnotesize{j}$};
  \node at ({360/(11) * (2)}:\outer) {$\footnotesize{a}$};
  \node at ({360/(11) * (3)}:\outer) {$\footnotesize{k}$};
  \node at ({360/(11) * (4)}:\outer) {$\footnotesize{f}$};
  \node at ({360/(11) * (5)}:\outer) {$\footnotesize{g}$};
  \node at ({360/(11) * (6)}:\outer) {$\footnotesize{h}$};
  \node at ({360/(11) * (7)}:\outer) {$\footnotesize{c}$};
  \node at ({360/(11) * (8)}:\outer) {$\footnotesize{e}$};
  \node at ({360/(11) * (9)}:\outer) {$\footnotesize{i}$};
  \node at ({360/(11) * (10)}:\outer) {$\footnotesize{b}$};

\draw[->, color = gray, bend right = 15] (j) to (d);
\draw[->, color = gray, bend left = 15] (j) to (a);
\draw[->, color = gray, bend right = 15] (k) to (a);
\draw[->, color = gray, bend right = 15] (f) to (k);
\draw[->, color = gray, bend right = 15] (g) to (f);
\draw[->, color = gray, bend left = 15] (g) to (h);
\draw[->, color = gray, bend left = 15] (h) to (c);
\draw[->, color = gray, bend right = 15] (e) to (c);
\draw[->, color = gray, bend right = 15] (i) to (e);
\draw[->, color = gray, bend left = 15] (i) to (b);
\draw[->, color = gray, bend right = 15] (d) to (b);

  \node at ({360/(11) * (1- (1/2))}:\inner) {\color{gray}$\scriptstyle{230}$};
  \node at ({360/(11) * (2-(1/2))}:\inner) {\color{gray}$\scriptstyle{012}$};
  \node at ({360/(11) * (3-(1/2))}:\inner) {\color{gray}$\scriptstyle{0}$};
  \node at ({360/(11) * (4-(1/2))}:\inner) {\color{gray}$\scriptstyle{301}$};
  \node at ({360/(11) * (5-(1/2))}:\inner) {\color{gray}$\scriptstyle{230}$};
  \node at ({360/(11) * (6-(1/2))}:\inner) {\color{gray}$\scriptstyle{01}$};
  \node at ({360/(11) * (7-(1/2))}:\inner) {\color{gray}$\scriptstyle{012}$};
  \node at ({360/(11) * (8-(1/2))}:\inner) {\color{gray}$\scriptstyle{30}$};
  \node at ({360/(11) * (9-(1/2))}:\inner) {\color{gray}$\scriptstyle{230}$};
  \node at ({360/(11) * (10-(1/2))}:\inner) {\color{gray}$\scriptstyle{012}$};
  \node at ({360/(11) * (11-(1/2))}:\inner) {\color{gray}$\scriptstyle{30}$};
  
\draw[->, color = gray, bend right = 30] (j) to node[left]{$\scriptstyle{230}$} (e);
\draw[->, color = gray, bend right = 40] (j) to node[left]{$\scriptstyle{230}$} (c);

\end{tikzpicture} \hspace{2 cm}
\raisebox{5 mm}{\labellist

\pinlabel {$a$} at -10 225
\pinlabel {$b$} at -10 187
\pinlabel {$c$} at -10 149
\pinlabel {$d$} at -10 111
\pinlabel {$e$} at -10 73
\pinlabel {$f$} at -10 34

\pinlabel {$a$} at 270 225
\pinlabel {$b$} at 270 187
\pinlabel {$c$} at 270 149
\pinlabel {$d$} at 270 111
\pinlabel {$e$} at 270 73
\pinlabel {$f$} at 270 34

\pinlabel {$g$} at 227 -16
\pinlabel {$h$} at 179 -12
\pinlabel {$i$} at 131 -12
\pinlabel {$j$} at 83 -12
\pinlabel {$k$} at 35 -12

\pinlabel {$g$} at 227 273
\pinlabel {$h$} at 179 275
\pinlabel {$i$} at 131 275
\pinlabel {$j$} at 83 275
\pinlabel {$k$} at 35 275
\tiny\pinlabel {$z$} at 246 247
\endlabellist
\includegraphics[scale=0.55]{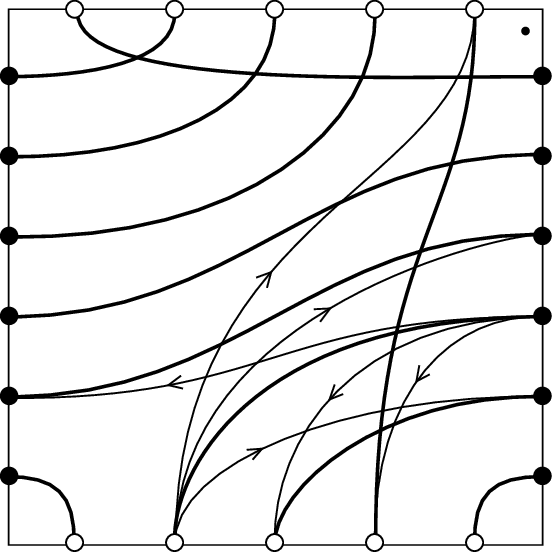}}
\vspace*{10pt}
\caption{The train track (right) associated with an extended type D structure (left). This extended type D structure is an extension of the type D structure in Figure \ref{fig:naive-track-example}; the extended coefficient maps are shown in gray.}
\label{fig:naive-track-example-extended}
\end{figure}

\subsection{Train tracks and matrices}
\label{subsec:extended tracks}

We have so far considered only naive train tracks constructed from $\Alg$-decorated or $\AlgExt$-decorated graphs, but we will have need to work with a larger class of train tracks in the marked torus, which we now formalize.

\piccaption[]{Examples of counterclockwise fishtails. These never appear in a toroidal train track. \label{fig:anti-oriented-fishtail}}
\parpic[r]{
 \begin{minipage}{45mm}
 \hspace{5mm}
 \labellist
 \tiny
 \pinlabel {$z$} at 63 65  \pinlabel {$z$} at 169 65 
 \endlabellist
\includegraphics[scale=0.5]{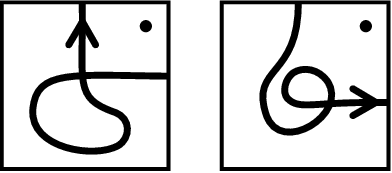}
  \end{minipage}%
}
\begin{definition}\label{def:toroidal-train-track}
A \emph{train track in the marked torus $T$} will refer to an immersed train track $\tracks$ in $T$ with the following properties:
(1) $\tracks$ is disjoint from the basepoint in $T$ and perpendicular to $\alpha$ and $\beta$; (2) 
 every intersection of $\tracks$ with $\alpha\cup\beta$ is a switch and called a \emph{primary switch} of $\tracks$; and (3)
no oriented path carried by $\tracks$ contains a counterclockwise fishtail (see Figure \ref{fig:anti-oriented-fishtail}). 
\end{definition}

Note that all the switches in a naive train track are primary, but this is not required by the definition. As we will see in the next subsection, one of the key advantages of the train track notation is that it allows us to push some of the switches in a train track into the interior of the fundamental domain for \(T\). 
In addition, the definition allows for train tracks where a section of the track simply terminates in a monovalent vertex. Although  tracks of this form cannot appear as the train track associated with an extendable type D structure, pairing with them provides a useful way of distinguishing between two such tracks ({\it c.f.} Section~\ref{subsec:Injectivity of g}).

While we will often simply use the term \emph{train track}, all train tracks we will discuss are in the marked torus (with the exception of those in Section \ref{sub:aside}) and should be understood to satisfy the conditions above.
Within this broader class of objects, we will be able to manipulate our naive train tracks geometrically to find much simpler representatives. Importantly, the discussion of pairing above will carry over directly to this more general class of train tracks.

Given a train track $\tracks$ in $T$ with $n$ primary switches, we can construct a $2n\times 2n$ matrix $M_\tracks$ encoding the essential information of $\tracks$.  We cut $T$ open along $\alpha$ and $\beta$ to form a square, with our usual convention that $\beta$ is horizontal (oriented rightward), $\alpha$ is vertical (oriented upward), and the basepoint is in the top right corner of the square. The primary switches give $2n$ marked points on the edges of the square, which we label $v_1, \ldots, v_{2n}$ clockwise, starting from the top right corner. We let $k$ denote the number of primary switches of $\tracks$ on $\alpha$, so that $v_i$ is on the right side of the square if $1\le i \le k$, the bottom side if $k+1 \le i \le n$, the left side if $n+1 \le i \le n+k$, and the top side if $n+k+1\le i \le 2n$. The $(i,j)$ entry of $M_\tracks$ counts immersed oriented paths in $\tracks$ through the square from $v_i$ to $v_j$. Such a path can be given a label recording how many times it passes the base point on its left. More precisely, any immersed path through the square from a primary switch $x$ to a primary switch $y$ can be projected to a path in the boundary of the square by sending each point in the interior of the path to the boundary following the path's leftward pointing normal vector. The resulting path connects the corner after $x$ (moving clockwise around the boundary of the square) to the corner before $y$. We precompose this path with the clockwise moving path from $x$ to the corner after $x$ and postcompose with the path from the corner before $y$ to $y$, resulting in a path from $x$ to $y$ in the boundary of the square. The assumption that $\tracks$ has no counterclockwise fishtails ensures that this projected path will have net clockwise movement around the boundary. We weight each path by $U^m$, where $m$ is the number of times (counted with sign) that the projected path passes the top right corner of the square. We set $U^2=0$ and count paths modulo 2, so $M_\tracks$ has coefficients in $\F[U]/U^2$.

\begin{definition}
We will say that train tracks $\tracks_1$ and $\tracks_2$ in $T$ are \emph{equivalent} if $M_{\tracks_1} = M_{\tracks_2}$. A train track $\tracks$ is said to be \emph{reduced} if no immersed oriented paths through the square $T\setminus(\alpha\cup\beta)$ begin and end on the same side of the square unless they are weighted by a positive power of $U$. A train track $\tracks$ is \emph{valid} if $M_\tracks^2 = U I_{2n}$, where $I_{2n}$ is the $2n\times 2n$ identity matrix. 
\end{definition} 
The significance of these restrictions on train tracks is found in the following proposition:

\begin{proposition}\label{prp:mod-track-matrix}
The following are equivalent:
\begin{itemize}
\item[(i)] Reduced extended type D structures with a chosen (ordered) set of generators. 
\item[(ii)] Equivalence classes of valid, reduced train tracks in $T$.
\item[(iii)] $2n\times 2n$ matrices $M$ over $\F[U]/U^2$, together with an integer $0\le k \le n$, such that $M^2 = UI_{2n}$ and $M$ is strictly block upper triangular modulo $U$ with respect to blocks of size $k$, $n-k$, $k$, and $n-k$.
\end{itemize}
\end{proposition}

\piccaption[]{Equivalent valid train tracks in $T$. \label{fig:sample-equivalence}}
\parpic[r]{
 \begin{minipage}{60mm}
 \labellist
 \tiny
 \pinlabel {$z$} at 110 106 \pinlabel {$z$} at 264 106 
 \endlabellist
\includegraphics[scale=0.6]{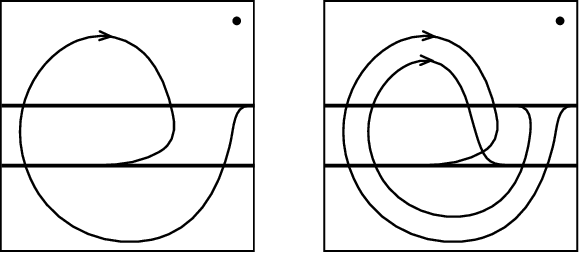}
  \end{minipage}%
}
Before proving this proposition it is worth pausing to consider a simple example. The pair of train tracks in Figure \ref{fig:sample-equivalence} are equivalent; the reader should construct the associated four-by-four matrix in each case and observe that these are equal, as required.  The key feature here is that the additional path in the train track on the right carries  the label $U^2$, which is set to zero in the matrix. 
On the other hand, an inequivalent train track is obtained by considering only the horizontal two-way tracks of this example. Although this train track is inequivalent to the ones given, it represents an isomorphic module. For generators $x$ and $y$ such that $D_{123012}(x)=y$ there is an isomorphism of modules over $\AlgExt$ given by $x'=x+\rho_{1230}\otimes y$ and $y'=y$, so that $D_{123012}(x')=0$. Note that this is an isomorphism between two different choices of extension (of a fixed type D structure) as $\AlgExt$ modules but not an isomorphism as $\F$ vector spaces. That is, this train track is weakly equivalent to those of Figure \ref{fig:sample-equivalence} in the sense of Definition \ref{def:weak-equiv} below.

\begin{proof}[Proof of Proposition \ref{prp:mod-track-matrix}]
Any extended type D structure $\tN$, along with a specified basis $B$, determines a naive train track $\tracks(\tN, B)$ as described above, and it is immediate from the construction that if  $\tN$ is reduced then so is $\tracks(\tN,B)$. This particular train track determines an equivalence class of train tracks. We must check that  $\tracks(\tN, B)$ is valid, but postpone this until after the discussion of matrices below.

Any valid reduced train track $\tracks$ determines a matrix $M_\tracks$ as described above. This is a $2n\times 2n$ matrix, where each row or column corresponds to primary switches on the boundary of the square $T\setminus(\alpha\cup\beta)$. These may be split into blocks according to which side each primary switch is on. The fact that paths in $\tracks$ have net clockwise rotation ensures that any path  labelled by 1 contributes to an entry above the diagonal in $M_\tracks$; the fact that $\tracks$ is reduced further ensures that no 1 appears in the blocks along the diagonal of $M_\tracks$. By definition, the fact that $\tracks$ is valid implies $M_\tracks^2 = UI_{2n}$.

Finally, a reduced extendable type D structure can be extracted from a matrix $M$ and an integer \(k\) as in (iii). The vector space $\iota_0 V$ has generators $x_1,\ldots,x_k$ and $\iota_1 V$ has generators $x_{k+1},\ldots,x_{n}$. Each generator $x_i$ corresponds to a pair of indices $i$ and $t(i)$, where \[
t(i)=\begin{cases}
n+k+1-i & \text{if}\quad  i \le k\\
2n + k - i + 1 & \text{if}\quad i > k.
\end{cases}\]An entry in the $j$ or $t(j)$ column and the $i$ or $t(i)$ row of $M$ contributes an $x_j$ term to the coefficient map $D_I(x_i)$, where $I$ is determined by the position in the matrix with respect to the block decomposition of $M$. More precisely, write $M=M_1+M_U$ where $M_1$ has entries that are $1$ or $0$ while $M_U$ has entries that are $U$ or $0$. Then the block decompositions of $M_1$ and $M_U$ with respect to blocks of size $k$, $n-k$, $k$, and $n-k$ recover the coefficient maps for the corresponding extended type D structure as follows:
\begin{equation}\label{eqn:block}M_1=\left[
\begin{array}{c|c|c|c}
0 &D_1&D_{12}& D_{123}\\ \hline
0 &0&D_2&D_{23} \\ \hline
0 &0&0&D_{3}\\ \hline
0 &0&0&0
\end{array}\right]
\quad
\quad
\quad
M_U= U \left[
\begin{array}{c|c|c|c}
D_{1230} &D_{12301}&D_{123012}& D_{1230123}\\ \hline
D_{230} &D_{2301}&D_{23012}& D_{230123} \\ \hline
D_{30}&D_{301}&D_{3012}&D_{30123}\\ \hline
D_0 &D_{01}&D_{012}&D_{0123}
\end{array}\right]
\end{equation}
Here the blocks are the matrices representing $D_I$, acting by \emph{right} multiplication, with respect to the basis $\{x_1, \ldots, x_k\}$ of $\iota_0 V$, the basis $\{x_{k+1}, \ldots, x_n\}$ of $\iota_1 V$, or these bases in reversed order, as appropriate. In the case that $M = M_{\tracks(\tN, B)}$ is the matrix associated with a pair $(\tN, B)$, the extended type D structure extracted from $M$ in this way exactly recovers $\tN$, as there is a one to one correspondence between nonzero terms in $M_1$ or $M_U$, paths in the train track $\tracks(\tN,B)$, and nonzero terms in the coefficient maps for $\tN$.

It remains to show that $\tracks = \tracks(\tN, B)$ is valid, or equivalently that the corresponding matrix $M_{\tracks(\tN, B)}$ squares to $UI_{2n}$, and conversely that $\widetilde\partial^2(x) = Ux$ in the extended type D structure determined by any matrix $M$ with $M^2 = UI_{2n}$. To see this, we look at the block decomposition of $M^2$, which can be computed from the block decomposition of $M_1$ and $M_U$ above. Let $\bar D_I$ denote the sum $\sum_{J\cup K = I} D_J D_K$ where $J$ and $K$ are nonempty sequences of indices in $\Z/4\Z$. Using this shorthand, we have:
\begin{equation*}M^2=\left[
\def\arraystretch{1.2}
\begin{array}{c|c|c|c}
0 & 0 & \bar D_{12}& \bar D_{123} \\ \hline
0 &0&0& \bar D_{23} \\  \hline
0 &0&0&0\\  \hline 
0 &0&0&0
\end{array}\right] + U \left[
\def\arraystretch{1.2}
\begin{array}{c|c|c|c}
\bar D_{1230} & \bar D_{12301}& \bar D_{123012}& \bar D_{1230123}\\ \hline
\bar D_{230} & \bar D_{2301}& \bar D_{23012}& \bar D_{230123} \\ \hline
\bar D_{30}& \bar D_{301}& \bar D_{3012}& \bar D_{30123}\\ \hline
\bar D_0 & \bar D_{01}& \bar D_{012}& \bar D_{0123}
\end{array}\right]
\end{equation*}
Note that, by slight abuse of notation, we are using $D_I$ to denote both a coefficient map and a matrix representing it, and the matrix acts by right multiplication. Thus the matrix $D_J D_K$ is the matrix representing the map $D_K \circ D_J$. It is now clear that the condition that $M^2 = U I_{2n}$ is precisely the condition on composition of coefficient maps for the corresponding reduced extended type D structure.
\end{proof}

To restrict an extended type D module $\tN$ to its underlying type D module $N$, we simply set $\rho_0 = 0$ in the algebra, and thus ignore all coefficient maps $D_I$ for which $I$ contains 0. It is clear that the analogous operation for the corresponding matrix is setting $U = 0$, resulting in a matrix with coefficients in $\F$. This motivates the following definitions:

\begin{definition}\label{def:weak-equiv}
A train track $\tracks$ in $T$ is \emph{weakly valid} if the corresponding matrix $M_\tracks$ squares to 0 modulo $U$. Two train tracks $\tracks_1$ and $\tracks_2$ are \emph{weakly equivalent} if the matrices $M_{\tracks_1}$ and $M_{\tracks_2}$ agree modulo $U$.
\end{definition}

The following is immediate from Proposition \ref{prp:mod-track-matrix}.

\begin{proposition}\label{prp:mod-track-matrix-weak}
The following are equivalent:
\begin{itemize}
\item[(i)] Reduced type D structures over $\Alg$, up to isomorphism as $\F$ vector spaces,  with a chosen (ordered) set of generators.
\item[(ii)] Weak equivalence classes of weakly valid, reduced train tracks in $T$.
\item[(ii)] $2n\times 2n$ matrices $M$ over $\F$, together with an integer $0\le k \le n$, such that $M^2 = 0$ and $M$ is strictly block upper triangular with respect to blocks of size $k$, $n-k$, $k$, and $n-k$.\qed\end{itemize}

\end{proposition}

Although we are mainly interested in reduced train tracks, we observe that valid almost reduced train tracks can be defined as those corresponding to almost reduced type D structures over $\Alg$.

We pause briefly to discuss pairing of train tracks in $T$, which is defined just as it was for naive train tracks in Section \ref{sec:pairing-train-tracks}. A pair $(\tracks_0, \tracks_1)$ is \emph{admissible} if there are no immersed annuli cobounded by $A(\tracks_0)$ and $D(\tracks_1)$, where $A(\tracks)$ denotes the result of including $\tracks$ into the first quadrant and extending vertically/horizontally and $D(\tracks)$ is the reflection of $A(\tracks)$ across the anti-diagonal.  Then $\sC(\tracks_0, \tracks_1)$ is the vector space generated by intersection points of $A(\tracks_0)$ and $D(\tracks_1)$, equipped with a linear map $d^\tracks$ that counts bigons not covering the basepoint with boundary carried by the pair $(\tracks_0,\tracks_1)$, with the requirement that both sides of the bigon are oriented paths from the initial intersection point to the terminal intersection point. Note that, unlike for naive train tracks, the orientation requirement is meaningful for an arbitrary train track in $T$. The key observation is that defining $\sC(\tracks_0, \tracks_1)$ does not require all the information in $\tracks_1$ and $\tracks_2$, it is in fact determined only by their weak equivalence classes.

\begin{proposition}
Let $\tracks_0$, $\tracks_0'$ be weakly valid reduced train tracks and $\tracks_1, \tracks'_1$ be weakly valid almost reduced train tracks such that the pairs $(\tracks_0, \tracks_1)$ and $(\tracks'_0, \tracks'_1)$ are admissible. If $\tracks_0$ is weakly equivalent to $\tracks_0'$ and $\tracks_1$ is weakly equivalent to $\tracks_1'$, then $\sC(\tracks_0, \tracks_1) = \sC(\tracks_0', \tracks_1')$ as chain complexes.
\end{proposition}

\begin{proof}
Note that to count bigons connecting two intersection points, one need only know the count of paths connecting each pair of primary vertices in each train track. The orientation requirement implies that we can ignore any path in $\tracks_0$ or $\tracks_1$ that has the puncture on the left side of the path; these are precisely the paths weighted by $U$ in $M_{\tracks_0}$ and $M_{\tracks_1}$. Thus the map $d^\tracks$ is determined by the weak equivalence class of $\tracks_0$ and $\tracks_1$.
\end{proof}

Note that under the hypotheses above $d^\tracks$ is a differential, making $\sC(\tracks_0, \tracks_1)$ a chain complex. Indeed, if $\tracks_0$ and $\tracks_1$ correspond to type D structures $N_0$ and $N_1$, respectively, then they are weakly equivalent to the naive train tracks representing $N_0$ and $N_1$, and it follows from Theorem \ref{thm:pairtracks} that $\sC(\tracks_0, \tracks_1)$ is isomorphic as a chain complex to $N_0^A \boxtimes N_1$.

\subsection{Simplifying valid train tracks} 
\label{subsec:simplify tracks}

Consider our running example and the train track in Figure \ref{fig:naive-track-example-extended}. We observe that this train track has two special properties. First, there are exactly two unoriented edges, one from each direction, incident to each switch. It follows that if the oriented edges are ignored, the unoriented edges form a collection of immersed curves. Second, the extra oriented edges can be grouped in pairs, where each pair connects the same two sections of the immersed curve, oriented the same way but turning opposite directions at each end and crossing once. We can slide the endpoints of each pair of oriented segments near each other along the unoriented segments, without changing the equivalence class of the train track, as shown in Figure \ref{fig:crossover-arrows}(a). (In doing so, we have passed from a naive train track in which every switch is primary, to one in which some of the switches lie in the interior of the fundatmental domain for \(T\).)
The resulting $X$ shaped pairs of oriented arrows will be called \emph{crossover arrows}, and to simplify diagrams these pairs will be replaced by a bold arrow (Figure \ref{fig:crossover-arrows}(b)).

\begin{figure}[ht]
\raisebox{-1.5 cm}{\includegraphics[scale=0.4]{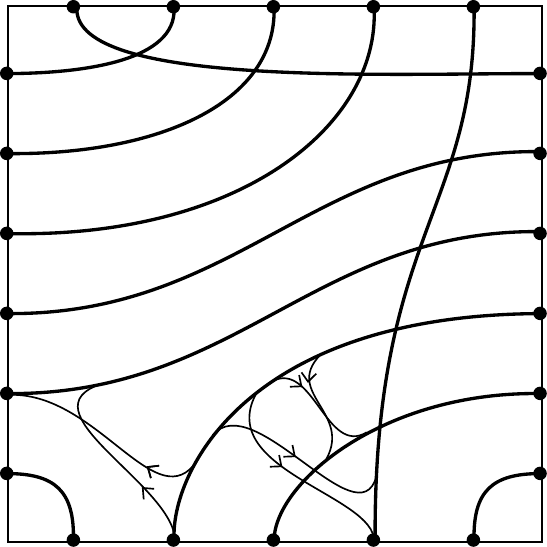}}
\hspace{3 cm}
\labellist
\pinlabel {$=$} at 81 18
\pinlabel {(a)} at -180, -65
\pinlabel {(b)} at 81, -65
\endlabellist
\includegraphics[scale=0.8]{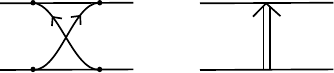}
\vspace{3 mm}
\caption{(a) The train track from Figure \ref{fig:naive-track-example-extended} homotoped so that the oriented edges form crossover arrows. (b) Diagrammatic shorthand for crossover arrows.}
\label{fig:crossover-arrows}
\end{figure}

The goal of the present subsection is to prove that these features are not unique to this example, but indeed hold for (some representative of) every equivalence class of valid reduced train tracks.

\begin{proposition}\label{prop:curves-plus-crossovers}
Any valid reduced train track in $T$ is equivalent to a train track $\tracks$ with the following form:

\begin{itemize}
\item Every switch of $\tracks$ has exactly one unoriented edge on each side, so that $\tracks$ is a collection of (unoriented) immersed curves along with additional oriented edges connecting points on the immersed curves;

\item the immersed curves restrict to horizontal and vertical segments outside of the square $[\frac{1}{4},\frac{3}{4}]\times[\frac{1}{4},\frac{3}{4}]$, which we call the \emph{corner box} since it is where the immersed curves change direction;

\item within the corner box, the immersed curves restrict to embedded arcs  connecting different sides of the corner box (we call these arcs \emph{corners}); 

\item the oriented edges appear in pairs that form crossover arrows (as in Figure \ref{fig:crossover-arrows});

\item the crossover arrows appear outside the corner box, and they move clockwise around the corner box.

\end{itemize}
Conversely, any train track of this form is a valid, reduced train track in $T$.
\end{proposition}

See Figure \ref{fig:curves-and-crossovers}(c) for an example of a train track of this form.

Before proving Proposition \ref{prop:curves-plus-crossovers}, we review some linear algebra of matrices over $\F[U]/U^2$. Any invertible matrix over $\F[U]/U^2$ can be written as the product of elementary matrices of the following two types:
$$A_{i,j} = \scalemath{.75}{\begin{blockarray}{cccccccc}
 & & & & & j & & \\
\begin{block}{c[ccccccc]}
 & 1 & & & & & & \\
 & & \ddots & & & & & \\
i& & & 1 & & 1 & & \\
 & & & & \ddots & & & \\
 & & & 0 & & 1 & & \\
 & & & & & & \ddots & \\
 & & & & & & & 1 \\
 \end{block}\end{blockarray}}  \qquad\qquad 
A^U_{i,j} = \scalemath{.75}{\begin{blockarray}{cccccccc}
 & & & & & j & & \\
\begin{block}{c[ccccccc]}
 & 1 & & & & & & \\
 & & \ddots & & & & & \\
i& & & 1 & & U & & \\
 & & & & \ddots & & & \\
 & & & 0 & & 1 & & \\
 & & & & & & \ddots & \\
 & & & & & & & 1 \\
 \end{block}\end{blockarray}} $$
For $A_{i,j}$ we require $i\neq j$; in $A^U_{i,i}$ the $(i,i)$ entry is $1+U$. Each of these elementary matrices is its own inverse. We remark that there is another type of elementary matrix, transposition matrices, but over $\F = \Z/2\Z$ these can be obtained by products of matrices of type $A_{i,j}$ .

\begin{lemma}\label{lem:linear_algebra}
Let $M$ be a $2n\times 2n$ matrix over the base ring $\F[U]/U^2$ that is upper triangular modulo $U$ such that $M^2 = U I_{2n}$. Then $M$ can be written as $P = P \bar M P^{-1}$ where $P$ is composition of elementary matrices of the form $A^U_{i,j}$ or of the form $A_{i,j}$ with $i<j$, and $\bar{M}$ has the following form:
\begin{itemize}
\item Each row and each column of $\bar M$ has exactly one nonzero entry;
\item The diagonal entries of $\bar M$ are zero;
\item If $1\le i < j \le 2n$, either the $(i,j)$ entry and the $(j,i)$ entry of $\bar M$ are both zero or the $(i,j)$ entry is 1 and the $(j,i)$ entry is $U$.
\end{itemize}
\end{lemma}
\begin{proof}
Since the elementary matrices are their own inverses, it is enough to show that $M$ can be reduced to a matrix $\bar M$ of the desired form by conjugating with elementary matrices of the form $A^U_{i,j}$ or of the form $A_{i,j}$ with $i<j$. Note that conjugating by these matrices preserves the fact that $M$ is upper triangular mod $U$ and that $M^2 = UI_{2n}$.

We proceed by induction on $n$. For $n = 1$, there are finitely many $2 \times 2$ matrices over $\F[U]/U^2$. Only four of these square to $U I_{2}$ and are upper triangular mod $U$:
$$\left[ \begin{array}{cc} 0 & 1 \\ U & 0\end{array} \right], 
\left[ \begin{array}{cc} 0 & 1+U \\ U & 0\end{array} \right], 
\left[ \begin{array}{cc} U & 1 \\ U & U\end{array} \right], \text{ and }
\left[ \begin{array}{cc} U & 1+U \\ U & U\end{array} \right].
$$
The first has the desired form already, and it is routine to check that the remaining three reduce to the first after conjugating with $A^U_{1,1}$, $A^U_{2,1}$, and $A^U_{1,1}A^{U}_{2,1}$, respectively.

If $n > 1$, choose $i$ and $j$ with $j-i$ as small as possible so that the $(i,j)$ entry of $M$ is 1 or $1+U$. Conjugating by $A_{j,k}$ has the effect of adding the $j$th column to the $k$th column and adding the $k$th row to the $j$th row. After  conjugating by $A_{j,k}$ if necessary for each $k>j$, we can ensure that the $(i,k)$ entry of $M$ is 0 or $U$. After conjugating with matrices of the form $A^U_{j,k}$ for any $k$, we can assume that the $i$th row has a 1 in the $j$th entry and zeros elsewhere. The fact that $M^2 = U I_{2n}$ then implies that the $j$th row has a $U$ in the $i$th entry and zeroes elsewhere. Similarly, after conjugating with $A_{k,i}$ with $k<i$ or with $A^U_{k,i}$ for any $k$ we can ensure that the only nonzero entry in the $j$th column of $M$ is a 1 in the $i$th entry. It follows from the fact that $M^2 = U I_{2n}$ that the only nonzero entry in the $i$th column is a $U$ in the $j$th entry.

After the conjugations described above, the matrix $M$ has the form desired in the $i$th and $j$th rows and columns. We need to show that rest of the matrix can be put in the desired form using conjugations with $A_{k,\ell}$ and $A^U_{k,\ell}$ with $k$ and $\ell$ not in $\{i,j\}$. This is equivalent to showing that the the $2(n-1)\times 2(n-1)$ matrix obtained from this one by deleting the $i$th and $j$th rows and columns can be reduced to the desired form, which we take as the inductive hypothesis.
\end{proof}

\piccaption[]{After adding a crossover arrow from $x_i$ to $x_j$, the path from $y$ to $x_i$ gives rise to a path from $y$ to $x_j$ (shaded). There is also a new path from $x_i$ to $y'$, which carries a weight of $U$. \label{fig:conjugation}}
\parpic[r]{
 \begin{minipage}{51mm}
 \hspace{5mm}
 \labellist
 \pinlabel {$y$} at 188 106
 \pinlabel {$x_i$} at 105 -8
 \pinlabel {$x_j$} at -10 106
 \pinlabel {$y'$} at 190 75
 \tiny
 \pinlabel {$\bu$} at 170 170
 \endlabellist
\includegraphics[scale=0.6]{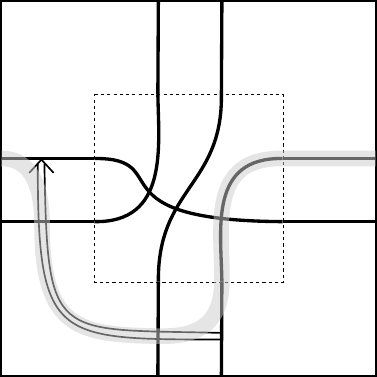}
 \vspace{1mm}
  \end{minipage}%
}
Lemma \ref{lem:linear_algebra} lets us relate $M$ to a matrix of a simple form through conjugation by elementary matrices; this in turn has a nice geometric interpretation in terms of train tracks. Let $\tracks$ be a train track that has only horizontal or vertical segments outside the corner box $[\frac{1}{4},\frac{3}{4}]\times[\frac{1}{4},\frac{3}{4}]$, and let $\tracks'$ be the train track obtained from $\tracks$ by adding a clockwise moving crossover arrow outside the corner box from the segment with endpoint $x_i$ to the segment with endpoint $x_j$. We claim that $M_{\tracks'}$ is obtained from $M_\tracks$ by conjugating with an elementary matrix. The counts of smooth paths are identical except that there are some new paths in $\tracks'$ involving the crossover arrow. For each path into $x_i$ in $\tracks$, $\tracks'$ will have an additional path into $x_j$, and for each path out of $x_j$ in $\tracks$, $\tracks'$ will have an additional path out of $x_i$; if the crossover arrow passes the basepoint in the top right corner of the square $m$ times, the new paths will be weighted by an additional factor $U^m$. An example of this is shown in Figure \ref{fig:conjugation}. It follows that $M_{\tracks'}$ is obtained from $M$ by adding $U^m$ times the $j$th row to the $i$th row and $U^m$ times the $i$th column to the $j$th column. In other words, we conjugate by $A_{i,j}$ if the crossover arrow does not pass the base point or by $A_{i,j}^U$ if the crossover arrow passes the basepoint once. (If the basepoint is passed more than once then $M_{\tracks'} \equiv M_{\tracks}$ modulo $U^2$).

\begin{proof}[Proof of Proposition \ref{prop:curves-plus-crossovers}]
Let $\tracks'$ be a valid reduced train track in $T$ and let $M = M_{\tracks'}$ be the corresponding matrix. Set $M = P \bar M P^{-1}$ as in Lemma \ref{lem:linear_algebra}; the matrix $\bar M$ corresponds to an extended type D structure, from which we can construct a train track $\bar \tracks$. To construct $\bar \tracks$, let $\bar \tracks$ have the same primary vertices as $\tracks'$, labelled $x_1,\ldots x_{2n}$ as described above, and add oriented corner edges from $x_i$ to $x_j$ if the $(i,j)$ entry of $\bar M$ is nonzero. The form of $\bar M$ implies that there will be exactly one incoming and one outgoing corner edge at each primary vertex, and that for each corner edge from $x_i$ to $x_j$ there is an oppositely oriented edge from $x_j$ to $x_i$. Replacing the pairs of oppositely oriented corner edges with single unoriented corner edges as usual, observe that there is exactly one unoriented corner edge at each primary vertex. It follows that $\bar \tracks$ is an immersed multicurve.

We observed that $M$ is obtained from $\bar M$ by conjugating with a sequence of elementary matrices of the form $A_{i,j}$ with $i<j$ or $A^U_{i,j}$. Each conjugation corresponds to adding a crossover arrow to $\bar \tracks$ near the boundary of the square. That is, adding the appropriate sequence of crossover arrows to $\bar \tracks$ produces a train track $\tracks$ whose   corresponding matrix is $M$. Since we only use type $A_{i,j}$ matrices if $i<j$, all the crossover arrows added move clockwise. Up to regular homotopy, we can assume that $\bar \tracks$ is horizontal or vertical outside of the corner box $[\frac{1}{4},\frac{3}{4}]\times[\frac{1}{4},\frac{3}{4}]$ and that the crossover arrows added lie outside the corner box. Since $\tracks$ is reduced, the arcs within the corner box connect different sides of the square. Thus $\tracks$ has the form claimed and $\tracks$ is equivalent to $\tracks'$ since the matrix associated to each is $M$.

Conversely, given any train track $\tracks$ of this form, there is a reduced train track $\bar \tracks$ obtained by removing the crossover arrows. $\bar \tracks$ is clearly valid and $M_\tracks$ is obtained from $M_{\bar \tracks}$ by conjugating with elementary matrices, which preserves the property that the matrix squares to $U I_{2n}$. It follows that $\tracks$ is valid as well as reduced.
\end{proof}

We remark that in many cases, a naive train track can be transformed into a train track of the form predicted by Proposition \ref{prop:curves-plus-crossovers} geometrically, without using the linear algebra above. For example, in the running example in Figure \ref{fig:crossover-arrows}(a), pairs of oriented edges have been homotoped together to form crossover arrows; it is clear that this homotopy preserves the equivalence class of train tracks. A further homotopy ensures that the crossover arrows lie outside of a corner box, as in Figure \ref{fig:curves-and-crossovers}(a). This train track has the desired form except that two of the crossover arrows run counterclockwise. This can be fixed by sliding the two counterclockwise arrows through the corner box and using some local moves that preserve the equivalence class of train track, which will be discussed in Section \ref{sub:calculus} (see Figure \ref{fig:crossover_moves}); in this case the local moves result in resolving a crossing of the immersed curve and adding an additional crossover arrow. In small examples, this geometric approach is often faster than the linear algebra used above. However, it would difficult to make this approach systematic for all cases. By passing from train tracks to matrices and back, the proof of Proposition \ref{prop:curves-plus-crossovers} takes advantage of the interplay between linear algebra and our geometric arguments.

\begin{figure}
\labellist
\pinlabel {(a)} at 150 -25
\endlabellist
\includegraphics[scale=0.4]{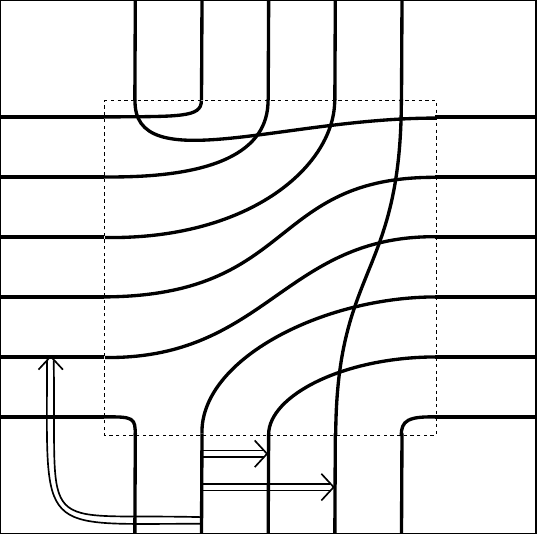} \qquad
\labellist
\pinlabel {(b)} at 150 -25
\endlabellist
\includegraphics[scale=0.4]{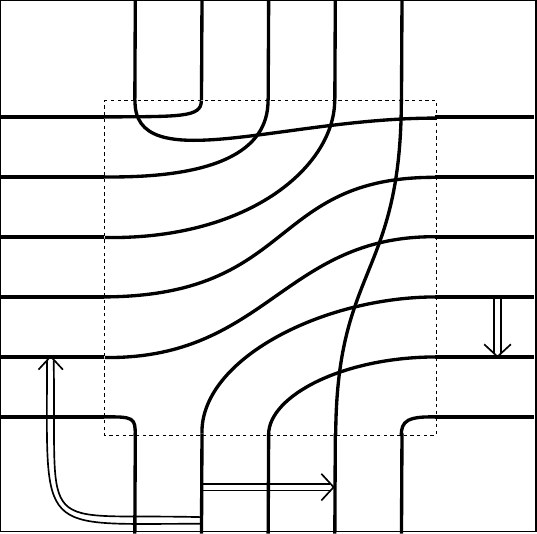} \qquad
\labellist
\pinlabel {(c)} at 150 -25
\endlabellist
\includegraphics[scale=0.4]{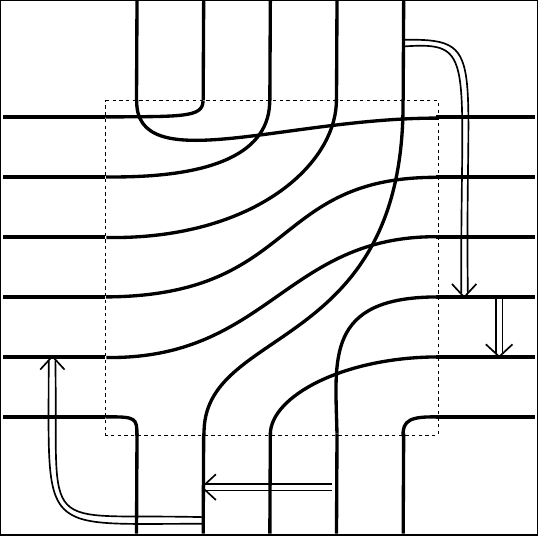}

\vspace{10 pt}
\caption{(a) A train track equivalent to the one in  in Figure \ref{fig:crossover-arrows}, after a slight homotopy. This has the form specified in Proposition \ref{prop:curves-plus-crossovers} except that two crossover arrows move counterclockwise. (b) The innermost counterclockwise arrow can be replaced by a clockwise arrow by pushing it through the corner region. (c) The remaining counterclockwise arrow can be replaced with two clockwise arrows by pushing through the corner region and resolving a crossing. Note that this train track is in the form guaranteed by Proposition \ref{prop:curves-plus-crossovers}.}\label{fig:curves-and-crossovers}
\end{figure}

\subsection{Crossover arrow calculus}\label{sub:calculus} 

\begin{figure}
\labellist
\pinlabel {(a)} at 10 340
\pinlabel {(b)} at 0 227
\pinlabel {(c)} at 330 230
\pinlabel {(d)} at -25 125
\pinlabel {(e)} at 145 15

\pinlabel {$\sim$} at 104 340
\pinlabel {$\sim$} at 270 340
\pinlabel {$\sim$} at 436 340

\pinlabel {$\sim$} at 75 227
\pinlabel {$\sim$} at 165 227
\pinlabel {$\sim$} at 418 227

\pinlabel {$\sim$} at 75 123
\pinlabel {$\sim$} at 255 123
\pinlabel {$\sim$} at 360 123
\pinlabel {$\sim$} at 465 123

\pinlabel {$\sim$} at 238 15
\pinlabel {$\sim$} at 335 15

\endlabellist
\includegraphics[scale=.6]{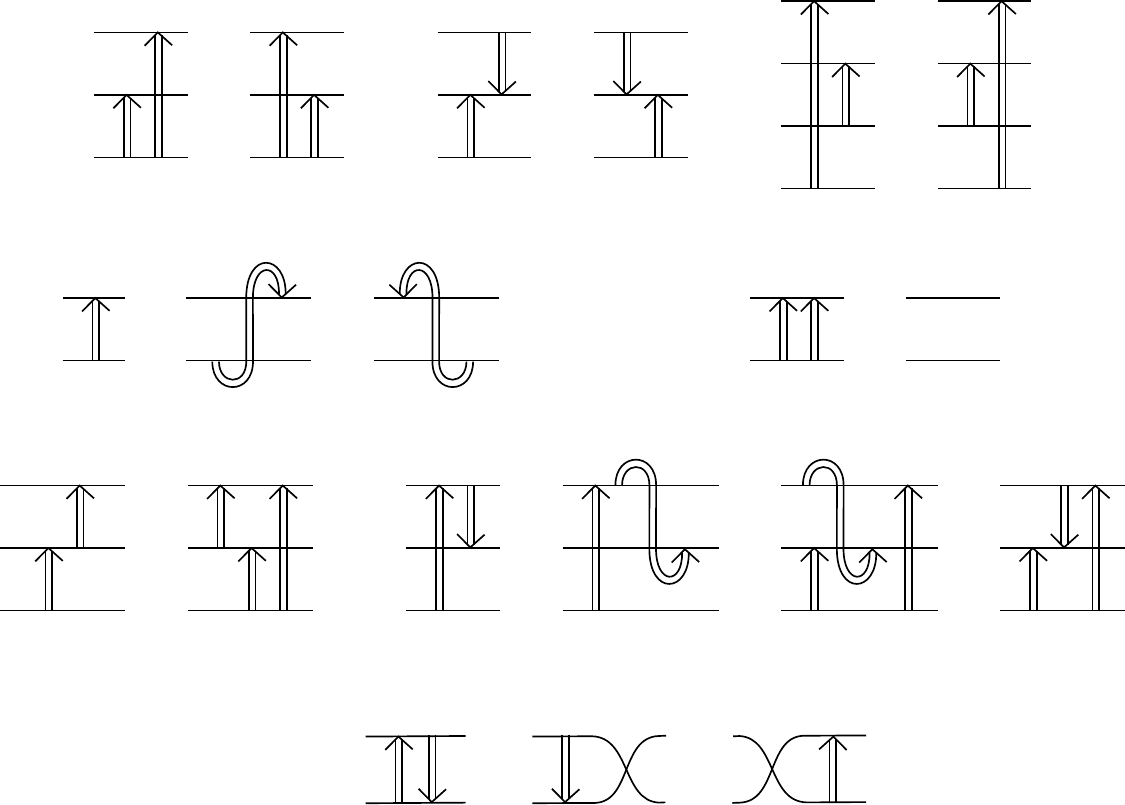}
\caption{Local moves for crossover arrows producing equivalent train tracks: (a) beginnings or ends of arrows slide past each other freely (three such instances are pictured, though this is not exhaustive; for instance, similar configurations can be obtained by reversing the direction of all arrows),  (b) U-turns can be added at the start and end of an arrow if they turn opposite directions; (c) a pair of adjacent parallel crossover arrows connecting the same two strands can be added or removed; (d) the start of one arrow may slide past the end of another arrow at the expense of adding a new arrow that is the composition of the two (two example configurations are shown); (e) opposing arrows can be replaced with one arrow and a crossing.}
\label{fig:crossover_moves}
\end{figure}

\begin{figure}
\labellist
\Huge
\pinlabel {$\leftrightarrow$} at 210 87
\pinlabel {$\leftrightarrow$} at 735 87

\endlabellist
\includegraphics[scale=.45]{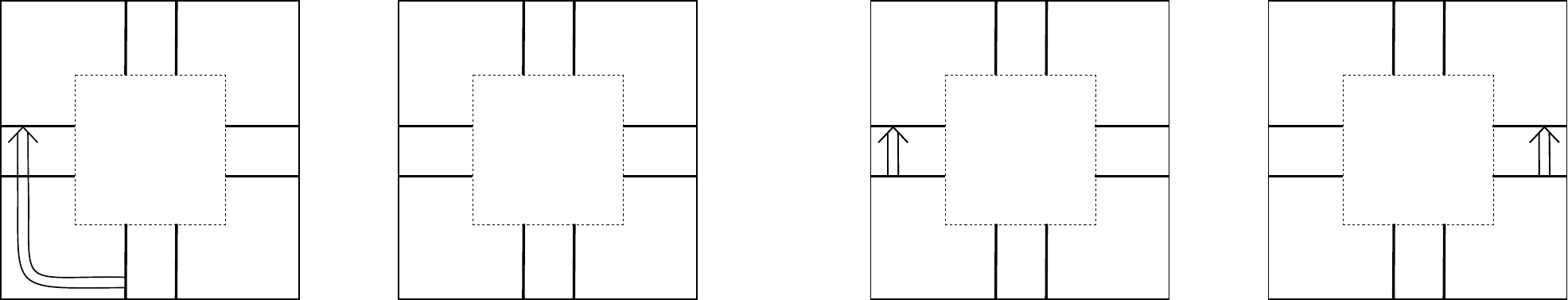}

\caption{Examples of the moves (M1), left, and (M2), right.}
\label{fig:movesM1M2}
\end{figure}

By Proposition \ref{prop:curves-plus-crossovers}, we can always represent a reduced extendable type D structure with a train track consisting of an immersed multicurve with crossover arrows. Our next aim is to simplify these train train tracks using certain geometric moves involving sliding crossover arrows along the immersed curves. First, we are allowed to slide arrows along the multicurve within the square $T\setminus (\alpha\cup\beta)$ without changing the equivalence class of the train track, as long as ends of arrows do not slide past each other. Moreover, the local moves shown in Figure \ref{fig:crossover_moves} produce equivalent train tracks. To see this, note that the mod 2 count of oriented paths between any pair of endpoints is the same on both sides of each replacement. Arrows may slide past each other within the square $T\setminus (\alpha\cup\beta)$ with the caveat that if the head of one arrow passes the tail of another arrow, we must add a new arrow that is homotopic to the composition of the two arrows. By composition, we mean the the arrows are stacked head to tail, where we assume the tail of the second arrow and the head of the first arrow lie on opposite sides of an immersed curve at the same point; to arrange this we may first need to add U-turns at the beginning and end of one arrow, which may be viewed as a regular homotopy of train tracks. Finally, a pair of opposing arrows connecting the same two segments of immersed curve can be replaced with a crossing and one arrow.

In addition to the local moves in Figure \ref{fig:crossover_moves}, we will make use of two additional moves to modify the train tracks:

\begin{itemize}
\item[(M1)] add or remove a clockwise crossover arrow that is in a neighborhood of the boundary of the square $T\setminus (\alpha\cup\beta)$ and which passes at least one corner of the square;

\item[(M2)] 
move a crossover arrow which is in a neighborhood of one side of the square $T\setminus (\alpha\cup\beta)$ to the opposite side of the square; in other words, slide an arrow across either $\alpha$ or $\beta$ in the torus $T$.
\end{itemize}

Examples of these moves are shown in Figure \ref{fig:movesM1M2}. In each move, we assume that there are no other crossover arrows between the arrow being (re)moved and the boundary of the square. More precisely, there is a neighborhood of the boundary of the square containing the crossover arrow in question but no other crossover arrows, and in which the immersed curve consists of horizontal and vertical segments. In contrast with the local moves described in Figure \ref{fig:crossover_moves}, the moves (M1) and (M2) change the equivalence class of the train track. However, these moves induce isomorphisms between the associated extended type D structures.

\begin{proposition}\label{prp:moves}
Suppose $\tracks'$ is obtained from $\tracks$ by an application of a sequence of the moves (M1) and (M2) together with the local moves depicted in Figure \ref{fig:crossover_moves}. If $\widetilde{N}'$ and $\widetilde{N}$ are the extended type D structures associated with $\tracks'$ and $\tracks$, respectively, then there is an isomorphism $\widetilde{N}'\cong\widetilde{N}$ as $\widetilde{\Alg}$ modules.
\end{proposition}

\begin{proof}
As observed, all local moves produce equivalent train tracks, and hence equality on the associated extended type D structures. To prove the proposition then, we need to establish changes of basis for each of the moves (M1) and (M2) that induce the desired isomorphisms. 

Suppose first that $\tracks'$ is obtained from $\tracks$ by an application of (M1), adding a crossover arrow from horizontal or vertical edge representing a generator $x$ of $\widetilde{N}$ to an edge representing a generator $y$. We claim that $\widetilde{N}'$ is related to $\widetilde{N}$ by the change of basis replacing $x$ with $x + \rho_I \otimes y$, where $I$ is the sequence of corners passed by the crossover arrow being added.

Consider for concreteness the case where the sequence of corners $I$ associated with the arrow being added is 1 (other cases of clockwise arrows are similar). Thus $x$ has idempotent $\iota_0$, $y$ has idempotent $\iota_1$, and we remove a crossover arrow from the left end of the horizontal segment corresponding to $x$, which we will call $v_x$, to the top end of the vertical segment corresponding to $y$, which we will call $v_y$. Adding the crossover arrow adds a new corner path starting at $v_x$ for each corner path starting at $v_y$, and a new corner path ending at $v_y$ for each one ending at $v_x$. The corner paths starting at $v_y$ correspond to $\rho_J \otimes z$ terms in $\tilde\partial(y)$ where the sequence $J$ begins with $2$. The new corner paths starting from $v_x$ pass one extra corner, so they correspond to terms $(\rho_1\rho_{J}) \otimes z$ in $\tilde\partial(x)$. In other words, we add $\rho_1 \tilde\partial(y)$ to $\tilde\partial(x)$ (note that multiplying $\tilde\partial(y)$ on the left by $\rho_1$ picks out exactly the terms for which the Reeb chord interval starts with 2). Similarly, each corner path ending at $v_x$ corresponds to a term $\rho_I \otimes x$ in $\tilde\partial(z)$ for some $z$ and some $I$ ending in 0, and for each the new corner path added corresponds to a $(\rho_{I} \rho_1) \otimes y$ term in $\tilde\partial(z)$. This is precisely the effect of the change of basis (over $\widetilde{\Alg}$) replacing $x$ with $x + \rho_1 \otimes y$.

Now suppose that $\tracks'$ is obtained from $\tracks$ by an application of (M2), where the crossover arrows slides between two horizontal or two vertical edges representing generators $x$ and $y$ of $\widetilde{N}$ and the crossover arrow is oriented from the edge representing $x$ to the edge representing $y$. In this setting, we claim that $\widetilde{N}'$ is obtained from $\widetilde{N}$ by the change of basis that replaces $x$ with $x+y$.

Consider a train track $\tracks$ in $T$ with associated extended type D structure $\widetilde{N}$. Let $x$ and $y$ be generators of $\widetilde{N}$, and suppose that $x$ and $y$ both have idempotent $\iota_0$ (the case where $x$ and $y$ both have idempotent $\iota_1$ is identical). Let $v_{i_x}$ and $v_{j_x}$ be the marked points on the right and left boundary, respectively, of the square $T\setminus(\alpha\cup\beta)$ associated to $x$, and let $v_{i_y}$ and $v_{j_y}$ be marked points on the right and left boundary of the square associated to $y$. Suppose $\tracks'$ is obtained by connecting the horizontal segments by crossover arrows at each end, one from $v_{i_x}$ to $v_{i_y}$ and one from $v_{j_x}$ to $v_{j_y}$; note that this is equivalent to move (M2), since we could equivalently add two cancelling arrows at one end of the horizontal strands and then slide one to the other end. The effect of this move on $M_{\tracks}$ is conjugation by $A_{i_x,i_y}A_{j_x,j_y}$. For each corner path out of $v_{i_y}$ or $v_{j_y}$ we add the corresponding corner path out of $v_{i_x}$ or $v_{j_x}$, and for each corner path into $v_{i_x}$ or $v_{j_x}$ we add the corresponding corner path into $v_{i_y}$ or $v_{j_y}$. Each corner path out of $v_{i_y}$ or $v_{j_y}$ corresponds to a term $z$ in $D_I(y)$ for some $I$, and the new corner path added corresponds to the same term in $D_I(x)$. Thus we add each term in $\partial(y)$ to $\partial(x)$. Each corner path into $v_{i_x}$ or $v_{j_x}$ corresponds to an $x$ term in $D_I(z)$ for some generator in $z$ and some $I$, so for each such term we add the term $y$ to $D_I(z)$. This is exactly the effect of the change of basis on $\widetilde{N}$ that replaces $x$ with $x+y$. This shows that move (M2) corresponds to a change of basis in the extended type D structure $\widetilde{N}$, as claimed.
\end{proof}

\subsection{Prospectus for simplification} With these observations in place, our goal is to use the local moves in Figure \ref{fig:crossover_moves} and moves (M1) and (M2) to simplify a given train track by systematically removing crossover arrows. Below, we describe our general strategy for doing this. In individual examples, it is often easy to simplify the train track by inspection, using the strategy as a guide.  In the next section, we will give a  formal algorithm showing that the strategy can always be applied. 

One might hope to remove all crossover arrows from the train track, resulting in simply a collection of immersed curves. This, it turns out, is not always possible; instead, we will remove all arrows which do not connect parallel immersed curves. To make sense of this, we must first define what we mean by parallelism. Associated with a crossover arrow, we have a pair of pointed curves \((\gamma, p)\) and \((\gamma',p')\), where \(p \in \gamma\) is the tail of the arrow, and \(p'\in \gamma'\) is its head. We identify \(\pi_1(T,p)\) with \(\pi_1(T,p')\) by choosing a path from \(p\) to \(p'\) which is disjoint from \(\alpha \cup \beta\), so after choosing orientations we can view \((\gamma,p)\) and \((\gamma',p')\) as defining \(\gamma, \gamma' \in \pi_1(T,p)\). 
\begin{definition}\label{def:parallel}
Pointed closed curves \((\gamma,p), (\gamma,p')\) are parallel if there is some \(\delta\in \pi_1(T,p)\) and \(k, k' \in \Z\) such that \(\gamma = \delta^{k},\) \(\gamma' = \delta^{k'}\). 
\end{definition}
Since \(k,k' \in \Z\), this definition does not depend on the choice of orientations. 
Note that the choice of basepoints is very important; for example, if \(p\) and \(p'\) lie on different components of \(\gamma  - (\alpha \cup \beta)\), \((\gamma, p)\) is usually not parallel to \((\gamma, p')\). 

Equivalently, assume that \(\gamma\)  is in minimal position with respect to \(\alpha\) and \(\beta\). 
Following the orientation on $\gamma$ starting at $p$ determines a (periodic) infinite sequence of signed intersections with $\alpha$ or $\beta$, which we can interpret as an infinite periodic  word in two variables and their inverses; two pointed curves are parallel if there is some choice of orientations for which they determine the same word. 

The definition of parallel pointed curves above suggests a local notion of parallelism. This can be made more precise with the notions of colors and weights that will be introduced in Section \ref{sub:removing}, but for now we simply say two pointed curves are \emph{locally parallel} moving in a given direction (either with the orientations or against the orientations) if the first signed intersection with $\alpha \cup \beta$ on each curve is of the same type. If pointed curves are not locally parallel we say they \emph{diverge} in the given direction. When two curves diverge, one curve diverges to the left and one diverges to the right. In general, curves will be locally parallel for some finite number of intersections and then diverge. Curves which never diverge are parallel in the sense of Definition \ref{def:parallel}; we will sometimes say such curves are \emph{globally parallel} to avoid confusion.

\begin{figure}[h]
\includegraphics[scale=0.4]{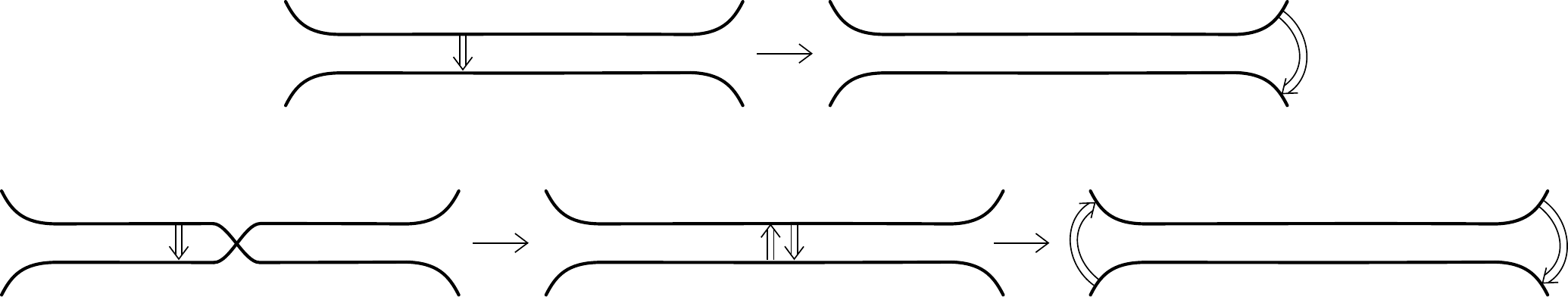} 
\caption{A schematic for removing a single crossover arrow that does not connect globally parallel curves. In the first case we slide the arrow until the curves diverge and the arrow moves left to right, and then remove the arrow. In the second case we resolve a crossing and then remove the two resulting arrows as in the first case.}\label{fig:removing-one-arrow}
\end{figure}

Suppose a train track consists of a collection of immersed curves and a single crossover arrow. If the crossover arrow does not connect (globally) parallel curves, then it can be removed using the moves (M1), (M2), and the local moves from Figure \ref{fig:crossover_moves}. Since the pointed curves connected by the arrow are not globally parallel, they must eventually diverge in both directions. We begin by sliding the crossover arrow along the curves in the direction of the orientation. If the curves are locally parallel in this direction, we slide the arrow across $\alpha$ or $\beta$ using (M2); repeating this as necessary we slide the arrow until the two curves diverge. If the arrow goes from the left curve to the right curve when the curves diverge, it can be realized as a clockwise moving crossover arrow around the corner box that passes at least one corner of the square; in this case the arrow can be  removed using (M1). If instead the arrow goes from the right curve to the left curve, then we slide the arrow along the curves in the opposite direction until they diverge. Once again, if the arrow goes from the left curve to the right curve at this point of divergence, it can be removed using (M1). If the arrow moves from right to left at both ends, then the curves must cross at some point in between. In this case we slide the arrow to the intersection and resolve the crossing using the local move (e) from Figure \ref{fig:crossover_moves}. The resulting train track has two crossover arrows, but when these slide along the curves in opposite directions they will be left-to-right, and thus removable, at both ends. This process is summarized in Figure \ref{fig:removing-one-arrow}. Note that in the second case case we have removed the original arrow and arrived at a collection of immersed curves with no crossover arrows, but the immersed curves are not the ones we started with because a crossing has been resolved.

If a train track contains many crossover arrows, any one of them can be removed by the procedure above, provided it does not connect parallel curves. It is plausible, then, that we can remove all crossover arrows that do not connect parallel curves by repeatedly removing one at a time.  The key potential issue here is that when multiple arrows are present, removing one of them may involve sliding it past others and this may create more arrows. We need to control the number of new arrows added to ensure that the process of removing arrows one at a time eventually terminates. To deal with this, we will introduce a bookkeeping tool in the form of a weight system for arrows, which measure how far an arrow needs to be pushed before the strands diverge. The strategy is to deal with the easiest to remove arrows (i.e. those with smallest weights) first. These weights will be infinite for crossover arrows connecting globally parallel curves, and we will show that there is an inductive algorithm that, at each step, increases the smallest weight (taken over all arrows) by one while controling the number of arrows added in the process. By compactness of the underlying curves, after a finite number of steps the only remaining arrows must have infinite weight. Figure \ref{fig:end-runnning-example} shows the outcome of this process applied to the train track in Figure \ref{fig:curves-and-crossovers}(c). Note that when performing this procedure by hand, it may not be necessary to explicitly specify this weight system.

\begin{figure}[h]
\includegraphics[scale=0.4]{figures/example-curves-and-crossovers_c} \hspace{20 mm}
\includegraphics[scale=0.4]{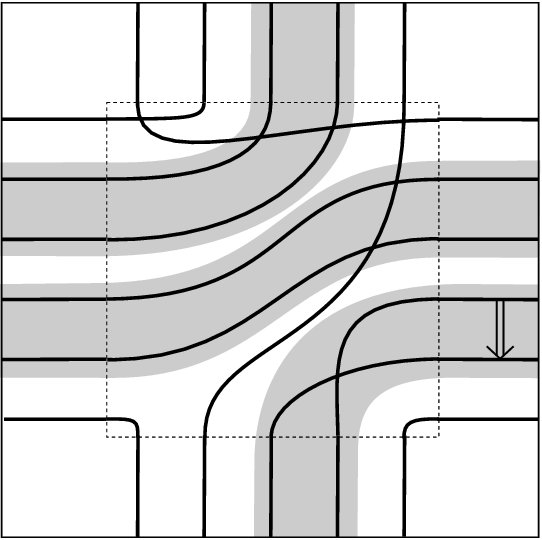}
\caption{The running example: On the left is the train track from Figure \ref{fig:curves-and-crossovers}(c). The arrow passing the top right corner of the square can be removed by an application of (M1); since it is not the outermost arrow it must first be moved closer to the boundary of the square, which introduces a new arrow, but this can also be removed by (M1). The arrow passing the bottom left corner of the square can also be removed by (M1). The arrow parallel to the bottom side of the square is slid downward, applying (M2) twice before being removed with (M1). The remaining crossover arrow, which connects two points on an immersed curve at which the curve is self parallel, cannot be removed.}\label{fig:end-runnning-example}
\end{figure}

%

\subsection{Arrow sliding algorithm} \label{sub:removing} We now formalize the strategy suggested above. By Proposition \ref{prop:curves-plus-crossovers}, we may replace any train track with a configuration of curves and crossover arrows; it remains to show that all crossover arrows can be eliminated unless they connect parallel (pointed) curves. 

 \labellist
 \pinlabel {$\vdots$} at 360 54  \pinlabel {$\vdots$} at 360 160
  \pinlabel {$\sigma_\bu$} at 367 290  \pinlabel {$\sigma_\circ$} at 367 427
   \pinlabel {$A$} at 310 293  \pinlabel {$A$} at 310 430
   \pinlabel {$A$} at 418 293  \pinlabel {$A$} at 418 430

 \tiny \pinlabel {$z$} at 243 344
 \endlabellist
\begin{figure}[h] \includegraphics[scale=0.35]{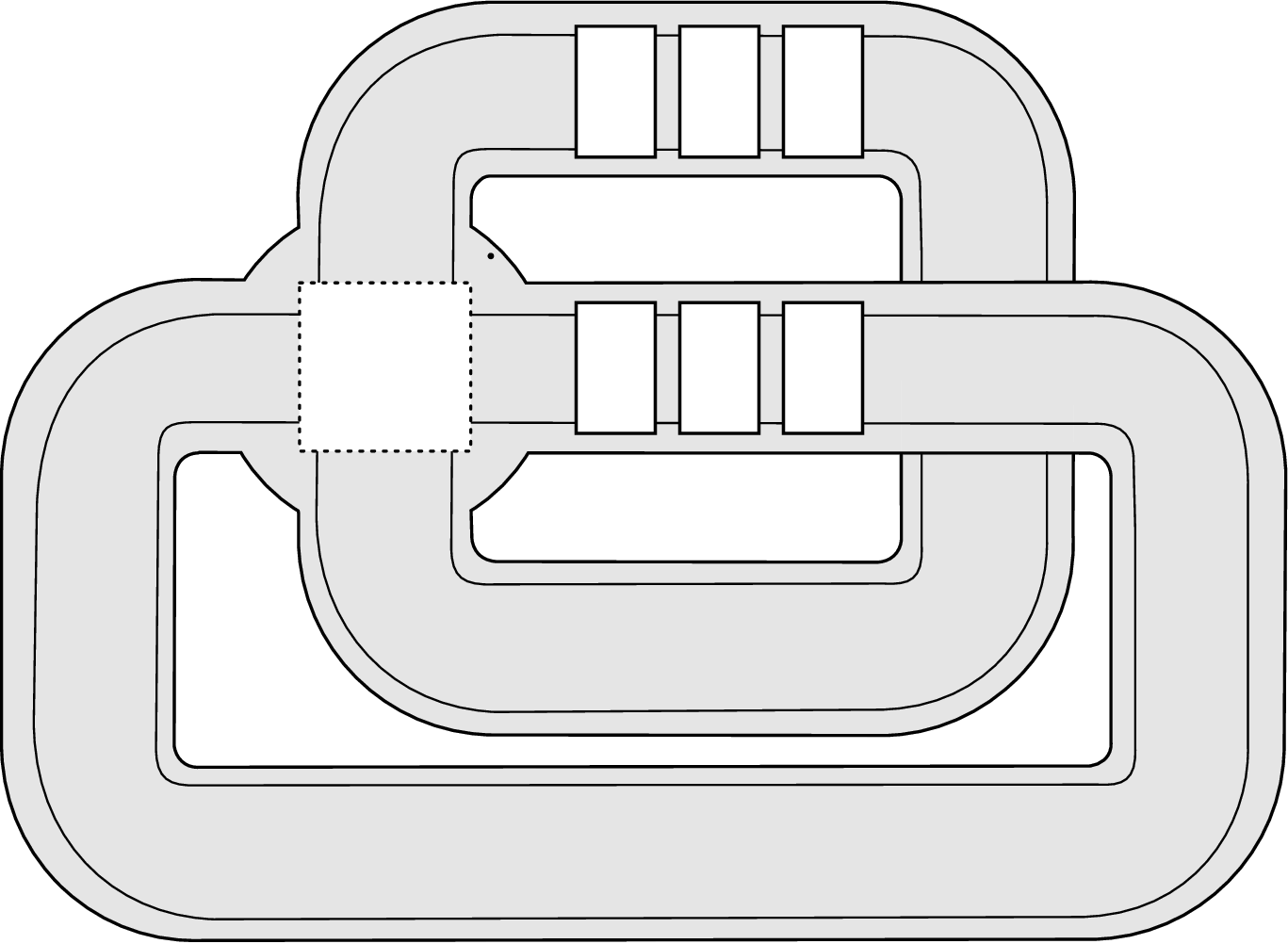}
\caption{The anatomy of a train track in $T$, according to Proposition \ref{prop:curves-plus-crossovers}: After applying (M1) moves to remove each of the (clockwise) crossover arrows covering a corner, all of the crossover arrows can be moved into the 1-handles (these are contained in the four boxes labeled $A$). Each one handle contains a permutation box (labelled $\sigma_\bu$ and $\sigma_\circ$ to agree with the notation for idempotents) containing an immersed collection of curve segments encoding the permutation. Finally, following previous sections, the dashed box in the 0-handle corresponds to the corner box.}\label{fig:anatomy}
\end{figure}

 \labellist
  \tiny \pinlabel {$z$} at 243 344
  \pinlabel {$\infty$} at 422 282.5
    \pinlabel {$(2,1)$} at 270 450.6
 \endlabellist
\begin{figure}[h] \includegraphics[scale=0.35]{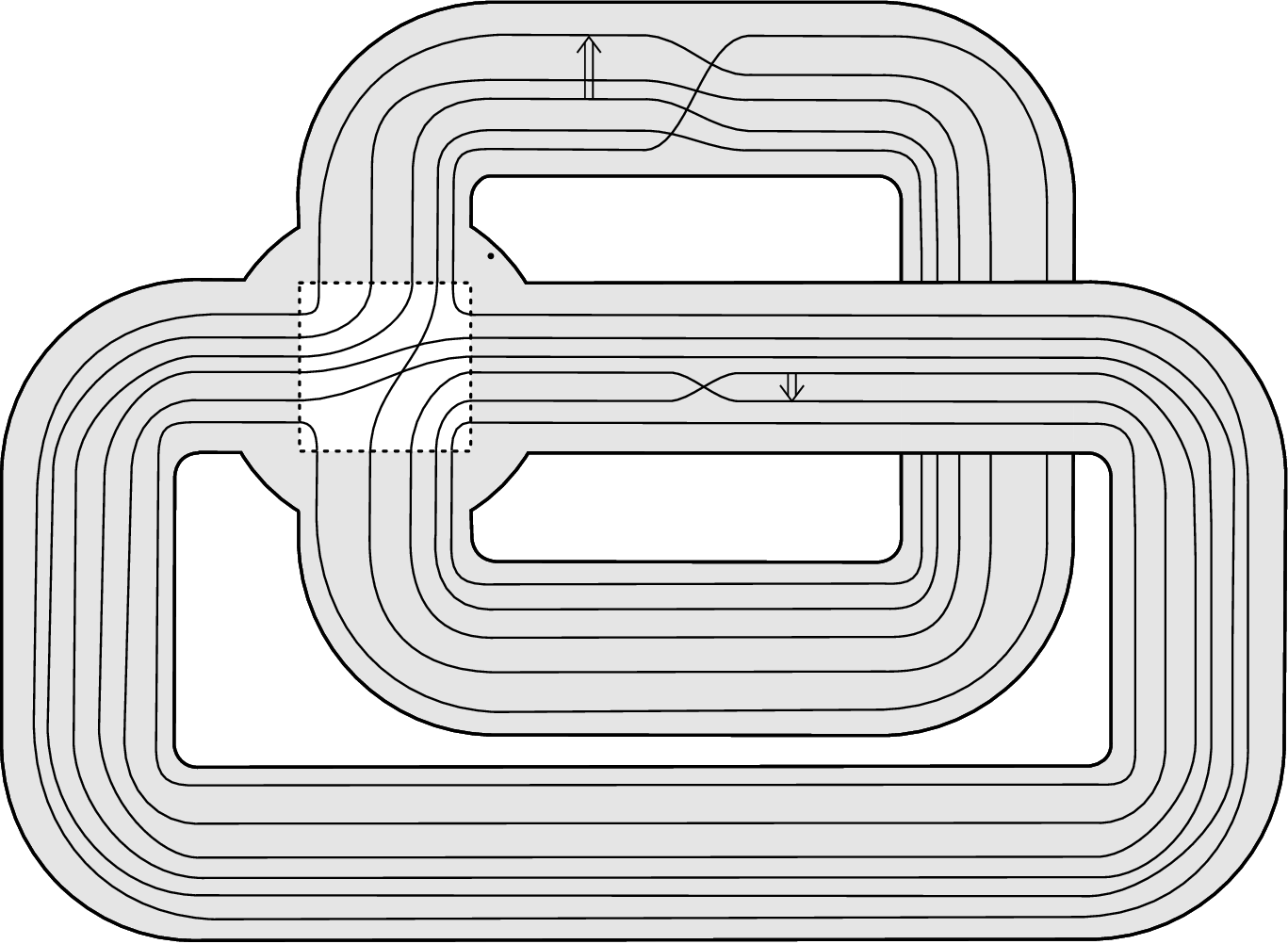}
\caption{An equivalent view of the train track from Figure \ref{fig:curves-and-crossovers}(c), where the two crossover arrows that passed corners have been removed by a change of basis using (M1) moves. Notice that, in each handle, the arrows and permutations satisfy Lemma \ref{Lemma1}. The crossover arrows have been labelled by their respective weights; comparing Figure \ref{fig:end-runnning-example} shows that the arrow with finite depth $(\wtp,\wtm)=(2,1)$ can be removed by applications of the moves (M1) and (M2), while the infinite  depth arrow persists.}\label{fig:anatomy-running-example}
\end{figure}

It will now be useful to make explicit the handle decomposition associated with the punctured torus $T$, in which the corner box corresponds to a 0-handle; see Figure \ref{fig:anatomy}. The corner box contains an arbitrary configuration of embedded arcs (possibly intersecting one another). Because all crossover arrows in the output of Proposition \ref{prop:curves-plus-crossovers} move clockwise, any arrows passing a corner of the corner box can be immediately removed by move (M1) and we may assume all other arrows lie in one of the two 1-handles.  We previously assumed that the immersed curves consist of purely horizontal or vertical segments outside the corner box (that is, in the 1-handles), but it will be convenient to allow these segments to cross each other so that the endpoints of these segments can be ordered independently on each end of the 1-handle. We collect all crossings into \emph{permutation boxes} near the middle of the handle, which we label \(\sigma_\bu\) and \(\sigma_\circ\); they contain a braid-like immersed curve diagram. The boxes labeled \(A\) are \emph{arrow boxes}; they contain parallel curve segments joined by an arbitrary configuration of crossover arrows. All of the crossover arrows are contained in these arrow boxes. Each strand in an arrow box has two directions: towards the corner box and away from the corner box (toward the permutation box). A specific example is shown in Figure \ref{fig:anatomy-running-example}.

\subsubsection*{Weights} Every crossover arrow in an arrow box can be assigned a weight, which is a value in the set $(\Z\setminus\{0\})^2\cup\{\infty\}$. If the strands on which the arrow starts and ends are globally parallel, the weight is \(\infty\). Otherwise, the weight is a pair \((\wtp,\wtm)\), where the integer \(\wtp\) ($w$-to) is defined as follows. Follow the strands on which the arrow begins and ends, starting by moving towards the corner box; at some point, the two strands will diverge (that is, leave the corner box through different edges). The non-zero integer $\wtp$ indicates that the arrow enters the corner box $|\wtp|$ times before the strands diverge; the sign of $\wtp$ records whether the arrow runs clockwise ($+$) or counterclockwise $(-)$ around the corner box when this happens. The non-zero integer $\wtm$ ($w$-from) is defined similarly by following the strands in the opposite direction, initially traveling away from the corner box, until they diverge. 

\begin{definition}
The {\it depth} of an arrow is  defined to be  $\min\{|\wtp|,|\wtm|\}$, or $\infty$ if the weight is \(\infty\). We will refer to a collection of immersed curves along with a collection of crossover arrows as a {\it curve configuration}; the depth of a curve configuration is the minimum depth of all its crossover arrows, or $\infty$ if there are no crossover arrows.
\end{definition}
 
By compactness of the curves, if an arrow has finite depth, then there is an absolute bound on that depth depending only on the number of strands in the configuration. Our aim is to prove:
\begin{proposition}
\label{prop:IncreaseDepth}
A curve configuration of depth \(m\) is equivalent to another curve configuration with the same number of strands and  depth no less than \(m+1\). 
\end{proposition} 
By induction, it will follow that any curve configuration is equivalent to an infinite depth curve configuration, that is, one in which the only crossover arrows run between parallel strands. Some additional structures and preliminary lemmas are required before proving Proposition \ref{prop:IncreaseDepth}.

\labellist
  \pinlabel {\scalebox{.5}{$z$}} at 243 344
  \tiny 
   \pinlabel {$nws$} at 316 321
      \pinlabel {$wsw$} at 316 303
       \pinlabel {$sww$} at 316 284
       \pinlabel {$ssw$} at 316 266
       \pinlabel $n$ at 71 468 \pinlabel $s$ at 71 352
        \pinlabel $w$ at 10 409 \pinlabel $e$ at 129 409
        \endlabellist
\piccaption[]{A depth 3 coloring in an arrow box; these strands are (lexicographically) ordered top-to-bottom since $n<w<s$ for the first letter and, for the two strands which begin with $s$, ${w<s<e}$ for the second letter. \label{fig:anatomy-colour}}
\parpic[r]{
 \begin{minipage}{63mm}
 \centering
\includegraphics[scale=0.28]{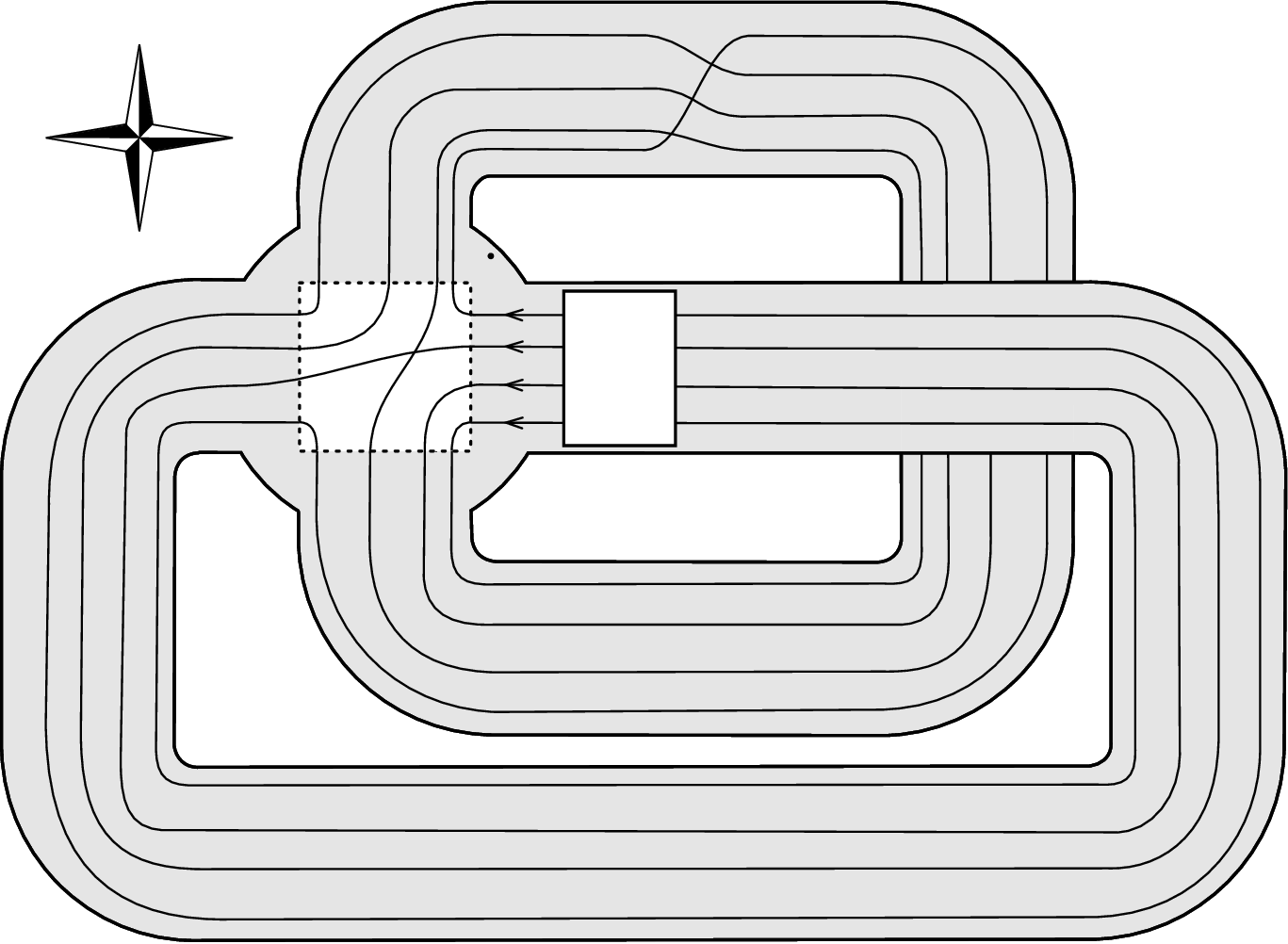}
  \end{minipage}%
}
{\em Colors and orders.} 
We now define a partial order on the set of strands in an arrow box. To do this, observe that any strand \(a\) in the box lifts to a quasi-geodesic in the universal cover of \(T\). By following the lift of \(a\) in the direction that initially moves towards the corner box, we get a ray \(\gamma_a\) in the cover. More explicitly, the universal cover of $T$ is a ribbon graph, which we can identify with a regular neighborhood of a  regular 4-valent tree properly  embedded in the hyperbolic plane \(\mathbb{H}^2\). The path \(\gamma_a\) takes in this tree can be described by an infinite string in the letters \(n,e,s,w\) recording which edge of the corner box the strand exits through each time it passes through the corner box, as illustrated in Figure \ref{fig:anatomy-colour}. By compactness of the curves in the configuration, this string is clearly periodic. Such a string associated with $a$ is called the coloring of $a$. Said another way, retracting the universal cover of $T$ to the Cayley graph for $\pi_1(T)\cong\langle n,e\rangle$, these colors corresponds to reduced words in the free group where $s=n^{-1}$ and $w=e^{-1}$.

Informally, the partial ordering on strands is given by $a < b$ if the strand $a$ is on the left relative to $b$ whenever the strands diverge. Equivalently, \(a<b\) means that if we take a  crossover arrow from \(a\) to \(b\) and push it in the positive direction until the strands diverge, the arrow will point in a counterclockwise direction when the strands diverge. This can be made precise in two ways. First, the Gromov boundary of the 4-valent tree  in $\mathbb{H}^2$ can be naturally identified with a Cantor set  \(C\subset S^1= \partial \mathbb{H}^2\).  Since \(\gamma_a\)  does not pass through the initial branch of the tree corresponding to the arrowbox it starts in, the set of possible endpoints of \(\gamma_a\) is the intersection of \(C\) with an interval \(I\subset S^1\) of length \(3 \pi/2\). Thus there is a natural order on the set of possible endpoints defined by declaring \(x<y\) if we can go counterclockwise from \(x\) to \(y\) in \(I\). We define \(a<b\) if \(\gamma_a(\infty)<\gamma_b(\infty)\). This ordering can also be described as a lexicographic order on colorings of strands, with the understanding that the order on the letters differs in each arrow box and thus in each position in the word. Each time the ray $\gamma_a$ passes through an arrow box and into the corner box, there are only three ways it can leave the corner box since backtracking along the 1-handle it came from is not allowed. At each pass through the corner box, we order the choices from rightmost to leftmost. In terms of letters, no backtracking means that, for instance, the letter $n$ is never followed by $s$. There are always three possible choices for the $(k+1)$st letter of a coloring, depending on $k$th letter, and these three choices inherit a linear ordering from the cyclic ordering $(n, w, s, e)$ by removing the disallowed letter. We then define $a < b$ if the color of $a$ is less than the color of $b$ in the lexicographical order---that is, if the first $k$ letters of the colors agree for some $k$ and the $(k+1)$st letter for $a$ is less than the $(k+1)$st letter for $b$ using the ordering obtained by disallowing the opposite of the $k$th letter.

Note that $<$ is a partial order, but $a$ and $b$ have an ordering relationship unless their colors agree; this happens precisely when the pointed curves starting with the strands $a$ and $b$ are globally parallel in the sense of Definition \ref{def:parallel}. We can define a sequence of weaker partial orders \(<_m\) on the set of strands in an arrow box using the truncation of the coloring of $a$ to the first $m$ letters, which we call the \emph{depth $m$ coloring of $a$}. We say that $a <_m b$ if the depth $m$ color of $a$ is less than the depth $m$ color of $m$ using the lexicographical ordering described above. The partial order $<_m$ does not distinguish between strands for which the rays $\gamma_a$ and $\gamma_b$ are locally parallel and do not diverge before passing through the corner box at least $m$ times. It is easy to see that \(a<_mb\) implies \(a<_{m+1}b\) and that \(<_m\) agrees with \(<\) for sufficiently large \(m\).  By a slight abuse of notation, we write  \(a \leq _m b\) to mean that $a < b$ or the depth $m$ colors of $a$ and $b$ agree, noting that \(a \leq _m b\) does not imply \(a<_m b\) or \(a=b\).

\subsubsection*{Ordering crossover arrows in an arrowbox}

In an arrowbox, where all strands are parallel, crossover arrows have a well-defined length, 
 measured by how many strands they cross. A collection of arrows  that point in the same direction can be sorted by length, as follows:

\begin{lemma}\label{Lemma0}
A configuration of crossover arrows that all point up is equivalent to a new configuration in which all the arrows point up and are sorted by length, in the sense that shorter arrows lie to the left of longer arrows. 
\end{lemma}
\begin{proof}
Let \(n+1\) be the total number of strands in the arrowbox.
We say a configuration is {\em \(k\)-sorted} if all arrows of length \(n\) appear on the right, followed by all arrows of length \(n-1\), then all arrows of length \(n-2\), {\it etc.} up to arrows of length \(n-k\). The length \(< n-k\) part of a configuration is the configuration obtained by deleting all arrows of length greater or equal to \(n-k\). 

\piccaption[]{Ordering arrows from left to right, following the steps used in the proof of Lemma \ref{Lemma0}. \label{fig:order-slide}}
\parpic[r]{
 \begin{minipage}{45mm}
 \centering
\includegraphics[scale=0.5]{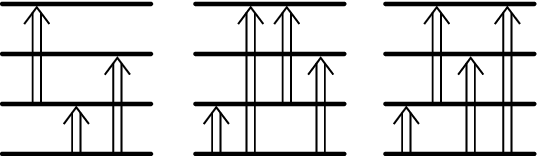}
  \end{minipage}%
}
We prove by induction on \(k\) that any configuration is equivalent to a \(k\)-sorted configuration whose length 
\(<n-k\) part is the same as the length \(< n-k\) part of the original configuration. 
 The base case \(k=0\) is clear, since an
arrow of length \(n\) slides freely past any other arrow. For the general case, given a configuration \(C_0\), we first apply the induction hypothesis to find a \(k\)-sorted configuration \(C_1\) whose length \(< n-k\) part agrees with \(C_0\). If there is no arrow of length $n-k-1$, then the configuration is $(k+1)$-sorted, and we are done. Otherwise choose the length \(n-k-1\) arrow which is farthest to the right among those arrows of  \(C_1 \) which are out of order, and slide it to the right, past the first arrow it encounters, to obtain a new configuration \(C_2\).

 If this slide did not create a new arrow (as in Figure~\ref{fig:crossover_moves}(a)) \(C_2\) is still \(k\) sorted and has the same length \(<n-k-1\) part as \(C_1\). If the slide created a new arrow (as in Figure~\ref{fig:crossover_moves}(d)) 
\(C_2\) has the same length \(<n-k-1\) part as \(C_1\) but is not $k$-sorted. In this case, we apply the induction hypothesis to find a configuration \(C_3\) which is \(k\)-sorted and has the same length \(<n-k\) part as \(C_2\).  (In particular, the position of the length \(n-k-1\) arrow has been switched, and we have not added any arrows of length \(n-k-1\)). Repeating this process, we eventually arrive at a \((k+1)\)-sorted configuration whose length \(<n-k-1\) part agrees with that of \(C_0\). 
\end{proof}

On a one-handle, the curve configuration consists of two arrowboxes and one permutation box, as shown in Figure~\ref{fig:anatomy}. We introduce the notational shorthand $[A_1,\sigma,A_2]$ to specify this data. We say that $[A_1, \sigma, A_2]$ and $[A_1', \sigma', A_2']$ are equivalent if they represent equivalent train tracks in the 1-handle, in the sense that the (mod 2) counts of paths from any endpoint at one end of the 1-handle to any endpoint at the other end agree. Note that when the local moves in Figure \ref{fig:crossover_moves} are applied within a 1-handle the resulting configurations are equivalent. In addition to sliding arrows past each other within each arrow box, as in Lemma \ref{Lemma0}, we can also slide arrows through the permutation box from one arrow box to the other, and we can resolve a crossing in the permutation box using the move in Figure \ref{fig:crossover_moves}(e), which changes the permutation $\sigma$ as well as the arrow boxes.

Our next goal is to show that we can replace any configuration $[A_1,\sigma,A_2]$ with an equivalent configuration $[A_1', \sigma', A_2']$ in which the crossover arrows are nicely arranged with respect to a chosen ordering of the strands on each end.
Fix an ordering $<$ on the strands in each arrowbox. We say that the arrow boxes are sorted with respect to $<$ if $b < a$ whenever there is a crossover arrow from \(a\) to \(b\).

\begin{lemma}\label{Lemma1} Any configuration $[A_1,\sigma,A_2]$  in a 1-handle is equivalent to a configuration $[A_1',\sigma',A_2']$, where the arrow boxes \(A_1'\) and \(A_2'\) are sorted with respect to the fixed ordering $<$.  Moreover, if the original configuration has depth \(m\), then the depth \(m\) colorings on the old and  new configurations are the same, and the new configuration has depth \(\geq m\).
\end{lemma}

\piccaption[]{Three equivalent configurations; the pair on the right both have the form guaranteed by Lemma \ref{Lemma1}. \label{fig:bundle-non-unique}}
\parpic[r]{
 \begin{minipage}{45mm}
 \centering
 \includegraphics[scale=0.5]{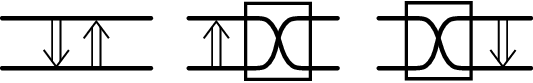}
  \end{minipage}%
} We remark that the configuration predicted by Lemma \ref{Lemma1} is not unique, as can be seen for the two-strand bundle with a pair of opposite arrows pictured in Figure \ref{fig:bundle-non-unique}. In general, when applying Lemma \ref{Lemma1} we choose the ordering $<$ to be the color ordering (or, more precisely, any total ordering on strands which is consistent with the partial order coming from colors). Note that in the resulting configuration, the order may no longer coincide with the color ordering, since resolving crossings in the permutation box changes the colors of the strands. However, the last part of the lemma implies that the depth $m$ partial ordering $<_m$ is the same for both configurations.

\begin{proof}[Proof of Lemma \ref{Lemma1}]

We begin by making some simplifications. First, 
if the original configuration has depth \(m\), we divide the strands into equivalence classes (bundles). Two strands belong to the same equivalence class if they run parallel for \(m-1\) steps in both directions. The hypothesis that the configuration has depth \(m\) means that any two strands joined by a crossover arrow are in the same equivalence class. Crossover arrows from one bundle do not interact with stands from a different bundle, so we can operate on each bundle separately; thus we assume without loss of generality that we are operating on a single bundle of strands. 

Next, we claim that if the statement holds for one  ordering $<$ on the strands in each arrow box, it holds for any other ordering $<'$.  To see this, note that $<'$  is related to $<$ by some permutation $\tau$ of the strands. We define $A^\tau$ to be the result of permuting the strands according to $\tau$ and carrying the crossover arrows along, so that an arrow from strand $i$ to strand $j$ in $A$ becomes an arrow from strand $\tau(i)$ to strand $\tau(j)$ in $A^\tau$; we refer to $A^\tau$ as the conjugation of the arrow box by $\tau$. By applying a homotopy to the strands in $A$ while fixing their endpoints, we can see that the arrow box $A$ is equivalent to the arrow box $A^\tau$ with a collection of crossings added on either end encoding the permutation $\tau$ on the left and $\tau^{-1}$ on the right. We can absorb the crossings on the inside of $A^\tau$ (that is, toward the permutation box) into the permutation box, ignore the crossings on the outside of $A^\tau$ (toward the end of the 1-handle), and apply the Lemma with $<$ to the remaining configuration. Undoing the homotopy, we see that this configuration being ordered with respect to $<$ is equivalent to the original configuration being ordered with respect to $<'$.

From now on, we will assume that the strands are ordered from top to bottom in \(A_1\), and bottom to top in \(A_2\), so that we are aiming to have all arrows in \(A_1'\) pointing up, and all arrows in \(A_2'\) pointing down.
We begin by arranging the bundle so that all crossings between strands are on the right and all crossover arrows are on the left. It is enough to show that the right-most down arrow can be moved past every up arrow, and to the other side of the crossings (at the expense of possibly altering the permutation between strands). An example is shown in Figure \ref{fig:Lemma-1-A}. Using Lemma \ref{Lemma0}, arrange the collection of up arrows so that the shortest arrows appear first. We will induct on  the length of the down arrow immediately to the left of the collection of up arrows. 

Suppose  that the right-most down arrow has length 1. This arrow slides past all length 1 up arrows, with the exception of the possibility illustrated in Figure \ref{fig:bundle-non-unique}, namely, the length 1 down arrow meets a length 1 up arrow between the same two strands. If this occurs, we replace the pair with a single up arrow and a new crossing between the strands. Effectively, the down arrow is replaced by a crossing and this new crossing slides to the right to compose with (and alter) the permutation. Notice that this will, in general, have a non-trivial effect on the remaining up arrows encountered, however this can only increase their length (if a starting or ending point is on one of the strands in question) and not, in particular, switch their direction. 

\begin{figure}[ht]
\includegraphics[scale=0.5]{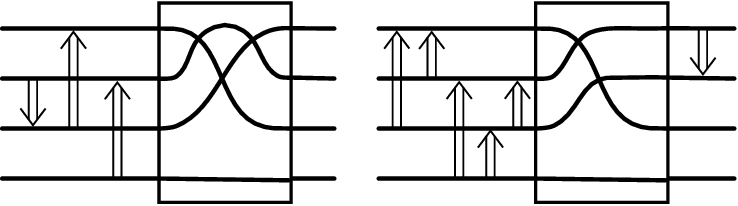}
\caption{Two equivalent views of the same bundle. The configuration on the left is in the form desired at the beginning of the algorithm; the configuration on the right shows the result of an application of the base case in the induction.} \label{fig:Lemma-1-A}
\end{figure}

If the length 1 down arrow does not meet a length 1 up arrow between the same strands, it slides freely past all remaining length $k>1$ arrows. New arrows may be produced in the process by composition, but these will have length $k-1$ and be (additional) up arrows; compare Figure \ref{fig:Lemma-1-A}. Finally, we slide the arrow through the crossings specified by the permutation $\sigma$. There are two cases: Either the arrow remains a down arrow on the other side of these crossings and we are done, or the arrow switches to an up arrow. In the latter case there must be a crossing between the strands that caused the swicth, which we resolve (see, for example, Figure \ref{fig:Lemma-1-A}). The result alters the permutation but produces an up arrow on the left and a down arrow on the right. This completes the base case for induction. 

\piccaption[]{Compositions with short (above) and long (below) arrows. \label{fig:Lemma-1-B}}
\parpic[r]{
 \begin{minipage}{35mm}
 \centering
 \includegraphics[scale=0.5]{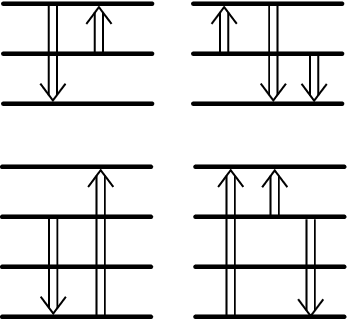}
  \end{minipage}%
}
Now suppose that the right-most down arrow has length $n$; our induction hypothesis is that arrows of length less than $n$ can be moved to the right of the up arrows. As before, the up arrows collected to the right are ordered by length. There are three groups in this collection: the short arrows (those of length less than $n$), the length $n$ arrows, and the long arrows (those of length greater than $n$). Sliding the down arrow past the short arrows may produce compositions (see Figure \ref{fig:Lemma-1-B}), but these new down arrows have length less than $n$ and are dealt with by the induction hypothesis. To pass the length $n$ arrows, we proceed as in the base case. That is, the arrow passes freely unless it meets an up arrow joining the same strands. In this case the arrow is replaced by a crossing, which slides to the right; no new down arrows are created as the remaining arrows are all long. Finally, if the arrow passes to the long arrows the arrow proceeds to the permutation (again without producing new down arrows, compare Figure \ref{fig:Lemma-1-B}) and we resolve a crossing in the permutation if needed.

Finally, we consider the effect of the operations we have performed on the weights of crossover arrows and orderings of strands. In the course of the proof, we have added transpositions between strands belonging to the same bundle. Strands belonging to the same bundle have the same depth \(m-1\) coloring, and any strand must pass at least once through the corner box before it goes over one of the new transpositions. Thus the depth \(m\) coloring of all strands is unchanged. From this, it follows that any arrow with depth \(\geq m\) in the old configuration has  depth \(\geq m\) in the new configuration as well. In the course of the proof, we may have created many new crossover arrows, but these all run between strands in the same bundle, so they have depth \(\geq m\) as well. 
\end{proof}

\begin{proof}[Proof of Proposition~\ref{prop:IncreaseDepth}] 
The argument is divided into four steps.

\subsubsection*{Step 1: Remove arrows with \(\wtp = -m\)} Order to depth $m$ by applying Lemma~\ref{Lemma1} to each one-handle, with respect to any extension of the partial ordering  \(\leq_{m}\) to a total ordering. Since the original configuration has depth \(m\), the ordering \(\leq_m\) is unaffected by this operation,  and it does not matter which handle we order first. The new configuration has depth \(m\) and is sorted with respect to \(\leq_m\).  
This means  that if we take an arrow with \(|\wtp|=m\) and push it \(m\) steps in the direction towards the corner box (to where its endpoints diverge), it will remain positively oriented, and thus will run clockwise around the corner box. Thus any such arrow actually has \( \wtp=m\), and there are no arrows with \(\wtp = -m\) in the new configuration.

\subsubsection*{Step 2: Remove arrows with \(\wtp = m\)} Fix an arrow box. If it contains any arrows with \(\wtp = m\), choose the one among them which is closest to the corner box. We can remove this arrow by pushing it in the positive direction until its ends diverge. As we do so, we may encounter other arrows with ends on one of the strands which our arrow runs between. We claim that these arrows can be pushed along in front of our chosen arrow (like a pile of snow accumulating in front of a snowplow) without  their ends diverging until the chosen arrow reaches the point where its ends diverge. To see this, note that every arrow in the arrow box which lies between the corner box and our chosen arrow has \(|\wtp|>m\), so it can be pushed along \(m\) steps without its ends diverging. Similarly, arrows which we encounter in any subsequent arrow box have depth at least \(m\), so they can be pushed \(m-1\) steps in any direction without diverging. Other arrows which these arrows would have to pass can likewise be pushed along (they need to move at most \(m-1\) steps). At the end of the process, we can remove the chosen arrow, and then just push all other arrows back to where they started. The net result is that we have removed the chosen arrow without making any other changes to the curve configuration. Repeating this process  removes all arrows with \(\wtp=m\). 

\subsubsection*{Step 3: Remove arrows with \(\wtm=-m\).} 
For a given arrow box \(A\), we will change the configuration so as to eliminate all arrows with \(\wtm=-m\) from \(A\). To do this, we first 
 apply Lemma~\ref{Lemma1} to the region consisting of the arrow box and its adjacent permutation 
 with ordering \(\leq_{m+1}\). This region is shown in  Figure~\ref{fig:Lemma-1-C}. 
\labellist
\pinlabel {$\sigma_\bu$} at 109 53
\endlabellist
\piccaption[]{Lemma \ref{Lemma1} applied to the dashed region;  arrows have $|\wtp|\ge m+1$ and $|\wtm|\ge m$.\label{fig:Lemma-1-C}}
\parpic[r]{
 \begin{minipage}{35mm}
 \centering
 \includegraphics[scale=0.5]{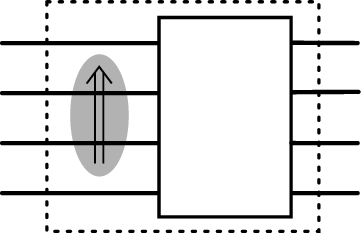}
  \end{minipage}%
}
 By Lemma~\ref{Lemma1}, the depth \(m\) colorings in the old and new configurations agree. Moreover, the depth \(m+1\) colorings on \(A\) are the same in both configurations. 
To see this, consider the number of steps a strand leaving the arrowbox towards the corner box takes to return to it. If it returns after one step through the corner box, it must re-enter through the right-hand side of the figure and exit through the left. Since all crossover arrows in the original arrowbox have \(|\wtp|\geq m+1\) (Steps 1 and 2), 
 any new transpositions we added in the course of applying the lemma will not change the next \(m\) labels in the coloring. Thus the first \(m+1\) labels in the original coloring will remain unchanged. A similar argument applies if  the strand takes two or more steps to return to the arrowbox. 

 \labellist
 \pinlabel {$\sigma'_\bu$} at 238 54
\pinlabel {$\underbrace{\phantom{aaaaaaaaaaaaaaaaa}}$} at 115 -3 
\pinlabel {$\underbrace{\phantom{aaaaaaaaaaaaaaaaa\i}}$} at 361 -3 
\pinlabel {$\underbrace{\phantom{aaaaaaaaaaaa}}$} at 527 -3 
\footnotesize
\pinlabel {$\wtp=m+1$ or $|\wtp|>m+1$} at 115 -23
\pinlabel {$|\wtm|\ge m$} at 115 -43
\pinlabel {$\wtp=m$ or $|\wtp|>m$} at 361 -23
\pinlabel {$|\wtm|>m$} at 361 -43
\pinlabel {$|\wtp|\geq m+1$} at 527 -23
\pinlabel {$|\wtm| \geq m$} at 527 -43
\endlabellist
\begin{figure}[ht]
\includegraphics[scale=0.5]{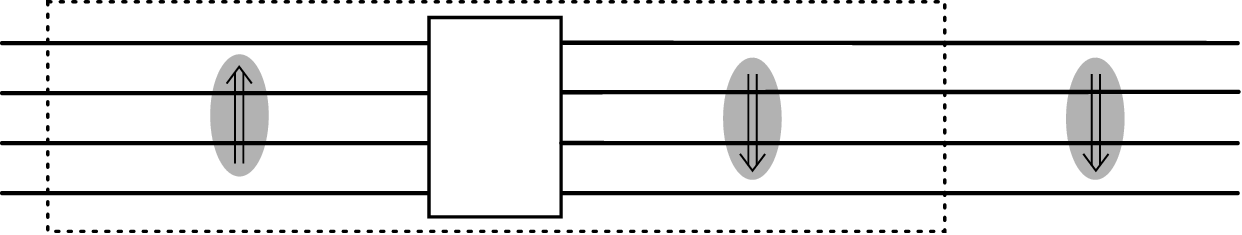}\vspace*{20pt}
\caption{A summary of weights after Lemma \ref{Lemma1} is applied to the region of the horizontal strands in the dashed box. Weights outside the dashed box are inherited from the previous step. 
}
 \label{fig:Lemma-1-E}
\end{figure}

 At this point, the configuration is as shown in Figure~\ref{fig:Lemma-1-E}. It is sorted to depth \(m+1\) on the left, and to depth \(m\) on the right. It follows from  Lemma~\ref{Lemma1} that every arrow in the dashed box has depth greater or equal to \(m\). In fact, every arrow in the dashed box has  \(\wtp \geq m+1\) on the left of the permutation, and \(\wtm\geq m+1\) on the right. To see this, note that when we apply Lemma~\ref{Lemma1}, the strands in each bundle used in the proof run parallel for \(m\) steps to the left.
 Finally, since the configuration is sorted, we see that   \(\wtp \neq -(m+1)\) for arrows on the left, and \(\wtp \neq -m\) for arrows on the right. In summary, the weights in the new configuration are as shown in the figure. 
 
In the process of applying Lemma~\ref{Lemma1}, we may have created some new arrows with \(\wtp = m\) in the right-hand side of the dashed region. Since the right-hand arrowbox in the 1-handle is sorted to depth \(m\), these can now be removed just as in Step 2. 
 
 Consider the arrow in \(A\) which is closest to the corner box. This arrow has \(|\wtp| \geq m+1 \geq 2\), so we can slide it through the corner box to obtain a new configuration. In its old position, the arrow had \(|\wtm| \geq m\), so the new arrow will have \(|\wtp| \geq m+1\). Similarly, the old arrow had \(\wtp = m+1\) or \(|\wtp|>m+1\), so the new arrow will have \(\wtm = m\) or \(|\wtm|>m\). Repeating this operation for each arrow in \(A\), we eventually arrive at a depth \(m\) configuration in which we have slid every arrow out of \(A\). 
 
Summarising our progress to this point: we have removed all arrows with \(\wtp=-m\) or \(\wtm=-m\) from our chosen arrowbox \(A\) without adding any new arrows with \(\wtp = -m\) or $\wtm=-m$ in any of the other arrowboxes. Therefore, by repeating this sequence of steps on each arrowbox, we arrive at a depth \(m\) configuration of curves in which there are no crossover arrows with $\wtp=-m$ or $\wtm=-m$. 

\subsubsection*{Step 4: Remove arrows with \(\wtm=m\).} There may remain arrows for which $\wtm=m$; the last step in the proof of Proposition~\ref{prop:IncreaseDepth} is to show that we can eliminate all arrows of depth $m$ without changing the rest of the configuration. This proceeds in exactly the same way as in Step 2. Fix a 1-handle, choose a direction \(\to\) on it, and consider all crossover arrows which have weight \(w_\to = m\) when sliding in  this direction. All the other arrows in the 1-handle have \(|w_\to| \geq m+1\), so we can slide the arrows with \(w_\to = m\) off one at a time without changing the weights of any other arrow, just as we did in Step 2.  Repeating this operation for each 1-handle and each direction, we arrive at a configuration in which every arrow has depth greater or equal to \(m+1\). 
\end{proof}

\subsection{The proof of Theorem \ref{thm:structure}}We are now positioned to prove the theorem that is the focus of the section. 

Fix a reduced, extendable type D structure and express it as an immersed train track. By Proposition \ref{prop:curves-plus-crossovers} this may be expressed as a curve configuration, a schematic for which is given in Figure \ref{fig:anatomy}. Then the algorithm detailed in Section \ref{sub:removing} ensures that this curve configuration can be simplified to one in which the only crossover arrows connect parallel curves in the marked torus. Indeed, the original type D structure is isomorphic as an $\Alg$-module to the type D structure determined by this simplified curve configuration.  It is convenient to introduce terminology for a train track of this form.

\begin{definition}
\label{def:curve-like}
A \emph{curve-like train track} is an $\Alg$-train track that has the form of immersed curves together with crossover arrows connecting parallel curves. 
\end{definition}

\piccaption[]{Equivalent configurations.\label{fig:make-a-local-system}}
\parpic[r]{
 \begin{minipage}{45mm}
 \centering
 \includegraphics[scale=0.5]{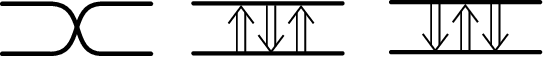}
  \end{minipage}%
}It remains to associate a local system with any collection of parallel curves that arise. Consider such a collection of $n$ parallel curves, and notice that  we can replace, without loss of generality, any crossings between two curves with a triple of alternating crossover arrows as in Figure \ref{fig:make-a-local-system}; compare Figure \ref{fig:crossover_moves}(e). Moreover, by applying the move (M2) we may assume that all of the arrows are contained in the 0-handle and run between curves whose endpoints are on two different sides of the corner box. 

\begin{figure}[ht]
\hspace*{0cm}\includegraphics[scale=0.5]{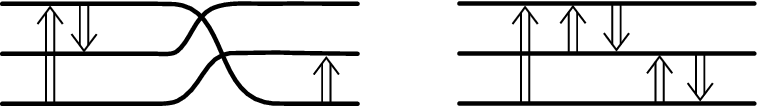} 
\qquad\raisebox{10pt}{$\left(\begin{matrix}0 &1 &0 \\ 1&0 &1 \\ 0&1& 1 \end{matrix}\right)$}
\caption{Extracting a local system from a bundle of curves, where the generators run from top to bottom.} \label{fig:simple-system}
\end{figure}

Interpreting this configuration as (part of) a train track, we take the segments on each side of the collection of arrows as generators and form a $n\times n$ matrix $(x_{ij})$ where the entry $x_{ij}$ is the mod 2 count of oriented paths from $x_i$ to $x_j$. The result is a local system of dimension $n$ over the immersed curve $\gamma$ carrying the bundle of curves. Note that to define this local system we must choose an orientation on $\gamma$ (for the example in Figure \ref{fig:simple-system} we count paths oriented left to right); choosing the opposite orientation corresponds to inverting the matrix.  \qed


\subsection{An aside on type D structures associated with general surfaces}\label{sub:aside}

Thus far we have focused on type D structures associated to the torus algebra $\Alg$, since these are the objects relevant for bordered Heegaard Floer homology of manifolds with torus boundary. However, the proof above applies more broadly to classify extendable type D structures associated with more general surfaces in terms of immersed curves in those surfaces.

We begin by defining algebras associated to a disk with a single marked point. For a given positive integer $m$ we can associate a quiver with the marked disk: there are $m$ vertices $\{0,\ldots,m-1\}$ placed on the boundary of the disk and $m$ edges $\{\rho_0,\ldots,\rho_{m-1}\}$ oriented clockwise along the boundary. We assume that the marked point is arbitrarily close to the boundary of the disk and is next to the edge $\rho_0$.  Fix the target of $\rho_i$ to be the point $i$; below it will be helpful to distinguish between the vertex $i$ in the quiver and the associated idempotent $\iota_i$ in the path algebra of the quiver. We follow the usual convention of writing $\rho_I$ for the product associated with a cyclically-ordered string $I$ from $\{1,2\ldots,m-1,0\}$. The path algebra of this quiver determines two algebras that are of interest to us: $\widetilde{\sB}$ is the result of quotienting by $\rho_I=0$ for any string $I$ containing more than one $0$, while $\sB$ is the result of quotienting  by $\rho_I=0$ for any string $I$ containing at least one $0$. The algebra $\widetilde{\sB}$ contains a distinguished central element $U=\sum_{|I|=m}\rho_I$.

\piccaption[]{Equivalent train tracks in the marked disk, according to Lemma \ref{lem:linear_algebra}. Here $m = 12$ and there is one primary vertex for each idempotent. \label{fig:tennis-ball}}
\parpic[r]{
 \begin{minipage}{60mm}
 \labellist
 \tiny
 \pinlabel {$z$} at 71 120.5 \pinlabel {$z$} at 228.5 121
 \endlabellist
  \vspace*{0mm}
\includegraphics[scale=0.6]{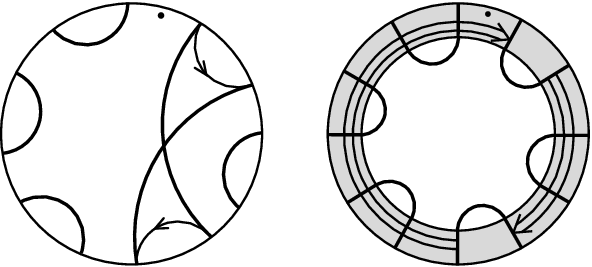}
  \vspace*{0mm}
  \end{minipage}%
}
Now a type D structure over $\sB$ may be defined as in Section \ref{sub:modules}. In the same way, we can  define the notion of an extension, which will be a module over $\widetilde{\sB}$. These (extended) type D structures can be encoded by train tracks in the marked disk, where generators with idempotent $\iota_i$ give rise to primary switches on the boundary of the disk in a neighborhood of the vertex $i$ and terms in the differential are encoded by paths through the disk. An example is shown in Figure \ref{fig:tennis-ball} (where, as usual, two-way track segments are recorded as unoriented edges). As such, in this setting, the type D structure and a choice of extension for it can be represented by a matrix $M$ over $\F[U]/U^2$ satisfying the hypotheses of Lemma \ref{lem:linear_algebra}, and hence may be expressed as $M=P\bar MP^{-1}$ where $P$ is a composition of elementary matrices of the form $A^U_{i,j}$ or of the form $A_{i,j}$ with $i<j$, and \(M\) is as in Proposition \ref{prop:curves-plus-crossovers}. Following the proof of Proposition \ref{prop:curves-plus-crossovers}, this new form may be interpreted in terms of the associated train tracks in the disk as follows: up to equivalence, all train tracks for which $M^2=UI_{2n}$ can be represented by 
(1) 
a collection of unoriented track segments connecting primary switches in the boundary of the disk; and 
(2) 
a collection of clockwise-moving crossover arrows confined to a small collar neighborhood of the boundary of the disk. 
For the example in Figure \ref{fig:tennis-ball}, the reader can check that $P = A^U_{6,1}A_{2,5}$ (where determining $M$ and $\bar M$ is left as an exercise), leading to the train track on the right of the figure.



For $m=4g$ and a specified handle attachment of $2g$ 1-handles, we can define an algebra associated with a punctured genus $g$ surface. We denote this algebra by $\Alg_F$, where $F$ is the surface along with the specified handle decomposition. Let $\iota_i$ and $\iota_j$ be idempotents of $\sB$ corresponding to points $i<j$ in the boundary of the disk. Adding a 1-handle $[-\epsilon,\epsilon]\times I$ to the disk so that $(0,0)$ is identified with $i$ and $(0,1)$ is identified with $j$ gives rise to a modified quiver: the points $i$ and $j$ are identified to form a quotient graph, but new paths are disallowed by setting $\rho_{i}\rho_{j+1}=0$, and $\rho_{j}\rho_{i+1}=0$. The new algebra $\Alg_F$ is a subalgebra of $\sB$ obtained by replacing the idempotents $\iota_i$ and $\iota_j$ with the single idempotent $\iota_i+\iota_j$. In other words: the points of the relevant idempotents are identified, while any and all possible {\em new} paths created are set to zero. The extended algebra $\widetilde\Alg_F$ is defined analogously as a quotient of $\widetilde\sB$. (The reader may recognize the geometric significance of the idempotent subring of $\Alg_F$ in general: the Grothendieck group of the appropriate associated category of modules is completely determined by this ring; see \cite{HLW}.)

Note that when $m = 4$ and $F$ is the punctured torus, with 1-handles identifying idempotents $\iota_0$ with $\iota_2$ and $\iota_1$ with $\iota_3$, the algebra $\Alg_F$ defined above coincides with the torus algebra $\Alg$ as defined in Section \ref{sub:modules}, and $\widetilde\Alg_F$ agrees with $\widetilde\Alg$. In this case we identify $\iota_\bu=\iota_0+\iota_2$ and $\iota_\circ=\iota_1+\iota_3$, and the relations from the 1-handle attachments account for the familiar $\rho_3\rho_2=\rho_2\rho_1=0$ in the definition of $\Alg$. When $m = 4g$ and $F$ is a punctured genus $g$ surface, the algebra $\Alg_F$ is the one-moving-strand part of the strand algebra associated to $F$ in \cite{LOT}. More specifically, the algebra $\sB$ is the algebra $\Alg(m, 1)$ defined in \cite[Section 3.1]{LOT} and $\Alg_F$ coincides with $\Alg(\mathcal{Z}, -g+1)$ defined in \cite[Section 3.2]{LOT}, where $\mathcal{Z}$ is a pointed matched circle representing the given handle decomposition of $F$.

In fact, we need not restrict to surfaces with one zero handle and one boundary component: we can construct algebras $\Alg_F$ and $\widetilde\Alg_F$ for other surfaces $F$ admitting a prescribed (finite) 0- and 1- handle decomposition, whenever each 0-handle carries a marked point (i.e. there is a distinguished $\rho_0$ for each 0-handle). We begin with the algebra $\oplus_{i=1}^n \sB^i$ associated with the collection of 0-handles, that is, the path algebra associated with a disjoint union of $n$ cycles. Each 1-handle attachment identifies some point $i_a$ (corresponding to $\iota^a_i\in\sB^a$) with $i_b$ (corresponding to $\iota^b_i\in\sB^b$) and gives rise to a quotient graph; this quiver has new relations imposed by $\rho_{i}^a\rho_{j+1}^b=0$ and $\rho_{j}^b\rho_{i}^a=0$ (just as in the case of a single 0-handle). Having attached all of the handles in the description of $F$, and appropriately quotienting the associated quiver at each stage, we arrive at an algebra $\sA_F$ (see Figure \ref{fig:handles-example} for an explicit example). A similar construction gives  $\widetilde{\sA}_F$. There is are close ties here with work of Zarev; see in particular the examples considered in \cite[Section 9]{Zarev}.

To describe an extendable type D structure over $\Alg_F$ for general surface $F$ with $n$ 0-handles, we can first describe a type D structure over $\oplus_{i=1}^n \sB^i$: the underlying vector space $\oplus_{i=1}^nV^i$ has left-action by the idempotents $\sI(\oplus_{i=1}^n \sB^i) = \oplus_{i=1}^n\sI( \sB^i)$ so that $V^i$ is a type D structure associated with the $i$th 0-handle. The essential observation is that any quotient identifying $i_a$ with $i_b$ in the quiver is compatible with  a left-action of the new idempotent subring where $\iota^a_i$ and $\iota^b_i$ are replaced by $\iota^a_i+\iota^b_i$, and hence gives a well-defined type D structure over the (subalgebra) associated with the quotient quiver. Notice that the matrices $M^i$ describing this type D structure are data that remains associated with the 0-handles in the decomposition of $F$. In particular, while the type D structures change (the underlying vector space changes, as well as the algebra), the $M_i$ are unchanged at every stage. We emphasise that, in the course of this construction, the relevant type D structures are {\em not} obtained by a sequence of quotients as differential modules, rather, the vector spaces are quotiented in a way compatible with a change to the idempotent subring. Moreover, such quotients are only possible when $\dim(\iota^a_iV^a)$ and $\dim(\iota^b_iV^b)$ agree. Geometrically, the train tracks representing the type D-structures in each 0-handle extend over the 1-handles as parallel lines connecting the relevant generators in the quotient; see Figure \ref{fig:handles-example}. This is the key observation required in order to establish Theorem \ref{thm:higher-genus}, below, treating type D structures associated with general surfaces. 

\begin{figure}[t]
\labellist \tiny
\pinlabel {$z$} at 105 298 
\pinlabel {$z$} at 518 347
\pinlabel {$z$} at 509 832 \pinlabel {$z$} at 742 832
\pinlabel {$z$} at 509 652 \pinlabel {$z$} at 742 652
\small
\pinlabel {$e_0^{(0,1)}$} at 481 800 \pinlabel {$e_0^{(1,1)}$} at 718 800
\pinlabel {$e_0^{(0,0)}$} at 481 579 \pinlabel {$e_0^{(1,0)}$} at 724 579
\pinlabel {$\widetilde{\sB}^{(0,1)}$} at -5 750 \pinlabel {$\widetilde{\sB}^{(1,1)}$} at 233 753
\pinlabel {$\widetilde{\sB}^{(0,0)}$} at -5 645 \pinlabel {$\widetilde{\sB}^{(1,0)}$} at 233 647
\pinlabel {$\rho_i^a$} at 50 472
\pinlabel {$\rho_j^b$} at 121 472
\pinlabel {$\Longrightarrow \ \rho_i^a\rho_j^b=0$} at 223 480
\pinlabel {$= \quad \rho_0^a$} at 195 435
\endlabellist
\includegraphics[scale=0.42]{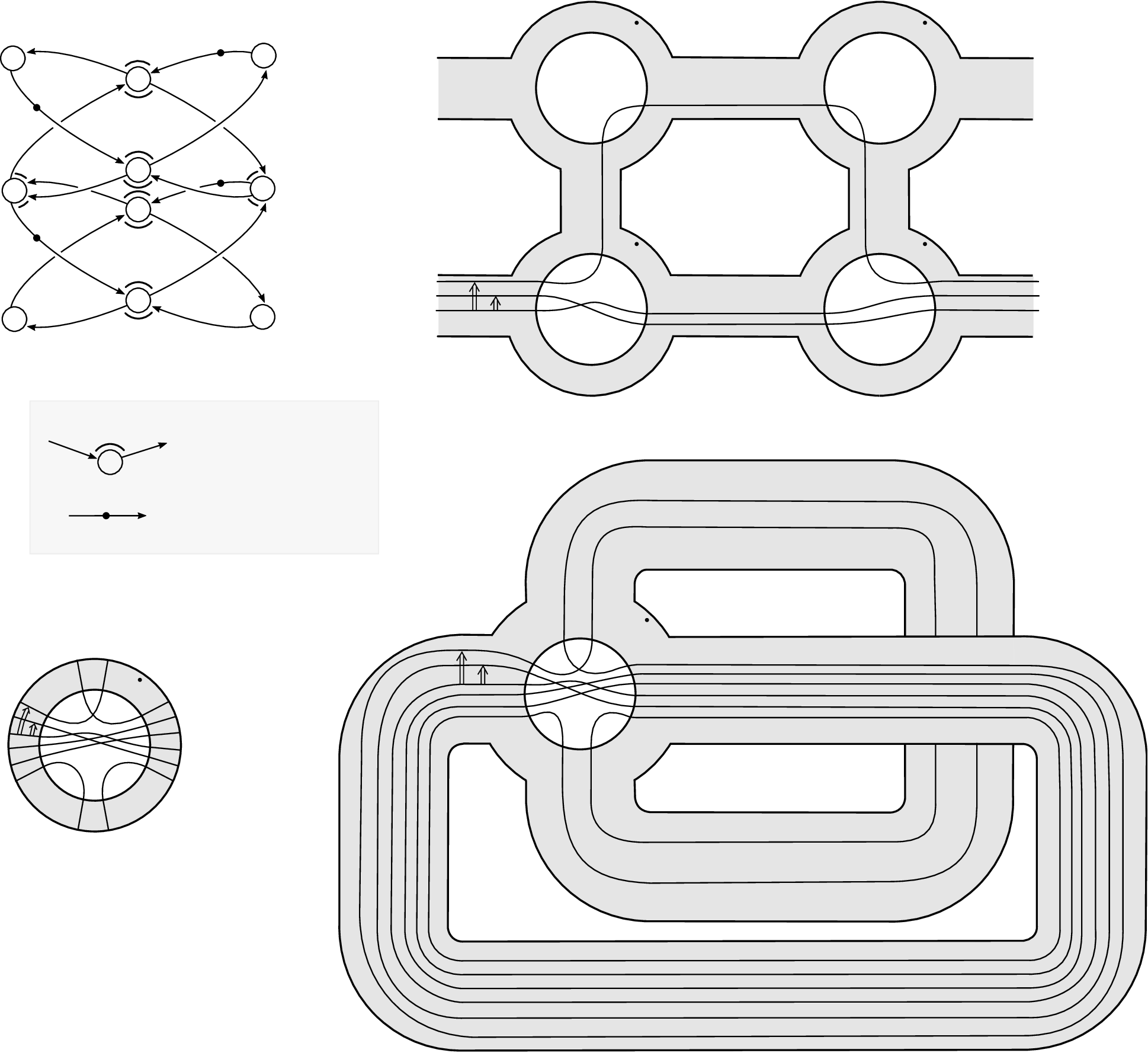}
\caption{Three (related) extended type D structures: over $\widetilde{\sB}$ (lower left) where $m=4$ has been omitted from the notation; over $\widetilde{\sA}$ (lower right); and over the algebra described by the quiver in the top left (upper right). The associated surface in this last example is made up of four 0-handles and six 1-handles (it should be viewed in the cylinder $S^1\times \R$), and each 1-handle attachment gives rise to a quotient of the quiver, with relations as described by the key in the shaded box. The end result is a subalgebra of $\widetilde{\sB}^{(0,0)}\oplus\widetilde{\sB}^{(1,0)}\oplus\widetilde{\sB}^{(0,1)}\oplus\widetilde{\sB}^{(1,1)}$. This last example covers the example associated with the torus. The reader can check that the matrices underlying the first two examples  are identical, while the last example requires 4 matrices (one for each 0-handle).}
\label{fig:handles-example}
\end{figure}

We note that the move (M1), along with the local moves for train tracks, apply in precisely the same way to train tracks in a marked disk, where the role of the corners of the square are played by points on the boundary of the disk between two adjacent quiver vertices (and outside of the neighborhoods of these vertices in which all primary switches lie). Thus by Proposition \ref{prop:curves-plus-crossovers} the train tracks within each 0-handle can be represented by arcs with clockwise moving crossover arrows, and any crossover arrows passing a corner can be immediately removed up to a change of basis in the corresponding module. These steps make sense for modules over $\sB$ before passing to the quotient $\Alg_F$ (i.e. before attaching 1-handles). In contrast, the move (M2) does not make sense without 1-handles, but once we pass to train tracks representing type D structures over $\Alg_F$ the obvious analog of (M2) is allowed: up to an appropriate change of basis, we may slide a crossover arrow near the boundary of a 0-handle across a 1-handle into the 0-handle on the other end.

We conclude by observing that the arrow sliding algorithm described in Section \ref{sub:removing} applies to train tracks in arbitrary surfaces. After removing crossover arrows which pass a corner of a 0-handle, all crossover arrows can be pushed into a 1-handle. The arrow sliding algorithm takes place primarily in one 1-handle at a time, where we rearrange a depth $m$ configuration of arrows in that 1-handle to remove all depth $m$ arrows. The inductive step follows from repeating this procedure in each 1-handle. The only difference for general surfaces is that the colors of strands can not be written as words in the letters $\{n, e, s, w\}$, but rather colors encode a path through the Cayley graph for $\pi_1(F)$. An ordering on colors, and thus a partial ordering on strands, can still be defined as before. So, while we are primarily interested in the torus algebra for this work, we have in fact shown:

\begin{theorem}\label{thm:higher-genus}Let $\Alg_F$ be the algebra associated with a marked surface $F$ equipped with a decomposition into finitely many 0- and 1-handles. Then any extendable type D structure over $\Alg_F$ may be realised as a collection of immersed curves in $F$ decorated with local systems. \end{theorem}

This recovers a result due to Haiden, Katzarkov, and Kontsevitch \cite{HKK} via rather different techniques.

\section{Decorated immersed curves as type D structures}\label{sec:cat}
\label{sec:categories}

The previous section shows how an extended type D structure over $\AlgExt$ can be interpreted as a collection of immersed curves decorated by local systems in the marked torus $T$. The goal of this section is to show that this in fact gives a bijection between the set of homotopy equivalence classes of extendable type D structures and the set of decorated immersed multicurves in the marked torus. Moreover, we show that the pairing on extendable type D structures coming from the box tensor product is given by the intersection Floer homology of the corresponding curves. 

\subsection{Intersection Floer homology}

In \cite{Abouzaid2008}, Abouzaid gave a combinatorial definition of the  intersection Floer homology of immersed curves in a symplectic surface  ({\it cf.} \cite{deSilvaRobbinSalamon} for a detailed discussion in the embedded case, and section 2 of \cite{AurouxSmith} for a recent summary). We will use a slight variation on this construction: rather than working with coefficients in a Novikov ring, as is standard in symplectic geometry, we use the notion of an admissible diagram, which was introduced by Ozsv{\'a}th and Szab{\'o} \cite{OSz2004} in the context of Heegaard Floer homology (see also Kotelskiy \cite{Kotelskiy2019}). 

Let \(S\) be a (possibly noncompact) orientable surface with \(\chi(S)<0\), so the universal cover \(\widetilde{S} \) is homeomorphic to \( \R^2\). A \emph{curve} \(\gamma\) in \(S\) is either an immersed closed curve \(\gamma\co S^1 \to  S\) that represents a nontrivial element of \(\pi_1(S)\) (a \emph{loop}), or a  properly immersed arc \(\gamma\co \R \to S\) (an \emph{arc}). If \(\gamma\) is an arc\, it lifts to an arc \(\widetilde{\gamma}\) in \(\widetilde{S}\). If \(\gamma\) is a loop, we consider the composition \( \gamma \circ p\co \R \to S\), where \(p\co\R \to S^1 \) is the covering map. This map lifts to an arc in \(\widetilde{S}\), which we again denote by \(\widetilde{\gamma}\). 
Of course, \(\widetilde{\gamma}\) is not unique in either case, but any two such lifts are related by the action of the group of deck transformations on \(\widetilde{S}\). 

\begin{definition}
If \(\widetilde{\gamma}\) is embedded, we say that \(\gamma\) is unobstructed.
\end{definition}

By \cite[Lemma 2.2]{AurouxSmith} this definition is equivalent to \(\gamma\) having no immersed fishtail. 
More generally, we define a multicurve \(\bgamma\), to be a finite union of curves \(\gamma_i\), and say that \(\bgamma\) is unobstructed if each \(\gamma_i\) is. 

 If \(\gamma\) is a loop, we fix a basepoint \(* \in S\) which in the image of \(\gamma\) but not a double point, and consider the covering space \(S_\gamma \to S\) defined by  
\(\pi_1(S_\gamma,*') = \langle \gamma \rangle \subset \pi_1(S,*)\). Let \(\overline{\gamma}\) be the lift of \(\gamma\) to \(S_\gamma\) that passes though the basepoint \(*'\). It is easy to see  that  \(S_\gamma \cong S^1 \times \R\), and that \(\gamma\) is unobstructed if and only if \(\bargamma\) is embedded.

\begin{definition}
Loops \(\gamma, \gamma' \subset S\) are \emph{commensurable} if there is an immersed loop \(\delta \subset S\) and integers \(n, n'\) such that \(\gamma\) is freely homotopic to \(\delta^{n}\) and \(\gamma' \) is freely homotopic to \(\delta^{n'}\).
\end{definition}

For curves in the marked torus \(T\), the condition that two curves are commensurable is equivalent to them being globally parallel, in the sense of Definition~\ref{def:parallel}. 
Group theoretically, this condition can be described as follows. As above, we assume our basepoint  \(*\) lies on \(\gamma\), so \(\gamma\) determines an element \(\gamma \in \pi_1(S)\). The free homotopy class of \(\gamma'\) determines a conjugacy class \(\{x\gamma'x^{-1}\} \subset \pi_1(S)\). Then \(\gamma\) and \(\gamma'\) are commensurable if and only if there is a cyclic subgroup \(C \subset \pi_1(S)\) containing both \(\gamma\) and \(x\gamma'x^{-1}\) for some \(x\).

\begin{lemma} 
\label{Lem:Sgamma}
Suppose \(S\) is noncompact so that \(\pi_1(S)\) is free. If \(\gamma, \gamma'\) are not commensurable, any lift of \(\gamma'\) to \(S_\gamma\) is an arc. If they are commensurable, there is precisely one lift of \(\gamma'\) to \(S_\gamma\) that is a loop. 
\end{lemma}

\begin{proof} The different lifts of \(\gamma'\) to \(S_\gamma\) correspond to conjugates \(x\gamma'x^{-1}\) of \(\gamma'\) in the free group \(\pi_1(S)\). The element \(x\gamma'x^{-1}\) lifts to a loop if and only if some power of \(x \gamma'x^{-1}\) is contained in \(\langle \gamma \rangle\), {\it i.e.} 
\(x (\gamma')^nx^{-1} = \gamma^m\) for some \(n,m\). Consider the subgroup \(G = \langle x\gamma'x^{-1}, \gamma\rangle \subset \pi_1(S)\). By hypothesis, \(\pi_1(S)\) is free, so \(G\) is free as well; \(G\) contains a nontrivial central element (namely \(x (\gamma')^nx^{-1} = \gamma^m\)) so it must be free of rank \(1\). Hence \(\gamma \) and \(\gamma'\) are commensurable if and only if some lift of \(\gamma'\) is a loop. 

Now suppose that \(\gamma'\) and \(\gamma\) are commensurable. After replacing \(\gamma'\) with \(x\gamma' x^{-1}\), we may assume that \(\gamma'\) and \(\gamma\) belong to the same cyclic subgroup of \(\pi_1(S)\), so there is some \(\delta \in \pi_1(S)\) with \(\gamma = \delta^r\), \(\gamma' = \delta^s\). Saying that \(\gamma'\) has a different lift that is a loop means there is some \(w \in \pi_1(S)\) with \(w \gamma' w^{-1} \neq \gamma'\) and \(w(\gamma')^nw^{-1} = \gamma^m\), or equivalently, \(w\delta^{sn} w^{-1} = \delta^{rm}\). 

Suppose this is the case, and consider the groups \(H = \langle w, \delta \rangle \subset \pi_1(S)\), and \(F_2\), the free group generated by \(w\) and \(\delta\). \(H\) is free, and there is an obvious surjection \(\pi\co F_2 \to H\). By Grushko's theorem, the rank of \(H\) is \(\leq 2\), and if it is equal to \(2\), \(\pi\) is an isomorphism. Now \(w\) and \(\delta\) satisfy the nontrivial relation \(w\delta^{sn} w^{-1} = \delta^{rm}\) in \(H\), so \(\pi \) cannot be an isomorphism. Thus \(H\) is free of rank \(1\), so it is abelian. But in this case  \(w\) commutes with \(\delta\), and thus also with \(\gamma'\). This contradicts our assumption that  \(w \gamma' w^{-1} \neq \gamma'\). We conclude that no other lift of \(\gamma'\) is a loop. 
\end{proof}

 Suppose that \(\gamma, \gamma'\) are unobstructed curves and that \(x, y \in \gamma \cap \gamma'\).  Let \(\pi_2(x,y)\) be the set of homotopy classes of bigons from \(x\) to \(y\).
 
\begin{lemma}
 If \(\gamma\) and \(\gamma'\) are incommensurable, \(\pi_2(x,y)\) is either empty or a single point. If  \(\gamma\) and \(\gamma'\) are commensurable, \(\pi_2(x,y)\) is either empty or a \(\Z\)-torsor. 
\end{lemma}

\begin{proof} The set \(\pi_2(x,y)\) is either empty or a torsor over \(\pi_2(x,x)\). The latter group can be computed from the long exact sequence 
\[\pi_2(S) \to \pi_2(x,x) \to \pi_1(\gamma)\oplus\pi_1(\gamma') \xrightarrow\alpha  \pi_1(S),\]
where \(\alpha(m,n) = \gamma^m(\gamma')^n\). (Note that this is only a map of sets, not a homomorphism.) We have  \(\pi_2(S)=0\), and \(\gamma, \gamma'\) represent nontrivial elements of \(\pi_1(S)\), so \(\alpha^{-1}(1)\) is either the identity element in \(\pi_1(\gamma)\oplus\pi_1(\gamma')\)  (if \(\gamma\) and \(\gamma'\) are incommensurable), or isomorphic to \(\Z\). 
\end{proof}

 Now we are ready to discuss the Floer chain complex. 
 Suppose that \(\bgamma = \cup \gamma_i\) and \( \bgamma' = \cup \gamma_j'\) are unobstructed transversally intersecting multicurves and that every component of \(\bgamma\) is a loop.
 As above, if \(x, y \in \bgamma \cap \bgamma'\), we let \(\pi_2(x,y)\) be the set of homotopy classes of bigons from \(x\) to \(y\), where a bigon from \(x\) to \(y\) is a continuous map \(\psi\co D^2 \to S\) such that \(\psi(-i) = x, \psi(i) = y\), the restriction of \(\psi\) to the right half of \(S^1\) factors through some \(\gamma_i\), and its restriction to the left half of \(S^1\) factors through some \(\gamma_j'\). It is immediate from the definition that \(\pi_2(x,y) = \emptyset\) unless \(x\) and \(y\)  belong to the same \(\gamma_i\) and \(\gamma_j'\).

 Roughly speaking, we want the Floer chain complex \(\CF(\bgamma, \bgamma')\) to be the \(\F\) vector space generated by \(\bgamma \cap \bgamma'\), with differential given by 
\(dx = \sum_y n(x,y) y\), where \(n(x,y)\) is the count of immersed bigons from \(x\) to \(y\), just as in Section~\ref{sec:pairing-train-tracks}.  Now \(\pi_2(x,y) = \emptyset\) unless \(x\) and \(y\) both belong to the same \(\gamma_i\) and the same  \(\gamma_j'\), so   \[\CF(\bgamma, \bgamma') \cong \bigoplus_{i,j} \CF(\gamma_i, \gamma_j').\]
Hence in what follows we can restrict attention to the case where \(\bgamma= \gamma\) and \(\bgamma' = \gamma'\) are individual curves.

The fact that \(d^2=0\) when \(\gamma\) and \(\gamma'\) are unobstructed follows from a combinatorial count of disks exactly as in \cite[Lemma 2.11]{Abouzaid2008}; Figure \ref{fig:dsquared} illustrates the idea of the proof.   
\begin{figure}[h]
\labellist
\pinlabel {$x$} at 305 535
\pinlabel {$y$} at 367 535
\pinlabel {$z$} at 245 535
\pinlabel {$x$} at 230 382
\pinlabel {$y$} at 290 382
\pinlabel {$z$} at 168 382
\pinlabel {$x$} at 434 382
\pinlabel {$y$} at 494 382
\pinlabel {$z$} at 390 382
\endlabellist
\includegraphics[scale=0.5]{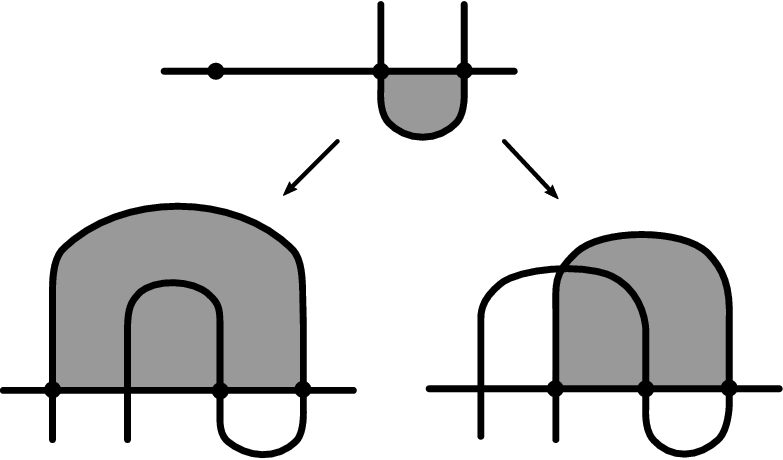}
\caption{Local pictures illustrating  why $d^2=0$: Either $d^2(x)=0$ owing to the fact that 4 bigons must be present (as illustrated on the left), or there must be a cusp in one of the $\gamma_i$ (as illustrated on the right).}\label{fig:dsquared}
\end{figure}

There are a few subtleties to consider. The first is the issue of local systems.  Abstractly, a local system on a curve \(\gamma\)  assigns a vector space \(V_x\) to each point in the domain of \(\gamma\), together with an isomorphism \(\phi_\alpha\co V_x \to V_y\) for each homotopy class of paths \(\alpha\) from \(x\) to \(y\), with the property that \(\phi_{\alpha\circ \alpha'} = \phi_\alpha \circ \phi_{\alpha'}\). If \(\gamma\) is an arc, two local systems on \(\gamma\) are isomorphic if and only if they have the same dimension. If \(\gamma\) is an oriented loop,  a local system on \(\gamma\)  is determined up to isomorphism by \(k = \dim V_x\) and a  matrix \(A \in GL_k(\F)\) (well defined up to conjugacy) that represents the monodromy \(\phi_\alpha\co V_x \to V_x\) where \(\alpha\) is a generator of \(\pi_1(S^1,x)\). We represent \(\gamma\) equipped with a local system as above by the triple \( \bgamma = (\gamma, k, A)\), where \(\gamma\) carries an orientation given by our choice of generator for \(\pi_1\). Changing the orientation on \(\gamma\)  has the effect of replacing \(A\) with \(A^{-1}\). 

If \(\bgamma, \bgamma'\) are curves as above equipped with local systems \((V,\phi)\)  and \((V',\phi')\), then we define
\[\CF(\bgamma,\bgamma')= \bigoplus_{x \in \bgamma \cap \bgamma'} V_x \otimes V_x' .\]
An immersed bigon \(\psi \in \pi_2(x,y)\)  induces a map 
\(d_\psi:  V_x \otimes V_x' \to V_y \otimes V_y'\) given by \(d_\psi = \phi_{\psi_1} \otimes \phi_{\psi_2}\) where \(\psi_1\) and \(\psi_2\) are the sides of \(\psi\) running along \(\gamma\) and \(\gamma'\) respectively.  Both \(\psi_1\) and \(\psi_2\) are oriented to point from \(x\) to \(y\). The differential on \(\CF(\bgamma,\bgamma')\) is given by 
\(d = \sum_{\psi} d_\psi\), where the sum runs over all immersed bigons in \(S\). (A more common convention is to use \(\CF(\bgamma,\bgamma')= \bigoplus_{x \in \bgamma \cap \bgamma'} \Hom(V_x, V_x')\). The two are related by replacing \(\bgamma' = (\gamma', k, A)\) with \((\gamma',k, (A^{-1})^T)\).) 

The second issue we must address is the finiteness of this sum. 
The definition we have given differs from the standard one in symplectic geometry \cite{Abouzaid2008}, which uses  coefficients in a Novikov ring and keeps track of the symplectic area of the disks. Since we are working with coefficients in \(\F\) rather than the Novikov ring, we must ensure that there are only finitely many immersed bigons   in order for the sum in the definition to make sense. If \(\gamma\) and \(\gamma'\) are incommensurable, this is automatic: there are finitely many \(x,y \in \gamma \cap \gamma'\);  \(\pi_2(x,y)\) contains either \(0\) or \(1\) representative, and each homotopy class contains at most one immersed bigon. 

If \(\gamma\) and \(\gamma'\) are commensurable, we must analyze the situation more closely. Since   the curves are commensurable, \(\alpha = \gamma^n\) is homotopic to \(\alpha'=\gamma'^m\) for some \(n,m\in \Z\). Since they  are unobstructed, \(\alpha\) and \(\alpha'\) lift to homotopic simple closed curves \(\widetilde{\alpha}\) and \(\widetilde{\alpha}'\)  in the covering space \(S_{\alpha}\).

The following lemma is well-known, but since the proof is short, we give it here. 

\begin{lemma}
\label{Lem:cover}
Suppose that \(\gamma\) is a loop, that \(p\co\overline{S} \to S\) is a covering map, and that \(\overline{\gamma}\co S^1 \to \overline{S}\) is a lift of \(\gamma\) to \(\overline{S}\). Then \(\CF(\gamma, \gamma') \cong \CF(\overline{\gamma},p^{-1}(\gamma')) \). 
\end{lemma}

\begin{proof} Since \(\overline{\gamma}\) is a lift of \(\gamma\), 
\(p\) gives a bijection between \(\bargamma \cap p^{-1}(\gamma')\) and \(\gamma \cap \gamma'\). Hence the two complexes have the same generators. 

Suppose \(\psi\co D^2 \to S\) is an immersed disk from \(x\) to \(y\), and let \(\overline{x}\) and \(\overline{y}\)  be the preimage of \(x\) and \(y\) in \(\bargamma\). Since \(D^2\) is simply connected, \(\psi\) lifts to an immersion
\(\overline{\psi}\co D^2 \to \overline{S}\) with \(\overline{\psi}(-i) = \overline{x}\). The side of \(\psi\) running along \(\gamma\) lifts to \(\bargamma\), so \(\psi(i) = \overline{y}\), and \(\overline{\psi} \) is an immersed disk from \(\overline{x}\) to \(\overline{y}\). Conversely, if  \(\overline{\psi} \) is an immersed disk from \(\overline{x}\) to \(\overline{y}\), \(p \circ \psi\) is an immersed disk from \(x\) to \(y\). It follows that the two complexes have the same differentials as well. 
\end{proof}

\begin{corollary} 
\label{cor:extra cover}
Suppose \(\gamma\) and \(\gamma'\) are commensurable, and \(\alpha, \alpha'\) are as above. Then the sum appearing in the boundary operator for \(\CF(\gamma, \gamma')\) is finite if and only if the sum for 
\(\CF(\widetilde{\alpha}, \widetilde{\alpha}')\) is. 
\end{corollary}

\begin{proof}
We apply the lemma to the cover \(p_\gamma\co S_\gamma \to S\). By Lemma~\ref{Lem:Sgamma}, \(p_\gamma^{-1}(\gamma')\) contains finitely many arcs \(\beta_1', \ldots, \beta_k'\) intersecting \(\bargamma\), along with a single loop \(\beta_0'\) that is commensurable with \(\bargamma\), so 
\[\CF(\gamma, \gamma') \cong \bigoplus_{i=0}^k \CF(\bargamma, \beta_i').\]
If \(\gamma^n = \gamma'^m\) then \(\beta_0 = \bargamma^n \) in \(\pi_1(S_\gamma) \cong \Z\). Applying the lemma again, this time to the \(n\)-fold-cyclic cover \(p\co S_\alpha \to S_\gamma\) we see that \(\CF(\bargamma, \beta_0) \cong 
\CF(p^{-1}( \bargamma), \overline{\beta_0}) = \CF(\widetilde{\alpha}, \widetilde{\alpha}')\).
For \(i>0\), \(\beta_i\) is an arc, so the differential in the complex \(\CF(\bargamma, \beta_i)\) is a finite sum. Hence the differential on \(\CF(\gamma, \gamma')\) is well-defined if and only if the differential on \(\CF(\widetilde{\alpha},\widetilde{ \alpha}')\) is. 
\end{proof}

 We define \(\sP_{\alpha, \alpha'}\) to be the  \emph{periodic domain} bounded by \(\alpha\) and \(\alpha'\), that is,  the domain of the compactly supported 2-chain in \(S_{\alpha}\) which is bounded by \(\widetilde{\alpha} - \widetilde{\alpha}'\). If  \(z \in S_\alpha  - \widetilde{\alpha} - \widetilde{\alpha}'\),  we let \(n_z(\sP_{\alpha, \alpha'})\) be the multiplicity of \(\sP_{\alpha, \alpha'}\) at \(z\). 

\begin{definition} We say that \(\gamma\) and \( \gamma'\) are in \emph{admissible position} if \(\sP_{\alpha, \alpha'}\) has both positive and negative multiplicities. 
\end{definition}

 Since \(\widetilde{\alpha}\) and \(\widetilde{\alpha'}\) are transverse embedded curves in \(S^1\times \R\), \(\gamma\) and \( \gamma'\) are in admissible position if and only if \(\alpha\) and \(\alpha'\) intersect.

 Suppose that  \(\gamma\) and \(\gamma'\) are in admissible position, that \(x, y\in \alpha \cap \alpha'\), and that \(\psi_1,\psi_2 \in \pi_2(x,y)\).  If \(\mathcal{D}(\psi_i)\) is the domain associated with \(\psi\), then  \(\mathcal{D}(\psi_1) = \mathcal{D}(\psi_2)  + k \sP_{\alpha, \alpha'} \) for some \(k \in \Z\). Since \(\sP_{\alpha, \alpha'}\) has both positive and negative multiplicities, there are only finitely many values of \(k\) for which all the multiplicities of \(\mathcal{D}(\psi_1)\) are positive, and hence only finitely many terms in the sum for the differential on \(\CF(\alpha, \alpha')\). Hence if \(\gamma\) and \(\gamma'\) are in admissible position, \(\CF(\gamma, \gamma')\) is well-defined.

More generally, if \(\bgamma = \cup \bgamma_i\) and \(\bgamma' = \cup \bgamma_j'\) are  unobstructed immersed curves equipped with local systems, we say \(\bgamma\) and \(\bgamma'\) are in admissible position if \(\gamma_i\) and \(\gamma_j'\) are in admissible position for each \(i\) and \(j\). If this is the case, the Floer chain complex \(\CF(\bgamma, \bgamma')\) is well defined. We will now explain how to compute its homology \(\HF(\bgamma, \bgamma')\).  We start with an elementary lemma, whose proof is left to the reader. 

\begin{lemma}
\label{Lem:LocalSystems}
If \(\bgamma_1 = (\gamma, k_1, A_1)\) and \(\bgamma_2 = (\gamma, k_2,A_2)\), then 
\(\CF(\bgamma_1\cup\bgamma_2,\bgamma') = \CF(\bgamma_3, \bgamma')\), 
where \(\bgamma_3 = (\gamma, k_1+k_2, A_1\oplus A_2)\). 
Similarly, if \(\bgamma_1 = (\gamma_1,k,A)\), where \(\gamma_1\) is the \(n\)-fold cover of \(\gamma\), then \(\CF(\bgamma_1,\bgamma') = \CF(\bgamma,\bgamma')\), where \(\bgamma=(\gamma, nk, A')\) and \(A'\) is a block matrix of the form
\begin{equation}
\label{eq:monodromy}
\left[\begin{array}{c|c|c|c}
0 & 0 & 0 & A \\  \hline
I_k & 0 & 0 & 0 \\  \hline
0 & I_k & 0 & 0 \\  \hline
0 & 0 & I_k & 0
\end{array} \right] 
\end{equation}
\end{lemma}

Next, we consider the effect of varying a loop  \(\gamma\) by a homotopy. 
Fix some \(\gamma'\); we say that the homotopy \(\gamma_t\) is \emph{good} if for each \(t\in \{0,1\}\), \(\gamma_t\) is unobstructed and the pair \((\gamma_t,\gamma')\) is admissible. If \((V_0, \phi_0)\) is a local system on \(\gamma_0\), the pullback to \(S^1 \times [0,1]\) induces a local system \((V_1, \phi_1)\) on \(\gamma_1\). 

\begin{theorem}
\label{Thm:Isotopy Invariance} Suppose that \(\bgamma_0, \bgamma'\) are unobstructed curves equipped with local systems. If \(\gamma_t\) is a good homotopy from \(\gamma_0\) to \(\gamma_1\), then \(\HF(\bgamma_0, \bgamma') \cong \HF(\bgamma_1, \bgamma')\). 
\end{theorem}

\begin{proof}
This is a standard result in the symplectic context, where we work with area preserving isotopies and coefficients in the Novikov ring. For a proof in this setting, see \cite[Section 4]{Abouzaid2008} or \cite[Lemma 2.12]{AurouxSmith}. We sketch the necessary changes to make the proof work in the context of admissible diagrams. As in the proof of Proposition 4.1 in \cite{Abouzaid2008}, we first pass to the cover \(S_{\gamma_0}\), and thus reduce to the case where \(\gamma_0\) and \(\gamma_1\) are embedded. After splitting the preimage of \(\gamma'\) into separate components and passing to a further cover, as in the proof of Corollary~\ref{cor:extra cover}, we can further assume that \(\gamma'\) is either (a) incommensurable with \(\gamma_0\) or (b) embedded and isotopic to \(\gamma_0\). Case (b) is actually a special case of the invariance of sutured Floer homology  \cite{JuhaszSFH}, since the  sutured Heegaard diagrams \((S^1\times [0,1], \gamma_0, \gamma')\) and \((S^1\times [0,1], \gamma_1, \gamma')\) are admissible sutured Heegaard diagrams representing the same sutured manifold. Henceforth we assume we are in case (a). 

By breaking the isotopy \(\gamma_t\) into a sequence of smaller isotopies, as in \cite[Lemma 4.2]{Abouzaid2008}, we may assume that \(\gamma_0\) and \(\gamma_1\) are parallel curves intersecting in two points \(\theta_1\) and \(\theta_2\). Then a (combinatorially defined) count of immersed triangles corresponding to composition with \(\theta_1\) and \(\theta_2\) gives chain maps \(f\co \CF(\gamma_0, \gamma') \to \CF(\gamma_1, \gamma')\) and \(g: \CF(\gamma_1, \gamma') \to \CF(\gamma_2, \gamma')\), where \(\gamma_2\) is a small translate of \(\gamma_0\) intersecting it in two points. Finally, the associativity relation (obtained by counting immersed quadrilaterals) ensures that \(f \circ g\) is homotopic to the natural map identifying \(\CF(\gamma_0,\gamma')\) with \(\CF(\gamma_2, \gamma')\). 

This argument carries over largely unchanged to the admissible setting, but we must check that the  counts of immersed triangles and quadrilaterals used to define the chain maps and homotopies are finite. We discuss the case of quadrilaterals; the argument for triangles is very similar. 
 In order to check that there are only finitely many homotopy classes that contribute to the sum, we must show that for a given \(N\), there are only finitely many periodic domains bounded by \(\gamma_0, \gamma_1, \gamma_2\) and \(\gamma'\) whose multiplicities are all between \(-N\) and \(N\). To do this, it suffices to find a basis \(\sP_i\) for the space of periodic domains and points 
\(z_i^\pm\) such that \(n_{z_i^+}(\sP_i) >0\), \(n_{z_i^-}(\sP_i)<0\), and 
\(n_{z_i^\pm}(\sP_j) = 0\) for \(i\neq j\). In this case we say the points \emph{separate} the periodic domains. Since we are in case (a) \(\gamma_0\) and \(\gamma'\) are incommensurable, and all periodic domains are bounded by  \(\gamma_0, \gamma_1\), and \( \gamma_2\). By choosing \(\gamma_2\) appropriately, we can assume the configuration is as shown in Figure~\ref{fig:periodic_domains}. The periodic domains \(\sP_1\) and \(\sP_2\)  form a basis and are separated by the points shown in the figure.
\end{proof}

\begin{figure}[h]
\labellist
\pinlabel {$\sP_1$} at 77 30 
\pinlabel {$\sP_2$} at 133 30 
\endlabellist
\includegraphics[scale=2]{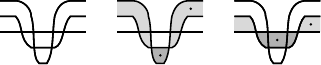}
\caption{The curves \(\gamma_0,\gamma_1, \gamma_2\) in the cylinder \(S_{\gamma_0}\), together with a basis of periodic domains.}\label{fig:periodic_domains}
\end{figure}

Recall that if \(\gamma, \gamma'\) are essential immersed curves in \(S\), their \emph{intersection number} is the minimal geometric intersection number between transverse curves \(\widehat{\gamma}, \widehat{\gamma}'\) that are homotopic to \(\gamma, \gamma'\). By a theorem of Freedman, Hass, and Scott \cite{FreedmanHassScott}, this minimal number is realized when \(\widehat{\gamma}, \widehat{\gamma}'\) are shortest length geodesics in the free homotopy classes of \(\gamma, \gamma'\). Similarly, if \(\gamma\) is primitive, the self-intersection number \(i(\gamma)\) is the minimal number of self-intersections among immersed curves homotopic to \(\gamma\). It is realized by the number of self-intersections of the shortest geodesic representing \(\gamma\). 

\begin{corollary}
\label{cor:ComputeHF}
 Suppose \(\bgamma = (\gamma, k, A)\) and \(\bgamma' = (\gamma', k', A')\) are primitive unobstructed immersed curves equipped with local systems. If \(\gamma\) and \(\gamma'\) are incommensurable, then \[\dim \HF(\bgamma, \bgamma') = kk' \cdot i(\gamma,\gamma').\] If \(\gamma' = \gamma^{-1}\), 
then \[\dim \HF(\bgamma, \bgamma') = 2 kk' \cdot i(\gamma) + 2  \dim(\ker \Phi_{A,A'}),\] where \(\Phi_{A,A'}\) is the endomorphism of \(\F^k \otimes \F^{k'}\) given by 
\(\Phi_{A,A'} = I \otimes I + A \otimes A'\). 
As a consequence, \( \HF(\bgamma, \bgamma)\) is always a nontrivial. 
\end{corollary}

If \(\gamma, \gamma'\) are commensurable but not primitive, we can use Lemma~\ref{Lem:LocalSystems} to reduce to the case where we have two parallel primitive curves. Hence the corollary enables us to compute 
\(\HF(\gamma, \gamma')\) for any two decorated curves \(\gamma, \gamma'\).  

\begin{proof} Let \(\widehat{\gamma}, \widehat{\gamma}'\) be shortest geodesics homotopic to \(\gamma, \gamma'\). Since \(\gamma\) and \(\gamma'\) are primitive, we can assume they are either incommensurable or \(\widehat{\gamma}' = \widehat{\gamma}^{\pm 1}\). In the first case, \(\CF(\bgamma, \bgamma') \cong \CF(\widehat{\bgamma}, \widehat{\bgamma}')\), and \(\CF(\widehat{\bgamma}, \widehat{\bgamma}')\) has \(kk' \cdot i(\gamma, \gamma')\) generators. To understand the differential, we consider the cover \(S_\gamma\). By section 3 of \cite{FreedmanHassScott}, each component of \(p_\gamma^{-1}(\widehat{\gamma}')\) intersects the lift of \(\widehat{\gamma}\) in either \(0\) or \(1\) point. Hence \(\pi_2(x,y) = 0\) for any two distinct generators of \( \CF(\widehat{\bgamma}, \widehat{\bgamma}')\). This proves the first statement. 

If \(\gamma'\) is homotopic to \(\gamma^{-1}\), \(\HF(\bgamma, \bgamma') \cong \HF(\widehat{\bgamma}, \bgamma'')\), where \(\gamma''\) is a small translate of \(\widehat{\gamma}\) intersecting it in two points. In this case \(p_\gamma^{-1}(\gamma'')\) consists of some arcs and a single loop that is a small translate of  \(\widehat{\bargamma}\) and intersects it in two points. Since \(\gamma''\) is a small translate of \(\widehat{\gamma}\), as before, each of the arcs intersects \(\widehat{\bargamma}\) in a single point. There are two such intersections for each self-intersection of \(\widehat{\gamma}\), so in total these contribute \(2kk' \cdot i(\gamma)\) generators, and no differential.
The final part of the chain complex is the summand \(\CF(\widehat{\bargamma}, \widehat{\bargamma})\) for which there are two intersection points, each contributing \(\Hom(V',V) \cong \F^k \otimes \F^{k'}\) to the set of generators, and two embedded disks. The differential in this complex is given by \(\Phi_{A,A'} \), so the dimension of its homology is \(2  \dim (\ker \Phi_{A,A'})\). 

For the final claim, reverse the orientation on \(\gamma\) to get 
\[\HF((\gamma,k,A), (\gamma,k,A)) = \HF((\gamma,k,A),(\gamma^{-1},k, A'))\] where \(A' = (A^{-1})^T\). The operator \(\Phi_{A,A'}\) can be viewed as a map from \(\End(\F^k)\) to itself. Thinking of  \( \End(\F^k)\) as the space of \(k\times k\) matrices with entries in \(\F\), we have \(\Phi_{A,A'}(X) = AXA^{-1} + X\).
The \(k \times k\) identity matrix is always in the kernel of \(\Phi\), so \(\HF((\gamma,k,A), (\gamma,k,A))\) is always nontrivial. 
\end{proof}

 \subsection{From Type D structures to curves} 

We now define  a set \(\Curves\), which is the target of the invariant \(\HFhat(M)\). Consider collections of finitely many oriented immersed unobstructed loops decorated with local systems in the marked torus; recall that by marked torus we mean the punctured torus $T$ with a choice of parametrizing curves $\alpha$ and $\beta$. Each decorated curve is a triple $(\gamma, k, A)$, where $\gamma$ is an oriented unobstructed immersed curve,  $k$ is a positive integer, and $A$ is a similarity class of $k\times k$ matrices over $\F$. As discussed earlier, the pair \((k,A)\) determines a \(k\)-dimensional local system on \(\gamma\). 

We impose an equivalence relation on collections of such curves, generated by the following relations, which can be applied to any subset of a collection:

\begin{itemize}
\item If \(\gamma\) is homotopic to \(\gamma'\), $(\gamma, k, A)$ is equivalent to $(\gamma', k, A)$.
\item  $(\gamma, k, A)$ is equivalent to $(-\gamma, k, A^{-1})$, where 
\(-\gamma\) denotes \(\gamma\) with the orientation reversed. 
\item The pair $\{(\gamma, k_1, A_1),(\gamma, k_2, A_2)\}$ is equivalent to the single curve $(\gamma, k_1+k_2, A_1\oplus A_2)$.
\item 
 If $\gamma$ is homotopic to $(\gamma')^m$ then $(\gamma, k, A)$ is equivalent to $(\gamma', mk, A')$, where $A'$ is a block matrix of the form
shown in equation~\eqref{eq:monodromy}. 
\end{itemize}
We define $\Curves$ to be the set of equivalence classes. 
Each equivalence class contains a representative which is 
 minimal in the sense that the underlying curves are all primitive and no two of them are homotopic. This representative is unique up to orientation reversal, conjugation of the monodromy, and  homotopy of the underlying curves. 
 
 Given  $\bgamma_1,\bgamma_2 \in \Curves$, we define their pairing to be $\langle \bgamma_1, \bgamma_2 \rangle = \HFa(\bgamma_1, r(\bgamma_2))$, where $r$ denotes the orientation reversing diffeomorphism of the marked torus that exchanges $\alpha$ and $-\beta$. 
Lemma~\ref{Lem:LocalSystems} and Theorem \ref{Thm:Isotopy Invariance} imply that the pairing is well-defined.


Let $\Extend$ denote the set of extendable type D structures over $\Alg$, up to homotopy equivalence.  The set $\Extend$ is equipped with a pairing given by the box tensor product: if $N_1, N_2 \in \Extend$, their pairing is  $\langle N_1, N_2 \rangle =N_1^A \boxtimes N_2$, where $N_1^A$ is the $\Ainfty$ module corresponding to $N^1$ under the duality described in Section \ref{subsec:graphs and loops} (see Remark \ref{rmk:pairtracks}).

 The main result of this section is:

\begin{theorem}
\label{Thm:TheBijection}
There is a bijection \(f \co \Extend \to \Curves\), with inverse $g\co\Curves \to \Extend$, satisfying
\(\langle N_1, N_2 \rangle \simeq \langle f(N_1), f(N_2)\rangle\) and, equivalently, \(\langle \gamma_1, \gamma_2 \rangle \simeq  \langle g(\gamma_1), g(\gamma_2) \rangle\). 
\end{theorem}

We first define the inverse \(g\co\Curves \to \Extend\).
Given a decorated immersed curve $\bgamma = (\gamma, k, A)$ in the marked torus, we construct a train track by replacing $\gamma$ with $k$ parallel copies of $\gamma$, writing $A$ as a product of elementary matrices $P_1 \cdots P_\ell$, and adding a crossover arrow between copies of $\gamma$ for each elementary matrix, where an elementary matrix of type $A_{ij}$ corresponds to an arrow connecting the $i$th copy of $\gamma$ to the $j$th copy of $\gamma$ and the arrows corresponding to $P_\ell, P_{\ell-1}, \ldots, P_1$ appear in order according to the orientation of $\gamma$. We assume that $\gamma$ is homotoped so that it has minimal intersection with $\alpha \cup \beta$, and we take care that no crossover arrows intersect $\alpha$ or $\beta$. The result is a valid reduced train track; this corresponds to an extended type D structure $\tN$, which has an underlying (extendable) type D structure $N$. We define $g(\bgamma)=N$. In doing so, we made some choices: namely a matrix representing the similarity class $A$ and an elementary decomposition of that matrix. Changing either choice corresponds to sliding crossover arrows around $\gamma$ by moves which  preserve the homotopy equivalence class of the corresponding extended type D module; thus $g$ is well-defined.
Finally, if \(\bgamma = \bigcup \bgamma_i\) is a union of decorated curves, we define \(g(\bgamma) = \bigoplus g(\bgamma_i)\). 

Next, we consider the effect of \(g\) on the pairing. 
 Recall that if  \(\tracks\) is a curve-like train track (as in  Definition~\ref{def:curve-like}), there is an associated local system  \((V, \phi)\) on the  underlying curve \(\gamma\).  If \(x \) is a point on \(\gamma\), \(V_x\) is the \(\F\) vector space spanned by vectors \(e_p\) for \(p \in S_x:= \tracks \cap \ell_x\), where \(\ell_x\) is a small  arc transverse to \(\gamma\) and passing through \(x\). If \(\alpha\) is an arc from \(x\) to \(y\) along \(\gamma\), then \(\phi_\alpha: V_x \to V_y\) is given by \(\phi_\alpha(e_p) = \sum n_{pq} e_q\), where \(n_{pq}\) counts the number of paths on from \(p\) to \(q\) on \(\theta_i\) which project to the path \(\alpha\) on \(\gamma\). 
Applying this procedure to the train track used to define \(g(\bgamma)\) returns the decorated curve \(\bgamma\). 
 
 Now suppose that \(\tracks_1, \tracks_2\) are two  valid reduced   train tracks, and consider the chain complex \(\sC(\tracks_1, \tracks_2)\),  defined in Section~\ref{subsec:extended tracks}.  If \(A(\tracks_1)\) and \(D(\tracks_2)\) are in admissible  position,  \(\sC(\tracks_1, \tracks_2)\) is  their intersection Floer chain complex. Note that \(A(\tracks_1)\) is isotopic to \(\tracks_1\), while \(D(\tracks_2) = r(A(\tracks_2))\) is isotopic to \(r(\tracks_2)\). 
 
 \begin{lemma}
 \label{lem:Two HFs}
Suppose that \(\tracks\) and \( \tracks'\) are curve-like train tracks with associated decorated curves \(A(\tracks) = \bgamma  = (\gamma, V, \phi)\) and \(D(\tracks_2)=\bgamma' = (\gamma', V', \phi')\), and that \(\tracks\) and \(\tracks'\) are in admissible  position. 
Then \(\sC(\tracks, \tracks') \cong \CF(\bgamma, \bgamma')\).
\end{lemma}

\begin{proof}
  Generators of 
 \(\sC(\tracks,\tracks')\) correspond to triples \((x,p,q)\), where \(x \in \gamma \cap \gamma'\), \(p \in S_{x}\) and \(q \in S'_{x}\), while generators of \(\CF(\bgamma, \bgamma')\) correspond to pairs \((x,A)\), where \(A \) is a generator of  \(V_x \otimes V_x'\). They are in bijection via the map that sends \((x,p,q)\) to \((x,e_p \otimes e_q)\). 
 
 Similarly, immersed disks with boundaries on \(\tracks\) and \(\tracks'\) are in bijection with triples \((\psi, \alpha, \alpha')\), where \(\psi\) is an immersed disk  with boundary on \(\gamma_1\) and \( \gamma_2\),  \(\alpha\) is a path on \(\tracks\) which projects to the part of the boundary of \(\psi\) which lies on \(\gamma\) (and similarly for \(\alpha'\)), and  \(\alpha\) and \(\alpha'\) have the same endpoints. Under the bijection above, the contribution to the differential on \(\sC(\tracks_1, \tracks_2)\) coming from \(\psi\) matches the contribution to the differential on \(\CF(\bgamma_1, \bgamma_2)\) coming from \(\psi\). 
 \end{proof}

\begin{proposition} 
\label{Prop:Pairing}
If \(\bgamma_1, \bgamma_2 \in \Curves\) then \(\langle \bgamma_1, \bgamma_2 \rangle \simeq \langle  g(\bgamma_1), g(\bgamma_2)\rangle.\) 
\end{proposition}

\begin{proof}
Let \(N_i = g(\bgamma_i)\), let  $\tracks_i$ be the  curve-like train track associated with \(\bgamma_i\), and consider the complex $\sC(\tracks_1, \tracks_2)$. 
In order for $\sC(\tracks_1, \tracks_2)$ to be defined, we need the pair to be admissible---there can be no immersed annuli carried by $A(\tracks_1)$ and $D(\tracks_2)$. To ensure this, we will modify $\tracks_2$ so that it has no component for which the underlying curve consists only of the following four types of segments:
\begin{center}
\includegraphics[scale=0.5]{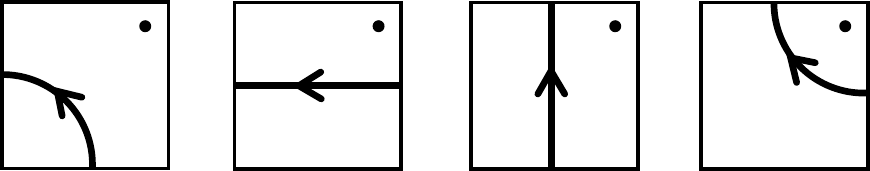}
\end{center}
This ensures that the pair is admissible, since any immersed annulus, when cut by $\alpha$ and $\beta$, consists only of the connecting pieces shown in Figure \ref{fig:bigon-pieces}. If any component of $\tracks_2$ does contain only the four segment types above, we choose one of the latter three types of segments and modify $\tracks_2$ by applying a finger-move isotopy to the given segment across either $\alpha$ or $\beta$, as in Figure \ref{fig:almost}. If $\tracks_2$ has parallel copies of the relevant segment we apply the isotopy to each copy, keeping them parallel, and ensure that the finger-move avoids any crossover arrows. Let $\tracks'_2$ denote the result of applying these finger-moves, if necessary, to $\tracks_2$. Note that $\tracks'_2$ is an almost reduced train track whose corresponding type D structure $N'_2$ is homotopy equivalent to $N_2$, where the homotopy equivalence replaces certain $\rho_{12}$, $\rho_{23}$, or $\rho_{123}$ arrows with zig-zags as in Figure \ref{fig:almost}; $N'_2$ is both bounded and almost reduced.

By Theorem \ref{thm:pairtracks}, $\sC(\tracks_1, \tracks'_2)$ is isomorphic to $N_1^A\boxtimes N'_2$, which is quasi-isomorphic to $ N_1^A \boxtimes N_2 = \langle g(\bgamma_1), g(\bgamma_2) \rangle$. On the other hand, 
by Lemma~\ref{lem:Two HFs}, $\sC(\tracks_1, \tracks'_2) \cong \CF(\bgamma_1, r(\bgamma_2'))$, where \(\bgamma_2' \) is isotopic to \(\bgamma_2\). Passing to homology gives the  statement of the proposition. 
\end{proof}

To prove Theorem~\ref{Thm:TheBijection}, we must show that \(g\) is injective and surjective. For surjectivity, we appeal to Theorem~\ref{thm:structure}. 
Given an element of $\Extend$, we choose a reduced representative $N$ and an extension $\tN$ of $N$. By Theorem~\ref{thm:structure}, \(N\) is homotopy equivalent to a reduced type D structure \(N'\) whose associated train track is a collection of curve-like train tracks. As we explained in the proof of Theorem~\ref{thm:structure}, to such a train track we can naturally associate an immersed decorated multicurve \(\bgamma\). To see that \(\bgamma\) is an element of \(\Curves\), we must check that each curve  \(\gamma \subset \bgamma\) is unobstructed, or equivalently, that  \(\widetilde{\gamma}\) is embedded in the universal cover. This follows from the fact that each individual segment of \(\gamma\) is embedded together with the fact that \(N'\) is reduced, so the path traced out by \(\gamma\) in the universal cover (thought of as a thickening of the regular 4-valent tree) never retraces itself. Then \(g(\bgamma) = N'\) is homotopy equivalent to our original \(N\), so \(g\) is surjective. 

\subsection{Injectivity of \(g\)}
\label{subsec:Injectivity of g}
 To complete the proof of Theorem~\ref{Thm:TheBijection}, we show that \(g\) is injective. 
Suppose that $N_1 = g(\bgamma_1)$ and $N_2=g(\bgamma_2)$  are homotopy equivalent to each other. Since $N_1$ and $N_2$ are homotopy equivalent, they have the same pairing behavior; that is, $$H_*(N_1^A\boxtimes N) \cong H_*(N_2^A\boxtimes N)$$ for any reduced type D structure $N$. 
For $i=1,2$, let $\tracks_i$ be a curve-like train track in \(T\) 
that represents \(\bgamma_i\).  It follows from Theorem \ref{thm:pairtracks} that $H_*(\sC(\tracks_1,  \tracks)) \cong H_* (\sC(\tracks_2,  \tracks))$ for any reduced weakly valid  train track $ \tracks$. We will show that \(\bgamma_i\) is determined by the groups $H_*(\sC(\tracks_i,  \tracks))$ as \(\tracks\) varies; thus \(\bgamma_1 = \bgamma_2\). 

Since $\tracks_i$ is a curve-like train track, a component of $\tracks_i$ consists of the following data: an immersed curve $\gamma$ in the punctured torus that intersects $\alpha\cup\beta$ minimally, an integer determining how many parallel copies of \(\gamma\) there are, and a collection of crossover arrows between parallel copies of $\gamma$. The curve $\gamma$ may be represented as a reduced cyclic word in $\{\alpha^{\pm 1}, \beta^{\pm 1}\}$ by traversing $\gamma$ and recording the intersections with the parametrizing curves $\alpha$ and $\beta$, with negative crossings recorded as inverse letters. Equivalently, we realize $\gamma$ as an element of the fundamental group of the marked torus, which is the free group on two generators. Here we take the generators $\alpha$ and $\beta$ of the fundamental group to be the simple closed curves dual to the parametrizing curves $\alpha$ and $\beta$. Note that viewing $\gamma$ as an element of the fundamental group requires choosing a basepoint, but choosing a different basepoint only changes the resulting word by cyclic permutation.

The underlying curves of $\tracks_i$ are determined by the results of pairing  $\tracks_i$ with certain immersed line segments. To see this, we first define the immersed segments in question. Given a reduced word $I$ in $\{\alpha^{\pm 1}, \beta^{\pm 1}\}$ we construct a  train track $\tracks_I$ from a sequence of horizontal and vertical segments, with each segment having endpoints near the center of the marked torus and crossing either $\alpha$ or $\beta$ exactly once. Each segment corresponds to a letter in the word $I$ depending on which parametrizing curve it crosses, where negative crossings are interpreted as an inverse letter. In particular, each letter $\alpha$ 
in $I$ gives rise to horizontal segment oriented rightward, 
and each letter $\beta$ 
in $I$ gives rise to a vertical segment oriented downward. 
We construct $\tracks_I$ by connecting these segments end to end, smoothing the corners, in the order determined by the word $I$; this results in an oriented immersed curve segment with the property that, when traversed, the intersections with  \(\alpha\) and \(\beta\) read off the word $I$. Some examples are shown in Figure \ref{fig:Segments_and_Duals} and Figure \ref{fig:Exp-for-inject-proof}. Note that traversing the segment with the opposite orientation gives the inverse word $I^{-1}$ in the free group, so $\tracks_I = -\tracks_{I^{-1}}$. We will generally not distinguish $\tracks_I$ from $\tracks_{I^{-1}}$, since the orientation is not relevant for pairing. Following Section \ref{sec:tracks}, $\tracks_I$ represents a type D structure with one generator for each segment and arrows determined by Figure \ref{fig:algebra-edges}. 
Note that this type D structure satisfies $\partial^2=0$ but it is not extendable. Equivalently (by Proposition \ref{prp:mod-track-matrix-weak}), the train track is only weakly valid.

For any given word $I$ in the free group generated by $\alpha$ and $\beta$ we can define a new word $\bar I$, which may be viewed as dual to $I$. Rather than define this new word algebraically, we define a procedure for constructing the  immersed line segment $\tracks_{\bar{I}}$ from $\tracks_{I}$.
More precisely, we will first construct \(D(\tracks_{\bar{I}})\), and then determine 
 $\tracks_{\bar{I}}$ from it. Assuming that the length of \(I\) is greater than \(1\), 
  $D( \tracks_{\bar{I}})$  is chosen so that (1) the first segment of $D( \tracks_{\bar{I}})$ intersects the first segment of $A(\tracks_I)$, (2) the last segment of $D(\tracks_{\bar I})$ intersects the last segment of $A(\tracks_I)$, and (3) the paths in $D( \tracks_{\bar I})$ and $A(\tracks_I)$ connecting these first and last intersections are homotopic rel endpoints; see Figure \ref{fig:Segments_and_Duals}. These rules uniquely determine the sequence of horizontal and vertical segments appearing in \(D(\tracks_{\bar{I}})\): for each step after the first, we choose the segment so that it exits the square (the fundamental domain for \(T\)) by the same side as \(A(\tracks_I)\). 
 In the special case where the length of \(I\) is \(1\), we introduce the convention that $\bar I = I$, so $D( \tracks_{\bar I})$ intersects \(A(\tracks_I)\) in a single point. Note that  reflecting across the antidiagonal in the torus takes $A(\tracks_I)$ to $D(\tracks_I)$ and $D( \tracks_{\bar I})$ to $A( \tracks_{\bar I})$ so that applying the dual operation to $\bar I$ recovers $I$.

 \begin{figure}
\labellist
\pinlabel {$\tracks_{\alpha\beta\alpha^{-1}\beta}$} at 70 85
\tiny 

\pinlabel {$\alpha$} at 16 174 
\pinlabel {$\alpha^{\!\!-\!1}$} at 24 219
\pinlabel {$\beta$} at 45 140 
\pinlabel {$\beta$} at 95 140

\pinlabel {$z$} at 112 241
\pinlabel {$z$} at 413 241
\pinlabel {$z$} at 718 241
\endlabellist
\includegraphics[scale=0.4]{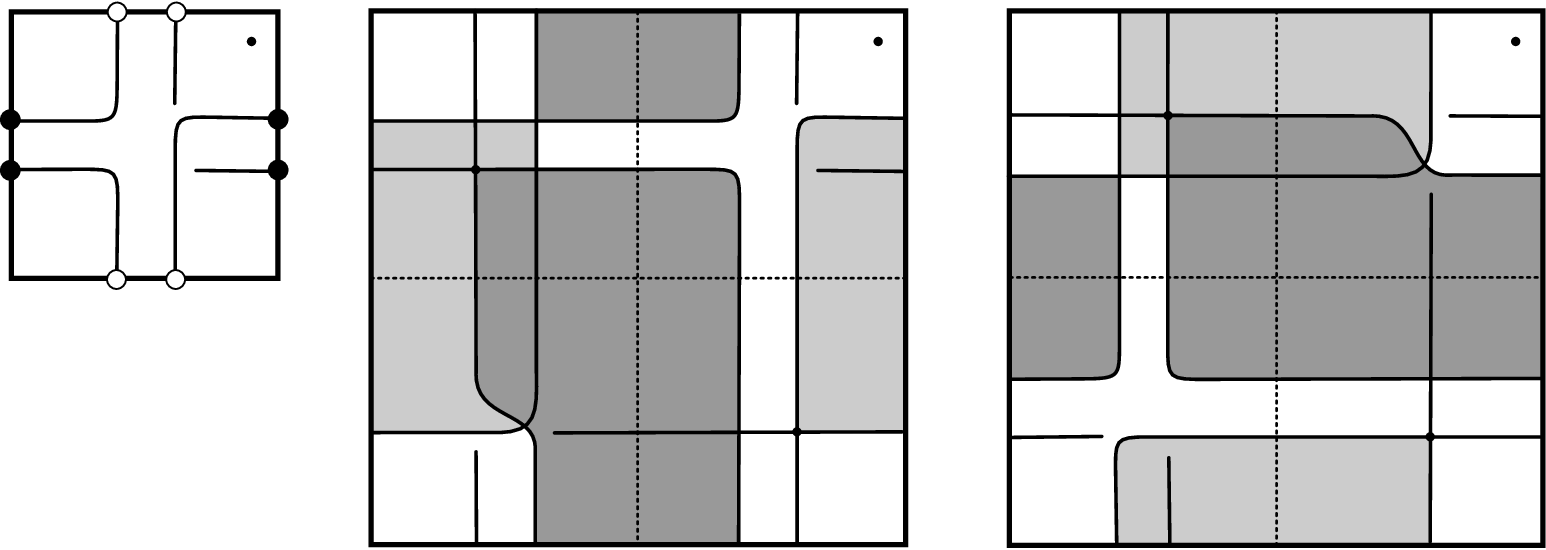}
\caption{Starting from the word $I=\alpha\beta\alpha^{-1}\beta$ we construct $\tracks_{I}$ (left). Next (center) we construct $A(\tracks_{I})$  by applying the operations of Section \ref{sec:pairing-train-tracks}, and then find $D(\tracks_{\bar{I}})$  via the construction in the text. The shaded bigons illustrate the homotopy rel endpoints between the two curves. Note that the start and end points of $D(\tracks_{\bar{I}})$  (as well as any corners) must lie in the lower left quadrant of the square. Finally, we reflect across the antidiagonal (right) to obtain $A(\tracks_{\bar{I}})$, and hence \(\tracks_{\bar{I}}\). This determines $\bar{I} = \alpha \alpha \beta^{-1}$. Since $A(\tracks_{\bar{I}})$ and $D(\tracks_{I})$ satisfy the same relationship as  $A(\tracks_{I})$ and $D(\tracks_{\bar{I}})$,  $\bar{\bar{I}} = I$.}\label{fig:Segments_and_Duals}
\end{figure}

These immersed curve segments will be used as test train tracks against which we will pair the curve-like train tracks $\tracks_i$ we hope to identify. 
It turns out that for any word $I$, the result of pairing with $\tracks_{\bar I}$ gives us information about how many times the word $I$ appears in the cyclic word defining a curve-like train track. More precisely, for a curve-like train track $\tracks_i$ and a word $I$, let $n_I(\tracks_i)$ be the number of times either $I$ or $I^{-1}$ appears in the cyclic words representing the underlying curves of $\tracks_i$, where each appearance is counted with multiplicity associated with the relevant cyclic word (i.e. the dimension of the local system associated with the corresponding curve component of $\tracks_i$).

\begin{lemma}
\label{Lem:Determine Curves}
For a curve-like train track $\tracks$, \(n_I(\tracks)\) is determined by 
 the dimension of the homology of $\sC(\tracks, \tracks_{\bar{I}})$ together with the values of \(n_J(\tracks)\) for all words $J$ for which $\bar J$ is a subword of $\bar I$.
\end{lemma}
%
%
%
%
%

\begin{proof}
The homology of the complex $\sC(\tracks, \tracks_{\bar I})$ does not depend on the local system on $\tracks$ (that is, on the crossover arrows) since  $\tracks_{\bar I}$ is not a closed curve. Without loss of generality then, assume that all local systems are trivial and $\tracks$ has no crossover arrows. Each component of $\tracks$ is represented by a cyclic word in $\alpha$ and $\beta$ with some multiplicity. Note that in $A(\tracks)$ there is a horizontal segment for each instance of $\alpha$ or $\alpha^{-1}$ in the words representing $\tracks$ (counted with multiplicity), and there is a vertical segment for each instance of $\beta$ or $\beta^{-1}$. Moreover, the immersed arc obtained by following $A(\tracks)$ from one horizontal or vertical segment to another can be interpreted as $A(\tracks_I)$ for a subword $I$ of the cyclic word representing the relevant component of $\tracks$, and the arc $A(\tracks_I)$ (ignoring orientation) appears for each instance of $I$ or $I^{-1}$ in the cyclic words representing $\tracks$.



We first observe that the dimension of $H_* (\sC(\tracks_, \tracks_\alpha))$ is precisely $n_\alpha(\tracks)$. This follows from the fact that $D(\tracks_\alpha)$ consists of a single vertical segment and no bigons are formed in this case, so intersecting $D(\tracks_\alpha)$ with $A(\tracks)$ simply counts the number of horizontal segments in $A(\tracks)$. Similarly, $n_\beta(\tracks)$ is determined by $H_*( \sC(\tracks_, \tracks_\alpha))$. We will demonstrate the proof for longer words $I$ in an example before commenting on the general case. Let $I = \alpha\beta^{-1}\alpha$, so \(A(\tracks_I)\) and \(D(\tracks_{\bar I})\) are as shown in Figure \ref{fig:Exp-for-inject-proof}; note that $\bar I = \alpha^{-1} \beta \alpha^{-1}$.  To compute the homology of $\sC(\tracks, \tracks_{\bar I})$ we first find the generators. These count $\alpha$-$\alpha$ or $\beta$-$\beta$ intersections, so there are $n_\alpha(\tracks)$ generators for each $\alpha$ or $\alpha^{-1}$ in $I$ and $n_\beta(\tracks)$ for each $\beta$ or $\beta^{-1}$ in $I$. In this example, the number of generators is given by \(2n_\alpha(\tracks) + n_\beta(\tracks)\). We next consider the differential. The length two subwords of $\bar I$ are $\bar J_1 = \alpha^{-1} \beta$ and $\bar J_2 = \beta \alpha^{-1}$, with dual words $J_1 = \alpha \beta^{-1}$ and $J_2 = \beta^{-1} \alpha$ (note that in this example $J_1$ and $J_2$ are subwords of $I$, but this need not be the case). Letting $J$ be either $J_1$ or $J_2$, for each instance of $J$ or $J^{-1}$ in the cyclic words representing $\tracks$, there is  a bigon between the portion of $A(\tracks)$ corresponding to $J$ and the portion of $D(\tracks_{\bar I})$ corresponding to $\bar J$. These bigons take the form of the lightly shaded bigon in Figure \ref{fig:Exp-for-inject-proof} for $J = J_1$ or the darkly shaded bigon in the figure for $J=J_2$. It is not difficult to see that all bigons contributing to the differential occur in this way. If all such bigons were disjoint then we would simply cancel two generators for each bigon. However, if two bigons share an endpoint then this computation is incorrect, as it would over-cancel. We next observe that if $A(\tracks)$ contains a segment that looks like $A(\tracks_{\alpha \beta^{-1} \alpha})$ then intersecting with $D(\tracks_{\bar I})$ produces a pair of bigons sharing an endpoint as in Figure \ref{fig:Exp-for-inject-proof}, and any pair of bigons formed with $D(\tracks_{\bar I})$ and sharing an endpoint must have this form. Thus we can correct for the over-cancellation by adding $2 n_{\alpha \beta^{-1} \alpha}(\tracks)$, and we have
\begin{equation*}
\dim H_*( \sC(\tracks, \tracks_{\bar I}) ) = 2 n_\alpha(\tracks) + n_{\beta}(\tracks) 
  - 2n_{\alpha\beta^{-1}}( \tracks) - 2n_{\beta^{-1}\alpha}( \tracks) 
  + 2n_{ \alpha\beta^{-1}\alpha}( \tracks)
\end{equation*}
The last term on the right is $n_I(\tracks)$ and all other terms on the right are $n_J(\tracks)$ where  $\bar J$ is a subword of $ \bar I$, as desired.

 \begin{figure}
\labellist
\pinlabel {$\tracks_{\alpha\beta^{-1}\alpha}$} at 70 85
\tiny 
\pinlabel {$z$} at 112 241
\pinlabel {$z$} at 413 241
\endlabellist
\includegraphics[scale=0.4]{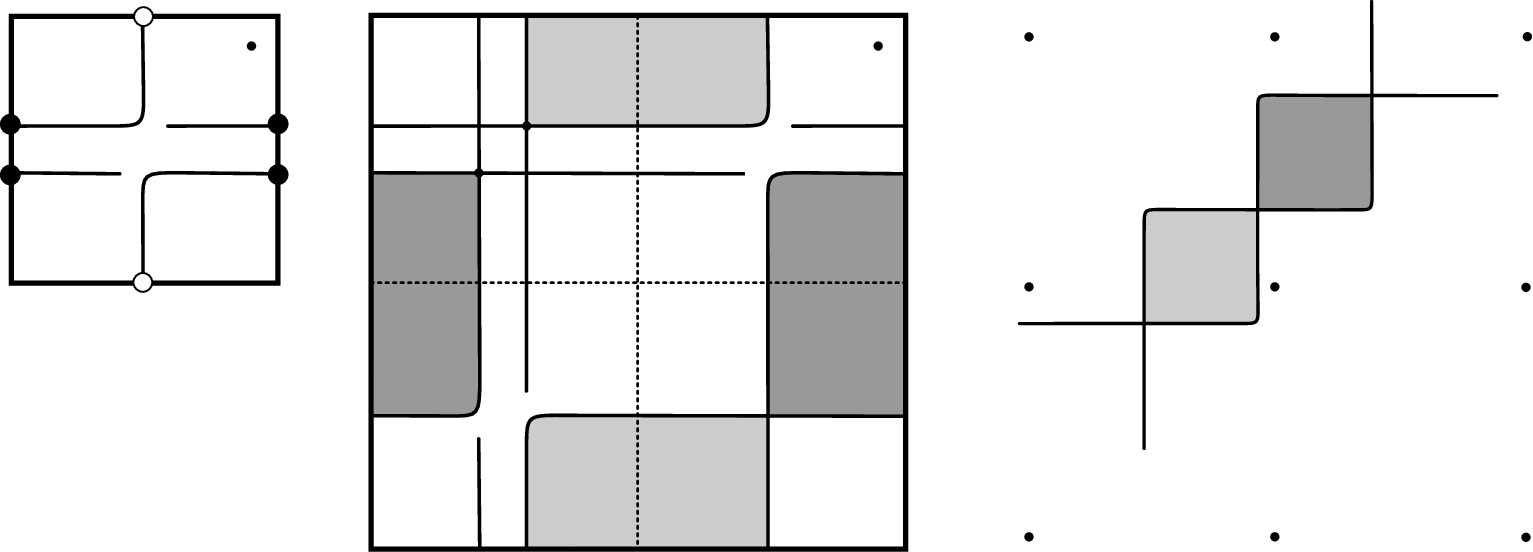}
\caption{As simple example illustrating the key step in the proof of Lemma \ref{Lem:Determine Curves}, the pairing of $A(\tracks_{I})$ with $D(\tracks_{\bar I})$ for $I = \alpha \beta^{-1} \alpha$, shown both in the punctured torus and a covering space.
}\label{fig:Exp-for-inject-proof}
\end{figure}

We now consider the case of an arbitrary word $I$. A key fact implicit in the example above is that any intersection point $x$ between $A(\tracks)$ and $D(\tracks_{\bar I})$ is an end of at most two bigons, and if there are two then they both start at $x$ or they both end at $x$. This fact holds in the general case, as we will show below. Assuming this, note that the Floer complex $C(\tracks, \tracks_{\bar I})$ is a direct sum of \emph{zig-zag chains}, i.e. complexes that can be represented by a chain of dots along a line connected by arrows on that line pointing in alternating directions. 
The dimension of $H_*(C(\tracks, \tracks_{\bar I}))$ is a function of the number of generators and the number of zig-zag chains of length $k$ for each $k > 0$; note that when we count zig-zag chains of length $k$ we will include subchains of longer chains, so for instance a chain of length three also contributes two to the count of length two chains and three to the count of length one chains. Now consider a bigon or chain of bigons starting at an intersection point $x$ and ending at an intersection point $y$, contributing a zig-zag chain to $C(\tracks, \tracks_{\bar I})$ (we do not assume the chain of bigons is maximal, so the zig-zag chain may be a subchain of a longer zig-zag chain in the complex). The portion of $D(\tracks_{\bar I})$ from the segment containing $x$ to the segment containing $y$ is $D(\tracks_{\bar J})$ for some subword $\bar J$ of $\bar I$. The portion of $A(\tracks)$ from the segment containing $x$ to the segment containing $y$ runs parallel to $D(\tracks_{\bar J})$ and so by the definition of the bar operation on segments this portion of $A(\tracks)$ is $A(\tracks_J)$. This indicates that $J$ or $J^{-1}$ appears in the cyclic word representing $\tracks$. Conversely, each time $J$ or $J^{-1}$ appears in the word representing $\tracks$ a corresponding subset of $A(\tracks)$ 
agrees with $A(\tracks_J)$ and forms such a chain of bigons with the subset $D(\tracks_{\bar J})$ of $D(\tracks_{\bar I})$. It follows that chains of bigons of that particular form are precisely counted by $n_J(\tracks)$, and the counts of all chains is determined by $n_J(\tracks)$ for all $J$ for which $\bar J$ is a subword of $\bar I$. If $n_J(\tracks)$ is known for all $J$ for which $\bar J$ is a strict subword of $\bar I$, then $n_I(\tracks)$ can be determined from $\dim( H_*(\tracks, \tracks_{\bar I}) )$.

It remains to show that the complex $C(\tracks, \tracks_{\bar I})$ breaks into zig-zag chains. This ultimately relies on the fact that $\tracks$ and $\tracks_{\bar I}$ are reduced, but it is especially apparent given the particular pairing position we assume for $C(\tracks, \tracks_{\bar I})$. In particular, recall that any bigon contributing to the differential in $C(\tracks, \tracks_{\bar I})$ must decompose into pieces of the forms shown in Figure~\ref{fig:bigon-pieces}. We can exclude the pieces of small bigons in the top row of the figure because $\tracks_{\bar I}$ is assumed to be reduced rather than only almost reduced. By inspecting these pieces it is straightforward to check that $x$  cannot be both a starting and an ending point of a bigon. For example, if $x$ is in the top left quadrant of the square, a bigon starting at $x$ is only possible if the path in $A(\tracks)$ starting rightward from $x$ first leaves the square on the top edge or if the path in $D(\tracks_{\bar I})$ starting downward from $x$ first leaves the square on the left edge, but a bigon ending at $x$ is only possible if both these paths first leave the square on the right or bottom edge. 

Inspecting the pieces in Figure~\ref{fig:bigon-pieces} 
reveals %
that an intersection point $x$ can be the start or end of at most one bigon covering a given quadrant locally near $x$. Indeed, the paths in $A(\tracks)$ and $D(\tracks_{\bar I})$ starting from $x$ in the directions required to surround a given quadrant near $x$ uniquely determine a chain of pieces like those in Figure~\ref{fig:bigon-pieces}, which ends the first time the paths leave the square on different sides. A bigon is formed only if the paths eventually cross, so that the chain of pieces gets to an ending piece, and once this happens there is no way of adding more pieces to get a second bigon with a corner at $x$. 
More precisely, again by inspecting Figure~\ref{fig:bigon-pieces}, observe that both the left and right boundary of any bigon, viewed as paths from the starting intersection point to the ending intersection point, must move monotonically upward and leftward in the following sense:
considering only the midpoints of the horizontal and vertical segments making up the path (recall that these are where intersections occur), the projection to the vector $(-1,1)$ in the plane is non-decreasing, and it is strictly increasing except possibly at the beginning or end of the path (this happens in the case of the first or third starting piece or the first or third ending piece in Figure~\ref{fig:bigon-pieces}). Now suppose there are two bigons starting at $x$ covering the same quadrant locally near $x$ (the argument for bigons ending at $x$ is similar). Call these bigons $B_1$ and $B_2$, call their endpoints $y_1$ and $y_2$, call their left boundaries by $\partial_L B_1$ and $\partial_L B_2$ and call their right boundaries by $\partial_R B_1$ and $\partial_R B_2$. Let $B_1$ be the bigon with the shorter left boundary, so that $\partial_L B_1 \subset \partial_L B_2$; it can be shown that $\partial_R B_2$ must a subset of $\partial_R B_1$, since otherwise $B_2$ would have a region of negative multiplicity near $y_1$ (at least in an appropriate cover). In fact, $\partial_L B_2$ must intersect $\partial_R B_1$ at some other point $y_3$ between $y_2$ and $y_1$. But in this case $y_3$ occurs before $y_1$ on $\partial_R B_1$ and after $y_1$ on $\partial_L B_2$, violating  the fact that both paths move monotonically in the upward/leftward direction. (Note that monotonicity of $\partial_L B_2$ implies that $y_3$ is strictly above/left from $y_1$ since neither of them is an end of $\partial_L B_2$).
\end{proof}

The preceding lemma (and induction on the length of $\tracks_{\bar I}$) shows that if the dimension of the homology of $\sC(\tracks, \tracks_{\bar J})$ is known for every word $J$, then the number of instances of $I$ or $I^{-1}$ in $\tracks$ can be determined for any word $I$. Since $\tracks$ has a finite number of curves that have a finite length, sufficiently long words will only occur a nonzero number of times if they repeat the entire cyclic word for one of the immersed curves in $\tracks$. In this way, the underlying immersed curves in $\tracks$ can be determined, with multiplicity.

It remains to show that if \(\tracks_1\) and \(\tracks_2\) are curve-like train tracks with the same underlying curves and multiplicities but with different local systems, then there is some \(\tracks'\) with \(\sC(\tracks_1,\tracks') \not \simeq \sC(\tracks_2, \tracks')\). Let \(\gamma\) be an underlying curve of \(\tracks_1\) and \(\tracks_2\) for which the local systems are non-isomorphic. We will take \(\tracks'\) to be curve-like with underlying curve \(\gamma\). Since the local systems on the other components of \(\tracks_1\) and \(\tracks_2\) have the same multiplicities, they will have the same pairing with \(\tracks'\). Thus we may assume that \(\tracks_i\) has a single underlying curve \(\gamma\) and that the monodromy of the local system is some \(k \times k\) matrix \(A_i\). Referring to the calculation in Corollary~\ref{cor:ComputeHF}, we see that it suffices to prove the following result. 

\begin{proposition}\label{prp:self-pair}
If $A_1$ and $A_2$ are $k\times k$ matrices over $\F$ which are not similar, then there exists a matrix $B$ such that $\rk( A_1\otimes B + I\otimes I) \neq \rk( A_2\otimes B + I\otimes I )$.
\end{proposition}

\begin{proof}
We outline the key steps but leave the details of the proof as an exercise to the reader. We assume that $A_1$, $A_2$, and $B$ are written in rational canonical form, and we may choose $B$ to have a single block, so that
$$B = \left[ \begin{matrix}
0 & 0 & \cdots & 0 & 1 \\
1 & 0 & \cdots & 0 & b_1 \\
0 & 1 & \cdots & 0 & b_2 \\
\vdots & \vdots & \ddots & \ddots & \vdots \\
0 & 0 & \cdots & 1 & b_{n-1}
\end{matrix} \right ]$$
It is sufficient to consider the case that $A_1$ and $A_2$ have this form as well. It follows that
$$ A\otimes B + I = \left[ \begin{matrix}
I & 0 & \cdots & 0 & A \\
A & I & \cdots & 0 & b_1 A \\
0 & A & \cdots & 0 & b_2 A \\
\vdots & \vdots & \ddots & \ddots & \vdots \\
0 & 0 & \cdots & 1 & I + b_{n-1} A
\end{matrix} \right ]$$
which row reduces to
$$\left[ \begin{matrix}
I & 0 & \cdots & 0 & M_0 \\
0 & I & \cdots & 0 & M_1 \\
0 & 0 & \cdots & 0 & M_2 \\
\vdots & \vdots & \ddots & \ddots & \vdots \\
0 & 0 & \cdots & 0 & I + M_{n-1}
\end{matrix} \right ]$$
where $M_0 = A$ and $M_i = A(b_i + M_{i-1})$ for $i>0$. The rank of this is determined by the rank of the lower right block. We choose the coefficients in $B$ so that $p(x) = x^n + b_1 x^{n-1} + \ldots + b_{n-1} x + 1$ is the minimal polynomial of $A_1$, which we assume without loss of generality does not divide the minimal polynomial of $A_2$. It follows that $A_1\otimes B + I\otimes I$ has rank $k(n-1)$ and $A_2\otimes B+I\otimes I $ has rank strictly greater than $k(n-1)$.
\end{proof}

In summary, we have shown that the map \(g\) is injective: if \(g(\bgamma_1) = g(\bgamma_2)\), then Lemma~\ref{Lem:Determine Curves} shows that \(\bgamma_1\) and \(\bgamma_2\) have the same underlying curves with multiplicities, and Proposition~\ref{prp:self-pair} shows that the local systems on each curve are isomorphic. 

We have now shown that  \(g\) is both injective and surjective. Together with Proposition~\ref{Prop:Pairing}, this completes the proof of Theorem~\ref{Thm:TheBijection}. 
\qed

 We conclude this section with the observation that, in particular, any two extensions of an extendable type D module correspond to the same decorated immersed curve and are thus homotopy equivalent.

\begin{proposition}\label{prop:unique-extension}
An extendable type D module over $\Alg$ has a unique extension up to homotopy equivalence.
\end{proposition}

\section{Bordered Floer invariants as decorated immersed curves} \label{sec:invariance}

We now return our attention to Floer theoretic invariants of manifolds with torus boundary. Our goal in this section is to prove Theorems~\ref{thm:invariance} and \ref{thm:pairing}. Suppose that \((M, \alpha, \beta)\) is a parametrized three-manifold with torus boundary. 
By Theorem \ref{thm:extension} (proved in the Appendix), the type D structure \(\CFD(M,\alpha, \beta)\) is extendable. 
The results of the preceding two sections imply that there is a well-defined collection of immersed curves with local systems associated with \(\CFD(M,\alpha,\beta)\). This collection of curves lives in the abstract torus \(T\), which we identify with 
\(T_M\) via a map that identifies the vertical edge of \(T\) with \(\alpha \subset \partial M\) and the horizontal edge of \(T\) with 
\(\beta \subset \partial M\). We denote the result by \(f_{\alpha, \beta}(\CFD(M,\alpha, \beta))\). 

It remains to check that this  collection of curves  does not depend on the choice of parametrization 
\((\alpha, \beta)\). 

\begin{proposition}\label{prp:Dehn-twists} There are equivalences 
\begin{align}
f_{\alpha,\beta}(\CFD(M,\alpha,\beta)) &\cong f_{\alpha,\beta\pm\alpha}(\CFD(M,\alpha,\beta\pm\alpha)) \label{eqn:alpha-twists}\\ 
f_{\alpha,\beta}(\CFD(M,\alpha,\beta)) &\cong f_{\alpha\pm\beta,\beta}(\CFD(M,\alpha\pm\beta,\beta)) \label{eqn:beta-twists}
\end{align} where $\cong$ denotes regular homotopy of curves with local systems.\end{proposition}

 \begin{proof} We will establish the equivalence (\ref{eqn:alpha-twists}); the equivalence (\ref{eqn:beta-twists}) is nearly identical, and is left to the reader. To this end, we recall the bimodules associated with a Dehn twist established in \cite{LOT-bimodules}. These give rise to 
 \[\widehat{T}_\alpha^{\pm 1} \boxtimes \CFD(M,\alpha,\beta) \cong \CFD(M,\alpha,\beta\pm\alpha)\]
 where $\widehat{T}_\alpha$ is the type DA bimodule denoted $\widehat{\mathit{CFDA}}(\tau_\mu)$ in the notation of \cite[Section 10]{LOT-bimodules}. This bimodule, and its inverse, are shown in Figure \ref{fig:Dehn-twist-bimodules}. Thus we need to show that $ f_{\alpha,\beta}(N) \cong f_{\alpha,\beta\pm\alpha}(\widehat{T}_\alpha^{\pm 1}\boxtimes N),$
where $N$ is the extendable type D structure $\CFD(M,\alpha,\beta)$.

\begin{figure}[ht!]
\begin{tikzpicture}[>=stealth'] 
\node at (0,0) {$\ast$}; 
\node at (-2,3) {$\bullet$}; 	
\node at (2,3) {$\circ$};
\draw[thick, shorten <=0.1cm] (2,3) arc (180:360:0.5);\draw[thick, ->] (3,3) arc (0:173:0.5); 
\draw[thick, shorten <=3pt, shorten >= 3pt, ->] (-2,3) to[bend left, looseness=0.75] (2,3); 
\draw[thick, shorten <=3pt, shorten >= 3pt, ->] (0,0) to[bend right, looseness=1] (-2,3);
\draw[thick, shorten <=3pt, shorten >= 3pt, <-] (0,0) to[bend left, looseness=1] (-2,3); 
\draw[thick, shorten <=3pt, shorten >= 3pt, <-] (0,0) to[bend right, looseness=1] (2,3);
\draw[thick, shorten <=3pt, shorten >= 3pt, ->] (0,0) to[bend left, looseness=1] (2,3); 
\node at (2.5,3.7) {$\scriptstyle \rho_{23}\otimes\rho_{23}$};
\node at (-0.3,4.25) {$\scriptstyle \rho_1\otimes \rho_1$};
	\node at (0,3.95) {$\scriptstyle +\rho_{123}\otimes \rho_{123}$};
	\node at (0.12,3.65) {$\scriptstyle +\rho_3\otimes (\rho_3,\rho_{23})$};	
\node at (-1.1,1.9) {$\scriptstyle \rho_{2}\otimes1$};
\node at (-2.1,0.8)  {$\scriptstyle \rho_{123}\otimes\rho_{12}$};\node at (-1.9,0.5)  {$\scriptstyle +\rho_{3}\otimes(\rho_{3},\rho_{2})$};
\node at (1.1,1.9) {$\scriptstyle 1\otimes\rho_{3}$};
\node at (1.9,0.8) {$\scriptstyle \rho_{23}\otimes\rho_{2}$};
\node at (7,0) {$\ast$}; 
\node at (5,3) {$\bullet$}; 	
\node at (9,3) {$\circ$};
\draw[thick, ->, shorten <=0.1cm] (5,3) arc (0:352:0.5);
\draw[thick, shorten <=0.1cm] (9,3) arc (180:360:0.5);\draw[thick, ->] (10,3) arc (0:173:0.5); 
\draw[thick, shorten <=3pt, shorten >= 3pt, ->] (5,3) to[bend left, looseness=0.75] (9,3); 
\draw[thick, shorten <=3pt, shorten >= 3pt, <-] (5,3) to (9,3);
\draw[thick, shorten <=3pt, shorten >= 3pt, ->] (7,0) to[bend right, looseness=1] (5,3);
\draw[thick, shorten <=3pt, shorten >= 8pt, <-] (7,0) to[bend left, looseness=1] (5,3); 
\draw[thick, shorten <=3pt, shorten >= 3pt, <-] (7,0) to[bend right, looseness=1] (9,3);
\draw[thick, shorten <=3pt, shorten >= 3pt, ->] (7,0) to[bend left, looseness=1] (9,3); 
\node at (4.5,3.7) {$\scriptstyle \rho_{12}\otimes(\rho_{123}\otimes\rho_2)$};
\node at (9.5,3.7) {$\scriptstyle \rho_{23}\otimes\rho_{23}$};
\node at (6.7,3.95) {$\scriptstyle \rho_1\otimes \rho_1$};
	\node at (7,3.65) {$\scriptstyle +\rho_{123}\otimes \rho_{123}$};
	\node at (7,2.75) {$\scriptstyle \rho_2\otimes (\rho_{23},\rho_{2})$};	
\filldraw [fill=white,white] (6,1.7) rectangle (7,2);\node at (6.2,1.9) {$\scriptstyle \rho_{2}\otimes(\rho_3,\rho_2)$};
\node at (4.8,0.8)  {$\scriptstyle \rho_{3}\otimes1+\rho_1\otimes\rho_{12}$};
\node at (8.15,1.9) {$\scriptstyle \rho_{23}\otimes\rho_{3}$};
\node at (8.8,0.8) {$\scriptstyle 1\otimes\rho_{2}$};
\end{tikzpicture}		
\caption{Graphical representations of the Dehn twist bimodules $\widehat{T}_\alpha$ (left) and $\widehat{T}_\alpha^{-1}$ (right), following \cite[Section 10]{LOT-bimodules}.} 	\label{fig:Dehn-twist-bimodules}
\end{figure}

We may work with one component of $\CFD(M,\alpha,\beta)$ at a time, so we will assume that $f_{\alpha,\beta}(N)$ consists of a single immersed curve $\gamma$ decorated by a local system. Let $\tracks$ be the corresponding curve-like train track, consisting of parallel copies of $\gamma$ connected by crossover arrows. Recall that the valid reduced train track $\tracks$ determines an extendable type D structure which is homotopy equivalent to $N$; since we may modify $N$ by homotopy equivalence, we will assume $\tracks$ corresponds directly to $N$. We can view the crossover arrows as all being concentrated in a single box, which slides freely along \(\gamma\). By sliding this box, we can ensure that no crossover arrow has its head or tail on a segment of curve corresponding to a \(\rho_3\) arrow. 

We claim that the effect on $\tracks$ of tensoring $N$ with $\widehat{T}_\alpha$ is easy to describe: it is applying a negative Dehn twist along $\alpha$ (and homotoping the result if necessary to achieve minimal intersection with $\beta$). Let $\tracks'$ be the train track obtained from $\tracks$ by applying a negative  Dehn twist in a thin strip along the left edge of the square $T\setminus(\alpha\cup\beta)$ (see Figure \ref{fig:dehn-twist-track}). The effect of this transformation can be described as follows:
\begin{itemize}
\item For each intersection $x$ with $\alpha$ we add a new intersection point $x'$ with $\beta$, such that $x$ and $x'$ are connected by a segment connecting the bottom and left edges of the square;
\item Each path from the bottom of the square to the left of the square ending at $x$ becomes a path from the bottom to the top ending at $x'$;
\item Each path from the right of the square to the left ending at $x$ becomes a path from the right to the top ending at $x'$;
\item Each path from the top of the square to the left ending at $x$ becomes a path from the top edge of the square to itself ending at $x'$; and
\item intersection points with $\beta$ and curve segments that do not meet the left edge of the square are uneffected.
\end{itemize}
We can attempt to read off a type D structure from this new train track in the usual way, although we note that the result may not be a valid type D structure (it may fail to satisfy $\partial^2 = 0$). In any case, the result is precisely what is obtained from $N$ by tensoring with $\widehat{T}_\alpha$, if we ignore the two operations in $\widehat{T}_\alpha$ with multiple $\Ainfty$ inputs. Indeed, referring to Figure \ref{fig:Dehn-twist-bimodules}, this has the following effect on $N$:
\begin{itemize}
\item For each $\iota_0$ generator $x$ we add a new $\iota_1$ generator $x'$ such that $x$ and $x'$ are connected by a $\rho_2$ arrow;
\item Each $\rho_2$ arrow ending at $x$ becomes a $\rho_{23}$ arrow ending at $x'$;
\item Each $\rho_{12}$ arrow ending at $x$ becomes a $\rho_{123}$ arrow ending at $x'$;
\item Each path $\rho_3$ arrow starting at $x$ becomes an arrow labelled by $1 = \rho_\varnothing$ starting at $x'$; and
\item $\iota_0$ generators, $\rho_1$, $\rho_{123}$ and $\rho_{23}$ arrows are unchanged.
\end{itemize}

To be precise, the new generator \(x'\) corresponds to \(* \boxtimes x\), while the new \(x\) (in \( \widehat{T}_\alpha \boxtimes N\)) corresponds to 
 \(\bullet \boxtimes x\).
Note that each $\rho_3$ in $N$ gives rise to a $\rho_\varnothing$ labelled arrow in the tensor product, and a corresponding bigon between $\tracks'$ and $\beta$. We will now use the assumption that $\tracks$ is curve-like and that there are no crossover arrows between segments corresponding to $\rho_3$ arrows in $N$; it follows that bigons between $\tracks'$ and $\beta$ appear in one of the three configurations shown in Figure \ref{fig:dehn-twist-track-bigons}. These three configurations arise from $\rho_3$ arrows in $\tracks$ that are followed by, respectively, backwards $\rho_1$ arrows, $\rho_{23}$ arrows, and $\rho_2$ arrows. Note that in the latter two configurations, adding the gray segments shown in the figure corresponds to incorporating the two operations of $\widehat{T}_\alpha$ that we ignored until now into the box tensor product with $N$---each ($\rho_3$, $\rho_{23}$) or $(\rho_3, \rho_2)$ sequence in $N$ gives becomes a $\rho_3$ in the tensor product. Let $\tracks''$ be the result of adding these gray lines to $\tracks'$. Note that $\tracks''$ is a valid (not necessarily reduced) train track and the corresponding type D structure is precisely $\widehat{T}_\alpha \boxtimes N$.

Now consider the effect of removing the bigons between $\tracks''$ and $\beta$. In the first configuration, removing the bigons by homotoping the parallel immersed curves has precisely the effect of canceling $\rho_\varnothing$ arrows in the underlying type D structures.  Each $\rho_1$ arrow into a generator $z$ for which there is a $\rho_\varnothing$ arrow from $x'$ to $z$ and a $\rho_2$ arrow from $x'$ to $x$ becomes a $\rho_{12}$ into $x$. In the second two configurations, the segments of $\tracks''$ forming the bigon can simply be deleted. This is because in the corresponding type D structure, the terminal endpoints of the $\rho_\varnothing$ arrows have no other incoming arrows, so the endpoints of the $\rho_\varnothing$ arrows can be canceled without adding any new arrows. In all three configurations, it is clear that the result of removing the bigons from $\tracks''$ in the appropriate way is simply the result of homotoping $\tracks'$ to be in minimal position with $\beta$.

We have shown that a curve-like train track associated with $\widehat{T}_\alpha\boxtimes N$ is obtained from $\tracks$ by applying a negative Dehn twist about $\alpha$. Thus, interpreting curve-like train tracks as decorated immersed curves, we have that $f_{\alpha,\beta}(\widehat{T}_\alpha\boxtimes N)$ is obtained from $f(N)$ by applying the same Dehn twist to the underlying curve $\gamma$ of $f(N)$. It follows that  $f_{\alpha,\beta}(N) \cong f_{\alpha,\beta+\alpha}( \widehat{T}_\alpha\boxtimes N)$, since the change of marking cancels the action of the Dehn twist.

A similar argument applies to the bimodule $\widehat{T}_\alpha^{-1}$, where tensoring by this bimodule has the effect of applying a positive Dehn twist about $\alpha$ to the corresponding train track and to $f(N)$. However, it is easier to  observe that $f_{\alpha,\beta}(N) = f_{\alpha,\beta}(\widehat{T}_\alpha \boxtimes \widehat{T}_\alpha^{-1}\boxtimes N)$ is obtained from $f_{\alpha,\beta}(\widehat{T}_\alpha^{-1}\boxtimes N)$ by applying a negative Dehn twist, so $f_{\alpha,\beta}(\widehat{T}_\alpha^{-1}\boxtimes N)$ is obtained from $f_{\alpha,\beta}(N)$ by applying a positive Dehn twist. Again, we see that $f_{\alpha,\beta}(N) \cong f_{\alpha,\beta-\alpha}( \widehat{T}_\alpha^{-1}\boxtimes N)$, since the change of marking cancels the action of the Dehn twist. \end{proof}

\begin{figure}
\includegraphics[scale=.7]{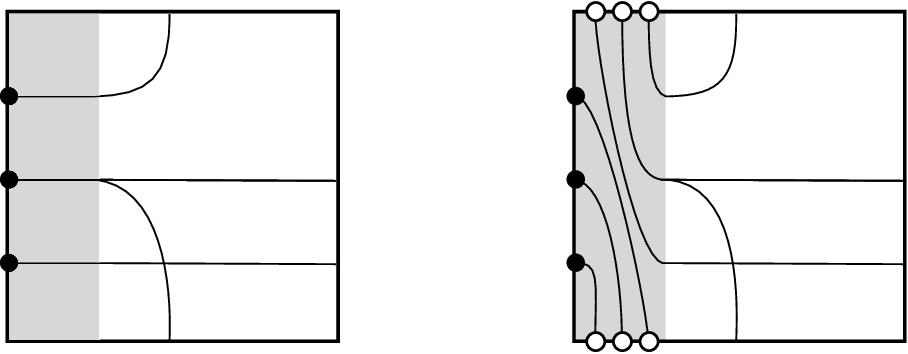}
\caption{Applying a Dehn twist to a train track}
\label{fig:dehn-twist-track}
\end{figure}

\begin{figure}
\includegraphics[scale=.5]{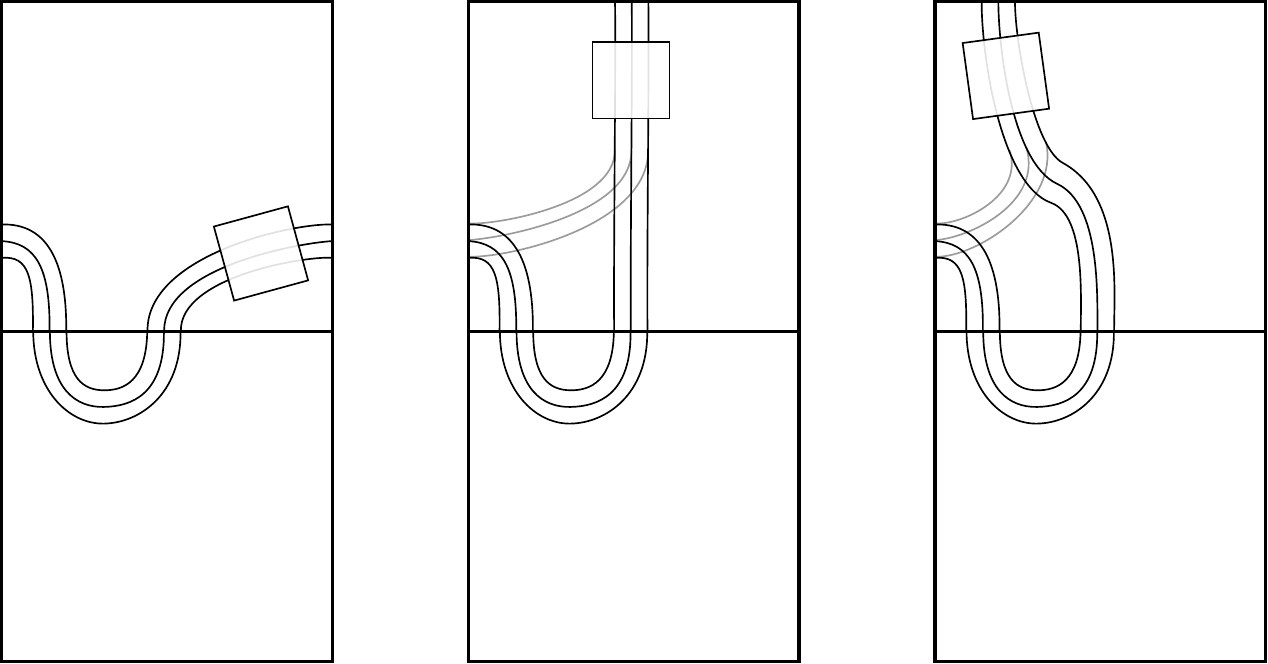}
\caption{Possible configurations of bigons in $\tracks'$ corresponding to $\rho_3$ arrows in $N$ that are followed by backward $\rho_1$ arrows, $\rho_{23}$ arrows, and $\rho_2$ arrows, respectively. The assumption that $\tracks$ has no crossover arrows between segments corresponding to $\rho_3$ arrows ensures that the bigons do not involve any crossover arrows; there could be crossover arrows in the boxes shown. The black segments are part of $\tracks'$ and $\tracks''$ is obtained from $\tracks'$ by adding the gray segments in the second and third configuration.}
\label{fig:dehn-twist-track-bigons}
\end{figure}

Since we can move between any two markings/bordered structures by a sequence of Dehn twists, Proposition \ref{prp:Dehn-twists} ensures that
$$f_{\alpha,\beta}(\CFD(M,\alpha,\beta)) \cong f_{\alpha',\beta'\;}(\CFD(M,\alpha',\beta'))$$
for any two pairs of parametrizing curves $(\alpha, \beta)$ and $(\alpha', \beta')$. 

We will write
\[\HFhat(M) = f_{\alpha,\beta}\left(\CFD(M,\alpha,\beta)\right).\]
As we have shown, this definition is independent of the choice of parametrization \((\alpha, \beta)\). 
This completes the proof of Theorem \ref{thm:invariance}.

We now turn to the proof of Theorem \ref{thm:pairing}. Recalling the set-up of the theorem, suppose $M_0$ and $M_1$ are manifolds with torus boundary and $h\co\partial M_1\to \partial M_2$ is a gluing map. Chose a parametrization $(\alpha_1, \beta_1)$ for $\partial M_1$ and fix the parametrization $(\alpha_0,\beta_0)$ of $\partial M_0$ by setting $\alpha_0 = h(\beta_1)$ and $\beta_0 = h(\alpha_1)$. Following the conventions for bordered Heegaard Floer invariants (see Section \ref{sec:conventions}), these parametrizations are consistent with the gluing map $h$; that is, by the pairing theorem in \cite{LOT}, $\HFhat(M_0\cup_h M_1)$ is given by the homology of $\CFA(M_0,\alpha_0,\beta_0)\boxtimes\CFD(M_1,\alpha_1,\beta_1)$. By the equivalence of categories proved Section \ref{sec:categories}, the homology of this box tensor product agrees with the pairing of the decorated immersed curves $f(\CFD(M_0,\alpha_0,\beta_0))$ and $f(\CFA(M_1,\alpha_1, \beta_1))$. To pair these curves we include them in the same marked torus, reflecting one across the anti-diagonal, and take intersection Floer homology. Note that as 3-manifold invariants, we think of $\curves{M_0} = f(\CFD(M_0,\alpha_0,\beta_0))$ and $\curves{M_1} = f(\CFA(M_1,\alpha_1, \beta_1))$ as living in two different marked tori, the parametrized boundaries of $M_0$ and $M_1$. Identifying these marked tori with a reflection across the anti-diagonal corresponds to a gluing map which takes $\alpha_1$ to $-\beta_0$ and $\beta_1$ to $\-\alpha_0$; this is precisely the map $\bar{h}$, the elliptic involution composed with $h$. Letting $\boldsymbol{\gamma_0}$ denote $\curves{M_0}$ and $\boldsymbol{\gamma_1}$ denote the image of $\curves{M_1}$ under $\bar{h}$, the pairing is given by $\HFa(\boldsymbol{\gamma_0}, \boldsymbol{\gamma_1})$, as claimed.

\begin{remark}\label{rmk:symmetry and pairing}
Theorem \ref{thm:pairing}, which follows from our graphical representation of type D structures in terms of train tracks (Section \ref{sec:tracks}) and the structure theorem for these tracks in terms of immersed curves with local systems (Section \ref{sec:extend}), inherits some features from the box tensor product in bordered Floer homology.  For instance, the appearance of $\bar h$ instead of $h$ was initially a surprise to us. Of course, if \(\curves{M_1}\) is invariant under the elliptic involution, there is no difference between using \(h\) and \(\bar h\).  This symmetry is known to hold  when $M_1$ is a  knot complement in $S^3$ by  a result of Xiu \cite{Xiu}, and when $M_1$ is a graph manifold as discussed in \cite{HW}. In the companion to this paper \cite{HRW-companion}, we show that it holds for all manifolds with torus boundary. 

\end{remark}

\section{Gradings}
\label{sec:gradings} The bordered invariants of $M$ come equipped with a mod 2 and a $\spinc$ grading. In this section, we explain how the invariant $\HFhat(M)$ can be enhanced to capture this information, leading to a proof of Theorem~\ref{thm:gradings}. While this does not capture all of the graded information in the bordered Heegaard Floer package, it is sufficient for the applications in this paper; a more thorough discussion of gradings, including the Maslov grading, appears in the companion to this paper \cite{HRW-companion}.

\subsection{The $\Z/2\Z$ grading.}
We first recall the \(\Z/2\Z\) grading in bordered Floer homology. 
For each spin$^c$ structure $\spin$, $\CFD(M,\alpha,\beta; \spin)$ admits a relative $\Z/2\Z$ grading $\gr^D$ as defined by Petkova in \cite{Petkova-decategorification}. (This grading may be identified with a specialization of the full grading package on the bordered invariants; see \cite[Appendix A]{HLW}.) The relative $\Z/2\Z$ grading satisfies $\gr^D(\partial x) = \gr^D(x) - 1$ and $\gr^D(a \otimes x) = \gr^D(a)+\gr^D(x)$ for $a\in\Alg$, where $\gr^D(\rho_1) = \gr^D(\rho_3) = 0$ and $\gr^D(\rho_2) = \gr^D(\rho_{123}) = \gr^D(\rho_{12}) = \gr^D(\rho_{23}) = 1$. Note that if $\CFD(M,\alpha,\beta;\spin)$ is connected (see below), the relative $\Ztwo$ grading is completely determined by this condition. The generators of $\CFA(M,\alpha,\beta;\spin)$ inherit a grading $\gr^A$ from the corresponding generators in $\CFD(M,\alpha,\beta;\spin)$, where the grading of generators with idempotent $\iota_0$ is reversed. A generator $x_0\otimes x_1$ in a box tensor product $\CFA(M_0,\alpha_0,\beta_0)\boxtimes\CFD(M_1,\alpha_1,\beta_1)$ inherits the grading $\gr^A(x_0) + \gr^D(x_1)$, which recovers the relative $\Ztwo$ grading on $\CFhat(M_0\cup M_1)$.

 The graph representing a $\Z/2\Z$ graded type D structure can be enhanced, replacing the vertex labeling $\{\bu,\circ\}$ with $\{\bullet^+, \bullet^-, \circ^+,\circ^-\}$, where $+$ designates \(\gr^A\) grading 0 and $-$ designates \(\gr^A\) grading 1. Referring to the conventions above, we see that an edge labeled \(2\) or \({123}\) joins two vertices with opposite sign, while an edge with any other label joins two vertices with the same sign. 

\piccaption[]{Vertex types for $\sA$-decorated graphs. \label{fig:vertex-types}}
\parpic[r]{
	\begin{minipage}{60mm}{
			\begin{tikzpicture}[>=stealth',scale=0.8] 
			\node at (0,0) {$\bullet$}; 
			\node at (0,1) {$\bullet$}; 
			\node at (0,2) {$\bullet$};
			\draw[thick, ->,shorten <=0.1cm] (0,0) -- (1,0); \node at (0.5,0.25) {$\scriptstyle{123}$};
			\draw[thick, ->, shorten <=0.1cm] (0,1) -- (1,1); \node at (0.5,1.25) {$\scriptstyle{12}$};
			\draw[thick, ->, shorten <=0.1cm] (0,2) -- (1,2); \node at (0.5,2.25) {$\scriptstyle{1}$}; \node at (0.5,-0.75){\bf{I}$_\bullet$};
			\node at (2,0) {$\bullet$}; 
			\node at (2,1) {$\bullet$}; 
			\node at (2,2) {$\bullet$};
			\draw[thick, <-,shorten <=0.1cm] (2,0) -- (3,0); \node at (2.5,0.25) {$\scriptstyle{12}$};
			\draw[thick, <-, shorten <=0.1cm] (2,1) -- (3,1); \node at (2.5,1.25) {$\scriptstyle{2}$};
			\draw[thick, ->, shorten <=0.1cm] (2,2) -- (3,2); \node at (2.5,2.25) {$\scriptstyle{3}$}; \node at (2.5,-0.75){\bf{II}$_\bullet$};
			\node at (4,0) {$\circ$}; 
			\node at (4,1) {$\circ$}; 
			\node at (4,2) {$\circ$};
			\draw[thick, <-,shorten <=0.1cm] (4,0) -- (5,0); \node at (4.5,0.25) {$\scriptstyle{1}$};
			\draw[thick, ->, shorten <=0.1cm] (4,1) -- (5,1); \node at (4.5,1.25) {$\scriptstyle{23}$};
			\draw[thick, ->, shorten <=0.1cm] (4,2) -- (5,2); \node at (4.5,2.25) {$\scriptstyle{2}$}; \node at (4.5,-0.75){\bf{I}$_\circ$};
			\node at (6,0) {$\circ$}; 
			\node at (6,1) {$\circ$}; 
			\node at (6,2) {$\circ$};
			\draw[thick, <-,shorten <=0.1cm] (6,0) -- (7,0); \node at (6.5,0.25) {$\scriptstyle{123}$};
			\draw[thick, <-, shorten <=0.1cm] (6,1) -- (7,1); \node at (6.5,1.25) {$\scriptstyle{23}$};
			\draw[thick, <-, shorten <=0.1cm] (6,2) -- (7,2); \node at (6.5,2.25) {$\scriptstyle{3}$}; \node at (6.5,-0.75){\bf{II}$_\circ$};
			\end{tikzpicture}} 
	\end{minipage}
}
We first observe that this data is equivalent to a choice of orientation on each edge of the underlying (undecorated) graph. Given a vertex \(v\), we partition the edges incident to \(v\) into two types, as shown in Figure \ref{fig:vertex-types}. That is, if \(v\) has label \(\bu\), incident edges are either of type {\bf I}$_\bu$ or type  {\bf II}$_\bu$, and similarly if \(v\) has label \(\circ\). This is a natural partition to consider in the context of extendable type D structures:

\begin{proposition}\label{prp:type-incedent} Suppose that $\Gamma$ is an $\sA$-decorated graph associated with a reduced extendable type D structure $N$. Then every vertex $\bu$ has at least one incident edge of type {\bf I}$_\bu$ and at least one  incident edge of type {\bf II}$_\bu$, and every vertex $\circ$ has at least one incident edge of type {\bf I}$_\circ$ and at least  one incident edge of type {\bf II}$_\circ$.
\end{proposition}

\begin{proof}This is essentially \cite[Proposition 3.3]{HW}.\end{proof}

In particular, when $\Gamma$ has valence 2 (and hence the associated train track is an immersed curve), there is exactly one of each edge type incident at every vertex.

This partition allows us to convert the system of signed vertices into an orientation on the underlying graph. The orientations on edges incident to any given vertex are assigned following the rules in Table \ref{tab:orient}. See Figure \ref{fig:oriented-graph} for this process carried out in an example.  Note that reversing the overall orientation on this graph corresponds to switching the grading of every vertex, and vice versa. For consistency with the gradings on the relevant bordered invariants, if an orientation in a given component is to be changed one must reverse orientations on every connected component corresponding to the relevant summand $\CFD(M,\alpha,\beta;\spin)$. 

\begin{table}[t]
	\begin{tabular}{|c|c|c|c|c|c|}
		\hline
		\rowcolor{lightgray} Edge type & Vertex sign  & Orientation   &Edge type & Vertex sign  & Orientation    \\
		\hline
		{\bf I}$_\bu$ &
		$+$ & 
		\begin{tikzpicture}[thick, decoration={
			markings,
			mark=at position 0.5 with {\arrow{>}}}] \draw[-,postaction={decorate}] (1,0) -- (0,0);\node at (-0.05,0) {$\bu$}; 
		\end{tikzpicture}
		&  {\bf I}$_\circ$ &
		$+$ & 
		\begin{tikzpicture}[thick, decoration={
			markings,
			mark=at position 0.5 with {\arrow{>}}}] \draw[-,postaction={decorate}] (0,0) -- (1,0);\node at (-0.05,0) {$\circ$}; 
		\end{tikzpicture}
		\\
		{\bf I}$_\bu$ &
		$-$ & 
		\begin{tikzpicture}[thick, decoration={
			markings,
			mark=at position 0.5 with {\arrow{>}}}] \draw[-,postaction={decorate}] (0,0) -- (1,0);\node at (-0.05,0) {$\bu$}; 
		\end{tikzpicture}
		&  {\bf I}$_\circ$ &
		$-$ & 
		\begin{tikzpicture}[thick, decoration={
			markings,
			mark=at position 0.5 with {\arrow{>}}}] \draw[-,postaction={decorate}] (1,0) -- (0,0);\node at (-0.05,0) {$\circ$}; 
		\end{tikzpicture}
		\\
		{\bf II}$_\bu$ &
		$+$ & 
		\begin{tikzpicture}[thick, decoration={
			markings,
			mark=at position 0.5 with {\arrow{>}}}] \draw[-,postaction={decorate}] (0,0) -- (1,0);\node at (-0.05,0) {$\bu$}; 
		\end{tikzpicture}
		&  {\bf II}$_\circ$ &
		$+$ & 
		\begin{tikzpicture}[thick, decoration={
			markings,
			mark=at position 0.5 with {\arrow{>}}}] \draw[-,postaction={decorate}] (1,0) -- (0,0);\node at (-0.05,0) {$\circ$}; 
		\end{tikzpicture}
		\\
		{\bf II}$_\bu$ &
		$-$ & 
		\begin{tikzpicture}[thick, decoration={
			markings,
			mark=at position 0.5 with {\arrow{>}}}] \draw[-,postaction={decorate}] (1,0) -- (0,0);\node at (-0.05,0) {$\bu$}; 
		\end{tikzpicture}
		&  {\bf II}$_\circ$ &
		$-$ & 
		\begin{tikzpicture}[thick, decoration={
			markings,
			mark=at position 0.5 with {\arrow{>}}}] \draw[-,postaction={decorate}] (0,0) -- (1,0);\node at (-0.05,0) {$\circ$}; 
		\end{tikzpicture}\\
		\hline
	\end{tabular}
	\vskip0.1in
	\caption{Producing the oriented graph given a signed $\sA$-decorated graph $\Gamma$ and the vertex types from Figure \ref{fig:vertex-types}.}\label{tab:orient}
\end{table}

\begin{figure}[ht!]
	\begin{tikzpicture}[scale=0.75,>=stealth', thick] 
	\def \radius {3.5cm} \def \outer {3.9cm}
	\node (d) at ({360/(11) * (1- 1)}:\radius) {$\bu$};
	\node (j) at ({360/(11) * (2- 1)}:\radius) {$\ci$};
	\node (a) at ({360/(11) * (3 -1)}:\radius) {$\bu$};
	\node (k) at ({360/(11) * (4 -1)}:\radius) {$\ci$};
	\node (f) at ({360/(11) * (5 -1)}:\radius) {$\bu$};
	\node (g) at ({360/(11) * (6 -1)}:\radius) {$\ci$};
	\node (h) at ({360/(11) * (7 -1)}:\radius) {$\ci$};
	\node (c) at ({360/(11) * (8 -1)}:\radius) {$\bu$};
	\node (e) at ({360/(11) * (9 -1)}:\radius) {$\bu$};
	\node (i) at ({360/(11) * (10 -1)}:\radius) {$\ci$};
	\node (b) at ({360/(11) * (11 -1)}:\radius) {$\bu$};

	\draw[->, bend right = 15] (d) to (j);
	\draw[->, bend left = 15] (a) to (j);
	\draw[->, bend right = 15] (a) to (k);
	\draw[->, bend right = 15] (k) to (f);
	\draw[->, bend right = 15] (f) to (g);
	\draw[->, bend left = 15] (h) to (g);
	\draw[->, bend left = 15] (c) to (h);
	\draw[->, bend right = 15] (c) to (e);
	\draw[->, bend right = 15] (e) to (i);
	\draw[->, bend left = 15] (b) to (i);
	\draw[->, bend right = 15] (b) to (d);
	
	\draw[->, bend right = 15] (d) to node[right]{$\scriptstyle{1}$} (i);
	\draw[->, bend right = 30] (d) to node[right]{$\scriptstyle{12}$} (e);
	\draw[->, bend right = 10] (d) to node[above]{$\scriptstyle{1}$} (h);
	\draw[->, bend right = 10] (j) to node[above]{$\scriptstyle{23}$} (g);

	\node at ({360/(11) * (1- (1/2))}:\outer) {$\scriptstyle{1}$};
	\node at ({360/(11) * (2-(1/2))}:\outer) {$\scriptstyle{3}$};
	\node at ({360/(11) * (3-(1/2))}:\outer) {$\scriptstyle{123}$};
	\node at ({360/(11) * (4-(1/2))}:\outer) {$\scriptstyle{2}$};
	\node at ({360/(11) * (5-(1/2))}:\outer) {$\scriptstyle{1}$};
	\node at ({360/(11) * (6-(1/2))}:\outer) {$\scriptstyle{23}$};
	\node at ({360/(11) * (7-(1/2))}:\outer) {$\scriptstyle{3}$};
	\node at ({360/(11) * (8-(1/2))}:\outer) {$\scriptstyle{12}$};
	\node at ({360/(11) * (9-(1/2))}:\outer) {$\scriptstyle{1}$};
	\node at ({360/(11) * (10-(1/2))}:\outer) {$\scriptstyle{3}$};
	\node at ({360/(11) * (11-(1/2))}:\outer) {$\scriptstyle{12}$};
	
	\node at ({360/(11) * (0)}:\outer) {$\footnotesize{+}$};
	\node at ({360/(11) * (1)}:\outer) {$\footnotesize{+}$};
	\node at ({360/(11) * (2)}:\outer) {$\footnotesize{+}$};
	\node at ({360/(11) * (3)}:\outer) {$\footnotesize{-}$};
	\node at ({360/(11) * (4)}:\outer) {$\footnotesize{+}$};
	\node at ({360/(11) * (5)}:\outer) {$\footnotesize{+}$};
	\node at ({360/(11) * (6)}:\outer) {$\footnotesize{+}$};
	\node at ({360/(11) * (7)}:\outer) {$\footnotesize{+}$};
	\node at ({360/(11) * (8)}:\outer) {$\footnotesize{+}$};
	\node at ({360/(11) * (9)}:\outer) {$\footnotesize{+}$};
	\node at ({360/(11) * (10)}:\outer) {$\footnotesize{+}$};

	\end{tikzpicture} \hspace{1 cm}
	\begin{tikzpicture}[scale=0.75, thick,decoration={
		markings,
		mark=at position 0.5 with {\arrow{>}}}] 
	\def \radius {3.5cm} \def \outer {3.9cm}
	\node (d) at ({360/(11) * (1- 1)}:\radius) {$\bu$};
	\node (j) at ({360/(11) * (2- 1)}:\radius) {$\ci$};
	\node (a) at ({360/(11) * (3 -1)}:\radius) {$\bu$};
	\node (k) at ({360/(11) * (4 -1)}:\radius) {$\ci$};
	\node (f) at ({360/(11) * (5 -1)}:\radius) {$\bu$};
	\node (g) at ({360/(11) * (6 -1)}:\radius) {$\ci$};
	\node (h) at ({360/(11) * (7 -1)}:\radius) {$\ci$};
	\node (c) at ({360/(11) * (8 -1)}:\radius) {$\bu$};
	\node (e) at ({360/(11) * (9 -1)}:\radius) {$\bu$};
	\node (i) at ({360/(11) * (10 -1)}:\radius) {$\ci$};
	\node (b) at ({360/(11) * (11 -1)}:\radius) {$\bu$};
	
	\draw[-, bend left = 15,postaction={decorate}] (j) to (d);
	\draw[-, bend left = 15,postaction={decorate}] (a) to (j);
	\draw[-, bend left = 15,postaction={decorate}] (k) to (a);
	\draw[-, bend left = 15,postaction={decorate}] (f) to (k);
	\draw[-, bend left = 15,postaction={decorate}] (g) to (f);
	\draw[-, bend left = 15,postaction={decorate}] (h) to (g);
	\draw[-, bend left = 15,postaction={decorate}] (c) to (h);
	\draw[-, bend left = 15,postaction={decorate}] (e) to (c);
	\draw[-, bend left = 15,postaction={decorate}] (i) to (e);
	\draw[-, bend left = 15,postaction={decorate}] (b) to (i);
	\draw[-, bend left = 15,postaction={decorate}] (d) to (b);
	
	\draw[-, bend left = 15,postaction={decorate}] (i) to (d);
	\draw[-, bend left = 30,postaction={decorate}] (e) to (d);
	\draw[-, bend left = 10,postaction={decorate}] (h) to (d);
	\draw[-, bend right = 10,postaction={decorate}] (j) to (g);
	
	\end{tikzpicture} 
	\caption{A $\Z/2\Z$ augmentation of an $\sA$-decorated graph (left), and the associated orientation graph (right).}\label{fig:oriented-graph}
\end{figure}

\begin{lemma}
	The conventions of Table~\ref{tab:orient} define a consistent orientation on each edge of the graph. 
\end{lemma}
\begin{proof}
	It is enough to check that edges labelled $2$ and $123$ change the grading and that all other labelings preserve it.
	It is easy to see from the table that the two ends of an edge have the same mod $2$ grading if they are both of the same type ({\it i.e.} both {\bf I} or both {\bf II}). Referring to Figure~\ref{fig:vertex-types}, we see that edges labeled $2$ and $123$ have one end of type 
	{\bf I} and one of type {\bf II}; other edges have ends of the same type. 
\end{proof}

\begin{remark} The condition from Proposition \ref{prp:type-incedent} ensures that every vertex of the graph has both an inward pointing and an outward pointing edge. If the graph is a loop corresponding to an immersed curve \(\gamma\) in \(T_M\), it follows that the \(\Z/2\Z\) grading is completely determined by the choice of an orientation on \(\gamma\). More generally, if the graph corresponds to an immersed curve with a connected local system, all the parallel curves in the local system must carry the same orientation. Thus to specify the \(\Z/2\Z\) grading it is again enough to specify an orientation on the underlying immersed curve. 
 \end{remark}

\subsection{The \(\spinc\) grading} Now we review the \(\spinc\) grading on bordered Floer homology. Recall that  each generator \(x\) of \(\CFD(M,\alpha,\beta)\) has an associated spin$^c$ structure \(\spin(x) \in \spinc(M)\). The elements of \(\spinc(M)\) are homology classes of nonvanishing vector fields on \(M\), and \(\spinc(M)\) has the structure of an affine set modeled on \(H^2(M) \cong H_1(M, \partial M)\) ({\it i.e.} an \(H^2(M)\)--torsor.) The same decomposition holds for $\CFA(M,\alpha,\beta)$, so that 
\[\CFD(M,\alpha,\beta ) = \bigoplus_{\spin \in \spinc(M)} \CFD(M,\alpha,\beta;\spin)\quad\text{and}\quad\CFA(M,\alpha,\beta ) = \bigoplus_{\spin \in \spinc(M)} \CFA(M,\alpha,\beta;\spin).\]
Restricting attention to the generators in a particular idempotent \(\iota\), we can also define a refined spin$^c$ grading \(\spinbar_\iota(x)\), which lives in an affine set \(\spinc(M, \iota)\) modeled on \(H^2(M, \partial M) \cong H_1(M)\).  Elements of \(\spinc(M,\iota)\) are homology classes of nonvanishing vector fields with prescribed behavior on \(\partial M\), and \(\spin(x)\) is the image of \(\spinbar_\iota(x)\) in \(\spinc(M)\).

To compare the refined gradings of two generators, we adopt the following. Given \(\spin \in \spinc(M)\), let 
\[ \spinc(M, \iota, \spin) = \{ \spinbar \in \spinc(M, \iota) \, | \, \spinbar = \spin \ \text{in} \ \spinc(M)\}.\]
If \(j_*\co H_1(\partial M) \to H_1(M)\) is the map induced by inclusion, 
\(\spinc(M, \iota, \spin)\) is an affine set modeled on \(H_M = \im j_* \cong H_1(\partial M)/\ker j_* \). 
We define \(\spinc(M,\spin) = \spinc(M, \iota_0, \spin) \cup 
\spinc(M, \iota_1, \spin)\). If \(x\) is a generator of \(\CFD(M,\alpha, \beta, \spin)\) in idempotent \(\iota\), we write \(\spinbar(x) = \spinbar_\iota(x) \in \spinc(M,\spin)\). 

Let \(2S \subset H_1(\partial M)\) be the kernel of the linear map 
\(H_1(\partial M) \to \Z/2\Z\) given by \(\alpha\mapsto 1\), \(\beta \mapsto 1\), and consider the lattices \(\frac {1}{2}H_1(M) \) and  \(S = \frac{1}{2} (2S) \) inside \(H_1(\partial M;\R)\). We have inclusions 
\(H_1(M) \subset  S \subset \frac{1}{2}H_1(M)\), and each lattice is index two in the next larger one. Let \(S_M\) be the  image of \(S\) in \(H_1(\partial M; \R)/\ker j_*\) so we have inclusions \( H_M \subset S_M \subset \frac{1}{2}H_M\), where each subgroup is index two in the next larger one. (Note that \(S\) and \(S_M\) depend on \(\alpha\) and \(\beta\), although this is not reflected in the notation.) 
 
\begin{table}[t]
	\begin{tabular}{|c|c|c|c|}
		\hline
		\rowcolor{lightgray} Labeled edge  & \(\spinbar(y)-\spinbar(x)\)   & Labeled edge & \(\spinbar(y)-\spinbar(x)\)   \\
		\hline
		\begin{tikzpicture}[thick, >=stealth'] \draw[->] (0,0) -- (1,0);\node [above] at (0.5,0) {$\scriptstyle{1}$};
		\node [left] at (0,0) {$x$}; \node [right] at (1,0) {$y$};\end{tikzpicture}
		& $-(\alpha+\beta)/2$  
		&\begin{tikzpicture}[thick, >=stealth'] \draw[->] (0,0) -- (1,0);\node [above] at (0.5,0) {$\scriptstyle{12}$};
		\node [left] at (0,0) {$x$}; \node [right] at (1,0) {$y$};\end{tikzpicture}
		& $-\beta$ \\
		\begin{tikzpicture}[thick, >=stealth'] \draw[->] (0,0) -- (1,0);\node [above] at (0.5,0) {$\scriptstyle{2}$};
		\node [left] at (0,0) {$x$}; \node [right] at (1,0) {$y$};\end{tikzpicture}& $(\alpha -\beta)/2$ 
		& \begin{tikzpicture}[thick, >=stealth'] \draw[->] (0,0) -- (1,0);\node [above] at (0.5,0) {$\scriptstyle{23}$};
		\node [left] at (0,0) {$x$}; \node [right] at (1,0) {$y$};\end{tikzpicture}
		& $\alpha$\\
		\begin{tikzpicture}[thick, >=stealth'] \draw[->] (0,0) -- (1,0);\node [above] at (0.5,0) {$\scriptstyle{3}$};
		\node [left] at (0,0) {$x$}; \node [right] at (1,0) {$y$};\end{tikzpicture}
		&  $(\alpha+\beta)/2$ 
		& \begin{tikzpicture}[thick, >=stealth'] \draw[->] (0,0) -- (1,0);\node [above] at (0.5,0) {$\scriptstyle{123}$};
		\node [left] at (0,0) {$x$}; \node [right] at (1,0) {$y$};\end{tikzpicture}
		&  $(\alpha-\beta)/2$ \\
		\hline
	\end{tabular}
	\vskip0.1in
	\caption{Grading shifts in $\CFD(M,\alpha,\beta)$ associated with labelled edges.}\label{tab:spinc}
\end{table}

\begin{lemma}
	\label{Lem:spindiff} \(\spinc(M,\spin)\) can be given the structure of an \(S_M\)--torsor such that (a) the action of \(H_M \subset S_M\) agrees with the pre-existing action of \(H_M\) on \(\spinc(M,\spin)\) and (b) if \(x, y\) are generators of \(\CFD(M, \alpha, \beta, \spin)\) that are joined by an arrow, then the difference \(\spinbar(y) - \spinbar(x)\) is given by Table~\ref{tab:spinc}. 
\end{lemma}

\begin{proof}
	This is  a rephrasing of \cite[Lemma 3.9]{RR}; compare \cite[Lemma 11.42]{LOT}. Specifically, the action of \(H_M\) on \(S_M\) has two cosets \(H\) and \(H'\). Each coset is an \(H_M\)--torsor, and  \(H' = H - (\alpha+\beta)/2\). Choose some \(H_M\)-equivariant identification \(\varphi\co\spinc(M, \iota_0,\spin) \to H\). For \(\spinbar \in \spinc(M, \iota_1,\spin)\)  define  \(\varphi (\spinbar) = \varphi (i^{-1}(\spinbar))-(\alpha + \beta)/2\), where \(i\co\spinc(M, \iota_0,\spin) \to  \spinc(M, \iota_1,\spin)\) is the \(H_M\)-equivariant map defined in \cite[Lemma 3.9]{RR}. Then \(\varphi(\spinc(M,\iota_1,\spin)) = H'\), so \(\varphi\) defines an \(H_M\) equivariant bijection between \(\spinc(M,\spin)\) and \(S_M\). This gives \(\spinc(M,\spin)\) the structure of an \(S_M\) torsor satisfying property (a). Property (b) follows from the corresponding list of grading shifts in \cite[Lemma 3.9]{RR}. 
\end{proof}

As in the introduction, we define \(p\co\barT_M \to T_M\) to be the covering map defined by the condition that \(\pi_1(\barT_M) \)  is the kernel of the composite map
$ \pi_1(T_M) \to H_1(T_M) \to H_1(\partial M) \to H_1(M).$ Equivalently, if 
\(\ell\) generates the kernel of \(j_*\co H_1(\partial M;\Z) \to H_1(M;\Z)\), then 
\(\barT_M\) is the quotient of \(H_1(\partial M;\R) \setminus H_1(\partial M;\Z)\) by the action of \(\ell\). The group of deck transformations on \(\barT_M\) is  isomorphic to \(H_M\). The class \(\ell \in H_1(\partial_M;\Z)\) is the {\em homological longitude}. It is an integer multiple of a primitive class \(\lambda \in H_1(\partial M; \Z)\), which is called the \emph{rational longitude}. 

Let \(q_\alpha, q_\beta \in T_M\) be the midpoints of the arcs \(\alpha\) and \(\beta\), and define \(S_\alpha = p^{-1}(q_\alpha) \subset \barT_M\), \(S_\beta = p^{-1}(q_\beta) \subset \barT_M.\)  Then \(S_\alpha \cup S_\beta\) is naturally an \(S_M\)--torsor and the subsets \(S_\alpha\) and \(S_\beta\) are identified with the \(H_M\)--cosets \(H\) and \(H'\). By Lemma~\ref{Lem:spindiff}, the \(S_M\)--torsor \(\spinc(M, \spin)\) can be identified with \(S_\alpha \cup S_\beta\) (as \(S_M\)--torsors) in such a way that \(\spinc(M,\iota_0,\spin)\) maps to \(S_\alpha\) and \(\spinc(M,\iota_1,\spin)\) maps to \(S_\beta\).  Equivalently, by identifying an arc with its midpoint, we can identify \(\spinc(M,\spin)\) with the set of lifts of \(\alpha\) and \(\beta\) to \(\barT_M\).  We fix one such identification and write \(\alpha_{\spinbar}\) or \(\beta_{\spinbar}\) for the lift of \(\alpha\) or \(\beta\) corresponding to \(\spinbar\). Note that any other such identification differs from our fixed one by the action of an element of \(H_M\).

In Section \ref{sec:tracks} we explained how to assign a train track \(\tracks(M,\alpha, \beta)\) in \(T_M \) to the type D structure  \(\CFD(M,\alpha,\beta)\). Since the differential on \(\CFD(M,\alpha,\beta)\) respects the decomposition into $\spinc$ structures, \(\tracks(M,\alpha,\beta)\) decomposes as a union of train tracks \(\tracks(M,\alpha,\beta;\spin)\), where \(\spin\) runs over 
\(\spinc(M)\).

\begin{proposition}
	\(\tracks(M,\alpha, \beta;\spin)\) lifts to a unique train track \(\bartracks(M,\alpha, \beta;\spin)\) in \(\barT_M\) with the property that if \(x\) is a switch of   \(\tracks(M,\alpha, \beta;\spin)\) lying on \(\alpha\) (resp. \(\beta\)) its lift lies on \(\alpha_{\spinbar(x)}\) (resp. \(\beta_{\spinbar(x)}\)). 	\end{proposition}

\begin{proof}
	We lift each switch \(x\) on \(\alpha\) (resp. \(\beta\)) to its preimage on  \(\alpha_{\spinbar(x)}\) (resp. \(\beta_{\spinbar(x)}\)). It remains to check that we can lift each section of track, which corresponds to some arrow in \(\CFD(M, \alpha, \beta)\), in a way that is compatible with the lifts of its endpoints.  Comparing the grading shifts in Table~\ref{tab:spinc} with Figure~\ref{fig:algebra-edges}, we see that this is indeed the case. Uniqueness of \(\bartracks(M,\alpha, \beta;\spin)\) is clear, since the given condition determines the positions of all the switches in the lift. 
\end{proof}

Finally, note that if  we had chosen a different identification of \(\spinc(M,\spin)\) with 
\(S_\alpha \cup S_\beta\), the resulting lift of \(\tracks(M,\alpha, \beta,\spin)\) would differ from the one constructed above by the action of an element of \(H_M\).

\begin{proof}[Proof of Theorem \ref{thm:gradings}]
	Orient the track \(\tracks(M, \alpha, \beta;\spin)\)  and lift it to a collection of oriented tracks 
	\(\bartracks(M, \alpha, \beta; \spin)\) in \(\Tbar_M\). The result is a compact train track in the cover \(\Tbar_M\), which we now simplify as in Section~\ref{sec:extend}. In fact, we can simplify \(\bartracks(M, \alpha, \beta;\spin)\) using exactly the same moves as we used to simplify \(\tracks(M, \alpha, \beta)\) down in \(T_M\). Each move we perform involves some arcs and/or crossover arrows that are supported in the 0-handle of the downstairs track. It is easy to see from the definition that all of the arcs involved lift up to a single 0-handle in the cover, so the move can be performed on the cover as well. 
	
	It is also easy to see that these moves respect the relative \(\Z/2\Z\) grading. For example, consider move (M1), which eliminates 
	a clockwise running crossover arrow which runs from \(x\) to \(y\), as in Proposition~\ref{prp:moves}. Let \(z\) be the generator connected to \(y\) by the two-way arc, so that \(\partial y\) contains a term of the form \(\rho_J \otimes  z\). The presence of the crossover arrow implies that \(\partial x\) contains a term of the form \(\rho_{IJ} \otimes z\). It follows that \(x\) and \(\rho_I \otimes y\) have the same \(\Z/2\Z\) grading, so the change of basis in which we replace \(x\) by \(x + \rho_I \otimes y\) respects the \(\Z/2\Z\) grading. The argument for move (M2) is very similar. 
	
	 We define \(\bartracks(M,\spin)\) to be this simplified train track; its image in \(T_M\) is the track obtained by simplifying \(\tracks(M,\spin)\). Finally, the same argument as in Section~\ref{sec:invariance} shows that 
	 \(\bartracks(M,\spin)\) does not depend on the choice of parametrization. We denote the resulting invariant by \(\curves{M,\spin}\).
	\end{proof}
	
Let \(\widehat{T}_M = H_1(\partial M;\R)/H_1(\partial M; \Z)\) be \(T_M\) with the puncture filled in. 	
	
\begin{corollary}\label{crl:pull-tight-to-longitude} Each component of \(\HFhat(M)\) is homotopic to some multiple of \(\ell\) in \(\widehat{T}_M\).  If \(b_1(M)=1\),  \(\HFhat(M, \spin)\) represents the homology class \(\ell \in H_1(\widehat{T}_M)\) for each \(\spin \in \spinc(M)\). If \(b_1(M)>1\), \(\HFhat(M,\spin)\) is null homologous in \(\widehat{T}_M\). 
\end{corollary}

\begin{proof} Filling in the punctures of \(\barT_M\) gives a covering space \(\widehat{\barT}_M\) of \(\widehat{T}_M\) that is homeomorphic to the cylinder \(H_1(\partial M; \R)/\langle \ell \rangle \). The image of a  component of \(\HFhat(M,\spin)\) in this covering space is homotopic to some multiple of \(\ell\).  Pushing everything back down to \(\widehat{T}_M\) gives the first statement. 

For the second,  recall that  the \(\Z/2\) grading induces a well-defined orientation on the components of \(\curves M\), so  it makes senses to talk about the homology class it represents. To compute this class, we Dehn fill \(M\) along some slope \(\alpha\) which is not the rational homological longitude, so \(b_1(M(\alpha)) = b_1(M)-1\). If \(\spin \in \spinc(M(\alpha))\), \(\HFhat(M(\alpha), \spin)\) is obtained by pairing \(\HFhat(M, \spin|_{M})\) with the image of some line \(L_{\alpha,\spin}\) with slope \(\alpha\) in \(\barT_M\). Since \(\alpha\) is not a multiple of \(\ell\), the image of \(L_{\alpha,\spin}\) in \(\barT_M\) is noncompact and generates the Borel-Moore homology \(H_1^{BM}(\widehat{\barT}_M) \simeq \Z\). In particular, the intersection number 
\(L_{\alpha, \spin} \cdot \ell = 1\), so if \([\HFhat(M,\spin) ] = k \ell \in H_1(\widehat{\barT}_M)\), then \(L_{\alpha,\spin} \cdot \HFhat(M,\spin) = k\). On the other hand, \(L_{\alpha,\spin} \cdot \HFhat(M,\spin) = \chi(\HFhat(M(\alpha), \spin))\). The result now follows from the well-known fact that \(\chi(\HFhat(M(\alpha), \spin)) = 1\) if \(b_1(M(\alpha)) = 0\), and is \(0\) otherwise. 
\end{proof}


To sum up, for each \(\spin \in \spinc(M)\),  \(\barcurves{M,\spin}\) is a well-defined collection of curves equipped with local systems in the cover \( \barT_M\).  When we want to emphasize the fact that we are thinking of the copy of \(\barT_M\)  containing \(\barcurves{M,\spin}\), we will denote it by \(\barT_{M,\spin}\).  

It is tempting to try to combine all the  \(\barcurves{M,\spin}\) for different \(\spin\) into a single well-defined collection of curves in \(\barT_M\)---for example, by identifying relative \(\spinc\) structures with the same value of \(\langle c_1(\spinbar),x\rangle\), where   \(x\) is a generator of \(H_2(M, \partial M)\). Although this can sometimes be made to work, it is not possible in general. A good example to consider is  when \(M\) is the twisted \(I\)-bundle over the Klein bottle (Example~\ref{ex:twisted I}b) where \(\spinc(M)\) consists of two \(\spinc\) structures \(\spin_0, \spin_1\) with the property that if \(\spinbar\in \spinc(M, \spin_i)\) then \(\frac{1}{2}\langle c_1(\spinbar), x\rangle \equiv i \mod{2}\).

\subsection{The refined pairing theorem}
 Suppose that \(Y = M_0 \cup_h M_1\). The relative \(\spinc\) grading on \(\CFhat(Y)\) can be recovered from the box tensor product 
$\CFA(M_0,\alpha_0,\beta_0)\boxtimes\CFD(M_1,\alpha_1,\beta_1)$ as follows. Consider the box tensor products \(x_0\boxtimes x_1\) and \(y_0 \boxtimes y_1\), where  $x_0,y_0$ in $\CFA(M_0,\alpha_0,\beta_0)$ and $x_1,y_1$ in $\CFD(M_1,\alpha_1,\beta_1)$. We have 
\[\spinbar(x_i)  - \spinbar(y_i) \in S_{M_i}.\] Note that  \(x_0\) and \(y_0\) are in the same idempotent if and only if \(x_1\) and \(y_1\) are, so either both differences are in \(H_{M_i}\) or neither one is. Recalling that 
\(H_{M_i} = H_1(\partial M_i)/\ker j_{i,*}\), where $j_{i,*}$ is the map induced by the inclusion $j_i\co H_1(\partial M_i) \to H_1(M_i)$, we see that the sum $\spinbar(x_0)  - \spinbar(y_0)+  \spinbar(x_1)  - \spinbar(y_1)$  is a well defined element of
$$ H_1(\partial M_0) / \left( \ker(j_{0,*})\oplus \ker(j_{1,*})\right) \cong H_1(Y). $$ 
It is equal to the difference \(\spin(x_0 \boxtimes x_1) - \spin(y_0 \boxtimes y_1)\).

Restriction gives a surjective map $$\pi\co \spinc(Y) \to \spinc(M_1) \times \spinc(M_2).$$  If \(\spin_i \in \spinc (M_i)\), it is not hard to see that \(\pi^{-1}(\spin_0 \times \spin_1)\) is a torsor over  \(H_Y= H_1(\partial M_0)/ \langle \ell_0, h_*(\ell_1) \rangle)\), where \(\ell_i\) is the 
homological longitude of \(M_i\). Let \(\barT_Y\) be the covering space of \(T_{M_0}\) whose fundamental group is the kernel of the natural map \(\pi_1(T_{M_0}) \to H_1(T_{M_0})  \to H_Y\); \(\barT_Y\) is the largest covering space of \(T_M\) that is covered by both 
\(\barT_{M_0}\) and \(\barT_{M_1}\). 

Let \(p_i\co\barT_{M_i} \to \barT_Y\) be the projection, and let \(\barp_1\) be the composition with \(p_1\) with the elliptic involution on \(\barT_Y\). (The elliptic involution is the map which descends from the map \(w \mapsto -w\) on the universal cover. It is well defined up to the action of the deck group.) The images \(p_0(\barcurves{M_0,\spin_0})\) and \(\barp_1(\barcurves{M_1,\spin_1})\)  are well defined up to the action of the deck group \(H_Y\). 

\begin{proposition}
	\label{prop:refined-pairing}
	Suppose that \(\spin \in \spinc(Y)\) has \(\pi(\spin) = \spin_0 \times \spin_1\). There is some \(\alpha_{\spin} \in H_Y\) such that  \(\HFhat(Y,\spin)\) is given by the pairing of \(p_0(\barcurves{M_0,\spin_0})\) with \(\alpha_\spin \cdot \barp_1(\barcurves{M_1, \spin_1})\) in \(\barT_Y\). Moreover, the relative \(\Z/2\Z\) grading of a generator of 
	\(\HFhat(Y, \spin)\) is given by the sign of the corresponding intersection point. 
\end{proposition}

\begin{proof}
	
	It suffices to check that the statement about \(\spinc\) structures holds for the train tracks \(\tracks(M_0;\spin_0), \tracks(M_1,\spin_1)\) (before simplification), since the pairings before and after simplification are the same. Suppose \(x_0,x_1\)
	are generators of \(\CFA(M_0), \CFD(M_1)\) that pair to give a generator of \(\CFhat(Y,\spin)\). Then we must have 
	\(x_0 \in \CFA(M_0,\spin_0)\) and \(x_1 \in \CFD(M_1,\spin_1)\), and we may choose \(\alpha_\spin\) so that the intersection between the train tracks \(p_0(\tracks(M_0);\spin)\) and \(\alpha_\spin \cdot \barp_1(\tracks(M_1;\spin))\) contains the intersection between the segments corresponding to \(x_0\) and \(x_1\). If \(y_0,y_1\) are other generators of \(\CFA(M_0,\spin_0)\) and \( \CFD(M_1,\spin_1)\),  \(y_0 \boxtimes y_1\) will be a generator of \(\CFhat(Y,\spin)\) if and only if the image of \(\spinbar(x_0) - \spinbar(y_0) - \spinbar(x_1) + \spinbar(y_1) \) is \(0\) in \(H_1(Y)\). This is equivalent to saying that \(\spinbar(x_0)-\spinbar(y_0)\) and \( \spinbar(x_1)-\spinbar(y_1)\) have the same image in \(H_1(Y)\), and this  occurs if and only if the segments corresponding to \(y_0\) and \(y_1\) intersect in \(p_0(\tracks(M_0;\spin_0)) \cap \alpha_\spin \cdot \barp_1(\tracks(M_1;\spin_1))\). 
	
Next, consider the relative $\Ztwo$ grading. Recall that fixing the $\Ztwo$ grading on $\CFA(M_0,\alpha_0,\beta_0)$ corresponds to choosing an orientation of each curve 
in $\boldsymbol{\gamma}_0 = \HFhat(M_0)$. In particular, upward (respectively downward) oriented $\alpha_0$ segments in $\boldsymbol{\gamma}_0$ correspond to $\circ^+$ (respectively $\circ^-$) generators of $\CFA(M_0,\alpha_0,\beta_0)$, and rightward (respectively leftward) oriented $\beta_0$ segments in $\boldsymbol{\gamma}_0$ correspond to $\bullet^+$ (respectively $\bullet^-$) generators of $\CFA(M_0, \alpha_0, \beta_0$). Fixing the $\Ztwo$ grading on $\CFD(M_1, \alpha_1, \beta_1)$ corresponds to choosing orientations on $\boldsymbol{\gamma_1} = \curves{M_1}$ with similar identification except that the grading is reversed on $\bullet$ generators. Reflecting across the anti-diagonal, it follows that upward (respectively downward) oriented segments in $\boldsymbol{\gamma}_1$ correspond to $\bullet^+$ (respectively $\bullet^-$) generators of $\CFD(M_1,\alpha_1,\beta_1)$ and rightward (respectively leftward) oriented segments in $\boldsymbol{\gamma}_1$ correspond to $\circ^-$ (respectively $\circ^+$) generators of $\CFD(M_1, \alpha_1, \beta_1$). The (relative) $\Ztwo$ grading on $\HFhat(M_0\cup_h M_1)$ is given by $\gr(x\otimes y) = \gr^A(x) + \gr^D(y)$, while the grading of $\HFa({\boldsymbol \gamma}_0, {\boldsymbol \gamma}_1)$ is given by intersection signs. It is straightforward to check that generators of $\HFhat(M_0\cup_h M_1)$ of the form $\circ^+ \otimes \circ^+$, $\circ^- \otimes \circ^-$, $\bullet^+\otimes\bullet^+$, and $\bullet^-\otimes\bullet^-$, which all have the same $\Ztwo$ grading, correspond exactly to positive intersection points and the remaining generators correspond to negative intersection points.
\end{proof}

\subsection{Examples}
\label{subsec:heights}

We conclude with some examples of how the \(\spinc\) grading is computed.

\begin{example}
Suppose that \(M\) is a homology \(S^1\times D^2\). In this case the homological longitude \(\ell\) coincides with the rational longitude \(\lambda\).   Take  \(x \in H_2(M,\partial M)\) with \(\partial x =  \lambda\), and  \(\mu \in H_1(M)\) with \(\mu \cdot \lambda =1\), so that \(\mu \cdot x =1\). If we use the parametrization \(\alpha = \mu\) and \(\beta = \lambda\), we have 
$$\iota_0  \CFD(M, \mu, \lambda) \simeq \mathit{SFH}(M,\gamma_\mu) \simeq \HFK(K_\mu) $$
	where \(K_\mu\) is the core of the Dehn filling \(M(\mu)\). 
\end{example}	

\piccaption[]{Lifts of \(\mu\) and their intersections with the curve representing the right-hand trefoil. \label{fig:HFKtrefoil}.}
	\parpic[r]{
 \begin{minipage}{40mm}
 \centering
 \includegraphics[scale=0.75]{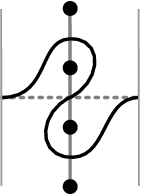}
  \end{minipage}%
}
There is a unique \(\spinc\) structure \(\spin\) on \(M\), and the set \(\spinc(M,\iota_0,\spin)\) is a \(\Z\)--torsor. Let \(\spinbar_i \in \spinc(M,\spin)\) be the relative \(\spinc\) structure with \(\langle c_1(\spinbar), x\rangle = 2i\), and let \(\mu_i\) be the associated lift of \(\mu\). There is a natural height function \(h\co\barT_M \to \R\) such that the height of the midpoint of \(\mu_i\) associated to \(\spinbar\) is \(i\), and  
\(HF(\HFhat(M,\spin_0), \mu_i) = \HFKhat(K_\mu, i)\) is the knot Floer homology of \(K_\mu\) in Alexander grading \(i\). Figure \ref{fig:HFKtrefoil} illustrates the computation for \(\HFKhat(T(2,3))\). The vertical segments in the middle of the diagram are \(\mu_{-1}, \mu_0\), and \(\mu_1\), so there is one generator in each of \(\spinbar_{-1}, \spinbar_{0}\), and \(\spinbar_1\).

	More generally, the same method can be used to compute 
	\(\HFK(K_\alpha)\) from \(\HFhat(M)\), where \(\alpha\) is any primitive curve on \(\partial M\). This and other relations between \(\HFhat(M)\) and knot Floer homology are discussed  in  \cite{HRW-companion}.
	
\begin{example}
Now suppose that \(M_0\) and \(M_1\) are homology \(S^1\times D^2\)'s, and  that \(h\) identifies their homological longitudes, so that \(H_1(Y) \simeq \Z\). There is a unique \(\spinc\) structure \(\spin_i\) on \(M_i\), but \(\spinc (Y)\) is a \(\Z\)-torsor, and  \(\barT_Y = \barT_{M_0} = \barT_{M_1}\).

 Let \(\mathfrak{t}_i \in \spinc(Y)\) satisfy 
\(\langle c_1(\mathfrak{t}_i), y \rangle = 2i\), where \(y\) generates \(H_2(Y)\). Then
\(\HFhat(Y, \mathfrak{t}_i) = HF(\gamma_0, \alpha_i \cdot \gamma_1)\) where \(\gamma_0 = \HFhat(M_0,\spin_0)\) and \(\gamma_1\) is obtained by starting with \(\HFhat(M_1, \spin_1)\) and applying the elliptic involution, which sends \(\mu_{k}\) to \(\mu_{-k}\). The action of \(\alpha_i\) shifts the height by \(i\) units, so the net effect is to superimpose \(\gamma_0\) and \(\gamma_1\) in such a way that \(\mu_k\) in the first diagram is identified with \(\mu_{i-k}\) in the second. 
\end{example}

In general, suppose we are given \(M\) whose rational longitude \(\lambda\)  has order \(k\) in \(H_1(M)\). Choose a class \(x \in H_2(M,\partial M)\) with \(\partial x = k  \lambda\), and a class \(\mu\) in \(H_1(M)\) with \(\mu \cdot \lambda =1\), so that \(\mu \cdot x =k\).  Then \(H_M \cong \Z \oplus \Z/k\), and there is a height function  \(h_x\co\barT_{M,\spin} \to \R\) that satisfies 
\(h_x(q_{\spinbar}) = \frac{1}{2} \langle c_1(\spinbar),x\rangle\). If \(\spinbar, \spinbar' \in \spinc(M, \iota_0, \spin)\), then 
\[h_x(q_{\spinbar})- h_x(q_{\spinbar'}) = \frac{1}{2} PD(c_1(\spinbar)-c_1(\spinbar')) \cdot x\]
is a multiple of \(k\). Hence all the heights for a given \(\spin\) are congruent to a single value modulo \({k}\), and there are \(k\) meridians at each such height. 
		
\begin{example}	
	Now suppose that we have two manifolds \(M_0\) and \(M_1\), with rational longitudes \(\lambda_i\) of order \(k_i\), that are identified by \(h\). Let \(k=\gcd(k_1,k_2)\), and define \(c_i=k_i/k\). Given \(x_i \in H_2(M_i,\partial M_i)\) as above, we can form a primitive  class \(y = c_2x_1-c_1x_2 \in H_2(Y)\). Given \(\spin_i \in \spinc(M_i)\), let \(\gamma_i = \curves{M_i,\spin_i}\). Then for  \(\mathfrak{t} \in \pi^{-1}(\spin_0,\spin_1)\),
	 \[\HFhat(Y,\mathfrak{t}) = \HF (p_0(\gamma_0),\alpha_{\mathfrak{t}} \cdot \overline{p}_1(\gamma_1)),\]
	 where \(\alpha_{\mathfrak{t}} \in H_Y = H_1(\partial M_0)/\langle k_0 \lambda,k_1 \lambda \rangle \simeq \Z\oplus \Z/k\Z\). 
	 
	 The height functions \(h_{x_i}\) on \(\barT_{M_i}\) induce height functions \(h_{y,i}:\barT_Y \to \R\) given by \(h_{y,1} (p_1(w)) = c_2 h_{x_1}(w) \) and \(h_{y,2} (p_1(w)) = -c_1 h_{x_1}(w) \). 
	 Then \( \langle c_1(\mathfrak{t}), y \rangle \) measures the shift in heights induced by 
	 \(\alpha_{\mathfrak{t}}\) so that
	\[ \langle c_1(\mathfrak{t}), y \rangle = 2(h_{y,1}(\alpha_t \cdot w)- h_{y,2}(w)) \]
	for all \(w \in \barT_Y\).
	
	\end{example}

\section{Applications}\label{sec:applications}

With invariance (Theorem \ref{thm:invariance}) and pairing (Theorem \ref{thm:pairing}) for $\HFhat(M)$ established, we can now provide proofs for the applications stated in the introduction.

\subsection{Peg-board diagrams}To facilitate the proofs, we will first discuss a convenient way of arranging the curves $\HFhat(M)$.


\piccaption[]{Pulling the invariant for the trefoil tight. \label{fig:trefoil-peg}}
\parpic[r]{
 \begin{minipage}{55mm}
 \centering
\includegraphics[scale=0.5]{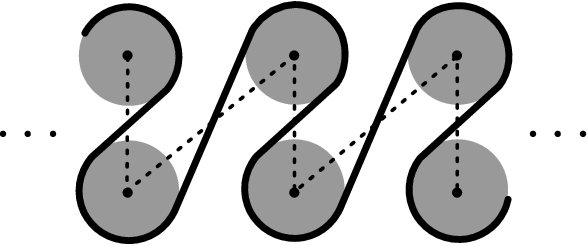}
  \end{minipage}%
  }Suppose \(\gamma\) is a component of  \(\HFhat(M)\). To understand the pairing between \(\gamma\) and some other curve it is helpful to ``pull the curves tight.'' Intuitively, we can imagine \(T_M\) as being  a flat Euclidean torus (like a board with the sides identified) and the puncture as represented by a peg pounded into the board. We place a rubber band in the homotopy class determined by \(\gamma\), and see what position it settles into. To draw this  curve, it is convenient to  first lift it to the cover \(\barT_{M}\) and then take the preimage of this lift in the larger cover \(\tildeT_{M}\). The latter space can be identified with flat Euclidean \(\R^2\) with a lattice of pegs---in other words, a peg-board. Hence we refer to these curves as \emph{peg-board diagrams}. Figure~\ref{fig:trefoil-peg} illustrates this process for the trefoil complement. 
We offer three different mathematical formulations of this idea, and explain how they are related.

We first describe the notion of an \emph{\(\epsilon\)-geodesic}. Fix a flat metric \(g\) on \(\widehat{T}_M =  H_1(\partial M;\R)/H_1(\partial M,\Z)\)  by choosing a parametrization of \(\partial M\) and using it to identify \(H_1(\partial M;\R)\) with \(\R^2\).  Consider a manifold \(T_{M,\epsilon}\) with a complete Riemannian metric \(g_\epsilon\) defined as follows.  \(T_{M,\epsilon}\) is the union of two parts. The first part is the complement of the ball of radius \(2\epsilon\) centered at the origin in \(T_M\), equipped with the flat metric \(g\). The second part is modeled on a surface of revolution in \(\R^3\), as illustrated in Figure~\ref{fig:metric-model}. The smooth curve  that we rotate to get the surface should agree with the \(x\)-axis for \(x>{2\epsilon}\), with the line \(x = \epsilon\) for \(y>\epsilon\), and should be given by the graph of some smooth function \(y = f(x)\) with \(f'(x)<0\) and \(f''(x)>0\) for \(x\in (\epsilon, 2\epsilon)\). 
We will refer to this second part of \(T_{M,\epsilon}\) as the \emph{peg}. The constructed surface \(T_{M,\epsilon}\) is homeomorphic to \(T_M\), and the conditions on \(f\) ensure that \(g_\epsilon\) is non-positively curved. The flat part of \(T_{M,\epsilon}\) has two connected components: one in the interior, and one that is boundary parallel.  It is clear from the construction that  \(T_{M,\epsilon}\) embeds in \(\widehat{T}_M \times \R\). Let \(p\co T_{M,\epsilon} \to T_M\) be the projection. 

\labellist
\tiny
\pinlabel $\epsilon$ at 30 35 \pinlabel $2\epsilon$ at 50 33
\pinlabel $\epsilon$ at 455 71 \pinlabel $2\epsilon$ at 479 72
\endlabellist
\begin{figure}
\includegraphics[scale=0.5]{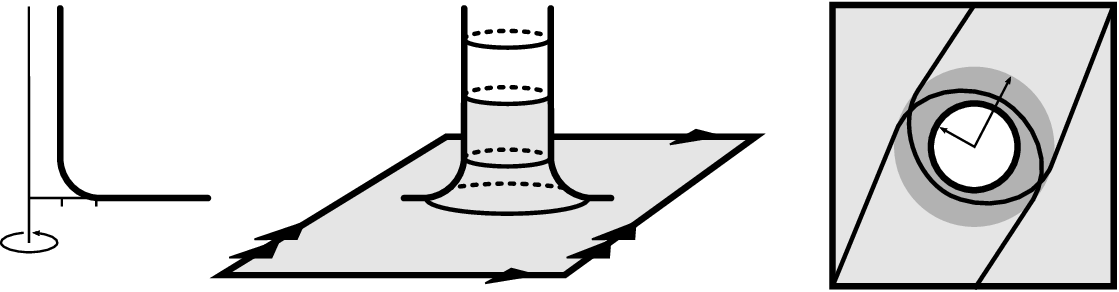}
\caption{A geometric model for the punctured torus: Attaching an infinite tube of radius $\epsilon$ to the flat torus results in a  complete metric that is non-positively curved and flat away from the smoothed attaching region. In the centre, we have shaded the compact region used in our appeal to Thorbergsson's work \cite{Thorbergsson1978}. On the right, we have projected to the torus and illustrated 
the invariant associated with the right-hand trefoil exterior; compare Figure \ref{fig:trefoil-peg}.}
\label{fig:metric-model}
\end{figure}

A curve \(\gamma \subset T_{M,\epsilon}\) is an \emph{\(\epsilon\)-geodesic} if it is geodesic for the metric \(g_\epsilon\). We summarize the relevant properties of such curves here: 

\begin{lemma} Any nontrivial free homotopy class of loops in \(T_M\) can be represented by an \(\epsilon\)-geodesic \(\gamma\). This representative is unique unless \(\gamma\) is entirely contained in the flat part of \(T_{M,\epsilon}\), {\it i.e.} is either boundary parallel or a Euclidean geodesic in the interior flat part. Any two distinct  \(\epsilon\)-geodesics intersect minimally and transversally. 
\end{lemma}

\begin{proof} Consider the compact subset \(K \subset T_{M,\epsilon}\) obtained as the complement of the boundary flat part. The only free homotopy classes that have representatives disjoint from \(K\) are boundary parallel. Such classes clearly have geodesic representatives---namely a meridional circle around the peg. By a theorem of Thorbegsson \cite[Theorem 3.2(i)]{Thorbergsson1978}, any other free homotopy class has a geodesic representative \(\gamma\). A standard argument using the Gauss--Bonnet theorem shows that if \(\gamma, \gamma'\) are two geodesic representatives in the same free homotopy class, they are entirely contained in the flat part of \(T_{M,\epsilon}\). Such a geodesic is contained in either the interior flat part of the torus, in which case it is a line, or in the boundary flat part, in which case it is boundary parallel. For the last statement, note that two distinct geodesics always intersect transversally, and the theorem of Freedman, Hass, and Scott \cite{FreedmanHassScott} ensures that they intersect minimally. 
\end{proof}

We say that a geodesic contained in the flat part of \(T_{M,\epsilon}\) is \emph{loose}; all other geodesics are \emph{tight}. A \emph{corner} of an \(\epsilon\)-geodesic \(\gamma_\epsilon\) is a connected component of the intersections of \(\gamma_\epsilon\) with the curved part of \(T_{M,\epsilon}\). The property of being loose is equivalent to having no corners.

Given an  \(\epsilon\)-geodesic \(\gamma\), we will construct another curve \(\gamma' \subset T_M\) that is homotopic to \(p(\gamma)\) and refer to \(\gamma'\) as an \emph{\( \epsilon\)-pegboard diagram}. The curve \(\gamma'\) is obtained by modifying the corners of \(\gamma\), as illustrated in Figure~\ref{fig:pegboard-corner}. To make this precise, let \(C_{2 \epsilon}\) be the circle of radius \(2\epsilon\) centered on the puncture in \(T_M\), so that \(C_{2\epsilon}\) is the boundary of the interior flat part of \(T_{M,\epsilon}\). 
If \(c\) is a corner of \(\gamma\), let \(p_1\) and \(p_2\) be the two points at which  $c$ intersects \(C_{2\epsilon}\). Let \(L\) be the perpendicular bisector of the line between \(p_1\) and \(p_2\), which also intersects the center of \(C_{2\epsilon}\).  Reflection in \(L\) induces an isometry of \(T_{M,\epsilon}\) that preserves the endpoints of the geodesic \(c\), and hence \(c\) itself. Let \(L_1\) and \(L_2\) be the Euclidean lines in \(T_M\) extending the parts of \(\gamma\) in the flat part of \(T_{M,\epsilon}\). Reflection in \(L\) exchanges \(L_1\) and \(L_2\), so  there is a circle \(C\) centered at the origin and tangent to \(L_1\) and \(L_2\).  Since \(L_1\) and \(L_2\) cut  \(C_{2\epsilon}\), the radius of \(C\) is less than \(2 \epsilon\). We can homotope \(c\) to obtain a curve \(c'\) that  runs along \(L_1\) until it reaches \(C\),  runs along \(C\) for some time, and then exits along \(L_2\). 

\labellist
\pinlabel $C_{2\epsilon}$ at 20 68 \pinlabel $c$ at 40 23
\pinlabel $\theta$ at 75 43 \pinlabel $\theta$ at 65 96
\pinlabel $L_2$ at 120 58 \pinlabel $L_1$ at 120 93
\pinlabel $C$ at 187 68
\pinlabel $c'$ at 335 23
\endlabellist
\begin{figure}
\includegraphics[scale=1]{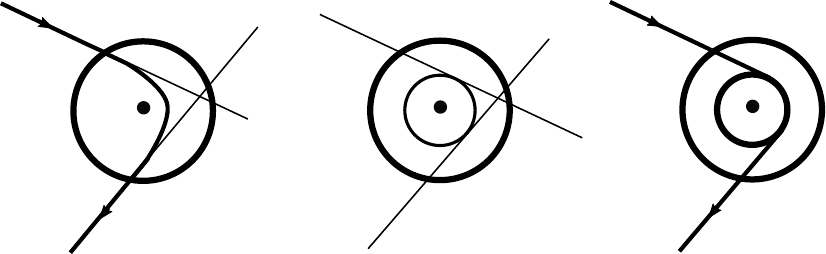}
\caption{Constructing a pegboard corner \(c'\) from a geodesic corner \(c\). The first step shows the original \(c\) together with the two lines \(L_1\) and \(L_2\). The second step shows the new circle \(C\), and the last shows the new curve \(c'\).}
\label{fig:pegboard-corner}
\end{figure}

 Let \(\theta(c)\) be the angle that \(L_1\) makes with \(C_{2\epsilon}\) (by symmetry, this is also the angle that \(L_2\) makes with \(C_{2\epsilon}\). There are two important quantities describing \(c\) that are determined by \(\theta(c)\). The first is  \(r(c)\), the radius of  the circle \(C\). The second is the length of the arc obtained by projecting \(c\) radially outward to \(C_{2\epsilon}\);  we denote this quantity by \(\psi(c)\). Observe that  \(r\) is a strictly decreasing function of \(\theta\), 
  while  \(\psi\) is a strictly increasing  function of \(\theta\).
  It follows that \(r\) is a strictly decreasing function of \(\psi\). 
 
 If \(c_1\) and \(c_2\) are two corners with distinct endpoints, the intersection number  \(i(c_1,c_2)\), is the minimal number of intersections between transversally intersecting curves in \(D_{2\epsilon} \setminus 0\) that are homotopic rel boundary to \(c_1\) and \(c_2\). We say \(c_1\) and \(c_2\) are in \emph{minimal position} if they realize the intersection number.  For example, if \(c_1\) and \(c_2\) are corners of \(\epsilon\)-geodesics, they are in minimal position.

\begin{lemma}\label{lem:corner-trans} Suppose that \(c_1'\) and \(c_2'\) are two corners constructed as above from geodesic corners \(c_1\) and \(c_2\). If \(c_1'\) and \(c_2'\) intersect transversally,  they are in minimal position. 
\end{lemma}
\begin{proof} Identify the universal cover of the punctured disk with \(\R \times (0,2\epsilon]\). The curves \(c_1'\) and \(c_2'\) are in minimal position if \(\widetilde{c}_1'\) and \(\widetilde{c}_2'\) are in minimal position for every lift 
\(\widetilde{c}_1'\)  of \(c_1'\) and every lift \(\widetilde{c}_2'\) of \(c_2'\). The shape of \(\widetilde{c}_1'\) and \(\widetilde{c}_2'\) is illustrated in  Figure~\ref{fig:corner-cover}. The depth of the trapezoidal shape made by the lift of \(c_i'\) is \(2 \epsilon- r(c_i)\), and its width is \(\psi(c_i)\). Since \(r\) is a strictly decreasing function of \(\psi\), the lifts have no excess intersection. 
\end{proof}

\labellist
\pinlabel $\widetilde{c}_1'$ at 160 86 \pinlabel $\widetilde{c}_2'$ at 500 80
\endlabellist
\begin{figure}[h]
\includegraphics[scale=0.5]{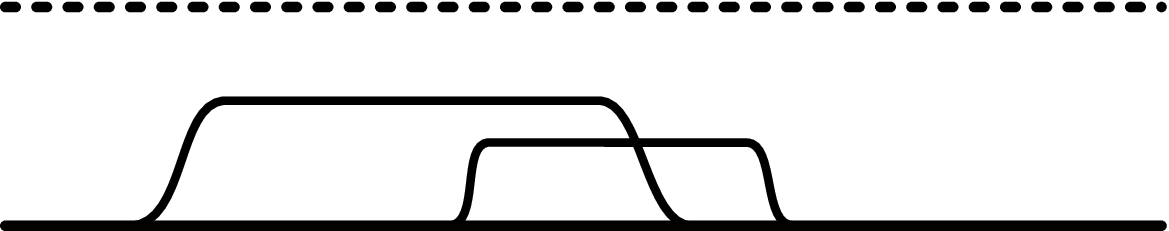}
\caption{The universal cover of the punctured disk, with lifts of \(c_1'\) and \(c_2'\). The key point is that the deeper shape is also wider. To compute the intersection number, we take the intersection of \(\widetilde{c}_1'\) with the orbit of \(\widetilde{c}_2'\) under the action of the deck group. }
\label{fig:corner-cover}
\end{figure}

If  \(r(c_1)=r(c_2)\), \(c_1'\) and \(c_2'\)  will not intersect transversally, but we can perturb them into minimal position by slightly changing \(r(c_2)\). Since \( r(c_1) = r(c_2)\) implies that  \(\psi(c_1) = \psi(c_2)\), \(c_1'\) and the perturbation of \(c_2'\) will be in minimal position   as long as \(c_1\) and \(c_2\) have distinct endpoints.

If \(\gamma\) is an \(\epsilon\)-geodesic, we define the \(\epsilon\)-pegboard diagram \(\gamma'\) to be the curve obtained by applying the operation above to every corner of \(\gamma\), so that 
an \(\epsilon\)-pegboard diagram is  composed of straight line segments together with circular arcs that are tangent to these line segments.

\begin{corollary}
If \(\gamma_1'\) and \(\gamma_2'\) are the \(\epsilon\)-pegboard diagrams induced by distinct \(\epsilon\)-geodesics \(\gamma_1\) and \(\gamma_2\), then they are in minimal position. 
\end{corollary}

\begin{proof}
Since \(\gamma_1\) and \(\gamma_2\) are distinct \(\epsilon\)-geodesics, they are in minimal position. By Lemma \ref{lem:corner-trans} \(\gamma_1'\) and \(\gamma_2'\) have the same number of intersections as \(\gamma_1\) and \(\gamma_2\). 
\end{proof}

Finally, we discuss the third  notion of peg-board position, which corresponds to letting the radius of the peg go to \(0\). We use our chosen parametrization of \(T_M\) to identify a fundamental domain for \(T_M\) with the square \([-\frac{1}{2}, \frac{1}{2}] \times [-\frac{1}{2}, \frac{1}{2}]\); the puncture is the point \((\pm \frac{1}{2},\pm \frac{1}{2})\), the vertical sides are identified with \(\alpha\), and the horizontal sides are identified with \(\beta\). 


Any essential loop in \(T_M\) is free homotopic to a curve \(\gamma\) with no backtracking. Such a loop is represented  by a cyclic word  in \(\alpha^{\pm1}\) and \(\beta^{\pm1}\), which we also call \(\gamma\). Recall that, by familiar abuse of notation, the letters of  the word \(\gamma\) correspond to the intersections of  the curve \(\gamma\) with the horizontal and vertical sides of the square; see Figure \ref{fig:weak-tight}, below. More precisely, if the \(i\)th letter is  \(\alpha^{\pm 1}\), the corresponding intersection point has coordinates \(\bv_i = (\pm \frac{1}{2}, x_i)\); if it is  \(\beta^{\pm 1}\), the intersection has coordinates \(\bv_i = (x_i, \pm \frac{1}{2})\) for some \(x_i \in [-\half, \half]\).  (Note that with this notation, the generators \(\alpha\) and \(\beta\) of \(\pi_1\) are \emph{dual} to the corresponding parametrizing curves.)

The curve \(\gamma\) is homotopic to a piecewise linear curve obtained by replacing the segment of \(\gamma\) between \(\bv_i\) and \(\bv_{i+1}\) with a line segment. Conversely, suppose that the word \(\gamma\) has \(n\) letters. Given \(\bx \in [-\frac{1}{2}, \frac{1}{2}]^n\), we let \(\gamma_\bx\) be the piecewise linear loop in \(T\) whose vertices are the points \(\bv_i\).  Consider \(F\co [-\frac{1}{2}, \frac{1}{2}]^n \to \R\) given by \(F(\bx) = \ell(\gamma_\bx)\), the length of \(\gamma_\bx\). We can express this length as follows. If we cut \(T_M\) open to form the square \([-\frac{1}{2}, \frac{1}{2}] \times [-\frac{1}{2}, \frac{1}{2}]\), the curve \(\gamma_\bx\) gets cut into a sequence of oriented line segments. Each point \(\bv_i\) has two preimages \((\pm \frac{1}{2}, x_i)\) in the square. We label these \(\bv_i^{\text{in}}\), \(\bv_i^{\text{out}}\) in such a way that the \(i\)th segment runs from \(\bv_i^{\text{out}}\) to \(\bv_{i+1}^{\text{in}}\). Then 
\(F(\bx) = \sum_{i=1}^{n} d(\bv_i^\text{out},\bv_{i+1}^\text{in})\), where \(d\) denotes the Euclidean distance, and \(\bv_{n+1} = \bv_1\).  

Note that since we allow \(x_i \in [-\frac{1}{2}, \frac{1}{2}]\), the points  \(\bv_i^{\text{in}}\), \(\bv_i^{\text{out}}\) may lie at the corners of the square corresponding to the punctures. If we have two consecutive points  \(\bv_i^{\text{out}}\), \(\bv_{i+1}^{\text{in}}\) which lie on the same corner of the square, the corresponding line segment can have zero length. 

\begin{definition} If \(\bx\) is a minimum for the function \(F\) we call \(\gamma_\bx\) a \emph{singular pegboard diagram} for \(\gamma\). \end{definition}

Since the domain of \(F\) is compact, every essential loop \(\gamma\) has at least one singular pegboard diagram \(\gamma_\bx\). Note that \(\gamma_\bx\) is not actually a curve in \(T_M\) unless \(\bx\) is in the interior of \([-\frac{1}{2}, \frac{1}{2}]^n\). However \(\gamma_\bx\) is homotopic to \(\gamma\) in \(\widehat{T}_M\), so \(\gamma_\bx\) lifts to  the covering space \(\widehat{\barT}_M\) obtained by filling in the punctures of  \(\barT_M\). If \(\bv_i\) is a vertex of \(\gamma_\bx\) for which \(x_i = \pm \frac{1}{2}\), we say \(\bv_i\) is a \emph{corner} of \(\gamma_\bx\). 

\labellist
\tiny
\pinlabel $\alpha$ at 76 120 
\pinlabel $\beta$ at 35 76
\endlabellist
\piccaption[]{A weakly tight curve is obtained from the word $\alpha\beta^{-1}\alpha^2\beta$. To make the curve easier to see, we've drawn a fundamental domain with the puncture in the middle.  \label{fig:weak-tight}}
\parpic[r]{
 \begin{minipage}{40mm}
 \centering
\includegraphics[scale=0.5]{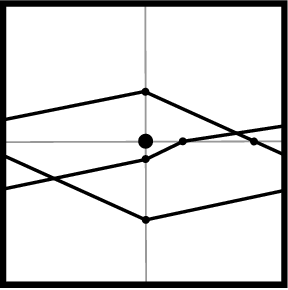}
  \end{minipage}%
  }
Note that it is possible for multiple points \(\bx\) to give the same underlying curve \(\gamma_\bx\). This occurs if a segment of \(\gamma_\bx\) containing three or more vertices lies along one of  the parametrizing curves, say \(\alpha\). In this case we can move the interior vertices that lie on \(\alpha\) freely along \(\alpha\) without changing the underlying curve; see the proof of Proposition \ref{prp:strong-weak-loose}, below. 

If \(\bx\) is an unique global minimum for \(F\), we say that \(\gamma_\bx\) is \emph{strongly tight}. If \(\bx\) is a global minimum for \(F\) and every other global minimum gives rise to the same underlying curve \(\gamma_\bx\) (as described in the paragraph above), we say \(\gamma_\bx\) is \emph{weakly tight.} Finally, if \(\gamma_\bx\) is a line in \(T_M\) and can be freely deformed to nearby parallel geodesics 
we say \(\gamma_\bx\) is \emph{loose}.  Observe that if \(\gamma_\epsilon\) is a loose \(\epsilon\)-pegboard diagram for \(\gamma\), the same curve is also a loose singular pegboard diagram for \(\gamma\). 

\begin{proposition}\label{prp:strong-weak-loose}
If \(\bx\) is a local minimum for \(F\), then \(\gamma_\bx\) is either strongly tight, weakly tight, or loose. 
\end{proposition}

\begin{proof}Write \(F = \sum_{i=1}^n f_i\), where either  \(f_i = \left(1+(x_i-x_{i+1})^2\right)^{1/2}\) or \(f_i = \left((x_i\pm \frac{1}{2})^2+(x_{i+1}\pm \frac{1}{2})\right)^{1/2}\). Since the functions \(f(u) = (1+u^2)^{1/2}\) and \(g(u,v) = (u^2+v^2)\) are both convex, it follows that \(f_i\) and \(F\) are convex functions as well. 
The convexity of $F$ places restrictions on its minima: 
Any local minimum \(\bx\) is necessarily a global minimum, and if  \(\bx\) and \(\bx'\) are two global minima, \(F\) is constant along the line segment that joins \(\bx\) and \(\bx'\). We refer to such a segment as a null line for \(F\). 

Suppose that \(\bx + t \bw\), \(t \in [0,1]\), is a null line for \(F\), and let \( \gamma_{\bx+t\bw}\) be the corresponding deformation of \(\gamma_\bx\). We divide the vertices of \(\gamma_{\bx+t\bw}\) into two types: those for which \(\bv_i(t)\) moves as \(t\) varies, and those for which it stays fixed. We claim that if there is one fixed vertex, then the underlying curve \(\gamma_{\bx+t\bw}\) stays fixed as \(t\) varies. To see this, suppose we have two fixed vertices \(\bv_{i}\) and \(\bv_{j}\), separated by some number of moving vertices. (Since the \(\bv_i\) are cyclically ordered, we can take \(\bv_i = \bv_j\) if necessary, so one fixed vertex suffices.) Since \(\bv_{i+1}\) is a moving vertex, we can find some \(t\in[0,1]\) for which \(\bv_{i+1}(t)\) lies  in the interior of the interval on which it moves. But for every \(t\in[0,1]\),  \(\bx + t\bw\) is a global minimum for \(F\). It follows that \(\bv_{i}(t), \bv_{i+1}(t)\), and \(\bv_{i+2}(t) \) must lie on a line; otherwise we could shorten \(\gamma_{\bx+t\bw}\) by moving \(\bv_{i+1}(t) \) and fixing the other \(\bv_i(t) \).  Repeating, we find that \(\bv_{i}(t), \bv_{i+1}(t), \ldots, \bv_{j}(t)\) are all collinear. But \(\bv_{i}\) and \(\bv_{j}\) were fixed by hypothesis, so the line they lie on is fixed as well. If this line happens to coincide with an \(\alpha\) or \(\beta\) curve, then the intervening points can move along it. In this case \(\gamma_\bx\) is weakly tight. Otherwise, they are fixed as well, and if this is the case for all segments of \(\gamma_\bx\), it is strongly tight.  In either case, the underlying curve remains the same. 

Finally, suppose that every vertex is moving, so we can find some small \(t\) such that all the \(\bv_i\) lie in the interior of the interval in which they move. The same argument as above shows that the \(\bv_i\) are all collinear, so the curve is loose. 
\end{proof}

We  relate the different notions of pegboard representative by showing that when \(\epsilon\) is small, an \(\epsilon\)-geodesic for a curve \(\gamma\) is close to a singular pegboard diagram for \(\gamma\). 
We say that curves \(\gamma_1, \gamma_2 \subset \widehat{T}_M\) are \(\delta\)-close if they can be parametrized so that \(|\gamma_1(t) - \gamma_2(t)| < \delta\) for all \(t\).

\begin{proposition}\label{prp:del-close}
Suppose \(\gamma_\bx\) is a singular pegboard representative for \(\gamma\) and that we are given  \(\delta>0\). Then there exists an \(\epsilon_0>0\) such that, for every \(\epsilon < \epsilon_0\), there is a \(\epsilon\)-geodesic \(\gamma_\epsilon\) representing  \(\gamma\) for which \(p(\gamma_\epsilon)\)  is \(\delta\)-close to \(\gamma_\bx\). 
\end{proposition}

\begin{proof} 

Suppose that  \(\gamma_\bx\) is  tight. 
Let \(X_\epsilon\) be the  image of the map \(p\co T_{M,\epsilon} \to T\); it is the complement of an open ball of radius \(\epsilon\) centered at the origin. There is a unique continuous map \(s\co X_\epsilon \to T_{M,\epsilon}\) satisfying \(p \circ s = 1_{X_\epsilon}\).  Let \(\gamma\) be a path in \(T_{M, \epsilon}\). Since \(T_{M,\epsilon}\) embeds in \(\widehat{T}_M \times \R\), the length of \(\gamma\) is greater than or equal to the length of \(p(\gamma)\). 

Let \(C(r)\) be the cube \([-r,r]^n\), and let \(C^\circ(r)\) be its interior.  Define \(F_{\epsilon}\co C(\frac{1}{2}-\epsilon) \to \R\) by 
\(F_{\epsilon}(\bx) = \sum_{i=1}^n d_\epsilon(s(\bv_i),s(\bv_{i+1}))\), where \(d_{\epsilon}\)  denotes the Riemannian distance in \(T_{M,\epsilon}\). Said another way, \(F_{\epsilon}(\bx)\) is the length of the piecewise geodesic arc \(\gamma_{\bx,\epsilon} \) in \(T_{M,\epsilon}\) obtained by joining the points \(s(\bv_i)\) and \(s(\bv_{i+1})\) by geodesic segments. 

\(F_{\epsilon}\) has the following properties. First, the fact that \(\ell(\gamma) \geq \ell(p(\gamma))\) implies that \(F_\epsilon \geq F\). Second, if \(|x_i|,|x_{i+1}| < \frac{1}{2} - 4 \epsilon\), then the geodesic between \(s(\bv_i)\) and \(s(\bv_{i+1})\) is the straight line in \(T_M\). Hence \(F= F_{\epsilon}\) on \(C(\frac{1}{2}-4\epsilon)\). Third, if \(\bx\) is a local minimum for \(F_\epsilon\), then \(\gamma_{\bx,\epsilon}\) is a geodesic. This follows from the usual principle that rounding a corner of a piecewise geodesic segment reduces its length.

Fix some \(r>0\), and consider:
\[N_r = \{\bz \in C(1/2) \, | \, d(\bx,\bz) \leq r \ \text{for some global minimum}\ \bx \ \text{for} \ F\}\] 
 \[N^\circ_r = \{\bz \in C(1/2) \, | \, d(\bx,\bz) < r \ \text{for some global minimum}\ \bx \ \text{for} \ F\}\] 
 Set \(Y_r = N_r \setminus N^\circ_r\) so that \(N_r\) and \(Y_r\) are closed subsets of \(C(\half)\) while \(N_r^\circ\) is an open subset.  
  No point of \(Y_r\) is a global minimum, so \(F(\by) > F(\bx)\) for all \(\by \in Y_r\).   \(Y_r\) is compact, so there is some \(\eta>0\) such that \(F(\by) \geq F(\bx) + \eta\) for \(\by \in Y_r\). Since \(F\) is continuous, we can find some \(\bz \in N_r \cap C^\circ(\half)\) with \(F(\bz) < F(\bx) + \eta\). Finally, choose \(\epsilon_0>0\) so that \(\bz \in C^\circ(\frac{1}{2}-4\epsilon_0)\). Suppose that \(\epsilon<\epsilon_0\) and let \(\bz'\) be a global minimum point for \(F_\epsilon\) on the compact set \(N_r \cap C(\frac{1}{2}-\epsilon)\). Using the first and second properties above, we see that for \(\by \in Y_r\),  \[F_\epsilon(\by)\geq F(\by) > F(\bz) = F_\epsilon(\bz).\] Hence \(\bz' \in N_r^\circ \cap C(\frac{1}{2}-\epsilon)\), which implies that \(\bz'\) is a local minimum for \(F_\epsilon\). By the third property, \(\gamma_{\bz', \epsilon}\) is a geodesic.  
Now if \(\epsilon<\delta/4\) it is easy to see that there is some \(r>0\) such that \(p(\gamma_{\bz,\epsilon})\) and \(\gamma_\bx\) are \(\delta\)-close whenever \(\Vert \bz - \bx \Vert < r\). Taking this value of \(r\) in the argument above gives the desired statement. 

If \(\gamma_\bx\) is loose, the statement is easy: we choose \(\epsilon_0\) small enough that there is a line in the flat part of \(T_{M,\epsilon_0}\) that is \(\delta\)-close to \(\gamma_\bx\).
\end{proof}

\begin{corollary}
Suppose that \(\gamma_0\) (resp. \(\gamma_\epsilon\)) is a singular pegboard diagram (resp. \(\epsilon\)-pegboard diagram) for \(\gamma\). If \(\gamma_0\) is tight, then \(p(\gamma_\epsilon) \to \gamma_0\) in the Hausdorff metric as \(\epsilon \to 0\).  
\end{corollary}

\begin{proof}
If \(\epsilon_i \to 0\), Proposition \ref{prp:del-close} tells us that there is some sequence of \(\epsilon_i\)-geodesics \(\{\gamma_{\epsilon_i}\}\) such that \(p(\gamma_{\epsilon_i})\) converges to \(\gamma_0\) in the Hausdorff metric. 
As \(\gamma_0\) is tight, \(\gamma_{\epsilon_i}\) is tight for all \(\epsilon_i\). Hence  \(\gamma_{\epsilon_i}\) is the unique \(g_{\epsilon_i}\) geodesic in the homotopy class of \(\gamma\).  So if \(\{\gamma'_{\epsilon_i}\}\) is any sequence of  \(\epsilon_i\)-geodesics for \(\gamma\), we must have  \(\{\gamma'_{\epsilon_i}\} =  \{\gamma_{\epsilon_i}\}\). 
\end{proof}

If \(\gamma_1, \gamma_2\) are \(\epsilon\)-geodesics (or pegboard diagrams), we can write 
\(i(\gamma_1,\gamma_2)= i^\circ(\gamma_1, \gamma_2) + \sum_{(c_1,c_2)} i(c_1, c_2) \)
where \(i^\circ(\gamma_1, \gamma_2)\) is the intersection number of the flat parts, of \(\gamma_1\) and \(\gamma_2\), and the  the sum runs over all pairs of corners \((c_1, c_2)\) for \(\gamma_1\) and \( \gamma_2\). 

For some arguments it is useful to have a similar formula for singular pegboard diagrams, but 
when we pass to the limit \(\epsilon\to 0\) to obtain a singular pegboard diagram \(\bargamma\), we lose  information about the homotopy class of the corners. To be precise, the underlying curve \(\bargamma\) does not depend on the homotopy class of the corners. However, if we remember not just \(\bargamma\), but the minimizer \(\bx\) that determines it, we can recover the homotopy class from the number of consecutive vertices \(\bv_i\) with \(d(\bv_i,\bv_{i+1}) = 0\) at the corner. 
Taking the limit of the formula above as \(\epsilon \to 0\), we obtain:

\begin{corollary} If \(\bargamma_1, \bargamma_2\) are singular pegboard diagrams such that no slope of \(\bargamma_1\) is also a slope of \(\bargamma_2\), then 
\[ i(\gamma_1, \gamma_2) = i^\circ(\bargamma_1, \bargamma_2) + \sum_{(c_1,c_2)} i(c_1,c_2) \]
where \( i^\circ(\bargamma_1, \bargamma_2)\) is the number of intersections away from the corners, and \(i(c_1,c_2)\) is a local intersection number determined by the homotopy data at the corners.\qed
\end{corollary}

To specify the homotopy class at a corner, we choose a small disk \(D\) centered at the puncture and project a path realizing the homotopy class out to \(\partial D\).  The resulting homotopy class is determined by the total angle \(\phi\) that it covers. (The angle \(\phi\) plays the same role as the quantity \(\psi\) for which we used for \(\epsilon\)-geodesics.)   Note that a corner with angle \(|\phi|<\pi\) cannot be realized as part of a singular pegboard diagram (since the resulting curve can be shortened). Hence the smallest possible angle always has 
\(\pi \leq |\phi| \leq 2\pi\). If the angle is \(\geq 2 \pi\), we say that \(\gamma\) 
\emph{wraps the peg} at that corner. The boundary case where \(\phi =\pi\) often requires special treatment in later arguments: we refer to such a corner as a \emph{\(\pi\)-corner} (or, sometimes, a \emph{straight corner}). When a distinction is required, non-$\pi$-corners then will be called \emph{true corners}.

If neither \(c_1\) nor \(c_2\) wraps the peg, the local intersection number 
\(i(c_1,c_2)\) is easily described. Let \(A_1\) and \(A_2\) be the arcs specifying the homotopy classes of \(c_1\) and \(c_2\). Then we have:
\begin{equation*}
i(c_1,c_2) = \begin{cases} 0 & \quad \text{if} \ A_1 \  \text{and} \ A_2 \  \text{are nested } \\
1 & \quad \text{if} \ A_1 \  \text{and} \ A_2 \  \text{are not nested  and }  \ A_1 \cap A_2 \ \text{has one component} \\
2 & \quad \text{if} \ A_1 \  \text{and} \ A_2 \  \text{are not nested  and } \  A_1 \cap A_2 \ \text{has two components}
\end{cases} 
\end{equation*}

\subsection{Solid torus like curves}\label{sub:STL}
Having introduced peg-board diagrams, we pause to discuss two key classes of immersed curves that arise. These are illustrated by the following two examples:

\begin{example}
\label{ex:twisted I}
(a) If $M$ is the solid torus $S^1\times D^2$, then $\curves M$ is the longitude $\partial D^2\times\{\text{pt}\}$ decorated with the trivial 1-dimensional local system.

(b) If $M$ is the twisted $I$-bundle over the Klein bottle, \(\CFD(M)\) was computed in \cite{BGW2013}; see also \cite{HW}. In our conventions, as an $\Alg$-decorated graph, it has two connected components that, with an appropriate choice of bordered structure, can be described as follows: 

\[ 
\begin{tikzpicture}[scale=0.75,>=stealth', thick] 
\def \n {2}
\def \radius {1.25cm}
\def \margin {5}
 \def \bigmargin {10}
\foreach \s in {1,...,\n}
{
  \node at ({360/\n * (\s - 1)}:\radius) {$\bu$};
  \node at ({360/(\n) * (\s- (1/2))}:\radius) {$\scriptstyle{12}$};
  \draw({360/\n * (\s - 1)+\margin}:\radius) arc ({360/\n * (\s - 1)+\margin}:{360/\n * (\s-(1/2))-\bigmargin}:\radius);
  \draw[->] ({360/\n * (\s - (1/2))+\bigmargin}:\radius) arc ({360/\n * (\s - (1/2))+\bigmargin}:{360/\n * (\s)-\margin}:\radius);
}
\end{tikzpicture}
\qquad
\begin{tikzpicture}[scale=0.75,>=stealth', thick]
\def \n {4}
\def \radius {1.25cm}
\def \margin {5} 
 \def \bigmargin {10}
\def \s{1}
  \node at ({360/\n * (\s - 1)}:\radius) {$\bu$};
  \node at ({360/(\n) * (\s- (1/2))}:\radius) {$\scriptstyle{1}$};
  \draw({360/\n * (\s - 1)+\margin}:\radius) arc ({360/\n * (\s - 1)+\margin}:{360/\n * (\s-(1/2))-\bigmargin}:\radius);
  \draw[->] ({360/\n * (\s - (1/2))+\bigmargin}:\radius) arc ({360/\n * (\s - (1/2))+\bigmargin}:{360/\n * (\s)-\margin}:\radius);

 \node at ({360/\n * (\s+1 - 1)}:\radius) {$\circ$};
  \node at ({360/(\n) * (\s+1- (1/2))}:\radius) {$\scriptstyle{3}$};
  \draw[<-]({360/\n * (\s +1- 1)+\margin}:\radius) arc ({360/\n * (\s +1- 1)+\margin}:{360/\n * (\s+1-(1/2))-\bigmargin}:\radius);
  \draw ({360/\n * (\s+1 - (1/2))+\bigmargin}:\radius) arc ({360/\n * (\s+1 - (1/2))+\bigmargin}:{360/\n * (\s+1)-\margin}:\radius);
  
   \node at ({360/\n * (\s+2 - 1)}:\radius) {$\bu$};
  \node at ({360/(\n) * (\s+2- (1/2))}:\radius) {$\scriptstyle{123}$};
  \draw({360/\n * (\s +2- 1)+\margin}:\radius) arc ({360/\n * (\s +2- 1)+\margin}:{360/\n * (\s+2-(1/2))-\bigmargin}:\radius);
  \draw[->] ({360/\n * (\s+2- (1/2))+\bigmargin}:\radius) arc ({360/\n * (\s+2 - (1/2))+\bigmargin}:{360/\n * (\s+2)-\margin}:\radius);
  
   \node at ({360/\n * (\s +3- 1)}:\radius) {$\circ$};
  \node at ({360/(\n) * (\s+3- (1/2))}:\radius) {$\scriptstyle{2}$};
  \draw({360/\n * (\s+3 - 1)+\margin}:\radius) arc ({360/\n * (\s+3 - 1)+\margin}:{360/\n * (\s+3-(1/2))-\bigmargin}:\radius);
  \draw[->] ({360/\n * (\s+3 - (1/2))+\bigmargin}:\radius) arc ({360/\n * (\s +3- (1/2))+\bigmargin}:{360/\n * (\s+3)-\margin}:\radius);
 
\end{tikzpicture}
\]

Both components are loop-type, and it is easy to see that $\curves M$ consists of the two immersed curves shown in Figure \ref{fig:alt-twisted-I}, each decorated with the trivial local system.
\end{example}

\piccaption[]{Pulling the invariant associated with the twisted $I$-bundle over the Klein bottle tight.  \label{fig:alt-twisted-I}}
\parpic[r]{
 \begin{minipage}{35mm}
 \centering
\includegraphics[scale=0.5]{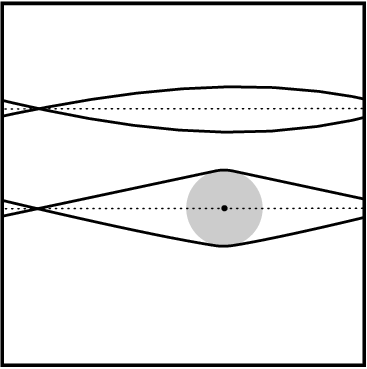}
  \end{minipage}%
  }
  
  The invariant of the solid torus pulls tight to a singular pegboard diagram lying on a line, and is loose. In contrast,  both curves for the twisted $I$-bundle over the Klein bottle lie on a single line when you pull them tight, but only one of them is a loose curve. The other curve (the lower one in the figure) has two \(\pi\)-corners. 
  The manifolds $M$ for which every component of $\curves M$ has this form form an interesting class. Consider the following definition, which was first made in \cite{solid-tori} (see also \cite{HW}):

\begin{definition}
Suppose \(M\) is a three-manifold with torus boundary, that \(\lambda \in H_1(\partial M)\) is the rational longitude, and that \(\mu \cdot \lambda = 1\). We say that 
$M$ is a \emph{Heegaard Floer homology solid torus} if 
$\CFD(M,\mu,\lambda)\cong\CFD(M,\mu+\lambda,\lambda)$. 
\end{definition}

A longer discussion of this class of manifolds is given in \cite[Section 1.5]{HRW-companion}. For our purposes, this definition can be rephrased in terms of curves; compare \cite[Theorem 27]{HRW-companion}. 

\begin{proposition}
\label{prop:HFST}
$M$ is a {Heegaard Floer homology solid torus} if and only if  the immersed multicurve $\curves M$ is homotopic to a curve that lies in a small neighbourhood of a representative  \(\lambda\) running through the puncture. 
\end{proposition}

Before giving the proof, we discuss the action of Dehn twists on an immersed closed curve. 
Let \((T, \alpha, \beta)\) be  a marked torus.  An immersed closed curve \(\gamma \subset T\) determines a reduced cyclic word, which we also denote by \(\gamma\), in \(\pi_1(T) \cong \langle \alpha, \beta \rangle\), where the generators \(\alpha\) and \(\beta\) are  closed curves dual to the parametrizing arcs. Let \(\ell(\gamma)\) be the length of this word. 

\begin{lemma}\label{lem:int-to-inf} If the intersection number \(i({\beta},\gamma)\neq 0\) then  
\(\ell(\tau^n_{{\beta}}(\gamma)) \to \infty\)  as \(n \to \infty.\) 
\end{lemma} 

\begin{proof}  Intersections of \(\gamma\) with the closed curve \({\beta}\) are in bijection with cyclic subwords of \(\gamma\) of the form \(\alpha \beta^{k}\alpha\) and \(\alpha^{-1} \beta^k \alpha^{-1}\) for \(k \in \Z\). The action of \(\tau^n_{\beta}\) replaces such subwords with \(\alpha \beta^{k+n}\alpha \) or \(\alpha^{-1}\beta^{k-n}\alpha^{-1}\), respectively. The resulting word is still cyclically reduced. As long is there as at least one intersection with   \({\beta}\),  its length goes to  \(\infty\) as \(n\) does, since the value of \(k\) in each subword is fixed. 
\end{proof}

\begin{proof}[Proof of Proposition~\ref{prop:HFST}]
Let \(\tau_\lambda\co \partial M \to \partial M\) be the Dehn twist along \(\lambda\). Then \(\tau_\lambda(\mu+\lambda) = \mu\) and \(\tau_\lambda(\lambda) = \lambda\), so \(M\) is a Heegaard Floer solid torus if and only if $\CFD(M,\mu,\lambda)\cong \widehat{T}_\lambda \boxtimes \CFD(M,\mu,\lambda)$, where \(\widehat{T}_\lambda\) is the Dehn twist bimodule considered in Section~\ref{sec:invariance}.  By the results of that section, this is equivalent to saying that \(\curves{M}\) is invariant under the action of \(\tau_\lambda\), {\it i.e.} that up to regular homotopy
\(\tau_\lambda(\curves{M})\) represents the same immersed multicurve as \(\curves{M}\). 

The hypothesis of the proposition is equivalent the statement that \(\curves{M}\) is homotopic to a curve that is disjoint from the closed curve representing \(\lambda\). If this is the case, \(\curves{M}\) is clearly preserved by the action of \(\tau_\lambda\). 

Conversely, suppose \(\curves{M}\) has a component \(\gamma\) for which \(i(\gamma, \lambda) \neq 0\). Then by Lemma \ref{lem:int-to-inf} \(\ell(\tau_\lambda^n(\gamma)) \to \infty \) as \(n \to \infty\) and  so, for sufficiently large \(n\), \(\tau_\lambda^n(\gamma)\) cannot be a component of \(\curves{M}\). It follows that \(\curves{M}\) is not fixed by the action of \(\tau_\lambda\), and \(M\) cannot be a Heegaard Floer solid torus. 
\end{proof}

If \(\curves M\) consists of a single loose curve with trivial local system in each \(\spinc\) structure, we say \(M\) is solid torus like. Gillespie studied such manifolds and showed that they are boundary compressible \cite[Corollary 2.9]{Gillespie}. More generally, 
we say \(\curves M\) is loose if each component of \(\curves M\) is a loose curve (possibly equipped with a nontrivial local system). Since \(\curves M\) lifts to \(\barT_M\), all such components must be parallel. 
We will prove the following generalization of Gillespie's result.

\begin{proposition}
	If \(\curves M\) is loose then \(M\) is boundary compressible. 
	\label{prop:loose}
	\end{proposition}

Before giving the proof, we recall some facts about knot Floer homology and the Thurston norm. 
Let \(\lambda\) be the rational longitude of \(M\), and suppose that \(\lambda\) has order \(k\) in \(H_1(M)\). 
Choose a meridian \(\mu \in H_1(\partial M)\) with \(\mu \cdot \lambda = 1\).  Then \(\mu\) is \(k\) times a primitive element of  \(H_1(M)\) . 

Fix a class \(x \in H_2(M,\partial M)\) with \(\partial x = \lambda\) and define
$$\HFK(K_\mu,x,i) = \bigoplus_{\langle c_1(\mathfrak{t}),x \rangle = 2i} \HFK(K_\mu, \mathfrak{t})$$
where the sum runs over \(\mathfrak{t} \in \spinc(M,\partial M)\). 
Then \[2 \cdot  \max \{i \,|\, \HFK(K,x,i)\neq 0\} = \|x\|_\mu\]
Here \(\|\cdot\|_\mu\) is the generalized Thurston norm given by 
\[\|x\|_\mu =\min_{[S]=x} c_\mu(S),\]
where the minimum is taken over all properly embedded orientable surfaces representing \(x\). If \(S\) is connected, \(c_\mu(S) = \max\{-\chi(S)+|\partial S \cdot \mu|, 0\}\); otherwise  
\(c_\mu(S) = \sum c_\mu(S_i)\), where \(S_i\) are the connected components of \(S\).  

As we described in Section~\ref{subsec:heights}, the \(\spinc\) decomposition of  \(\HFK(K_,\mu)\) can be computed by intersecting 
 \(\curves{M,\spin}\) with lifts of  \(\mu\) to  \(\barT_{M,\spin}\), where \(\spin\) runs over \(\spinc(M)\). 
The choice of \(x\) determines a height function on \(\barT_{M,\spin}\) such that intersections of \(\curves{M,\spin}\) with lifts of the meridian at height \(i\)  correspond to elements of \( \HFK(K_\mu,x,i)\). 
	 
Similarly, if \(Y\) is a closed manifold, and  \(y \in H_2(Y)\) is a primitive class, we define 
$$\HFhat(Y,y,i) = \bigoplus_{\langle c_1(\spin),[\Sigma]\rangle = 2i} \HFhat(Y, \spin).$$
Then  \( \|y\| =2 \cdot  \max \{i \, | \, \HFhat(Y,y,i) \neq 0\}\). 

\begin{proof}[Proof of Proposition~\ref{prop:loose}]

   Fix a class \(x\) as above, and choose a norm-minimizing surface \(S\) realizing \(x\). After tubing to remove excess boundary components, we may assume that \(\partial S\) consists of precisely \(k\) parallel copies of \(\lambda\). Since \(\lambda\) has order \(k\) in \(H_1(M)\), all \(k\) boundary components lie on a single component of \(S\). Let \(S_0\) be the component with boundary, and consider the modified complexity \(\widetilde{c}(S) = -\chi(S_0) + \sum_{i>0} c(S_i)\), where the sum runs over the other components of \(S\). Note that 
   \(\widetilde{c}(S) = c(S)\) unless \(k=1\) and \(S_0=D^2\).

By hypothesis, \(\curves{M}\) is a union of loose curves in \(T_M\). Each individual curve lifts to some \(\barT_{M,\spin}\) as a circle parallel to \(\lambda\). This circle is  at some height \(n \in \Z\) with respect to the height function determined by \(x\). 
Let \(n_+\) and \(n_-\) be the heights of the highest and lowest circles in \(\curves{M}\). 
Then \(n_+\) is the largest value of 
\(i\)  for which \(\HFK (K_\mu,x,i)\)  is nontrivial, and \(n_-\) is the lowest. Since knot Floer homology is conjugation-symmetric,  we must have \(n_-=-n_+\).  The relation between \(\HFK\) and the Thurston norm implies that \(\widetilde{c}(S) = 2n_+ -k \). 

Now consider the closed manifold \(Y = M \cup_h M\), where \(h(\lambda) = - \lambda\) and \(h(\mu) = \mu\), and let \(y \in H_2(Y)\) be the class obtained by doubling \(x\). Then \(\HFhat(Y,y,i)\) is calculated by pairing \(\curves{M,\spin} \) with \(\alpha \cdot \overline{h}(\curves{M,\spin'})\), where \(\spin, \spin'\) run over \(\spinc(M)\), and \(\alpha\) runs over all deck transformations which shift height by \(i\). The largest value of \(i\) for which \(\HFhat(Y,y,i)\) is potentially nontrivial is obtained by pairing the highest set of curves in \(\curves{M}\) with the lowest set of curves in \(\overline{h}(\curves{M})\). The latter is simply the image of the highest set of curves in \(\curves{M}\) under \(\overline{h}\), so the maximal value of \(i\) is 
\(n_+-n_- = 2n_+\).  We can view the highest set of curves as a single local system \((V,A)\) on \(S^1\), so  $$\HFhat(Y,y,2n_+) = HF (S^1_{(V,A)},S^1_{(V,A)}).$$ The latter group is nonzero by Corollary~\ref{cor:ComputeHF}.  Since \(\HFhat(Y,y,2n_+) \neq 0\), we see that \(\|y\| = 4n_+\).
  
   On the other hand, the double \(DS\) of \(S\) represents \(y\) and has complexity \(2c(S)\). If \(M\) is boundary incompressible, then \(S_0 \neq D^2\), so \(c(S) = \widetilde{c}(S) = 2n_+-k\). In this case, we would have 
  \(c(DS)=4n_+-2k <\|y\|\), which is impossible. We conclude that \(S_0 = D^2\) and  \(M\) was boundary compressible. 	
	\end{proof}

\begin{remark}
The proof shows that if \(\curves{M}\) is loose then it consists of curves that are all at a fixed height (up to overall shift, we may assume that these are all at height \(0\)). 
	Loose manifolds with curves at other heights do not exist. One possible explanation for this would be that \(\curves{M,\spin}\)  should represent an exact Lagrangian in \(\barT_M\). 
\end{remark}


\subsection{A dimension inequality for pinches} 

Consider the following construction: Let $Y$ be a rational homology sphere of the form $M_0\cup_h M_1$, so that each $M_i$ is a rational homology solid torus. Denote by $\lambda_i$ the rational longitude in each $\partial M_i$; recall that this slope is characterized by the property that some number of like oriented copies of $\lambda_i$ bounds a properly embedded surface in $M_i$. The Dehn surgery $Y_0=M_0(h(\lambda_1))$ is the result of pinching $M_1$ in $Y$. Note that $h(\lambda_1)$ is a well defined slope in  $\partial M_0$.
 In certain settings there is an associated degree $n$ map $Y\to Y_0$, where 
\(n\) is the order of  $i_*([\lambda_1])$ in $H_1(M_1)$. In particular, when both $M_0$ and $M_1$ are integer homology soild tori, there is always a degree one map associated with a pinch.

\begin{theorem}\label{thm:pinch}
Given a rational homology sphere $Y=M_0 \cup_h M_1$, where $M_0$ and $M_1$ are rational homology solid tori, there is an inequality 
\[\dim \HFhat(Y) \ge n m \dim \HFhat(Y_0)\]
where $Y_0=M_0(h(\lambda_1))$, $n$ is the order of $i_*([\lambda_1])$ in $H_1(M_1)$, and $m=|\spinc(M_1)|=|H_1(M_1,\partial M_1)|$. 
\end{theorem}

The proof of Theorem \ref{thm:degree-one-pinch} follows from a special case of Theorem \ref{thm:pinch} (with $n=m=1$). 



\begin{proof}[Proof of Theorem \ref{thm:pinch}]
Let $\boldsymbol{\gamma}_0$ and $\boldsymbol{\gamma}_1$ denote the decorated curves $\curves{M_0}$ and $\bar{h}(\curves{M_1})$, respectively, in $T_{M_0}=\partial M_0\setminus z$. By the gluing theorem, $\HFhat(Y)$ is equivalent to the Floer homology $\HFa(\boldsymbol{\gamma}_0, \boldsymbol{\gamma}_1)$ and in particular the dimension is at least the intersection number $i(\boldsymbol{\gamma}_0, \boldsymbol{\gamma}_1)$  between $\boldsymbol{\gamma}_0$ and $\boldsymbol{\gamma}_1$. 
Let $\frac{p}{q}$ be the slope of the curve in $T_{M_0}$ determined by $h(\lambda_1)$, and let $L_{p,q}$ be the corresponding simple closed curve $p\alpha_0 + q\beta_0$. Since $Y$ is a rational homology sphere we may assume that the slope of $\lambda_0$ is not $\frac{p}{q}$; in particular, no component of $\boldsymbol{\gamma}_0$ is commensurable to $L_{p,q}$, and by the gluing theorem $\dim \HFhat(Y_0) = i(\boldsymbol{\gamma}_0, L_{p,q})$.

Let $\hat{\boldsymbol{\gamma}}_1$ be the collection of curves obtained from $\boldsymbol{\gamma}_1$ by replacing any non-trivial local systems with trivial ones of the same dimension and pulling the curves tight in the torus $\partial M_0$ (ignoring the basepoint). More precisely, after replacing non-trivial local systems, we delete any curve components that are nullhomotopic  in the (non-punctured) torus $\partial M_0$, and we replace the remaining curves with minimal length representatives of their homotopy class in the flat torus. The resulting collection of curves is loose and represents the same homology class in $\partial M_0$ as $\boldsymbol{\gamma}_1$; we call $\hat{\boldsymbol{\gamma}}_1$ the {\it loose representative} of $\boldsymbol{\gamma}_1$. Recall that for each spin$^c$-structure $\spin$ of $M_1$, the multicurve $\curves{M_1, \spin}$ is homologous, ignoring the basepoint, to the homological longitude $n \lambda_1$; see Corollary \ref{crl:pull-tight-to-longitude}. Thus, for each spin$^c$-structure of $M_1$, the relevant components of $\boldsymbol{\gamma}_1$ are homologous in $\partial M_0$ to $n$ copies of $L_{p,q}$. It follows that $\hat{\boldsymbol{\gamma}}_1$ contains at least $nm$ copies of $L_{p,q}$, so that
\[i(\boldsymbol{\gamma}_0, \hat{\boldsymbol{\gamma}}_1) \ge n m \cdot i(\boldsymbol{\gamma}_0, L_{p,q}) = n m \dim\HFhat(Y_0).\]
It remains to show that 
$i(\boldsymbol{\gamma}_0, \boldsymbol{\gamma}_1) \ge i (\boldsymbol{\gamma}_0, \hat{\boldsymbol{\gamma}}_1)$. 


\labellist
\tiny
\pinlabel $\boldsymbol{\gamma}_1$ at 16 71 
\pinlabel $\hat{\boldsymbol{\gamma}}_1$ at 280 15
\endlabellist
\piccaption[]{Sliding a corner over the peg. The corner wraps around a circle of some radius between $\epsilon$ and $2\epsilon$ (the inner and outer radii of the shaded annulus). We replace the segment of curve through the annulus with a segment on the opposite side of the peg following the circle $C_{2\epsilon}$. \label{fig:main-pinch-move}}
\parpic[r]{
 \begin{minipage}{60mm}
 \centering
\includegraphics[scale=0.5]{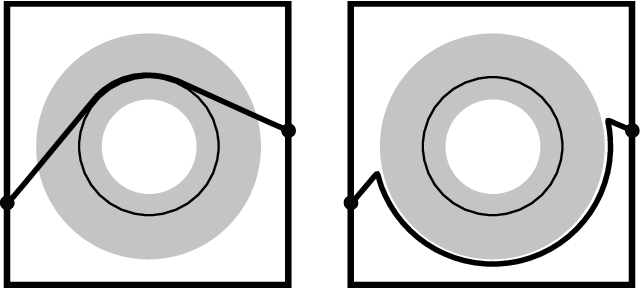}
 \end{minipage}%
  }We will fix a small $\epsilon$ and work in an $\epsilon$-pegboard diagram. Note that $\hat{\boldsymbol{\gamma}}_1$ can be obtained from ${\boldsymbol{\gamma}}_1$ by a sequence of the following moves: (i) replacing a local system with a trivial one of the same dimension; (ii) resolving a self-intersection to split off a closed component; (iii) deleting a component, (iv) homotopy in $T_{M_0}$ to put curves in $\epsilon$-pegboard position; and (v) passing the curve through a peg near a corner in the peg board diagram, as in Figure \ref{fig:main-pinch-move}. Moves (i) and (ii) clearly do not change the number of intersections with $\gamma_0$ at all, and move (iii) cannot increase the intersection number. Move (iv) also does not increase intersection number, since curves in an an $\epsilon$-pegboard diagram realize the minimal intersection for their homotopy classes in $T$. The remaining move, (v), requires some care, since it may introduce new intersection points; see Figure \ref{fig:types}.
  
After removing nullhomologous components, no component of ${\boldsymbol{\gamma}}_1$ is homotopic to a component of ${\boldsymbol{\gamma}}_0$; it follows that ${\boldsymbol{\gamma}}_0$ and ${\boldsymbol{\gamma}}_1$ intersect minimally and transversely in the $\epsilon$-pegboard diagram. By first applying (ii) as needed, we may assume that at each corner of ${\boldsymbol{\gamma}}_1$ the curve changes direction by at most an angle of $\pi$; that is, there is no peg-wrapping in $\gamma_1$. Now fix a corner $c_1$ of $\gamma_1$ and consider the effect of applying (v) at this corner; let $\hat{c}_1$ denote the modified corner segment along the circle $C_{2\epsilon}$. Since ${\boldsymbol{\gamma}}_1$ is unchanged outside the disk of radius $2\epsilon$, to compute the new intersection number we only need to consider the intersection of $\hat{c}_1$ with the corners of ${\boldsymbol{\gamma}}_0$. For each corner of ${\boldsymbol{\gamma}}_0$, replacing $c_1$ with $\hat{c}_1$ either adds two intersection points, preserves intersection number, or removes two intersection points (see Figure \ref{fig:types}). Replacing $c_1$ with $\hat{c}_1$ will not increase the overall intersection number as long as ${\boldsymbol{\gamma}}_0$ has at least as many type (d) corners as type (a) corners.

\begin{figure}[h]
\labellist
\tiny
\pinlabel $c_1$ at 70 100 
\pinlabel $\hat{c}_1$ at 70 14 
\pinlabel $c_0$ at 80 83

\pinlabel $c_1$ at 250 100 
\pinlabel $\hat{c}_1$ at 250 14 
\pinlabel $c_0$ at 256 36

\pinlabel $c_1$ at 430 85 
\pinlabel $\hat{c}_1$ at 430 14
\pinlabel $c_0$ at 460 105

\pinlabel $c_1$ at 610 100 
\pinlabel $\hat{c}_1$ at 610 14 
\pinlabel $c_0$ at 599 50
\endlabellist
\includegraphics[scale=0.5]{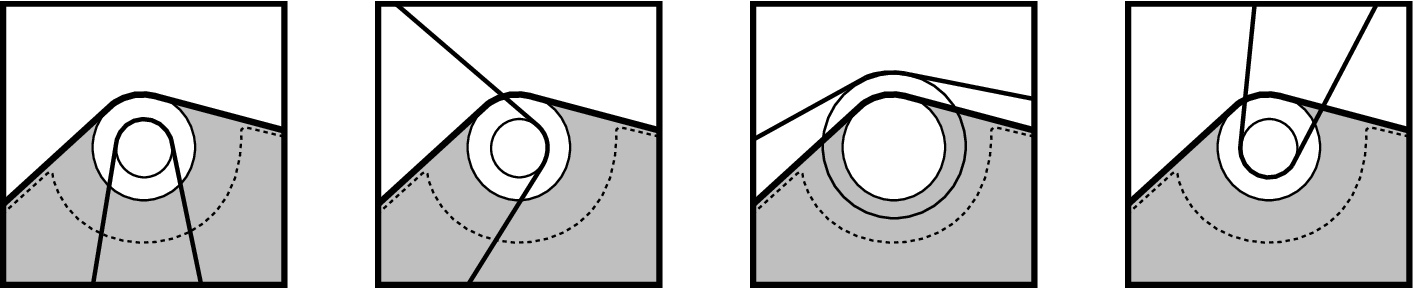}

(a) \hspace{2.45 cm} (b) \hspace{2.45 cm} (c) \hspace{2.45 cm} (d)

\caption{For a fixed corner $c_1$ of $\boldsymbol{\gamma}_1$, a corner $c_0$ of $\boldsymbol{\gamma}_0$ has one of four types. The inside of the corner $c_1$ is shaded, and the dashed line indicates the modified corner $\hat{c}_1$ resulting from move (v).}
\label{fig:types}
\end{figure}

\labellist
\tiny
\pinlabel $s_{\text{in}}$ at -9 39
\pinlabel $s_{\text{out}}$ at 157 77 
\pinlabel {\bf inside} at 72 15
\pinlabel $s_{\text{in}}$ at 173 39
\pinlabel $s_{\text{out}}$ at 339 77
\pinlabel $\bar s_{\text{in}}$ at 174 53 
\pinlabel $\bar s_{\text{out}}$ at 340 62
\pinlabel $\phi$ at 254 91
%
\endlabellist
\begin{figure}[h]
\includegraphics[scale=0.5]{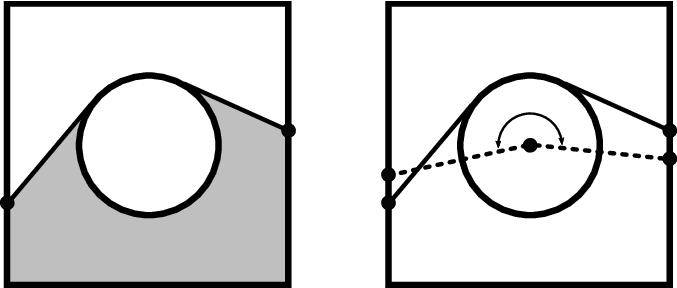} 
\caption{A square tile representing the torus $T_{M_0}$ together with the section $s$ of $\boldsymbol{\gamma}_1$ containing a chosen corner $c_1$. The inside of $s$ is determined by containment of the peg, as shown on the left; on the right, the corresponding segment $\bar{s}$ in a singular pegboard diagram. Note that $s$ lies in a neighborhood of width $4\epsilon$ around $\bar{s}$. We are most interested in the case that the angle $\phi$ associated with $c_1$ in the singular pegboard diagram is strictly bigger than $\pi$ (that is, $c_1$ is not a $\pi$-corner), so that $\bar{s}$ is not a straight line. In this case, the inside of $s$  covers strictly less than half the boundary of the square.}
\label{fig:parts}
\end{figure}


To analyze this situation, we decompose $T_{M_0}$ into a square tile (obtained by taking the closure of $T_{M_0}\smallsetminus (\alpha_0 \cup \beta_0)$), and consider the component of $\boldsymbol{\gamma}_1$ containing the chosen corner $c_1$ as in Figure \ref{fig:parts}. This segment of curve, which we call $s$, begins on the boundary of the square tile at at $s_{\text{in}}$, wraps some amount around a circle with radius between $\epsilon$ and $2\epsilon$ centered at the puncture, then exits the square tile at $s_{\text{out}}$. Note that $s$ divides the square tile into two sides; we call the side that contains the peg the {\it inside} of the corner $c_1$. Similarly, any corner $c_0$ of $\boldsymbol{\gamma}_0$ can be extended to a segment through the square tile; we will call this segment $x$, with endpoints $x_{\text{in}}$ and $x_{\text{out}}$ on the boundary of the tile. We can now more precisely classify corners of $\boldsymbol{\gamma}_0$ into the four types shown in Figure \ref{fig:types} by considering how the segment $x$ interacts with $s$. We say the corner $c_0$ is of type (a) if both endpoints lie on the inside of $c_1$, type (b) if exactly one endpoint lies inside $c_1$, type (c) if both endpoints lie outside $c_1$ and $c_0$ turns the same direction as $c_1$ with no peg-wrapping, and type (d) if both endpoints lie outside $c_1$ and $c_0$ wraps the peg or turns the opposite direction from $c_1$. Note that for type (a) corners, $c_0$ necessarily changes direction more than $c_1$ and so wraps closer to the puncture, giving the nested configuration in the figure. In this case $c_0$ is disjoint from $c_1$ but intersects the modified corner $\hat{c}_1$ twice. For type (b) corners, $c_0$ may turn either more or less tightly than $c_1$; in either case, $c_0$ intersects both $c_1$ and $\hat{c_1}$ once. For type (c) corners, $c_1$ must be nested inside of $c_0$, so $c_0$ is disjoint from both $c_1$ and $\hat{c}_1$. Finally, type (d) corners have two intersections with $c_1$ but are disjoint from $\hat{c}_1$.  Observe that replacing $c_1$ with $\hat{c}_1$ will not increase the overall intersection number as long as ${\boldsymbol{\gamma}}_0$ has at least as many type (d) corners as type (a) corners.

\labellist
\tiny
\pinlabel $c^i$ at 73 75  
\pinlabel $x^i_{\text{in}}$ at 48 -7 
\pinlabel $x^i_{\text{out}}$ at 106 -7
\pinlabel $x^{i-1}_{\text{out}}$ at 50 154
\pinlabel $x^{i+1}_{\text{in}}$ at 108 154
\endlabellist
\piccaption{Labelling endpoints for corners of $\boldsymbol{\gamma}_0$.
 \label{fig:in-out}}
\parpic[r]{
 \begin{minipage}{30mm}
 \centering
\includegraphics[scale=0.5]{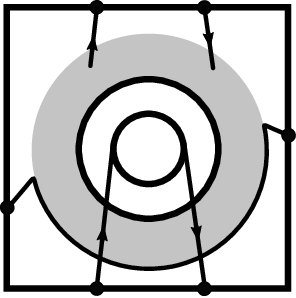}
 \end{minipage}%
  } We now need to consider two cases with different behavior, depending on whether the curve changes direction at the corner $c_1$ in a singular pegboard diagram. Let $\bar{s}$ denote the piecewise linear segment of curve through the square tile and containing the corner $c_1$ in a singular peg-board diagram for $\boldsymbol{\gamma}_1$; in other words, $\bar{s}$ is the limit of $s$ if $\epsilon$ is taken to zero. We denote the endpoints of $\bar{s}$ by $\bar{s}_{\text{in}}$ and $\bar{s}_{\text{out}}$. We first consider the case that the singular peg-board representative of $\boldsymbol{\gamma}_1$ changes direction at $c_1$; that is, the segment $\bar{s}$ is not a straight line. In this case $\bar{s}_{\text{in}}$ and $\bar{s}_{\text{out}}$ cut the boundary of the square tile into two unequal pieces, with the smaller piece corresponding to the inside of $c_1$. Note that $s_{\text{in}}$ and $s_{\text{out}}$ are close to $\bar{s}_{\text{in}}$ and $\bar{s}_{\text{out}}$, respectively, where close means within a small multiple of $\epsilon$; in particular, since $s$ lies in the $2\epsilon$-neighborhood of $\bar s$ and since $\bar s$ meets the boundary of the square in an angle of at least $\tfrac \pi 4$, the distances are bounded by $2\sqrt{2} \epsilon$. Since we can take $\epsilon$ to be arbitrarily small, it follows that the inside of $c_1$ covers strictly less than half the boundary of the square tile. We now run through all the corners $\{c_0^i\}$ of $\boldsymbol{\gamma}_0$, labeling the endpoints of the corresponding curve segments in the square tile by $x^i_{\text{in}}$ and $x^i_{\text{out}}$, as in Figure \ref{fig:in-out} (the index $i$ should be interpreted cyclically, with the corners ordered as they appear traversing the curve). Note that the curve $\boldsymbol{\gamma}_0$ may pass through the square tile (without interacting with the peg) many times between $x^i_{\text{out}}$ and $x^{i+1}_{\text{in}}$; we ignore these portions of the curve.  Since $\boldsymbol{\gamma}_0$ does not change direction between corners, the slope of the segment leaving the tile at $x^i_{\text{out}}$ agrees with the slope of the segment entering the tile at $x^{i+1}_{\text{in}}$. It follows that $x^{i+1}_{\text{in}}$ is close to the antipodal point to $x^i_{\text{out}}$ on the boundary of the square (i.e. the point obtained by a half-rotation about the center of the square), where again close means within a small fixed multiple of $\epsilon$. More precisely, if we consider the analogous points $\bar{x}^i_{\text{out}}$ and $\bar{x}^{i+1}_{\text{in}}$ coming from a singular pegboard representative of $\boldsymbol{\gamma}_1$, we observe that these are antipodal points since the segments connecting them to the center of the square have the same slope. Because $\boldsymbol{\gamma}_0$ lies in a $2\epsilon$-neighborhood of the singular peg-board representative, $x^i_{\text{out}}$ is within $2\sqrt{2}\epsilon$ of $\bar{x}^i_{\text{out}}$, $x^{i+1}_{\text{in}}$ is within $2\sqrt{2}\epsilon$ of $\bar{x}^{i+1}_{\text{in}}$, and  $x^{i+1}_{\text{in}}$ is a distance of less than $4\sqrt{2}\epsilon$ from the antipodal point to $x^i_{\text{out}}$. For sufficiently small $\epsilon$, it follows that if $x^i_{\text{out}}$ is on the inside of $c_1$ then $x^{i+1}_{\text{in}}$ must be on the outside of $c_1$ and not close to $s_{\text{in}}$ or $s_{\text{out}}$ in the aforementioned sense. We also observe that if $c_0^i$ is a type (c) corner and $x^i_{\text{in}}$ is not close to $s_{\text{in}}$ or $s_{\text{out}}$, then $x^{i+1}_{\text{in}}$ is on the outside of $c_1$ and not close to $s_{\text{in}}$ or $s_{\text{out}}$. This follows from the fact that $c_1$ lies entirely inside $c_0$ and the antipodal point to $x^i_{\text{out}}$ lies either outside $c_0$ or close to $x^i_{\text{in}}$. In particular, if $c_0^i$ is type (a), then $x^{i+1}_{\text{in}}$ must lie outside of $c_1$ and $c_0^{i+1}$ must be type (b), (c), or (d). Moreover, if it is type (b) or (c) then $x^{i+2}_{\text{in}}$ lies outside $c_1$, and we can repeat the argument with $c_0^{i+2}$; we must reach a type (d) corner before reaching another type (a) corner. Thus there are at least as many type (d) corners as type (a) corners in $\boldsymbol{\gamma}_0$, and applying move (v) at the corner $c_1$ does not increase the total intersection number.
\labellist
\tiny
\endlabellist
\piccaption{Segments of $\boldsymbol{\gamma}_0$ interact with a neighbourhood of $\boldsymbol{\gamma}_1$ in two distinct ways.  \label{fig:pi-corner}}
\parpic[r]{
 \begin{minipage}{30mm}
 \centering
\includegraphics[scale=0.5]{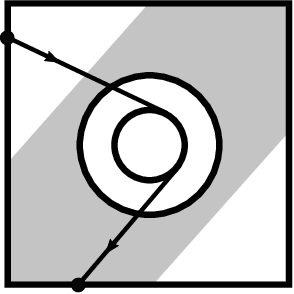}
 \end{minipage}%
  }It remains to deal with the case that the corner $c_1$ is straight in a singular pegboard diagram. The argument above breaks down in this case, since the inside of the corner may cover more than half of the boundary of the square tile, and indeed applying move (v) to a single corner that is straight in the singular diagram may increase the total intersection number. By applying (v) at all corners that are not straight in the singular diagram, along with moves (i)-(iv), we can reduce to the case that all corners of $\boldsymbol{\gamma}_1$ are straight; that is, $\boldsymbol{\gamma}_1$ is a Heegaard Floer solid torus curve. We now argue directly that for such a curve the loose representative $\hat{\boldsymbol{\gamma}}_1$ realizes the minimal intersection with $\boldsymbol{\gamma}_0$. To see this consider an annular neighbourhood of $\boldsymbol{\gamma}_1$ containing the radius $2\epsilon$ circle around the peg; see Figure \ref{fig:pi-corner}. Now decomposing $\boldsymbol{\gamma}_0$ into segments connecting the corners $\{c_0^i\}$ shows two distinct behaviours. We may ignore segments of $\boldsymbol{\gamma}_0$ that are contained in the neighbourhood of $\boldsymbol{\gamma}_1$, since the loose representative $\hat{\boldsymbol{\gamma}}_1$ does not intersect these segments of $\boldsymbol{\gamma}_0$. We consider then only the intersections with those segments that leave the  neighbourhood of $\boldsymbol{\gamma}_1$, and observe that the original curve $\boldsymbol{\gamma}_1$ and the loose representative $\hat{\boldsymbol{\gamma}}_1$ have the same number of intersections with these segments. As a result, we can replace $\boldsymbol{\gamma}_1$ with $\hat{\boldsymbol{\gamma}}_1$ without increasing the number of intersection points. 
 \end{proof}

\subsection{Heegaard Floer homology for toroidal manifolds} As another immediate consequence of the geometric interpretation of bordered Floer invariants, we will now prove Theorem \ref{thm:total-dim-toroidal}, which states that if $Y$ contains an essential torus that is separating then $\dim \HFhat(Y) \ge 5$.

\begin{proof}[Proof of Theorem \ref{thm:total-dim-toroidal}]
By cutting along the torus, realize $Y$ as $M_0 \cup_h M_1$ where $M_0$ and $M_1$ are manifolds with incompressible torus boundary and $h\co\partial M_1 \to \partial M_0$ is a diffeomorphism. Let $\boldsymbol{\gamma_0} = \curves{M_0}$ and $\boldsymbol{\gamma_1} = \bar h(\curves{M_1})$ be the relevant decorated immersed multicurves in $\partial M_0 \setminus z_0$. 

We first set some homological prerequisites in place. A manifold with torus boundary $M$ has rational longitude $\lambda$ and homological longitude $p\lambda$. In particular, $p\lambda$ bounds a properly embedded surface in $M$; and $\langle\lambda\rangle$ generates a  $\Z/p\Z$  subgroup in $H_1(M;\Z)$. Applying universal coefficients and Poincar\'e duality then gives a  $\Z/p\Z$  subgroup in $H^2(M;\Z)$, which in turn guarantees $p$ torsion spin$^c$-structures on $M$ that we denote $\{\spin_i\}_{i=0}^{p-1}$. Finally, observe that for any choice of meridian $\mu$, that is, a slope at distance one from $\lambda$, the Dehn filling $M(\mu)$ inherits $\Z/p\Z\subset H^2(M(\mu);\Z)$; by abuse of notation, we denote the associated $p$ torsion spin$^c$-structures on $M(\mu)$ by $\{\spin_i\}_{i=0}^{p-1}$ as well. With this homological data in hand, we turn to Heegaard Floer homology. In particular, note that $\HFhat(M(\mu),\spin_i)$ is non-trivial for each $i=0,\ldots,p-1$. This appears to be known to experts, and appeals to work of Lidman \cite{LidmanHFinfty}; an argument, due to the second author, is written down in \cite[Proposition 8.9]{KWZ-thin}. In particular, by applying our pairing theorem, this non-vanishing result guarantees that there is at least one curve in $\HFhat(M,\spin_i)$ for each $i=0,\ldots,p-1$. Moreover, if we forget the punctures, this curve is homologous to  $p\lambda$; compare Corollary \ref{crl:pull-tight-to-longitude}. 

As in the proof of Theorem \ref{thm:pinch}, there are two somewhat distinct cases to consider, with Heegaard Floer solid torus type curves requiring a different approach. Recall that a corner of $\boldsymbol{\gamma_i}$ is a $\pi$-corner if the associated angle in a singular pegboard diagram is exactly $\pi$, so that the corner pulls tight to a straight line, and that we call all other corners true corners. If $\boldsymbol{\gamma_i}$ has no true corners then $M_i$ is a Heegaard Floer solid torus. Note that if $\boldsymbol{\gamma_i}$ has one true corner, then it must have at least two true corners. This is because if $\boldsymbol{\gamma_i}$, a closed curve, changes direction at one corner it must change direction again to get back to its initial slope. The one potential exception is if $\boldsymbol{\gamma_i}$ contains a single true corner wrapping the peg an integer number of full wraps so that the curve does not change direction; however, it follows from Corollary \ref{cor:tangent_slopes} in the next section that such curves do not arise.

A key observation is that, if $\boldsymbol{\gamma_0}$ is not a Heegaard Floer solid torus curve, the minimal intersection number between $\boldsymbol{\gamma_0}$ and $\boldsymbol{\gamma_1}$ satisfies
\begin{equation}\label{eq:intersection-lower-bound}
i(\boldsymbol{\gamma_0}, \boldsymbol{\gamma_1}) \ge 2 \#\{\text{true corners of }\boldsymbol{\gamma_1}\} + \#\{\pi\text{-corners of }\boldsymbol{\gamma_1}\} + \#\{\text{loose components of }\boldsymbol{\gamma_1}\}.
\end{equation}
To show this, we first consider a fixed true corner $c_1$ of $\boldsymbol{\gamma_1}$ and argue that the combined intersection number of $c_1$ with all the corners $\{c_0^i\}$ of $\boldsymbol{\gamma_0}$ is at least two. As in the proof of Theorem \ref{thm:pinch}, we assign the corners of $\boldsymbol{\gamma_0}$ types as in Figure \ref{fig:types}, and we extend each corner $c_0^i$ to a segment $x^i$ through the square tile with endpoints $x^i_{\text{in}}$ and $x^i_{\text{out}}$. If $\boldsymbol{\gamma_0}$ has a type (d) corner, we are done since this corner alone contributes two intersection points with $c_1$. If $\boldsymbol{\gamma_0}$ has a type (a) corner, then the argument in the proof of Theorem \ref{thm:pinch} shows there must also be a type (d) corner. By assumption $\boldsymbol{\gamma_0}$ has at least one true corner. If $\boldsymbol{\gamma_0}$ has a true corner of type (b), take this to be $c_0^i$ and orient the curve so that $x^i_{\text{out}}$ is on the inside of $c_1$. We then have that $x^{i+1}_{\text{in}}$, being within a small neighborhood of the antipodal point of $x^i_{\text{out}}$, lies outside of $c_1$ so that the next corner $c_0^{i+1}$ must have type (b), (c), or (d). If $c_0^{i+1}$ is type (b) or (d) we are done, since this contributes at least one more intersection with $c_1$ in addition to the one intersection from $c_0^i$. The only way that $c_0^{i+1}$ can have type (c) is if the segment from $c_0^i$ to $x^i_{\text{out}}$ coincides with one of the segments defining $c_1$ when we pass to a singular peg-board diagram and $c_0^{i+1}$ is a $\pi$-corner. In this case, the same reasoning shows that $c_0^{i+1}$ is type (b), (c), or (d) and can only be type (c) under the same conditions. Since we must eventually reach a true corner, we eventually reach a type (b) or (d) corner, contributing at least the second intersection with $c_1$ (note that if we reach a type (b) corner, it can not be the corner $c_0^i$ we started with since the incoming segment has the same slope as the outgoing segment of $c_0^i$, and $c_0^i$ was not a $\pi$-corner). Finally, if $\boldsymbol{\gamma_0}$ has a true corner of type (c) then there is a straight line segment through the puncture in the square tile such that both endpoints lie strictly outside of this corner. Up to reparametrization we may assume this line is horizontal, so that the corner represents a local maximum of $\boldsymbol{\gamma_0}$. Being a closed curve, $\boldsymbol{\gamma_0}$ must also have a local minimum, and a corner at which a local minimum occurs is necessarily type (d). Thus the true corner $c_1$ contributes at least two intersections to $i(\boldsymbol{\gamma_0}, \boldsymbol{\gamma_1})$. 

If the fixed corner $c_1$ of $\boldsymbol{\gamma_1}$ is a $\pi$-corner, a similar argument shows the contribution is at least one. Note that any type (b) or (d) corner in $\boldsymbol{\gamma_0}$ contributes at least one intersection point with $c_1$. Any corner that is type (c) with respect to $c_1$ must be a $\pi$-corner, and by assumption $\boldsymbol{\gamma_0}$ has at least one true corner, so we are left with considering a true type (a) corner $c_0^i$. In this case, the endpoints $x^i_{\text{in}}$ or $x^i_{\text{out}}$ lie on the inside of $c_1$, and at least one of them is strictly on the inside $c_1$ (in the sense that it is at least a distance of $4\sqrt 2 \epsilon$ from the points $s_{\text{in}}$ or $s_{\text{out}}$). Suppose $x^i_{\text{out}}$ lies strictly inside of $c_1$; it follows that $x^{i+1}_{\text{in}}$ lies strictly on the outside of $c_1$ and $c_0^{i+1}$ is type (b) or (d), contributing the needed intersection point. 

Finally, we show that that a loose component of $\boldsymbol{\gamma_1}$ contributes at least one intersection. Indeed, if it did not then $\boldsymbol{\gamma_0}$ would have to lie in a neighborhood of a line with the same slope as the loose component, but this would imply that $\boldsymbol{\gamma_0}$ is a Heegaard Floer solid torus. This completes the proof of equation~\eqref{eq:intersection-lower-bound}. 

We first consider the case that neither $M_0$ nor $M_1$ is a Heegaard Floer solid torus. Equation \eqref{eq:intersection-lower-bound} then implies that $i(\boldsymbol{\gamma_0}, \boldsymbol{\gamma_1}) \ge 4$. Suppose that $i(\boldsymbol{\gamma_0}, \boldsymbol{\gamma_1}) = 4$. By Equation \eqref{eq:intersection-lower-bound} we must have that $\boldsymbol{\gamma_1}$ contains a single component with exactly two corners, and thus exactly two segments in a singular pegboard diagram. Interchanging the roles of $M_0$ and $M_1$ we see that the same is true for $\boldsymbol{\gamma_0}$. It is also clear that neither curve can carry a non-trivial local system, since this would multiply the number of intersection points by the dimension of the local system. Let $p_1$ and $p_2$ be the slopes of the two segments in a singular peg-board diagram for $\boldsymbol{\gamma}_0$, and let $q_1$ and $q_2$  be the slopes of the two segments for $\boldsymbol{\gamma}_1$. Observe that, since neither curve is a Heegaard Floer solid torus, \(p_1 \neq p_2\)  and \(q_1 \neq q_2.\)


The intersection points counted so far all occur at the corners; as $\epsilon$ goes to zero in an $\epsilon$-pegboard diagram, these intersection points approach the puncture. We will now consider intersections away from the corners. We observe that any two segments of slope $r_1$ and $r_2$ in the singular pegboard diagram intersect $d$ times, where $d = \Delta(r_1, r_2)$ is the distance between the slopes (calculated by taking minimal geometric intersection number). One of these intersection points is at the endpoints, but the remaining $d-1$ intersection points occur in the interior of the segments; these intersections remain in an $\epsilon$-pegboard diagram and are a finite distance from the peg. It follows that if $i(\boldsymbol{\gamma}_0, \boldsymbol{\gamma}_1) = 4$ then $\Delta(p_1, q_1)$, $\Delta(p_1, q_2)$, $\Delta(p_2, q_1)$, and $\Delta(p_2, q_2)$ are all at most 1. 
First suppose that the four slopes $p_1, p_2, q_1,$ and $q_2$ are all distinct. Up to reparametrization, we may assume $p_1 = \infty$; it follows that $q_1$ and $q_2$ are distinct integers. Since $p_2$ is distance 1 from each of these and distinct from $\infty$, we must have $p_2 = n$, $q_1 = n-1$ and $q_2 = n+1$. This case is ruled out, because it would imply that the rational longitude $\lambda_1$ of $M_1$ has order two, but as above this implies the existence of another component and gives $i(\boldsymbol{\gamma}_0, \boldsymbol{\gamma}_1) > 4$. Now instead suppose both curves have a segment of the same slope; up to reparametrization and reindexing, we may assume $p_1 = q_1 = \infty$. It follows that $p_2$ and $q_2$ are integers, with $|p_2 - q_2| \le 1$. Up to reparametrization and switching the roles of $\boldsymbol{\gamma}_0$ and $\boldsymbol{\gamma}_1$, we may assume that $p_2 = 0$ and $q_2$ is either 0 or 1. In the first case, the curves $\boldsymbol{\gamma}_0$ and $\boldsymbol{\gamma}_1$ coincide. This is the one setting in which an $\epsilon$-pegboard diagram does not give transverse intersection but, perturbing slightly from an $\epsilon$-pegboard diagram, it is a simple exercise to check that the minimal intersection $i(\boldsymbol{\gamma}_0, \boldsymbol{\gamma}_1)$ is 4. However, in this minimal position there is an immersed annulus, so achieving admissibility requires adding two additional intersection points and $\dim \HFa(\boldsymbol{\gamma}_0, \boldsymbol{\gamma}_1) = 6$. The remaining case is depicted in Figure \ref{fig:ZHS3-small} and has 5 intersection points.
\begin{figure}[t]
\includegraphics[scale=.8]{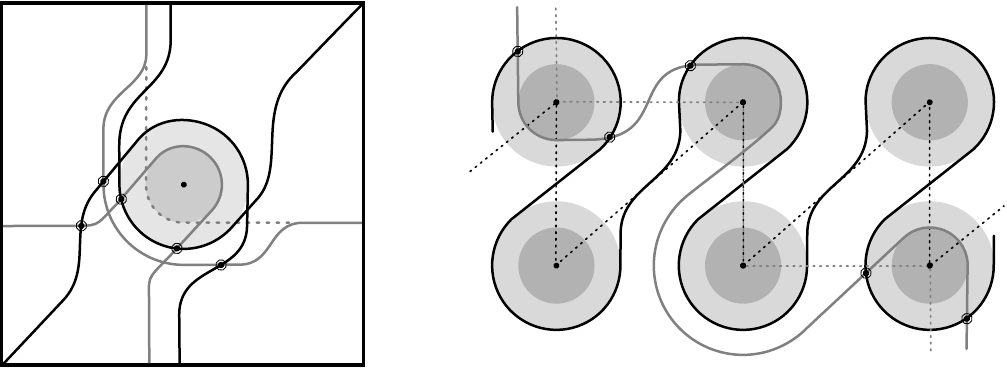}
\caption{The pairing of $\boldsymbol{\gamma_0}$ (gray) consisting of two segments of slopes 0 and $\infty$ with $\boldsymbol{\gamma_1}$ (black) consisting of two segments of slopes 1 and $\infty$, shown both in the torus and in the covering space $\R^2\setminus\Z^2$. This corresponds to a gluing of two trefoil complements. The result of the gluing is an integer homology sphere, since there is $\pm 1$ intersection point counted with sign. Note that, for the purposes of illustration, neither curve is pulled tight.}
\label{fig:ZHS3-small}
\end{figure}

We next consider the case that $M_1$ is a Heegaard Floer solid torus but  $M_0$ is not. If $\boldsymbol{\gamma_1}$ has true corners, the argument above applies so that $i(\boldsymbol{\gamma}_0, \boldsymbol{\gamma}_1) > 4$ unless $\boldsymbol{\gamma_1}$ has a single component and exactly two corners. It is easy to see that such a curve cannot be the invariant of a Heegaard Floer solid torus. If both segments in a singular peg-board diagram have the same slope then either the curve is nullhomologous or it is homologous to twice a primitive curve; in the first case there must be another component since $\boldsymbol{\gamma_1}$ is homologous to the rational longitude $\lambda_1$ of $M_1$, and in the second case there must be another component because if $\lambda_1$ has order two there are at least two spin$^c$ structures of $M_1$.  We conclude that $\boldsymbol{\gamma_1}$ has only $\pi$-corners. It is still the case that each $\pi$-corner of $\boldsymbol{\gamma_1}$ and each loose component of $\boldsymbol{\gamma_1}$ contributes one to $i(\boldsymbol{\gamma}_0, \boldsymbol{\gamma}_1)$; in fact, a loose component that is $p$ times a primitive curve contributes at least $p$ intersection points. If $p$ is the order of the rational longitude of $M_1$, then $\boldsymbol{\gamma_1}$ has at least $p$ components which each contribute at least $p$ intersection points, so $i(\boldsymbol{\gamma}_0, \boldsymbol{\gamma}_1) \ge p^2$. If $i(\boldsymbol{\gamma}_0, \boldsymbol{\gamma}_1) < 5$ then $p = 2$ and $\boldsymbol{\gamma_1}$ has exactly two components, each of which is one of the two curves associated with the twisted $I$-bundle (see Figure \ref{fig:alt-twisted-I}). Since $M_1$ has incompressible boundary, at least one of the components of $\boldsymbol{\gamma_1}$ is tight. It is now a simple exercise that $i(\boldsymbol{\gamma}_0, \boldsymbol{\gamma}_1) \ge 6$. The lower bound of four given previously records two intersections with a loose component and two intersections near corners in a tight component. However, if $\boldsymbol{\gamma}_0$ has two segments with slope different from the slope of $\lambda_1$ then there will be two type (d) corners, and if one segment of $\boldsymbol{\gamma}_0$ has the same slope as $\lambda_1$ this segment will have two additional intersection points, so the tight component of $\boldsymbol{\gamma}_0$ contributes at least four intersections.

The final case to consider is when both $M_0$ and $M_1$ are Heegaard Floer solid tori; that is, each multicurve pulls tight to a straight line in a singular pegboard diagram. These curve-sets may be divided into sets of loose and tight components. By irreducibility and Proposition~\ref{prop:loose}, both $\boldsymbol{\gamma}_0$ and $\boldsymbol{\gamma}_1$ must have at least one tight component. We claim that each multicurve must also have at least one loose component. In fact, each must have a loose component for which the associated local system has 1 as an eigenvalue (note that when discussing local systems on loose components we will assume the underlying curve is primitive by applying Lemma \ref{Lem:LocalSystems} if needed). To see this, consider pairing with $L_i$, a simple closed curve in $\partial M_i$ representing the rational longitude $\lambda_i$ of $M_i$. The loose curve $L_i$ pairs trivially with any tight component of $\boldsymbol{\gamma}_i$, since the curves can be made disjoint without any immersed annuli. Any loose component of $\boldsymbol{\gamma}_i$ decorated with a local system $A$ pairs with $L_i$ to give Floer homology of dimension $2\dim(\ker(I+A))$; if $I+A$ is invertible (equivalently if 1 is not an eigenvalue of $A$) then this contribution is again trivial. Since the Floer homology $\HFa(\boldsymbol{\gamma}_i, L_i)$ computes $\HFhat$ of the Dehn filling $M_i(\lambda_i)$, which must be nontrivial, there must be at least one component of $\boldsymbol{\gamma}_i$ which is loose and with 1 as an eigenvalue of the local system.
  
\labellist
\tiny
\endlabellist
\piccaption[]{Pairing two twisted $I$-bundles over the Klein bottle: At left, the loose components contribute two intersections and the tight components contribute two intersections. Imposing admissibility to the tight components as on the right, brings the intersection number to 6.  \label{fig:4-admiss}}
\parpic[r]{
 \begin{minipage}{60mm}
 \centering
\includegraphics[scale=0.5]{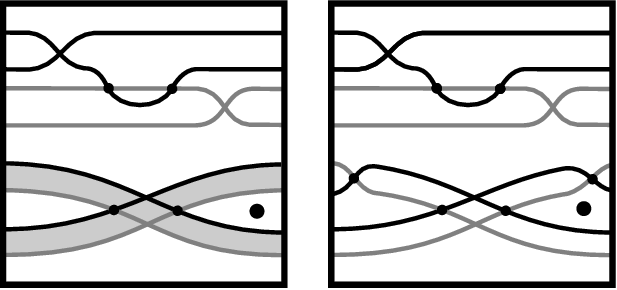}
 \end{minipage}%
  } Let $p_i$ be the order of the rational longitude $\lambda_i$ of $M_i$. Recall that $\boldsymbol{\gamma}_i$ has at least $p_i$ components, each homologous to at least $p_i$ copies of $\lambda_i$. From this it easily follows that if $Y$ is a rational homology sphere, $i(\boldsymbol{\gamma}_0, \boldsymbol{\gamma}_1) \ge p_0^2 p_1^2$. Moreover, if a tight component of $\boldsymbol{\gamma}_i$ has only $\pi$-corners then $p_i$ must be at least two, and if this holds for both curves then $i(\boldsymbol{\gamma}_0, \boldsymbol{\gamma}_1) \ge 16$ for a rational homology sphere gluing. There are a few pathological examples to consider (which likely do not arise from manifolds) of tight Heegaard Floer solid torus curves with $p_i = 1$, but such curves traverse the underlying primitive curve at least three times and still give a lower bound of 9 for rational homology sphere gluings. The more interesting case is when $h(\lambda_1) = \lambda_0$; we will reparametrize so that both curves are horizontal. In this setting, any pair of tight components contribute at least two intersection points to Floer homology; to see this note that each curve must pass above the peg at least once and below the peg at least once, and a (rightward moving) segment moving from a corner below the peg to a corner above the peg always intersects a segment moving from a corner above the peg to a corner below the peg. Any pair of loose curves with local systems $A_0$ and $A_1$ for which $1$ as an eigenvalue also have Floer homology of dimension at least two. Although the minimal intersection number is zero, by Corollary \ref{cor:ComputeHF} the dimension of Floer homology is given by $2\ker(I + A_0 \otimes  A_1)$; since 1 is an eigenvalue of both $A_0$ and $A_1$, it is also an eigenvalue of $A_0 \otimes A_1$ and so $I + A_0 \otimes  A_1$ is singular. To summarize, if  $x_i$ is the number of tight components of $\boldsymbol{\gamma}_i$ and $y_i$ is the number of loose components of $\boldsymbol{\gamma}_i$ with 1 as an eigenvalue of the local system, then $\dim \HFa(\boldsymbol{\gamma}_0, \boldsymbol{\gamma}_1)$ is bounded below by $2x_0x_1+2y_0y_1$. The discussion above ensures that $x_i$ and $y_i$ are both nonzero for $i = 0,1$, and so this bound is at least 4. Moreover, the bound is only 4 if $x_0 = x_1 = y_0 = y_1 = 1$. In this case, both rational longitudes must have order two (since each multicurve has only two components), and there is a unique tight curve which wraps twice around the torus following the horizontal line. Since the tight component of $\boldsymbol{\gamma}_0$ is commensurate with the tight component of $\boldsymbol{\gamma}_1$, one can check that pairing these curves in fact contributes at least 4 to $\dim \HFa(\boldsymbol{\gamma}_0, \boldsymbol{\gamma}_1)$, raising the lower bound in this case to 6. Note that this bound can be realized by letting $M_0$ and $M_1$ both be twisted $I$-bundles over the Klein bottle; this example is shown in Figure \ref{fig:4-admiss}.
\end{proof}

The case in Figure \ref{fig:ZHS3-small} is realized by gluing two trefoil complements. For instance, it arises from gluing the complement of two right handed trefoil complements by identifying 0-framed longitude with meridian and meridian with 1-framed longitude. This provides an example of a toroidal integer homology sphere $Y$ realizing $\dim \HFhat(Y) = 5$.

With this example in hand, consider the function $\mathpzc{f}$ taking values in the natural numbers defined by 
\[\mathpzc{f}(n)=
\min\{\dim\HFhat(Y) \,|\, Y\ \text{is a rational homology sphere with\ } n\ \text{separating JSJ\ tori} \}\]
and note that $\mathpzc{f}(0)=1$ (realized by the  three-sphere or the Poincar\'e homology sphere) and  $\mathpzc{f}(1)=5$ (realized by gluing two trefoil exteriors as in Theorem \ref{thm:toroidal-detailed}, see Figure \ref{fig:ZHS3-small}). For Heegaard Floer aficionados, we propose: 
\begin{question}\label{qst:growth} Is it possible to determine $\mathpzc{f}$ or describe properties of $\mathpzc{f}$ as $n$ grows?\end{question}

If we further assume that $Y$ is an L-space, we can in fact strengthen the result of Theorem \ref{thm:total-dim-toroidal} by enumerating all examples with $\dim \HFhat(Y) < 7$; see Theorem \ref{thm:toroidal-detailed}. This will make use of a result specific to L-spaces. Toward establishing this result, we now turn our attention to properties of the immersed curve invariants relevant to L-spaces.

\subsection{Characterizing L-space slopes} Given a manifold with torus boundary $M$, let $\sL_M$ denote the set of L-space slopes, that is, the subset of slopes giving rise to L-spaces on Dehn filling, and let $\sL^\circ_M$ denote the interior of $\sL_M$. This section gives a characterization of the set $\sL^\circ_M$ in terms of the collection of curves $\curves{M}$.

 Given a spin$^c$ structure $\spin$ on $M$, we will define a set of slopes $S(M, \spin) \subset \RP$ associated with the multicurve $\curves{M, \spin}$. For each $\epsilon > 0$, we consider an $\epsilon$-pegboard diagram for the curves $\curves{M, \spin}$. If $\curves{M, \spin}$ has single component equipped with the trivial local system we define $S_\epsilon(M, \spin)$ as the set of tangent slopes of the $\epsilon$-pegboard representative of $\curves{M, \spin}$. Otherwise, if $\curves{M, \spin}$ is disconnected or has a non-trivial local system we define $S_\epsilon(M, \spin)$ to be all of $\RP$. Note that when computing the set of tangent slopes in a $\epsilon$-pegboard diagram, it is convenient to lift to the covering space $\overline{T}_M$. We now define $S(M, \spin)$ by taking $\epsilon\to 0$, in the sense that $S(M, \spin)=\bigcap_{n=N}^\infty S_{\frac{1}{n}}(M, \spin)$. We define $S(M)$ to be the union of $S(M, \spin)$ over all spin$^c$ structures.


We can define a related set $S^{sing}(M, \spin)$ using a singular pegboard diagram. If \(\curves{M, \spin}\) has a single component carrying a trivial local system, $S^{sing}(M, \spin)$ is the set of slopes of tangent lines to the singular peg-board representative, where we say a line $L$ through a corner is a tangent line if $L$ coincides with one of the two segments at that corner or if both segments lie on the same side of $L$. As with $S(M, \spin)$, we set $S^{sing}(M, \spin) = \RP$ if for some \(\spinc\) structure \(\spin\), $\curves{M,\spin}$ has more than one curve  or a non-trivial local system.  It is straightforward to see that $S(M) = S^{sing}(M) $ when \(\curves{M}\) does not wrap any pegs. {\it A priori}, it seems like the two sets may differ in the presence of peg-wrapping, but we will prove in Corollary~\ref{cor:tangent_slopes} below that if there is any peg-wrapping, we must have \(S^{sing}(M) = S(M) = \RP\).

Using tangent slopes in peg-board diagrams, we can give a new interpretation of L-space slopes for a loop-type manifold.

\begin{theorem}\label{thm:slope_detection}
If $M$ is a manifold with torus boundary then $\sL^\circ_M$ is the complement in $\Q P^1$ of the set $S^{sing}_\Q(M) = S^{sing}(M)\cap\Q P^1$. 
\end{theorem}

\begin{proof}By the pairing theorem, $\HFhat(M(\alpha))$ is equivalent to the intersection Floer homology of $\curves{M}$ and a straight line with slope $\alpha$, which we denote $L_\alpha$. Suppose first that $\alpha$ is a non-L-space slope, so that the Dehn filling $M(\alpha)$ is not an L-space and there is some spin$^c$-structure $\bar\spin$ of $M(\alpha)$ so that $\dim\HFhat(M(\alpha);\bar\spin)>1$. It follows that (for any homotopy representative of $\curves{M}$) there are two intersection points $x$ and $y$ between $\curves{M}$  and the line $L_\alpha$ with the same spin$^c$-grading. This means that $x$ and $y$ correspond to intersection points $\tilde x$ and $\tilde y$ between (a lift of) $L_\alpha$ and $\curves{M, \spin}$ in the covering space $\R^2 \setminus \Z^2$, for some spin$^c$-structure $\spin$ of $M$. If $\curves{M, \spin}$ contains more than one curve then $S^{sing}(M) = \RP$ by definition and $\alpha \in S_\Q^{sing}(M)$, so suppose that $\curves{M, \spin}$ contains a single curve. The intersection points $\tilde x$ and $\tilde y$ exist for any homotopy representative of $\curves{M, \spin}$ (in particular the pegboard representatives for arbitrarily small $\epsilon$) and since $L_\alpha$ can be taken to be some finite distance away from every peg, we can in fact take $\tilde x$ and $\tilde y$ to be intersections of the singular peg-board representative of $\curves{M, \spin}$ with a lift of $L_\alpha$. Note that the corners of the singular representative for $\curves{M, \spin}$ can be smoothed without affecting intersections with $L_\alpha$ and that $S^{sing}(M)$ is precisely the set of tangent slopes after this smoothing. By the (extended) mean value theorem, there is a point on the smoothing of the singular representative of $\curves{M, \spin}$ for which the tangent line has slope $\alpha$. We conclude that $\alpha\in S^{sing}_\Q(M)$ as claimed. We have proved that $\sL_M^c \subset S^{sing}_\Q(M)$; in fact, since $S^{sing}_\Q(M)$ is closed (as a subset of $\Q P^1$), we have $\overline{\sL_M^c} = (\sL^\circ_M)^c \subset S^{sing}_\Q(M)$. 

\piccaption[]{Intersection points for a slope in $S^{sing}(M)$.}
\parpic[r]{
 \begin{minipage}{34mm}
 \centering
\includegraphics[scale=0.6]{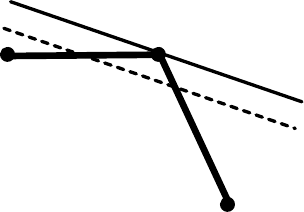}
  \end{minipage}%
  }
Conversely consider a slope $\alpha$ in $S^{sing}_\Q(M)$. We will work with the singular peg-board representative of $\curves{M}$. If $S^{sing}(M) = \{\alpha\}$, then $\alpha$ is the rational longitude of $M$ and thus not an L-space slope. Otherwise, $\alpha$ is in the interval of tangent slopes determined by some corner $c$; let $s_1$ and $s_2$ denote the two line segments meeting at $c$. First suppose $\alpha$ is in the interior of the interval for the corner $c$. Let $L_\alpha$ be the line of slope $\alpha$ through $c$; since $L_\alpha$ is a tangent line, $s_1$ and $s_2$ both lie on one side of line $L_\alpha$. Let $L'_\alpha$ be a small pushoff of $L_\alpha$ that intersects the segments $s_1$ and $s_2$ and is disjoint from all lattice points. For sufficiently small $\epsilon$ (smaller than the minimum distance between $L'_\alpha$ and any lattice point), replacing the singular representative for $\curves{M}$ with the radius $\epsilon$ peg-board representative preserves these two intersection points. These points give two generators with the same spin$^c$ grading in the intersection Floer chain complex of $\curves{M}$ with $L'$. Since both $L'_\alpha$ and $\curves{M}$ are $\epsilon$-pegboard representatives, they are in minimal position and both generators survive in homology. It follows that $M(\alpha)$ is not an L-space, that is, $\alpha \not\in \sL_M$. Now suppose $\alpha$ is a boundary of the interval of slopes determined by the corner $c$; there are slopes arbitrarily close to $\alpha$ in the interior of this interval, and these slopes are not in $\sL_M$, so $\alpha \not\in \sL^\circ_M$.
\end{proof}

Theorem \ref{thm:slope_detection} is in fact true for arbitrary immersed curves in $\Tp$  (using an appropriate definition of the set of L-space slopes $\sl$ for a curve), not just for the collections of curves associated with manifolds. For curves arising from manifolds, however, we can replace $S^{sing}(M)$ with $S(M)$. To prove this, we first make the following observation about Floer simple manifolds: 
\begin{lemma}\label{lem:Floer_simple_slopes}
If $M$ is Floer simple then $S(M) \neq \RP$.
\end{lemma}
\begin{proof}
If $M$ is Floer simple, then by \cite[Proposition 6]{HRRW} $M$ has simple loop type, that is, for some parametrization $(\alpha, \beta)$ of $\partial M$, $\CFD(M, \alpha, \beta)$ is a collection of loops (one for each \(\spinc\) structure on \(M\)) consisting only of $\rho_1$, $\rho_3$ and $\rho_{23}$ arrows. It follows that the corresponding curves in the plane travel only up and right. Clearly $S(M)$ contains only positive slopes (relative to the chosen parametrization), and so $S(M) \neq \RP$. Note that reparametrizing the boundary changes this interval but does not change whether or not it is all of $ \RP$.
\end{proof}

\begin{corollary}\label{cor:tangent_slopes}
For any manifold $M$, $S(M) = S^{sing}(M)$.
\end{corollary}
\begin{proof}
First observe that if $M$ is not loop type (that is, the associated invariant carries a non-trivial local system) then $ S^{sing}(M)=S(M)=\RP$ by definition. Consider then the case when $M$ is loop type. In the absence of peg-wrapping $S(M) = S^{sing}(M)$. If there is peg-wrapping, then $S(M) = \RP$, and by Lemma \ref{lem:Floer_simple_slopes}, $M$ is not Floer simple. Equivalently, $\sL^\circ_M = \emptyset$ and by Theorem \ref{thm:slope_detection}  $S^{sing}(M) = \RP$.
\end{proof}

We remark that, as an immediate corollary of Theorem \ref{thm:slope_detection}, the set $\sL^\circ_M$ is either empty, an open interval with rational endpoints, or all of $\Q P^1$ but the rational longitude $\lambda$; this is because $S(M)$ contains at least $\lambda$ and is a closed interval. More generally, it is possible to characterize the set $\sL_M$ in terms of the tangent slopes in a singular peg-board diagram: roughly speaking, $\sL_M$ is the complement of the set $S'(M)$ obtained by taking the union over all corners of the \emph{interior} of the set of tangent slopes determined by the corner, with the convention that a degenerate corner with two segments of slope $\alpha$ contributes $\RP \setminus \{\alpha\}$, along with the slope of any segment such that $\HFhat(M)$ turns left at both of the corners connected by the segment or turns right at both corners. It follows that $\sL_M$ is either empty, a single point, a closed interval with rational endpoints, or $\Q P^1 \setminus \{\lambda\}$; this recovers 
\cite[Theorem 1.6]{RR} (see also  \cite[Theorem 1.2]{HW}).

\subsection{The L-space gluing theorem}
We now turn to the proof of Theorem \ref{L-space-gluing-theorem}. In previous works \cite{HW, RR, HRRW} the authors and S. Rasmussen prove this L-space gluing criterion when the manifolds have simple loop type (equivalently, when the manifolds are Floer simple). Interpreting bordered Floer homology as immersed curves with local systems allows us to give an elegant proof of this gluing theorem without requiring the simple loop type hypothesis on the manifolds.

\begin{proof}[Proof of Theorem \ref{L-space-gluing-theorem}]
By the pairing theorem, $M_0 \cup_h M_1$ is equivalent to the intersection Floer homology of $\boldsymbol{\gamma}_0 = \curves{M_0}$ and $\boldsymbol{\gamma}_1 = \bar{h}(\curves{M_1})$. Suppose first that $M_0 \cup_h M_1$ is not an L-space, implying that there is some $\spin$ so that $\dim\HFhat(M_0 \cup_h M_1,\spin)>1$. In other words, there are two intersection points $x$ and $y$ between $\boldsymbol{\gamma}_0$ and $\boldsymbol{\gamma}_1$ with the same spin$^c$ grading. This means, in particular, that $x$ and $y$ are intersection points of $\boldsymbol{\gamma}_{0,\spin_0} = \curves{M_0, \spin_0}$ and $\boldsymbol{\gamma}_{1,\spin_1} = \bar h(\curves{M_1, \spin_1})$ for some $\spin_0 \in \spinc(M_0)$ and $\spin_1 \in \spinc(M_1)$. We will assume for $i \in \{0,1\}$ that $\boldsymbol{\gamma}_{i,\spin_i}$ contains only one curve with a trivial local system, since otherwise $\sL^\circ_{M_i}$ is empty and $\sL^\circ_{M_0} \cup h( \sL^\circ_{M_1} ) \neq \Q P^1$. Let $c_0$ be the path from $x$ to $y$ in $\boldsymbol{\gamma}_{0,\spin_0}$ and let $c_1$ be the path from $x$ to $y$ in $\boldsymbol{\gamma}_{1,\spin_1}$. The fact that $x$ and $y$ have the same spin$^c$ grading implies that $[c_0-c_1] = 0 \in H_1(T)$, or equivalently that the paths lift to form a bigon in $\R^2$, the corners of which are lifts $\tilde x$ and $\tilde y$ of $x$ and $y$. Let $\alpha$ be the slope of the line segment connecting $\tilde x$ and $\tilde y$. By the (extended) mean value theorem, $\tilde c_0$ and $\tilde c_1$ each contain a point with slope $\alpha$.

This argument applies for any homotopy representative of $\boldsymbol{\gamma}_0$ and $\boldsymbol{\gamma}_1$. In particular, if we take $\epsilon$-pegboard diagrams relative to peg radius $\epsilon = \frac{1}{n}$ we find a slope $\alpha_n$ that is a tangent slope to both curves. It follows that $\alpha_n \in S_{\frac{1}{n}}(M_0) \cap h(S_{\frac{1}{n}}(M_1))$. Here, by abuse of notation, $h$ refers to the map on slopes induced by $h$; note that, as maps on slopes, $\bar{h}$ and $h$ agree since the elliptic involution preserves slopes. Since $S_{\frac{1}{n}}(M_0) \cap h(S_{\frac{1}{n}}(M_1))$ for positive integers $n$ gives a nested sequence of nonempty closed sets there is some slope $\alpha$ in the total intersection, which is simply $S(M_0) \cap h(S(M_1))$. Moreover, we can take this $\alpha$ to be rational, since $S(M_0)$ and $h(S(M_1))$ have rational endpoints. By Theorem \ref{thm:slope_detection} and Corollary \ref{cor:tangent_slopes}, $S_\Q(M_0) = (\sL^\circ_{M_0})^c$ and $S_\Q(M_1) = (\sL^\circ_{M_1})^c$. Thus $\alpha$ is not in $\sL^\circ_{M_0} \cup h( \sL^\circ_{M_1} )$.

Conversely, suppose that there is a slope $\alpha$ not in $\sL^\circ_{M_0} \cup h( \sL^\circ_{M_1} )$. Equivalently, $\alpha$ is in both $S^{sing}_\Q(M_0)$ and $h(S^{sing}_\Q(M_1))$, so $\alpha$ is a tangent slope of both $\boldsymbol{\gamma}_0$ and $\boldsymbol{\gamma}_1$ in a singular pegboard diagram. Let $c_0$ and $c_1$ be components of $\boldsymbol{\gamma}_0$ and $\boldsymbol{\gamma}_1$, respectively, for which $\alpha$ is a tangent slope to the singular pegboard representative. We may assume that $c_0$ and $c_1$ are not solid torus like components. If, for example, $c_0$ were solid torus like then $\alpha$ would be the rational longitude of $M_0$ and every component of $\boldsymbol{\gamma}_0$ would have a tangent line of slope $\alpha$; we could then replace $c_0$ by another component, at least one of which is not solid torus like because $M_0$ is not solid torus like. By reparametrizing $\partial M_0$ and $\partial M_1$, we may assume that $\alpha = 0$.

Fixing a small peg radius $\epsilon$, consider $\epsilon$-pegboard representatives of $c_0$ and $c_1$ with a chosen orientation on each curve. Recall that each corner wraps around a circle of different radius between $\epsilon$ and $2\epsilon$, with corners that change direction more wrapping closer to the puncture (as usual, if two radii agree we perturb one of them slightly to achieve transversality). Since $c_0$ and $c_1$ are in minimal position, to show that $M_0 \cup_h M_1$ is not an L-space it is enough to find two intersection points between $c_0$ and $c_1$ with opposite $\Ztwo$-grading. It is easy to find two such generators if both curves have peg-wrapping. In this case the curves $c_0$ and $c_1$ each contain a portion of sucrve bounding a teardrop enclosing the puncture. Suppose the relevant corners of $c_0$ and $c_1$ wrap around circles radius $\epsilon_0$ and $\epsilon_1$ around the puncture. Suppose, without loss of generality, that $\epsilon_0 < \epsilon_1$, so that the circle of radius $\epsilon_0$ is contained in the teardrop bounded by $c_1$. The curve $c_0$ also contains two line segments tangent at their ends to the circle of radius $\epsilon_0$, one oriented toward the circle and one oriented away from it. These two segments clearly intersect the teardrop portion of $c_1$ with opposite sign, and the two corresponding generators of $\HFa(\boldsymbol{\gamma}_0,\boldsymbol{\gamma}_1)$ have opposite $\Ztwo$ grading. Thus we may assume that there is no peg-wrapping in $c_0$.

In the singular peg-board diagram for $c_0$, the line segments connecting corners can be classified as moving upwards, downwards, or horizontally (once we have fixed an orientation on $c_0$). The segments can not be all upward or all downward, since then $\alpha = 0$ would not be a tangent slope. If the slopes are not all horizontal, we can choose a corner so that the segments preceding and following the corner are up and down, up and horizontal, or horizontal and down. Note that the corresponding corner in an $\epsilon$-pegboard diagram for $c_0$, which we suppose wraps around a circle of radius $\epsilon_0$, has a point of horizontal tangency on the top side of that circle at the point $(0, \epsilon_0)$. If every line segment in $c_0$ is horizontal, then every corner has  horizontal tangency to the peg, and we can chose the corner so that this point of tangency is on the top side of the peg. Moving away from this point of tangency in either direction, $c_0$ moves either horizontally or downwards. If it moves horizontally it wraps once around the torus and returns to the peg at another horizontal tangency on the top of the peg; after some number of horizontal segments like this, $c_0$ must move downwards. Thus, $c_0$ contains an upward moving piece $c_0^I$ oriented toward the point $(0,\epsilon_0)$ and a downward moving piece $c_0^O$ oriented away from $(0,\epsilon_0')$, where $\epsilon_0'=\epsilon_0$ if there are no horizontal segments in between. Combining $c_0^O$ and $c_0^I$ gives a curve segment as in Figure \ref{fig:gluing-proof-1}(a); note that if there are horizontal segments in between $c_0^O$ and $c_0^I$ we ignore them---in this case the endpoints do not match up perfectly since $\epsilon_0' \neq \epsilon_0$, but by extending one segment until it crosses the other one we get a piecewise smooth version of the segment in Figure \ref{fig:gluing-proof-1}(a). Similarly, we can choose a corner of $c_1$ such that the preceding segment in the singular pegboard diagram moves downward or horizontally and the following segment moves horizontally or upward, and such that in an $\epsilon$-pegboard diagram there is a horizontal tangency to the peg at $(0, -\epsilon_1)$. Let $c_1^I$ be the horizontal or downward piece of $c_1$ oriented toward $(0, -\epsilon_1)$ and let $c_1^O$ be the horizontal or upward piece oriented away from $(0, -\epsilon_1)$ (or, potentially, $(0, -\epsilon_1')$), as in Figure \ref{fig:gluing-proof-1}(b). Note that we ignore any full wraps around the peg as well as any horizontal segments between $c_1^I$ and $c_1^O$. We consider these curves in one fundamental domain of $\Tp$, the square $[-\frac{1}{2}, \frac{1}{2}]\times[-\frac{1}{2}, \frac{1}{2}]$. The intersection of $c_0^I$ and $c_0^O$ with the boundary of this square can be any point with non-positive $y$-coordinate; the intersection of $c_1^I$ and $c_1^O$ with the boundary of this square can have any non-negative $y$-coordinate or $y$-coordinate $-\epsilon_1$.
It is clear that if both $c_0^I$ and $c_0^O$ intersect the boundary of the square $[-\frac{1}{2}, \frac{1}{2}]\times[-\frac{1}{2}, \frac{1}{2}]$ with strictly negative $y$-value or if both $c_1^I$ and $c_1^O$ intersect the boundary of the square with strictly positive $y$-value (equivalently, if the corresponding segments in the singular diagram are not horizontal), then there are two intersections points with opposite $\Ztwo$ grading (see Figure \ref{fig:gluing-proof-1}(c)).

The remaining case is when one of the specified segments in both $c_0$ and $c_1$ correspond to horizontal segments in a singular diagram. We consider the case that $c_0^O$ and $c_1^O$ are both horizontal moving to the right (see Figure \ref{fig:gluing-proof-2}); cases when the incoming segments or both segments are horizontal are similar. In this case, we do not get two intersection points of $c_0^I \cup c_0^O$ and $c_1^I \cup c_1^O$ if we restrict to one fundamental domain of $\Tp$. However, $c_0^O$ and $c_1^O$ both wrap once around  $\Tp$ horizontally, leading to new corners of $c_0$ and $c_1$, respectively. Let $\epsilon_0''$ and $\epsilon_1''$ be the peg radii associated with these two corners. If $\epsilon_1'' < \epsilon_0''$, then $c_0^O$ intersects $c_1^O$ before reaching the next corner (Figure \ref{fig:gluing-proof-2}(a)). If $\epsilon_1'' > \epsilon_0''$, then $c_1$ must not have peg-wrapping at this corner. It follows that $c_1$ must continue by leaving the peg horizontally or upwards. If it leaves the peg moving upwards, it will intersect with  $c_0^I \cup c_0^O$ (Figure \ref{fig:gluing-proof-2}(b)). If it leaves the peg horizontally, it wraps around $\Tp$ once more to a new corner with some radius $\epsilon_1'''$, and we repeat the argument using this corner in place of the corner with radius $\epsilon_1''$. In either case, we find two intersection points between $c_0$ and $c_1$ with the opposite sign, as desired.
\end{proof}

\begin{figure}
\begin{overpic}[scale=0.5]{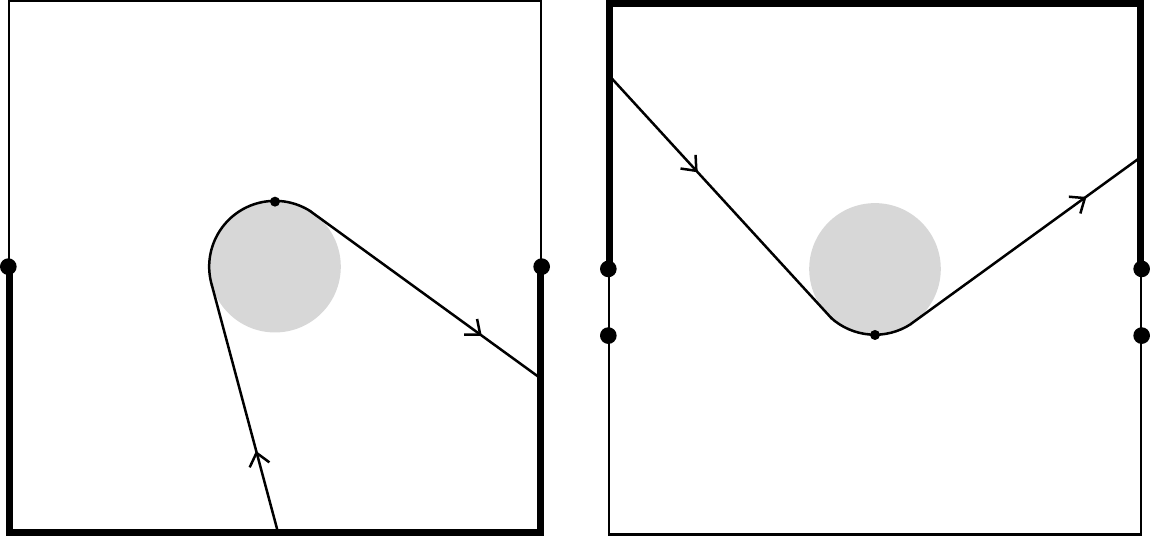}
\put(16, 12){$c_0^I$}
\put(36, 23){$c_0^O$}
\put(61, 23){$c_1^I$}
\put(85, 20){$c_1^O$}
\end{overpic}
\quad
\includegraphics[scale=0.5]{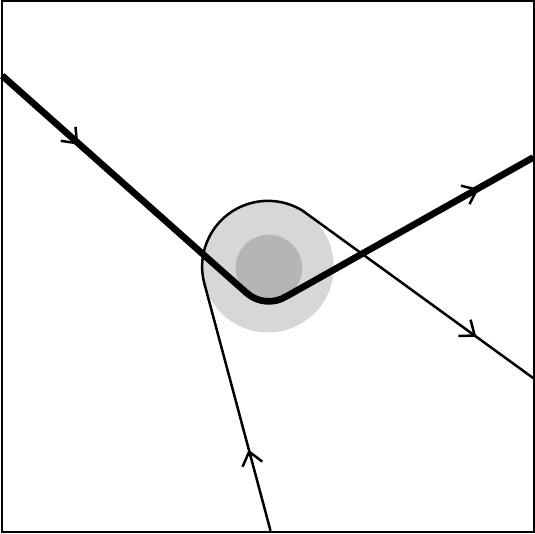}

(a) \hspace{4.5 cm} (b) \hspace{4.5 cm} (c)

\caption{(a) The segments $c_0^I$ and $c_0^O$ in $c_0$; these segments can enter/leave the square anywhere in the bottom half of its boundary. (b) The segments $c_1^I$ and $c_1^O$ in $c_1$; these segments can enter/leave the square anywhere in the top half of its boundary or though the points $(\pm\frac{1}{2}, -\epsilon_1)$. (c) The combined curve segments $c_0^I \cup c_0^O$ and $c_1^I \cup c_1^O$ have two intersections of opposite signs if the endpoints of $c_0^I$ and $c_0^O$ are strictly in the bottom half of the square or if the endpoints of $c_1^I$ and $c_1^O$ are strictly in the top half of the square.}
\label{fig:gluing-proof-1}
\end{figure}

\begin{figure}
\includegraphics[scale=0.6]{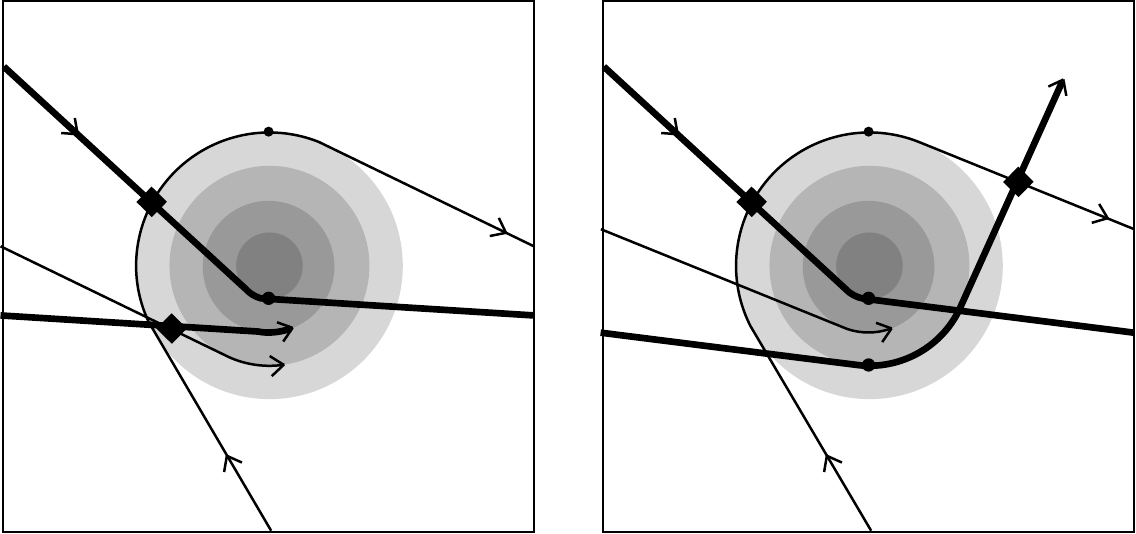}

(a) \hspace{5.5 cm} (b)

\caption{The combined segments $c_0^I \cup c_0^O$ and $c_1^I \cup c_1^O$ in an $\epsilon$-pegboard diagram when $c_0^O$ an $c_1^O$ both correspond to horizontal segments in singular diagrams for $c_0$ and $c_1$. There are two cases, depending on which curve wraps with smaller radius at the next corner; in either case there are two intersection points between $c_0$ and $c_1$ with opposite sign.}
\label{fig:gluing-proof-2}
\end{figure}

\subsection{Enumerating toroidal L-spaces} As our first application of the L-space gluing theorem, we classify small toroidal L-spaces. 

\begin{theorem}\label{thm:toroidal-detailed}
If $Y$ is a prime toroidal L-space then $|H_1(Y)|=5$ or $|H_1(Y)|\ge 7$. Moreover, if $|H_1(Y)|=5$ then $Y$ is one of precisely 4 manifolds obtained by gluing trefoil exteriors. 
\end{theorem}

\begin{proof}
We begin by setting conventions. Consider manifolds  $M_0$ and $M_1$ with torus boundary, and let $h\co \partial M_1\to \partial M_0$ be an orientation reversing homeomorphism. Since we are interested in rational homology spheres $Y= M_0\cup_h M_1$, we may assume that both $M_0$ and $M_1$ are rational homology solid tori, with slopes $\lambda_i\in\partial M_i$ representing the rational longitude for $i=0,1$. Recall that $$|H_1(M_0\cup_h M_1;\Z)| = t_0t_1|\lambda_0||\lambda_1|\Delta(\lambda_0,h(\lambda_1))$$ where  $t_i$ is the order of the torsion subgroup of $H_1(M_i;\Z)$ and \(|\lambda_i|\) is the order of $\lambda_i$ in $H_1(M_i;\Z).$  The integer $\Delta(\lambda_0,h(\lambda_1))$ is the distance between the relevant slopes, that is, the minimal geometric intersection in $\partial M_0$. If we fix \(\spinc\) structures \(\spin_i\) on \(M_i\), the quantity \(|\lambda_0||\lambda_1|\Delta(\lambda_0,h(\lambda_1))\) is the number of \(\spinc\) structures on \(Y\) that restrict to
\(\spin_0\) on \(M_0\) and \(\spin_1\) on \(M_1\). 

Since \(M_i\) is boundary incompressible, there is at least one \(\spin_i \in \spinc(M_i)\) for which \(\gamma_i = \curves{M_i, \spin_i}\) is not loose. 
The argument used in the proof of Theorem~\ref{thm:total-dim-toroidal} shows that the curves representing \(\gamma_0\) and \(\gamma_1\) intersect in at least four points. Since \(Y\) is an L-space, this implies that \(|\lambda_0||\lambda_1|\Delta(\lambda_0,h(\lambda_1)\geq 4\). To have \(|H_1(Y;\Z)|\leq 6\), we must have \(t_0=t_1=1\), 
{\it i.e.} both \(M_i\) are integer homology solid tori. 

By abuse of notation we identify $h$ with a matrix $\left(\begin{smallmatrix} s & r \\ q& p\end{smallmatrix} \right)$, acting on the right, where $(1,0)$ is identified with $\lambda_1$ and $(0,1)$ represents a meridian $\mu_1$. With this notation, one checks that $|r|=\Delta(\lambda_0,h(\lambda_1))$. Towards minimizing $|H_1(M_0\cup_h M_1;\Z)|$, we may assume that the $M_i$ are integer homology solid tori, so that $|\lambda_0|=|\lambda_1|=1$ and $|H_1(M_0\cup_h M_1;\Z)|=|r|$.  

Now suppose that $Y$ is an L-space; the intervals $\sL_i=\sL_{M_i}$ are necessarily non-empty. Since \(H_1(M_i) = \Z\), the \(M_i\) cannot be Floer homology solid tori. Let \(b_i \lambda_i+a_i \mu_i\) be an endpoint of \(\sL_i\), where without loss of generality we assume \(a_i>0\). The characterization of \(\sL_i\) proved in \cite{RR} implies that \(\sL_i\) is contained either in the interval 
\([a_i/b_i, a_i/(b_{i}+1)]\) 
or in the interval 
\([a_i/(b_i-1),a_i/b_i])\). 
(Here our notation denotes the cyclic interval not containing \(0\) with the given endpoints.)  By choosing \(\mu_i\) appropriately, we may assume that \(0>b_i>-a_i\) in the first case, and that 
\(0<b_i<a_i\) in the second, so that \(\sL_i \subset [-\infty, -1]\) or \(\sL_i \subset [1,\infty]\). Finally, by reversing the global orientation on \(Y\), we may assume that \(\sL_1 \subset[1,\infty]\).

 We consider the images $h(0,1)=(q,p)$, $h(1,0)=(s,r)$, and $h(1,1)=(q+s,p+r)$ relative to the interval of slopes $\sL_0$. If $Y$ is an L-space then we must have $1<h(\alpha)<\infty$ for each of these slopes or $-\infty<h(\alpha)<-1$ for each of these slopes. 

First consider the case where \(\sL_0 \subset[1,\infty]\). Then \begin{align*}
1<\textstyle\frac{p}{q}<\infty &\Rightarrow 1<|q|<|p|\\
1<\textstyle\frac{r}{s}<\infty &\Rightarrow 1<|s|<|r|\\
\end{align*}
where $p$ and $q$ have the same sign and $r$ and $s$ have the same sign. Similarly, considering the inverse homeomorphism $h^{-1}=\left(\begin{smallmatrix} -p & r \\ q& -s\end{smallmatrix} \right)$, 
\begin{align*}
1<\textstyle-\frac{s}{q}<\infty &\Rightarrow 1<|q|<|s|\\
1<\textstyle-\frac{r}{p}<\infty &\Rightarrow 1<|p|<|r|\\
\end{align*}
 so that $p$ and $r$ have opposite signs and $q$ and $s$ have opposite signs. Additionally, note that $\det(h)=ps-qr=-1$ since $h$ reverses orientation. Without loss of generality, we may assume that $r>0$; one checks that there are no matrices with $r=1,2$ satisfying this list of constraints. 

When $r=3,4$ the only possible matrices are $\left(\begin{smallmatrix} 2 & 3\\ -1& -2\end{smallmatrix} \right)$, $\left(\begin{smallmatrix} 3 & 4\\ -2& -3\end{smallmatrix} \right)$. However a homeomorphism of the form $\left(\begin{smallmatrix} r-1 & r\\ 2-r& 1-r\end{smallmatrix} \right)$ is always ruled out since $h(1,1)=(1,1)$ in this case (so a boundary L-space slope is not mapped to an interior L-space slope, as required). The same argument applies in the case  $r=6$ 

 From this discussion we observe that in fact we seek to list matrices with determinant $-1$ satisfying $1<\frac{p+r}{q+s}<\frac{r}{s}<\frac{p}{q}<\infty$. For \(r=5\), this  leaves $
 \left(\begin{smallmatrix} 3 & 5\\ -1& -2\end{smallmatrix} \right)$ and $
 \left(\begin{smallmatrix} 2 & 5\\ -1& -3\end{smallmatrix} \right)$ 
 as a complete list of possible L-space gluings in this case.
 
 The case where \(\sL_0 \subset [-\infty,-1]\) is similar. Potential examples of L-spaces for which $r<5$ or $r=6$ are ruled out, by the requirement that $-\infty<\frac{p+r}{q+s}<\frac{r}{s}<\frac{p}{q}<-1$, leaving 
  $\left(\begin{smallmatrix} -3 & 5\\ 2& -3\end{smallmatrix} \right)$ and 
$ \left(\begin{smallmatrix} -2 & 5\\ 1& -2\end{smallmatrix} \right)$
  as a complete list of L-space gluings for \(r=5\).
  
  Next, we claim that in each of these four cases, we must have \(\sL_1 = [1, \infty]\). We provide the argument when \(h\) is given by \(\left(\begin{smallmatrix} s & r\\ q& p\end{smallmatrix} \right)=\left(\begin{smallmatrix} 2 & 5\\ -1& -3\end{smallmatrix} \right) \); the other cases are very similar. As observed above, 
  if \(a/b\) is the right endpoint of \(\sL_1\), then \(\sL_1 \subset [a/(b+1),a/b]\). 
  Since neither $\mu_0$ nor $\lambda_0+\mu_0$ are elements of $\sL_0^\circ$ both $h^{-1}(\mu_0)$ and $h^{-1}(\lambda_0+\mu_0)$ must be contained in $\sL_1^\circ$. We calculate that this includes the slopes $-\frac{s}{q}=2$ and $\frac{r-s}{q-p}=\frac{3}{2}$, respectively, in the set $\sL_1^\circ$,
	and deduce that \(\frac{a}{b+1}<\frac{3}{2}\) and \(\frac{a}{b}>2\) or, equivalently,	\(\frac{b}{a}+\frac{1}{a}>\frac{2}{3}\) and \(\frac{b}{a}<\frac{1}{2}\).
	It follows that \(a<6\), and a direct check shows that there are no open intervals $(a/(b+1),a/b)$ with \(0<b<a<6\) containing both $\frac{3}{2}$ and $2$. Thus the right-hand endpoint of \(\sL_1\) is \(\infty\), which implies that \(M_1\) has an integer-homology sphere L-space filling. In turn, this implies that \(\sL_1 =[2g-1,\infty]\), where \(g\) is the genus of \(M_1\). Since 
  	\(\frac{3}{2},2\in \sL_1^\circ\) it must be that \(g=1\). 
  	
  Reversing the roles of \(M_0\) and \(M_1\), we see that \(\sL_0 = [1,\infty]\) or \(\sL_1=[-\infty,-1]\). In either case, both \(M_0\) and 
  \(M_1\) are prime, boundary incompressible and have two different L-space homology sphere fillings. It follows from work of Ghiggini \cite{Ghiggini2008} that they are both homeomorphic to trefoil complements. 
\end{proof}

We expect that there will only be finitely many  toroidal L-spaces with \(|H_1(Y)|=7\), but cannot prove it, since we can't classify Floer simple manifolds with the same Alexander polynomial as \(T(2,5)\). In contrast, it is possible to obtain infinitely many toroidal L-spaces with \(|H_1(Y)|=8\) by gluing the trefoil complement to the twisted $I$-bundle over the Klein bottle.

\subsection{Satellite L-space knots} Given a pattern knot $P$ and a companion knot $C$, both in the three-sphere, denote by $P(C)$ the result of forming a satellite knot (note that this depends on some additional choices). The following is a conjecture of  Hom, Lidman, and Vafaee  \cite[Conjecture 1.7]{HLV2014}:
\begin{conjecture}\label{con:sat}
If $P(C)$ is an L-space knot, then $P$ and $C$ are L-space knots as well. 
\end{conjecture}
To conclude the section we prove that this conjecture holds. 
\begin{proof}[Proof of Theorem \ref{thm:sat-intro}]
Consider the toroidal L-space $Y$ resulting from surgery on a satellite knot $K$ and write  $Y=M_0\cup_h M_1$.  By Theorem \ref{L-space-gluing-theorem}, $\sL_{M_0}^\circ\cup h(\sL_{M_1}^\circ)=\Q P^1$ so that the manifolds $M_0$ and $M_1$ are, in fact, Floer simple manifolds (that is, the sets $\sL_{M_i}^\circ$ are non-empty). In particular, $M_0$ the complement of an L-space knot $C$ in $S^3$ (the companion knot). Since the Seifert longitude $\lambda_C$ of $C$ is not an L-space slope, it must be that $\lambda_C\in h(\sL_{M_1}^\circ)$ so that $\alpha_P=h^{-1}(\lambda_C)$ gives an L-space $M_1(\alpha_P)$.    

Now consider the pattern knot $P$ in $D^2\times S^1$, obtained from $M_K\setminus M_0$, where the boundary of $D^2\times S^1$ is framed so that $\alpha_P\simeq \{\text{pt}\}\times S^1$. Note that filling this manifold along $\alpha_P$ gives a knot in $S^3$ (the pattern knot), which we will also denote by $P$. Further, $P$ must be an L-space knot since there is more than one choice $\alpha$, which existed by Floer simplicity of $M_K$ (and was required for the construction of $Y$). This establishes the conjecture.\end{proof}

There is a converse to this statement that only requires a weaker version of the L-space gluing theorem from \cite{HRRW}; see Hom \cite{Hom}.

\appendix
\section{Three-manifold invariants are extendable}

  \label{sec:Extension}
  
We outline the proof of  Theorem~\ref{thm:extension}, which is mainly a matter of assembling the correct elements from \cite{LOT} and checking that they still apply in our situation.  We will assume the reader is familiar with the notation and terminology of \cite{LOT}. 
  
  Suppose that \((M, \alpha_1, \alpha_2)\) is a compact oriented three-manifold with parametrized \(T^2\) boundary, and that \( \mathcal{H} = (\Sigma, \balpha, \bbeta)\) is a bordered Heegaard diagram representing it. (There is a notation clash here with our previous use of \(\alpha, \beta\) for the parametrization.) 
  
  Our first step is to describe the moduli spaces of pseudo-holomorphic maps used to defined the generalized coefficient maps \(D_I\). The extended torus algebra \(\widetilde{A}\) has a natural basis \(\langle \rho_I \rangle\), where each 
  \(\rho_I\) corresponds to a Reeb chord on the ideal contact boundary  \(\mathcal{Z} = \partial\overline{\Sigma}\).  Suppose that \(x, y\) are generators for  \(\mathcal{H}\), that \(B \in \widetilde{\pi}_2(x, y)\), and that \(\rhoarrow\) is a sequence of Reeb chords whose product is equal to \(\rho_I\). 
  (Here, the notation  \(\widetilde{\pi}_2\) indicates that we consider domains which may have nonzero multiplicity at \(z\), as in \cite[Section 10.2]{LOT}). 
  
  We consider decorated sources as in \cite[Definition 5.2]{LOT}, but with boundary punctures labeled by arbitrary Reeb chords \(\rho_I\) on \(\mathcal{Z}\), including those which pass through the basepoint \(z\). (Sources of this type are not sufficient for a bordered theory of \(\mathit{HFK}^-\), since they lack interior punctures, but they are good enough for our needs.) 
  Given a decorated source \(S^\triangleright\) 
   we can consider the moduli space \(\widetilde{\sM}^B(x,y;S^\triangleright)\) of pseudoholomorphic maps  \(U\co S \to \Sigma \times [0,1] \times \R\)  which represent the homology class \(B\)  as in  \cite[Definition 5.3]{LOT}, but omitting axiom (M-9) (the requirement that \(S\) has 0 multiplicity near the basepoint). Similarly, if \(\vec{P}\) is an ordering of the punctures of \(S\), we can consider the open subset  \(\widetilde{\sM}^B(x,y;S^\triangleright;\vec{P})\) for which projection to the \(\R\) direction induces the given ordering \(\vec{P}\). Finally, we let 
   \(\sM^B(x,y;S^\triangleright;\vec{P})\) be the corresponding reduced moduli space.

   Most of the analysis of pseudo-holomorphic curves in \cite[Sections 5.2-5.6]{LOT} carries over unchanged to this context, with the notable exception of the fact that boundary degenerations (see \cite[Lemma 5.48]{LOT})  can (and do) occur. 
   More specifically, the proof of \cite[Lemma 5.48]{LOT} 
   relies on the fact that \(\pi_2(\Sigma, \alpha) = 0\). Since we are allowing sources whose boundary maps to 
   \(\rho_0\), we must instead consider \(\widetilde{\pi}_2(\Sigma, \alpha) = \Z\). The generator of this group 
   is the domain \([\Sigma]\), whose boundary is  \( \sum_{i=0}^3 \rho_i\). (Notice that since we do not 
   allow sources with    interior punctures, \(\widetilde{\pi}_2(\Sigma, \beta)=0\).) 
   
   Now suppose we are given a sequence of Reeb chords \(\vec{\rho}\) and that \(B\) is compatible with \(\vec{\rho}\) in the sense of  
    \cite[Definition 5.68]{LOT}. Following \cite[Definition 5.68]{LOT}, we let \(\sM^B(x,y;\vec{\rho})\) 
    be the union of those moduli spaces 
   \(\sM^B(x,y;S^\triangleright)\) whose constituent pseudo-holomorphic curves are {\em embedded} and for which \([\vec{P}]= \vec{\rho}\).

  Finally, following  \cite[Definition 6.3]{LOT}, we let 
  $$ D_I(x) = \sum_{y} \sum_{\{\vec{\rho} |a(- \vec{\rho}) = \rho_I\}} \sum_{\{B| \mathrm{ind}(B, \vec{\rho})=1\}} \#(\sM^B(x,y;\vec{\rho})) y $$
  and define 
 \( \widetilde{\partial}(x) = \sum_I \rho_I \otimes D_I(x)\), where the sum runs over all indices \(I\) which contain at most one \(0\). 
 
 \begin{proof}[Proof of Theorem \ref{thm:extension}]
 We must show that \(\widetilde{\partial}^2(x) = W \otimes x\) modulo terms of the form 
\(\rho_I  \otimes  y\), where \(I\) contains at least two \(0\)'s. 
To do so, we argue as in the  proofs of \cite[Proposition 6.7]{LOT} and  \cite[Proposition 11.30]{LOT}. 
 Write
 $$ \widetilde{\partial}^2(x) = \sum_I \sum_ y n_{xy,I} \rho_I y. $$
 Fix \(y\) and \(I\) which contains at most one \(0\), and consider the moduli space
$$\sM(x,y;\rho_I) = \bigcup_{\{\vec{\rho}| a(-\vec{\rho})=\rho_I\}} \bigcup_{\{B |\mathrm{ind}(B,\vec{\rho})=2\} }
\sM^B(x,y;\vec{\rho})$$
As in the proof of \cite[Theorem 6.7]{LOT}, \(n_{xy,I}\) is the mod 2 number of ends
\(\sM(x,y;\rho_I)\) which do not correspond to boundary degenerations. 

Suppose that boundary degenerations appear as ends of \(\sM(x,y;\rho_I)\).
Index considerations dictate that such a degeneration must be of the form \(\varphi_1 \vee \varphi_2\), where \(\varphi_1 \in \sM^{B_1}\)
for some \(B_1 \in \pi_2(\Sigma, \alpha)\) with index \(k\), and \(\varphi_2 \in \sM^{B_2}\), where 
\(B_2 \in \pi_2(x,y)\) has index \(-k\). Assuming we have chosen a generic almost complex structure, all moduli spaces have their expected dimensions, and \(k=0\). 
Thus \(\varphi_2\) is a union of trivial strips. The domain of \(B_1\) is 
some  multiple of the generator of \(\pi_2(\Sigma, \alpha)\). Since \(0\) occurs in \(I\) at most once, 
the domain of \(B_1\) must be \(\Sigma\), which implies that \(|I|=4\). We conclude that if \(|I| \neq 4\) or
or \(x \neq y\), there are no boundary degenerations and \(n_{xy,I}=0\). The case \(|I|=4\) and \(x=y\) was 
studied by Lipshitz, Ozsv{\'a}th, and Thurston in  \cite[Proposition 11.30]{LOT}. They showed that 
the number of boundary degenerations in this case is equal to \(1 \mod 2\), so \(n_{xy,I}=1\). 
\end{proof}

\bibliographystyle{plain}
\bibliography{REF/bibliography}

\end{document}